\documentclass[10pt,letterpaper]{book}

\newif\ifUseBIBLATEX
\UseBIBLATEXtrue
\newif\ifbigMargin
\bigMargintrue

\ifbigMargin
    \usepackage[text={5in, 8.25in},marginparwidth=1.5in,centering]{geometry}
\else
    \usepackage[text={6.5in, 9.25in},centering]{geometry}
\fi

\usepackage{fancyhdr}
\fancypagestyle{revfrontmatter}{
    \pagenumbering{roman}
    \cfoot{\thepage}
}

\fancypagestyle{revstyle}{
    \pagenumbering{arabic}
    \fancyfoot[LE,RO]{Page \thepage}
    \fancyfoot[RE,LO]{arXiv Version 2}
    \renewcommand{\chaptermark}[1]{\markboth{ {\slshape{##1}}}{}}
    \fancyhead[RE,LO]{\text{\nouppercase{\rightmark}}}
    \fancyhead[LE,RO]{\textsl{\nouppercase{\leftmark}}}
    \cfoot{}
}

\usepackage[usenames,dvipsnames]{xcolor}
\usepackage{svg}

\usepackage{amssymb}
\usepackage{amsfonts}
\usepackage{amsmath}
\usepackage{amsthm}
\usepackage{thmtools}
\usepackage{thm-restate}
\usepackage{mathrsfs}
\usepackage{nicefrac}
\usepackage{xfrac}
\usepackage{scalerel} 

\usepackage{datetime}

\usepackage[
    colorlinks=true,
    linkcolor=blue,
    citecolor=black
]{hyperref}  
\usepackage[
    capitalize
]{cleveref}  
\newtheorem{theorem}{Theorem}[section]
\newtheorem{lemma}[theorem]{Lemma}

\newtheorem{proposition}[theorem]{Proposition}
\theoremstyle{remark}
\newtheorem{remark}[theorem]{Remark}
\crefname{lemma}{Lemma}{Lemmas}
\crefname{theorem}{Theorem}{Theorems}
\crefname{proposition}{Proposition}{Propositions}
\crefname{figure}{Figure}{Figures}
\crefname{chapter}{Section}{Sections}
\Crefname{chapter}{section}{sections}
\crefname{section}{Section}{Sections}
\Crefname{section}{section}{sections}
\crefname{subsection}{Section}{Sections}
\Crefname{subsection}{section}{sections}
\crefname{appendix}{Appendix}{Appendices}
\Crefname{appendix}{appendix}{appendices}

\usepackage{lmodern}

\usepackage{titlesec}
\titleformat{\chapter}[display]
  {\LARGE\bfseries\sffamily}{\chaptertitlename\ \thechapter}{20pt}{\Huge}
\titleformat*{\section}{\Large\bfseries\sffamily}
\titleformat*{\subsection}{\large\bfseries\sffamily}
\titleformat*{\subsubsection}{\bfseries\sffamily}
\titleformat*{\paragraph}{\itshape\sffamily}

\usepackage[normalem]{ulem}
\usepackage{lineno}

\newenvironment{rcases}
  {\left.\begin{aligned}}
  {\end{aligned}\right\rbrace}

\usepackage{graphicx}
\usepackage{placeins}
\usepackage[skip=6pt]{caption}
\usepackage{epsfig}
\usepackage{multicol,multirow}
\usepackage{changepage}
\usepackage{placeins}
\usepackage{setspace}
\usepackage{appendix}

\usepackage{url}

\ifUseBIBLATEX
    \usepackage[
    backend=biber,
    style=ext-alphabetic,
    giveninits=true,
    maxbibnames=99,
    minalphanames=3,
    maxalphanames=3,
    url=false
    ]{biblatex}
    \DeclareFieldFormat[article,inbook,incollection,inproceedings,patent,thesis,unpublished]{titlecase:title}{\MakeSentenceCase*{#1}}
    \AtBeginBibliography{\small}
    \AtEveryBibitem{%
      \clearfield{issn} 
      \clearfield{doi} 
      \ifentrytype{online}{}{
        \clearfield{url}
      }
    }
\else
    \newcommand{\cites}[1]{\cite{#1}}
\fi

\usepackage{algorithm}
\usepackage{algorithmicx}   
\usepackage[noend]{algpseudocode}  
\algdef{SE}[SUBALG]{Indent}{EndIndent}{}{\algorithmicend\ }
\algtext*{Indent}
\algtext*{EndIndent}
\algdef{SE}[DOWHILE]{Do}{doWhile}{\algorithmicdo}[1]{\algorithmicwhile\ #1}%
\newcommand{\codecomment}[1]{\textcolor{gray}{\# #1}}

\usepackage{bm} 
\providecommand{\mathbold}[1]{\bm{\mathsf{#1}}}
\newcommand{\R}{\mathbb{R}}

\newcommand{\E}{\mathbb{E}}
\newcommand{\Herm}{\mathbb{H}}
\newcommand{\vct}[1]{\bm{#1}}
\newcommand{\Ao}{\mtx{A}}
\newcommand{\Aa}{\mtx{\hat{A}}}
\newcommand{\Alift}{\mtx{A}_{\mu}}

\newcommand{\mtx}[1]{\mathbold{#1}}

\usepackage{stmaryrd}
\newcommand{\llb}{\llbracket}
\newcommand{\rrb}{\rrbracket}
\newcommand{\idxs}[1]{{\llb #1 \rrb}}
\newcommand{\trans}{*}
\newcommand{\fslice}{{:}}  
\newcommand{\lslice}[1]{{:}{#1}}  
\newcommand{\tslice}[1]{{#1}{:}}  

\DeclareMathOperator{\range}{range}
\DeclareMathOperator{\cond}{cond}
\DeclareMathOperator{\rank}{rank}
\DeclareMathOperator{\diag}{diag}
\DeclareMathOperator{\trace}{tr}

\newcommand{\e}{\epsilon}
\makeatletter
\newcommand{\doublehat}[1]{%
\begingroup%
  \let\macc@kerna\z@%
  \let\macc@kernb\z@%
  \let\macc@nucleus\@empty%
  \hat{\raisebox{.25ex}{\vphantom{\ensuremath{#1}}}\smash{\hat{#1}}}%
\endgroup%
}
\makeatother
\usepackage{float}
\floatstyle{ruled}
\newfloat{method}{htbp}{lok}
\floatname{method}{Method}
\crefname{method}{Method}{Methods}
\Crefname{method}{Method}{Methods}

\DeclareMathOperator{\vectorize}{vec}
\DeclareMathOperator{\FFT}{DFT}
\DeclareMathOperator*{\startimes}{\scalerel*{\circledast}{\sum}}
\DeclareMathOperator*{\argmin}{arg\,min}
\DeclareMathOperator*{\argmax}{arg\,max}

\DeclareMathOperator{\Span}{span}

\newcommand{\LDLt}{LDL}

\newcommand{\RandNLA}{RandNLA}
\newcommand{\Nystrom}{Nystr\"{o}m}
\newcommand{\RandLAPACK}{\textsf{RandLAPACK}}
\newcommand{\RBLAS}{\textsf{RandBLAS}}
\newcommand{\RandBLAS}{\RBLAS{}}

\newcommand{\BLAS}{\textsf{BLAS}}
\newcommand{\BLASlev}[1]{\textsf{BLAS #1}}
\newcommand{\LINPACK}{\textsf{LINPACK}}
\newcommand{\PLASMA}{\textsf{PLASMA}}
\newcommand{\LAPACK}{\textsf{LAPACK}}
\newcommand{\LAPACKpp}{\textsf{LAPACK++}}
\newcommand{\Ristretto}{\textsf{Ristretto}}
\newcommand{\SciPy}{\textsf{SciPy}}

\newcommand{\code}[1]{\texttt{#1}}
\newcommand{\MAGMA}{\textsf{MAGMA}}
\newcommand{\SLATE}{\textsf{SLATE}}
\newcommand{\ScaLAPACK}{\textsf{ScaLAPACK}}

\newcommand{\RSVDPACK}{\textsf{RSVDPACK}}
\newcommand{\IDLib}{\textsf{ID}}
\newcommand{\RandomOneTwo}{\textsf{Random123}}
\newcommand{\LibSkylark}{\textsf{LibSkylark}}
\newcommand{\LowRankApprox}{\textsf{LowRankApprox.jl}}
\newcommand{\Elemental}{\textsf{Elemental}}
\newcommand{\Tensorlab}{\textsf{Tensorlab}}
\newcommand{\cuSOLVE}{\textsf{cuSOLVE}}
\newcommand{\SciKitCUDA}{\textsf{SciKit-CUDA}}

\usepackage{overpic}
\usepackage{subcaption}
\let\emph\relax 
\DeclareTextFontCommand{\emph}{\itshape}

\usepackage{etoc}
\usepackage{enumitem}

\renewcommand*\etoctoclineleaders
    {\hbox{\normalfont\normalsize\hbox to .75ex {\hss~.~\hss}}}
\newcommand*{\EndEtocSectionEntry}
    {\leaders\etoctoclineleaders\hfill
     \marginpar{\hspace{-1.75cm}\makebox[0.75cm][r]{\bfseries\normalsize\etocpage}}%
     \par }
\newcommand*{\EndEtocSubsectionEntry}
    {\leaders\etoctoclineleaders\hfill
     \marginpar{\hspace{-1.75cm}\makebox[0.75cm][r]{\mdseries\normalsize\etocpage}}%
     \par }

\newcommand{\setminitocstyle}{

    \etocsettocstyle
    {\noindent\rule{\linewidth}{0.5pt}
    \begin{adjustwidth}{0.75cm}{1.5cm}}
    {\vspace{-0.5em}\end{adjustwidth}\noindent\rule{\linewidth}{0.5pt}}
    
    \etocsetstyle{section}
    {
    	
    	\begin{itemize}[itemsep=0.5ex, parsep=0ex, topsep=0ex]
    }
    {\normalsize\bfseries\rmfamily\item[\etocnumber{}]}
    {\etocname{} \EndEtocSectionEntry{}}
    {\end{itemize}}
    
    \etocsetstyle{subsection}
    {
    	\begin{itemize}[itemsep=0.25ex, topsep=0ex, parsep=0ex]
    }
    {\normalfont\item[\etocnumber{}]}
    {\etocname{} \EndEtocSubsectionEntry{}}
    {\end{itemize}}

}

\newcommand{\minitoc}{
    \setminitocstyle{}
    \localtableofcontents{}
}
\newcommand{\dominitoc}{{}}

\newcommand{\jwdnote}[1]{$$$}
\newcommand{\rjmnote}[1]{$$$}
\newcommand{\lgnote}[1]{$$$}
\newcommand{\michael}[1]{$$$}
\newcommand{\omnote}[1]{$$$}
\newcommand{\hlnote}[1]{$$$}
\newcommand{\michal}[1]{$$$}
\newcommand{\piotr}[1]{$$$}

\title{Randomized Numerical Linear Algebra \\ {\Large{A Perspective on the Field With an Eye to Software}} }
    
\date{April 12, 2023} 

\begin{document}

\pagenumbering{gobble}

\maketitle

\chapter*{Preface}

Randomized numerical linear algebra -- \RandNLA{}, for short -- concerns the use of randomization as a resource to develop improved algorithms for large-scale linear algebra computations.
The origins of contemporary \RandNLA{} lay in theoretical computer science, where it blossomed from a simple idea: randomization provides an avenue for computing \textit{approximate} solutions to linear algebra problems more efficiently than deterministic algorithms.
This idea proved fruitful in and was largely driven by the development of scalable algorithms for machine learning and statistical data analysis applications.
However, the true potential of \RandNLA{} only came into focus once it began to integrate with the fields of numerical analysis and ``classical'' numerical linear algebra.
Through the efforts of many individuals, randomized algorithms have been developed that provide full control over the accuracy of their solutions and that can be every bit as reliable as algorithms that might be found in libraries such as \LAPACK{}. 

The spectrum of possibilities offered by \RandNLA{} has created a virtuous cycle of contributions by numerical analysts, statisticians, theoretical computer scientists, and the machine learning community.
Recent years have even seen the incorporation of certain \RandNLA{} methods into MATLAB, the NAG Library, NVIDIA's cuSOLVER, and SciKit-Learn.
In view of these developments, we believe the time is right to accelerate the adoption of \RandNLA{} in the scientific community.
In particular, we believe the community stands to benefit significantly from a suitably defined  ``\RandBLAS{}'' and ``\RandLAPACK{},'' to serve as standard libraries for \RandNLA{}, in much the same way that \BLAS{} and \LAPACK{} serve as standards for deterministic linear algebra.

This monograph surveys the field of \RandNLA{} as a step toward building meaningful \RandBLAS{} and \RandLAPACK{} libraries.
\cref{sec1:intro} primes the reader for a dive into the field and summarizes this monograph's content at multiple levels of detail.
\cref{sec2:rblas} focuses on \RandBLAS{}, which is to be responsible for \textit{sketching}.
Details of functionality suitable for \RandLAPACK{} are covered in the five \nameCrefs{sec2:rblas} that follow.
Specifically, \cref{sec3:LS_and_optim,sec4:lowrank,sec5:more_drivers} cover least squares and optimization, low-rank approximation, and other select problems that are well-understood in how they benefit from randomized algorithms.
The remaining \nameCrefs{sec7:lev_scores} -- on statistical leverage scores (\cref{sec7:lev_scores}) and tensor computations (\cref{sec8:tensors}) -- read more like traditional surveys.
The different flavor of these latter \nameCrefs{sec7:lev_scores} reflects how, in our assessment, the literature on these topics is still maturing.

We provide a substantial amount of pseudo-code and supplementary material over the course of five appendices.
Much of the pseudo-code has been tested via publicly available Matlab and Python implementations.

\clearpage

\pagestyle{revfrontmatter}

\subsection*{Authors}

Riley Murray, ICSI, LBNL, and University of California, Berkeley\\
rjmurray@berkeley.edu \\[0.70em]
James Demmel, University of California, Berkeley\\
demmel@berkeley.edu\\[0.70em]
Michael W. Mahoney, ICSI, LBNL, and University of California, Berkeley\\
mmahoney@stat.berkeley.edu \\[0.70em]
N. Benjamin Erichson, ICSI and Lawrence Berkeley National Laboratory\\
erichson@icsi.berkeley.edu\\[0.70em]
Maksim Melnichenko, University of Tennessee, Knoxville\\
mmelnic1@vols.utk.edu\\[0.70em]
Osman Asif Malik, Lawrence Berkeley National Laboratory\\
oamalik@lbl.gov \\[0.70em]
Laura Grigori, INRIA Paris and J.L. Lions Laboratory, Sorbonne University\\
laura.grigori@inria.fr\\[0.70em]
Piotr Luszczek,  University of Tennessee, Knoxville\\
luszczek@icl.utk.edu\\[0.70em]
Micha{\l} Derezi\'nski, University of Michigan\\
derezin@umich.edu\\[0.70em]
Miles E. Lopes, University of California, Davis\\
melopes@ucdavis.edu\\[0.70em]
Tianyu Liang, University of California, Berkeley\\
tianyul@berkeley.edu\\[0.70em]
Hengrui Luo, Lawrence Berkeley National Laboratory\\
hrluo@lbl.gov \\[0.70em]
Jack Dongarra, University of Tennessee, Knoxville\\
dongarra@icl.utk.edu

\clearpage
\subsection*{Acknowledgements}

Many individuals from the community gave detailed feedback on earlier versions of this monograph that were not circulated publicly.
These individuals include Mark Tygert, Cameron Musco, Joel Tropp, Per-Gunnar Martinsson, Alex Townsend, Daniel Kressner, Alice Cortinovis, Ilse Ipsen, Sergey Voronin, Vivak Patel, Daniel Maldonado, Tammy Kolda, Florian Schaefer, Ramki Kannan, and Piyush Sao -- each of them has our sincere gratitude for their assistance.

In addition, we thank the following people for providing input on the earliest stages of this project:  Vivek Bharadwaj, Younghyun Cho, Jelani Nelson, Mark Gates, Weslley da Silva Pereira, Julie Langou, and Julien Langou.

This work was partially funded by an NSF Collaborative Research Framework: Basic ALgebra LIbraries for Sustainable Technology with Interdisciplinary Collaboration (BALLISTIC), a project of the International Computer Science Institute, the University of Tennessee's ICL, the University of California at Berkeley, and the University of Colorado at Denver (NSF Grant Nos.\ 2004235, 2004541, 2004763, 2004850, respectively) \cite{BALLISTIC}.
Any opinions, findings, and conclusions or recommendations expressed in this material are those of the author(s) and do not necessarily reflect the views of the National Science Foundation.
MWM would also like to thank the Office of Naval Research, which provided partial funding via a Basic Research Challenge on Randomized Numerical Linear~Algebra.

\subsection*{Release History}

\begin{itemize}
    \item 11/13/2022. The first publicly-available draft of this monograph was circulated as a technical report at SC22.
    \item 02/22/2023: arXiv V1.
    Improvements were made to Section 3.2.2 based on comments from Joel Tropp.
    Helpful feedback from Robert Webber led to improvements throughout Sections 4 and 5.2.
    Clarifications and greater detail were added to Appendix B.2 following comments from Ilse Ipsen.
    \item 04/12/2023: arXiv V2. Section 5.1 was slightly revised based on valuable comments from Oleg Balabanov.
    Section 5.3 was rewritten and expanded following helpful discussions with Tyler Chen.
    Sections 4.1.1 and 4.5 now mention an important piece of software that Mark Tygert brought to our attention.
    Revisions were made to Section 5.2.2 to more clearly and accurately characterize methods from the literature.
\end{itemize}

\setcounter{tocdepth}{1}
\dominitoc
\tableofcontents
\setcounter{tocdepth}{2}

\clearpage
\pagestyle{revstyle}
\chapter{Introduction}
\label{sec1:intro}

\minitoc

\label{page:sec1_toc}

This introductory section has three principal goals: 
to motivate our subject and clarify common misconceptions that surround it (\S \ref{subsec:sec1:our_world}); 
to explain this monograph's scope and overarching structure (\S \ref{subsec:astro_view}); 
and to help direct the reader's attention through section-by-section summaries (\S \ref{subsec:sec1:birds_eye}).
Many readers may benefit from our ``survey of surveys'' (\S \ref{subsec:recommended_reading}), and all should at least briefly consult the section on notation and definitions (\S \ref{subsec:notation}).

\section{Our world}\label{subsec:sec1:our_world}
%
%
Numerical linear algebra (NLA) concerns algorithms for computations on matrices with numerical entries.
Originally driven by applications in the physical sciences, it now provides the foundation for vast swaths of applied and computational mathematics. 
The cultural norms in this field developed many years ago, in large part from recurring themes in problem formulations and algorithmically-useful structures in matrices.
However, more recently, NLA has also been motivated by developments in machine learning and data science.
Applications in these fields also have their own themes of problem formulations and structures in data, often of a very different nature than those in more classical applications.

\paragraph{A dire situation.}
While communities that rely on NLA now vary widely, they share one essential property: a ravenous appetite for solving larger and larger problems.
For decades, this hunger was satiated by complementary innovations in hardware and software.
However, this progress should not be taken for granted.
In particular, there are two factors that increasingly present obstacles to scaling linear algebra computations to the next~level. 
\begin{itemize}
    \item \textit{Space and power constraints in hardware.}
    Chips today have billions of transistors, and these transistors are packed into a very small amount of space.
    It takes power to run these transistors at gigahertz frequencies, and power generates heat.
    It is hard for one hot thing to dissipate heat when surrounded by millions of other hot things.
    Too much heat can fry a chip.

    These constraints are known to industry and research community alike,
    and often referred to as the breakdown of the Dennard's
    Law~\cite{Dennard:scaling,Bohr:30yearDennard}
    and the sunsetting of Moore's Law~\cite{Moore:1965:cramming}.
    The former represents the infamous \emph{power wall} due to the
    inability of dissipating the heat produced by the processors leading to
    flattened curve of clock frequency increases.
    The latter introduced the post-Moore era of heterogeneous computing~\cite{vetter:2018:hetero} leading to plethora specialized hardware targeting individual application spaces.
    
    The end result of all this?
    A situation where ``more powerful processors'' are just scaled-out
    versions of ``less powerful processors,'' for both commodity and
    server-tier hardware.
    Any algorithm that does not parallelize well is fundamentally limited in its ability to leverage these advances.
    If one’s pockets are deep enough, then one can try to get around this with purpose-built accelerators.
    But even then, there remains the matter of programming those accelerators, and high-performance implementations of classical NLA algorithms are anything but~simple.
    
    \item \textit{NLA's maturity as a field.}
    Software can only improve so much without algorithmic innovations.
    At the same time, linear algebra is a very well-studied topic, and most algorithmic breakthroughs in recent years have required carefully exploiting structures present in specific problems.
    Identifying new and useful problem structures has been increasingly difficult, often requiring deep knowledge of NLA alongside substantial domain expertise.
\end{itemize}
If we are to continue scaling our capabilities in matrix computations, then it is essential that we leverage all technologies that are on the table.

\paragraph{An underutilized technology.}
This monograph concerns \textit{randomized numerical linear algebra}, or \textit{\RandNLA{}}, for short.
Algorithms in this realm offer compelling advantages in a wide variety of settings.
Some provide an unrivaled combination of efficiency and reliability in computing approximate solutions for massive problems.
Others provide fine-grained control when balancing accuracy and computational cost, as is essential for practitioners who are operating at the limits of what their machines can handle.
In many cases, the practicality of these algorithms can be seen even with elementary MATLAB or Python implementations, which increases their suitability for adapting to new hardware by leveraging similarly powerful abstraction layers.
Finally, although truly high-performance implementations are more complicated, they remain relatively easy to implement when given the right building blocks.

But we are getting ahead of ourselves.
What do we mean by ``randomized algorithms,'' as the term is used within \RandNLA{}?
First and foremost, these are algorithms that are probabilistic in nature.
They use randomness as part of their internal logic to make decisions or compute estimates, which they can go on to use in any number of ways.
These algorithms do not presume a distribution over possible inputs, nor do they assume the inputs somehow possess intrinsic uncertainty.
Rather, they use randomness as a tool, to find and exploit structures in problem data that would seem ``hidden'' from the perspective of classical NLA.

\paragraph{What's this about ``finding hidden structures?''}\label{page:hidden_structures}
Consider the problem of highly overdetermined least squares, i.e., the problem of solving 
\begin{equation}\label{eq:sec1:ols}
    \min_{\vct{x} \in \R^n} \|\mtx{A}\vct{x} - \vct{b}\|_2^2,
\end{equation}
where $\mtx{A}$ has $m \gg n$ rows.
It is well-known that if the columns of $\mtx{A}$ are orthonormal then \eqref{eq:sec1:ols} can be solved in $O(mn)$ time by setting $\vct{x} = \mtx{A}^{\trans}\vct{b}$, where $\mtx{A}^{\trans}$ is the transpose of $\mtx{A}$.
The trouble, of course, is that the columns of $\mtx{A}$ are very unlikely to be orthogonal in any interesting application, 
and the standard algorithms for solving this problem take $O(mn^2)$ time.

However, what if, by some miracle, we could easily find
an $n \times n$ matrix $\mtx{C}$ for which $\mtx{A}\mtx{C}^{-1}$ was column-orthonormal?
In this case, we could compute the exact solution $\vct{x} = (\mtx{C}^{-1})(\mtx{C}^{-1})^{\trans}\mtx{A}^{\trans}\vct{b}$ in time
\begin{equation*}
     O\left(mn + n^3\right)
\end{equation*}
by suitably factoring $\mtx{C}$.
Now, randomized algorithms do not work miracles, but at times they can \textit{approach} the miraculous.
In the specific case of \eqref{eq:sec1:ols}, randomization can be used to quickly identify a basis in which $\mtx{A}$ is \textit{nearly} column-orthonormal.
Any such basis can be incorporated into a standard iterative method from classical NLA.
With such an approach, one can reliably solve \eqref{eq:sec1:ols} to $\epsilon$-error in time
\begin{equation*}
    O\left(mn \log\left(\tfrac{1}{\epsilon}\right) + n^3\right) 
\end{equation*}
(where we ask forgiveness for being vague about the meaning of ``$\epsilon$'').

%
%


\paragraph{Back to the big picture.}
The approach to least squares described above has been known for well over ten years now.
Since then, an entire suite of compelling results on \RandNLA{} has been established.
What's more, the literature also documents the existence of high-performance proof-of-concept implementations that testify to the practicality of these methods.
Indeed, as we explain below, randomized algorithms have been developed to address the same basic challenges as classical methods.
Randomized algorithms are also very well-suited to address many upcoming challenges with which classical algorithms struggle. 
\RandNLA{} as a field is slowly achieving a certain level of maturity.

Despite this, substantial interdisciplinary gaps have impeded technology transfer from those doing research in \RandNLA{} to those who might benefit from it.
This stems partly from the absence of work that organizes the \RandNLA{} literature in a way that supports the development of high-quality software.

This monograph is our attempt at addressing that absence.
With it, we aim to provide a principled and practical foundation for developing high-quality software to address future needs of large-scale linear algebra computations, for scientific computing, large-scale data analysis and machine learning, and other related applications.
Our particular approach is informed by plans to develop such high-quality libraries -- a ``\RandBLAS{}'' and ``\RandLAPACK{},'' if you will.
Towards this end, we have implemented and tested many of the algorithms described herein in both MATLAB\footnote{\url{https://github.com/BallisticLA/MARLA}} and Python.\footnote{\url{https://github.com/BallisticLA/PARLA}}
We provide more context on our approach and scope in \cref{subsec:astro_view}.
But first, we elaborate on the value propositions of \RandNLA{} and the role of randomness in these algorithms.

\subsection{Four value propositions of randomization}\label{subsec:value_prop_randnla}

Our goal here is to introduce (and only introduce!) some value propositions for \RandNLA{}.
We do this for as broad an audience as possible and we have attempted to keep our introductions short.
While these descriptions are unlikely to convince a skeptic, they should at least set the agenda for a debate.

\paragraph{Background: time complexity and FLOP counts.}
In the sequential RAM model of computing, 
an algorithm's \textit{time complexity} is its worst-case total number of reads, writes, and elementary arithmetic operations, as a function of input size.
Precise expressions for time complexity can be hard to come by and difficult to parse.
Therefore it is standard to describe complexity asymptotically, with big-$O$ notation.

In NLA, we also care about how the size of an algorithm's input affects the number of arithmetic operations that it requires.
Arithmetic operations are presumed to be floating point operations (``flops'') by default, and it is common to refer to an algorithm's \textit{flop count} as a function of input size.
Flop counts almost always agree asymptotically with time complexity.
But, in contrast with time complexity, flop counts are often given with explicit constant~factors.
In determining the constant factors accurately one must consider subtleties such as whether a fused multiply-add instruction counts as one or two ``flops.''
We do not consider such subtleties in this monograph.
%
%

\paragraph{Fighting the scourge of superlinear complexity.}

The dimensions of matrices arising in applications typically have semantic meanings.
They might represent the number of points in a dataset, or they might be affected by the ``fidelity'' of a linear model for some nonlinear phenomenon.
A scientist who relies on matrix computations will inevitably want to increase the size of their dataset or the fidelity of their model.
This is often difficult because the complexity of classical algorithms for high-level linear algebra problems rarely scale linearly with the semantic notion of problem size.
This brings us to the first value proposition of \RandNLA{}.

\begin{quote}
    \emph{For many important linear algebra problems, randomization offers entirely new avenues of computing approximate solutions with linear or nearly-linear complexity.}
\end{quote}

To get a sense of why this matters, suppose that one needs to compute a Cholesky decomposition of a dense matrix of order $n$.
The standard algorithm for this takes $n^3 / 3$ flops.
At time of writing, a higher-end laptop can do this calculation for $n = 10{,}000$ in about one second.
However, if $n$ is the semantic notion of problem size, and if one wants to solve a problem ten times as large, then the calculation with $n = 100{,}000$ takes over 15 minutes.

There are two lessons in that simple example.
The first is that superlinear complexity can be crippling when it comes to solving larger linear algebra problems.
The second is that an informed user does well to think of their problem size in a more realistic way.
In the case of this example, one should go into the problem thinking in terms of the number of free parameters in an $n \times n$ positive definite matrix: around 50 million when $n = 10{,}000$ and around 5 billion when $n = 100{,}000$.

\begin{remark}
    One of \RandNLA{}'s success stories is a fast algorithm for computing sparse approximate Cholesky decompositions of so-called \textit{graph Laplacians}.
    In order to keep the length of this monograph under control, we have opted \textit{not} to include algorithms that only apply to sparse matrices.
    However, we do provide algorithms for computing approximate eigendecompositions of regularized positive semidefinite matrices, and these algorithms can be used to solve linear systems faster than Cholesky in certain applications.
\end{remark}

\paragraph{Resisting the siren call of galactic algorithms.}
The problem of multiplying two $n \times n$ matrices is one of the most fundamental in all of NLA.
If we only consider asymptotics, then the fastest algorithms for this
task run in less than $O(n^{2.38})$ time.
%
%
%
However, the fastest method that is practical (Strassen's algorithm), runs in time $O(n^{\log_{2}7})$.

The trouble with these ``fast'' algorithms is that they have massive constants hidden in their big-$O$ complexity.
Such algorithms are called \textit{galactic}, owing to common comparisons between the size of their hidden constants and the number of stars or atoms in the galaxy.
And with this, we arrive at the second value proposition of \RandNLA{}.
\begin{quote}
    \emph{For a handful of important linear algebra problems, the asymptotically fastest (non-galactic) algorithms for computing accurate solutions are, in fact, randomized.}
\end{quote}
\noindent Highly overdetermined least squares (see page \pageref{page:hidden_structures}) is one such problem.

\paragraph{Striking the Achilles' heel of the RAM model.}
The RAM model of computing, although useful, is not high-fidelity.
Indeed, even in the setting of a shared-memory multi-core machine, it fails to account for the fact that moving data from main memory, through different levels of cache, and onward to processor registers is \textit{much} more expensive than elementary arithmetic on the same data.
This fact has been appreciated even in the earliest days of \LAPACK{}'s development, over 30 years ago.
Its principal consequence is that even if the time complexities of two algorithms match up to \textit{and including} constant factors, their performance by wallclock time can differ by orders of magnitude.
This is the third value proposition of \RandNLA{}.

\begin{quote}
    \emph{Randomization creates a wealth of opportunities to reduce and redirect data movement.
    Randomized algorithms based on this principle are significantly faster than the best-available deterministic methods by wallclock time.}
\end{quote}
\noindent Randomized algorithms for computing full QR decompositions with column pivoting fit this description.

\paragraph{Finite-precision arithmetic: once a curse, now a blessing.}
Finite-precision arithmetic and exact arithmetic are different beasts, and this has real consequences for NLA.
For one thing, this limitation introduces many technicalities in understanding accuracy guarantees, even for seemingly straightforward problems like LU decomposition.
It is tempting to view it as a curse.
However, if we accept it as given, then it can be used to our advantage.
Certain computations can be performed with lower precision without compromising the accuracy of a final result.

This perspective brings us to our final value proposition, stated in terms of the concept of sketching, defined momentarily.

\begin{quote}
    \emph{In \RandNLA{}, it is natural to perform computations on sketches of matrices in lower-precision arithmetic.
    Depending on how the sketch is constructed, one can be (nearly) certain of avoiding degenerate situations that are known to cause common deterministic algorithms to fail.}
\end{quote}

\subsection{What is, and isn't, subject to randomness}\label{subsec:sec1:what_is_and_isnt_random}


\paragraph{Sampling sketching operators from sketching distributions.}
We are concerned with algorithms that use random linear dimension reduction maps called \textit{sketching operators}.
The sketching operators used in \RandNLA{} come in a wide variety of forms.
They can be as simple as operators for selecting rows or columns from a matrix, and they can be even more complicated than algorithms for computing Fast Fourier Transforms.
We refer to a distribution over sketching operators as a \textit{sketching distribution}.
Given this terminology, we can highlight the following essential fact.
\begin{quote}
    \emph{For the vast majority of \RandNLA{} algorithms, randomization is only used when sampling from the sketching distribution.}
\end{quote}
From an implementation standpoint, one should know that while sketching distributions can be quite complicated, the sampling process always builds on some kind of basic random number generator.
Upon specifying a seed for the random number generator involved in sampling, \RandNLA{} algorithms become every bit as deterministic as classical algorithms.

\paragraph{Forming and processing sketches.}
When a sketching operator is applied to a large data matrix, it produces a smaller matrix called a \textit{sketch}.
A wealth of different outcomes can be achieved through different methods for processing a sketch and using the processed representation downstream.
\begin{quote}
    \emph{Some processing schemes inevitably yield rough approximations to the solution of a given problem.
    Other processing schemes can lead to high-accuracy approximations, if not exact solutions, under mild assumptions.}
\end{quote}
Across these regimes, one of the most popular trends in algorithm analysis is to employ a two-part approach.
In the first part, the task is to characterize algorithm output in terms of some simple property of the sketch.
In the second part, one can employ results from random matrix theory to bound the probability that the sketch will possess the desired property.

\paragraph{Confidently managing uncertainty.}
The performance of a numerical algorithm is characterized by the accuracy of its solutions and the cost it incurs to produce those solutions.
Naturally, one can expect some variation in algorithm performance when using randomized methods.
Luckily, we have the following.
\begin{quote}
    \emph{Most randomized algorithms ``gamble'' with only one of the two performance metrics, accuracy or cost.
    Through optional algorithm parameters, users retain fine-grained control over one of these two metrics.
    }
\end{quote}
Furthermore, when cost is controllable, the algorithm parameters can be adjusted to influence accuracy; when accuracy is controllable, they can be adjusted to influence cost.
The effects of these influences can sometimes be masked by variability in run-to-run performance.
However, there is a general trend in \RandNLA{} algorithms of becoming more predictable as they are applied to larger problems.
At large enough scales, many randomized algorithms are nearly as predictable as deterministic ones.


\section{This monograph, from an astronaut's-eye view
}
\label{subsec:astro_view}

This monograph started as a development plan for two C\texttt{++} libraries for \RandNLA{}, primarily working within a shared-memory dense-matrix data model.
We prepared a preliminary plan for these libraries in short order by leveraging existing surveys.
However, after pausing our writing for some number of months to receive feedback from community members, we found ourselves with many unanswered questions that would affect our implementations.
Before long we found ourselves in a cycle of diving ever-deeper into the \RandNLA{} literature with an eye to implementation, each time coming up with more answers and more questions.

This monograph does not answer every question we came across in the foregoing months.
Rather, it represents what we know at a time when the best way to answer our remaining questions is to focus on developing the libraries themselves.
Therefore we provide the reader with this --- a monograph that aggregates material from over 300 references on classical and randomized NLA --- which functions partly as a survey and partly as original research.
In it, we present new (unifying) taxonomies, candidate application areas, and even a handful of novel theoretical results and~algorithms.

Although its scope has greatly increased, the original purpose of this monograph informs its structure.
It also contains a number of clear statements about plans for our C\texttt{++} libraries.
Therefore, while we do not want to give the impression that this monograph's value depends on its connections to specific pieces of software, we provide the following remarks on our planned libraries up-front.

\begin{quote}
The first library, \RandBLAS{}, concerns basic sketching and is the subject of \cref{sec2:rblas}.
Our hope is that \RandBLAS{} will grow to become a community standard for \RandNLA{}, in the sense that its API would see wider adoption than any single implementation thereof.
In order to achieve this goal we think it is important to keep its scope narrowly focused.

The second library, \RandLAPACK{}, concerns algorithms for traditional linear algebra problems (\cref{sec3:LS_and_optim,sec4:lowrank,sec5:more_drivers}, on least-squares and optimization, low-rank approximation, and additional possibilities, respectively) and advanced sketching functionality (\cref{sec7:lev_scores,sec8:tensors}).
The design spaces of algorithms for these tasks are \textit{large}, and
we believe that powerful abstractions are needed for a library to leverage this fact.
Consistent with this, we are developing \RandLAPACK{} in an object-oriented programming style wherein \textit{algorithms are objects}.
Such a style is naturally instantiated with functors when working in C\texttt{++}.
\end{quote}

We have written this monograph to be modular and accessible, without sacrificing depth.
The modularity manifests in how there are almost no technical dependencies across \cref{sec3:LS_and_optim,sec4:lowrank,sec5:more_drivers,sec7:lev_scores,sec8:tensors}.
For the sake of accessibility, each \nameCref{sec3:LS_and_optim} gives background on its core subject.
We use two strategies to provide accessibility without sacrificing depth.
First, we make liberal use of appendices.
In them, the reader can find proofs, background on special topics, low-level algorithm implementation notes, and high-level algorithm pseudocode.
Second, our citations regularly indicate precisely where a given concept can be found in a manuscript. 
Therefore, if we give too brief a treatment on a topic of interest, the reader will know exactly where to look to learn more.

\subsubsection{A word on ``drivers'' and ``computational routines''}
We designate most algorithms as either \textit{drivers} or \textit{computational routines}.
These terms are borrowed from \LAPACK{}'s API.
In general, drivers solve higher-level problems than computational routines, and their implementations tend to use a small number of computational routines.
In our context,
\begin{quote}
    \textit{drivers} are only for traditional linear algebra problems,
\end{quote}
while
\begin{quote}
    \textit{computational routines} address a mix of traditional linear algebra problems and specialized problems that are only of interest in \RandNLA{}.
\end{quote}
\cref{sec3:LS_and_optim,sec4:lowrank} cover drivers \textit{and} the computational routines behind them; they are the most comprehensive \nameCrefs{sec1:intro} in this monograph.
\cref{sec5:more_drivers} also covers drivers, but at less depth than the two that precede it.
In particular, it does not identify algorithmic building blocks that would be considered computational routines.
Meanwhile, the advanced sketching functionality in \cref{sec7:lev_scores,sec8:tensors} would \textit{only} be considered for computational routines.


One reason why we use the ``driver'' and ``computational routine'' taxonomy is to push much of the \RandNLA{} design space into computational routines.
This is essential to keeping drivers simple and few in number.
However, it has a side effect: since choices made in the computational routines decisively affect the drivers, it is hard to state theoretical guarantees for the drivers without being prescriptive on the choice of computational routine.
This is compounded by two factors.
First, we prefer to \textit{not} be prescriptive on choices of computational routines within drivers, since there is always a possibility that some problems benefit more from some approaches than others.
Second, even if we recommended specific implementations, it would be very complicated to characterize their performance with consideration to the full range of possibilities for their tuning parameters.

As a result of all this, we make relatively few statements about performance guarantees or computational complexity of driver-level algorithms.
While this is a limitation of our approach, we believe it is not severe.
One can supplement this monograph with a variety of resources discussed in \cref{subsec:recommended_reading}.

\section{This monograph, from a bird's-eye view}
\label{subsec:sec1:birds_eye}

Section-by-section summaries are provided below to help direct the reader's attention.
While space limitations prevent them from being comprehensive, they are effective for what they are.
They assume familiarity with standard linear algebra concepts, including least squares models, singular value decomposition, Hermitian matrices, eigendecomposition, and positive (semi)definiteness.
We define all of these concepts in \cref{subsec:notation} for completeness.
Finally, as one disclaimer, some problem formulations below have slight differences from those used in the sections themselves.

\paragraph{Essential notation and conventions.}
The adjoint of a linear operator $\mtx{A}$ is denoted by $\mtx{A}^{\trans}$. 
When $\mtx{A}$ is a real matrix, the adjoint is simply the transpose.
Vectors have column orientations by default, so the standard inner product of two vectors $\vct{u},\vct{v}$ is $\vct{u}^{\trans}\vct{v}$.

We sometimes call a vector of length $n$ an \textit{$n$-vector}.
If we refer to an $m \times n$ matrix as ``tall'' then the reader can be certain that $m \geq n$ and reasonably expect that $m > n$.
If $m$ is much larger than $n$ and we want to emphasize this fact, then we write $m \gg n$ and would call an $m \times n$ matrix ``very tall.''
We use analogous conventions for ``wide'' and ``very wide'' matrices.

\phantomsection
\subsection*{Basic Sketching (\cref{sec2:rblas}) }
\label{subsec:sec2_summary}

This section documents our work toward developing a \RandBLAS{} standard.
It begins with remarks on the Basic Linear Algebra Subprograms (\BLAS{}), which are to classical NLA as we hope the \RandBLAS{} will be to \RandNLA{}.

\cref{sec2:rblas:high-level-plan} addresses high-level design questions for a \RandBLAS{} standard.
By starting with a simple premise, we arrive at the conclusion that it should provide functionality for \textit{data-oblivious sketching} (that is, sketching without consideration to the numerical properties of the data).
We then offer our thoughts on how such a library should be organized and how it should handle random number generation.

\cref{subsec:sketching_prelims} summarizes a variety of concepts in sketching.
In it, we answer questions such as the following.
\begin{itemize}
    \item What are the geometric interpretations of sketching?
    \item How does one measure the quality of a sketch?
    \item What are the ``standard'' properties for the first and second moments of sketching operator distributions?
    When and how are these properties important in \RandNLA{} algorithms?
\end{itemize}
Detail-oriented readers should consider \cref{subsec:sketching_prelims} alongside \cref{subapp:effective_distortion}, which presents a novel concept called \textit{effective distortion} that is useful in characterizing the behavior of randomized algorithms for least squares and related problems.

\cref{subsec:dense_skops,subsec:sparse_skops,subsec:srfts} review the three types of sketching operator distributions that the \RandBLAS{} might support.
These types of distributions consist of dense sketching operators (e.g., Gaussian matrices), sparse sketching operators, and sketching operators based on subsampled fast trigonometric transforms (such as discrete Fourier, discrete cosine, and Walsh-Hadamard transforms).
As we explain in \cref{subsec:sparse_skops}, we consider row-sampling and column-sampling as particular types of sparse sketching.
The interested reader is referred to \cref{subapp:SJLTs} for details on a class of sparse sketching operators that is distinct from row or column sampling.
These details include notes on high-performance implementations that have not appeared in earlier literature.

Our chapter on basic sketching concludes with \cref{subsec:multi_sketch}, which presents a handful of elementary sketching operations that are not naturally represented by a linear transformation that acts only on the columns or only on the rows of a matrix.
These operations arise in the fastest randomized algorithms for low-rank approximation.

\phantomsection
\subsection*{Least Squares and Optimization (\cref{sec3:LS_and_optim})}
This is one of three \nameCrefs{sec3:LS_and_optim} that cover driver-level functionality, and it is one of two that discuss drivers \textit{and} computational routines.
It is narrower in scope but greater in depth than the other \nameCrefs{sec3:LS_and_optim} that address drivers.

\paragraph{Problem classes.} In \cref{subsec:optim_problem_classes} we consider a variety of least squares problems within a common framework. 
The framework describes all problems in terms of an $m \times n$ data matrix $\mtx{A}$ where $m \geq n$.
Given $\mtx{A}$, any pair of vectors $(\vct{b},\vct{c})$ of respective lengths $(m, n)$ can be considered along with a parameter $\mu \geq 0$ to define ``primal'' and ``dual'' \textit{saddle point problems}.
The primal problem is always
\begin{equation}
    \min_{\vct{x}\in\R^n}\left\{\|\mtx{A}\vct{x} - \vct{b}\|_2^2 + \mu\|\vct{x}\|_2^2 + 2\vct{c}^{\trans}\vct{x}\right\}.\label{eq:sec0_saddle_opt_x} \tag{$P_{\mu}$}
\end{equation}
The dual problem takes one of two forms, depending on the value of $\mu$:
\begin{equation}\label{eq:sec0_dual_saddle}
    \begin{rcases}
     & \min_{\vct{y} \in \R^m} \left\{\|\mtx{A}^{\trans}\vct{y} - \vct{c}\|_2^2 + \mu\|\vct{y} - \vct{b}\|_2^2\right\} & \quad\text{if } \mu > 0 \quad \\
    & \min_{\vct{y} \in \R^m}\left\{ \|\vct{y} - \vct{b} \|_2^2 \,:\, \mtx{A}^{\trans}\vct{y} = \vct{c} \right\} & \quad\text{if } \mu = 0 \quad
    \end{rcases}.
    \tag{$D_{\mu}$}
\end{equation}
Special cases of these problems include overdetermined and underdetermined least squares, as well as ridge regression with tall or wide matrices.
\cref{subapp:error_metrics} gives background on accuracy metrics, sensitivity analysis, and error estimation methods that apply to the most prominent problems under this umbrella.

\cref{subsec:optim_problem_classes} considers one type of problem that does not fit nicely into the above framework.
Specifically, for a positive semidefinite linear operator $\mtx{G}$ and a positive parameter $\mu$, it also considers the \textit{regularized quadratic} problem
\begin{equation}\label{eq:sec0_unconstr_quadratic_opt}
    \min_{\vct{w}} \vct{w}^{\trans}(\mtx{G} + \mu\mtx{I})\vct{w} - 2\vct{h}^{\trans}\vct{w}.\tag{$R_{\mu}$}
\end{equation}
We note that \eqref{eq:sec0_saddle_opt_x} and \eqref{eq:sec0_dual_saddle} can be cast to this form when $\mu$ is positive.
However, to make this reformulation would be to obfuscate the structure in a saddle point problem, rather than reveal it.

\paragraph{Drivers.}
We start in \cref{subsubsec:sketch_and_solve} by covering a low-accuracy method for overdetermined least squares known as \textit{sketch-and-solve}.
This method is remarkable for the simplicity of its description and its analysis.
It is also the first place where our newly-proposed concept of effective distortion provides improved insight into algorithm behavior.

\cref{subsubsec:sketch_and_precond,subsubsec:nys_pcg} concern methods for solving problems \eqref{eq:sec0_saddle_opt_x}, \eqref{eq:sec0_dual_saddle}, and \eqref{eq:sec0_unconstr_quadratic_opt} to high accuracy.
These methods use randomization to find a \textit{preconditioner}.
The preconditioner is used to implicitly change the coordinate system that describes the optimization problem, in such a way that the preconditioned problem can easily be solved by iterative methods from classical NLA.
These methods are intended for use with certain problem structures (e.g. $m \gg n$) that we clearly identify.

The broader idea of sketch-and-solve algorithms has been successfully used for kernel ridge regression (KRR -- see \cref{subapp:primer_krr} for a primer).
In \cref{subsubsec:krr_sketch_and_solve}, we reinterpret two algorithms for approximate KRR as sketch-and-solve algorithms for \eqref{eq:sec0_unconstr_quadratic_opt}.
We further identify how the \textit{sketched problems} amount to saddle point problems with $m \gg n$.
\cref{subapp:KRR_AM15_SASAP} details how the saddle point framework is useful in the more complicated of these two settings.

\paragraph{Computational routines.}
The computational routines that we cover in \cref{subsec:optim_comp_routines} only pertain to drivers based on random preconditioning.
We kick off our discussion in \cref{subsubsec:tech_background_saddle} with background on saddle point problems.
Then, \cref{subsec:precond_gen} addresses preconditioner generation for saddle point problems when $m \gg n$.
It opens with a theoretical result (\cref{prop:left_sketch_precond}) characterizing the spectrum of the preconditioned data matrix $\mtx{A}$ (see also \cref{subsubapp:we_sketch_subspaces}) before providing a comprehensive overview of implementation considerations.
Special attention is paid to how one can generate the preconditioner when $\mu > 0$ at no added cost compared to when $\mu = 0$.
In \cref{subsec:saddle_nys_precond}, we extend recently proposed methods from the literature to define novel low-memory preconditioners for regularized saddle point problems.
Finally, \cref{subsubsec:det_saddle_solve} reviews a suite of deterministic iterative algorithms from classical NLA that are needed for randomized preconditioning algorithms.

\phantomsection
\subsection*{Low-rank Approximation (\cref{sec4:lowrank}) }
\label{subsec:sec4_summary}

Low-rank approximation problems take the following form.
\begin{quote}
    Given as input an $m \times n$ target matrix $\Ao$, compute suitably structured factor matrices $\mtx{E}$, $\mtx{F}$, and $\mtx{G}$ where
    \begin{equation*}\label{eq:sec0:LRA}
    \begin{array}{
    cccccc}
    \Aa & := & \mtx{E} & \mtx{F} & \mtx{G} \\
    m\times n & & m\times k & k \times k & k\times n
    \end{array} 
    \end{equation*}
    approximates $\Ao$.
    The accuracy of the approximation $\Aa \approx \Ao$ may vary from one application to another, but we require that $k \ll \min\{m, n\}$.
\end{quote}
This \nameCref{sec4:lowrank} summarizes the massive design spaces of randomized algorithms for such problems, as documented in the existing literature.
One of its core contributions is to clarify what parts of this design space are relevant in what situations.

\paragraph{Problem classes.}
\cref{subsec:supported_lowrank} starts by explaining the significance of the SVD and eigendecomposition in relation to principal component analysis.
From there, it introduces the reader to a handful of \textit{submatrix-oriented decompositions} -- CUR, one-sided interpolative decompositions (one-sided ID), and two-sided interpolative decompositions (two-sided ID) -- along with their applications.
\cref{subsec:supported_lowrank} concludes with guidance on how one should and should-not quantify approximation error in low-rank approximation problems.
We note that this background is \textit{much} more detailed than that \cref{subsec:optim_problem_classes} provided on least squares and optimization problems.
This extra background will be important for many readers.

\paragraph{Drivers.}
\cref{subsec:lowrank_drivers} gives concise yet comprehensive overviews for \RandNLA{} algorithms for SVD and Hermitian eigendecomposition (\S \ref{subsubsec:svd_algs} and \ref{subsubsec:herm_eig_algs}) as well as CUR and two-sided interpolative decomposition (\S \ref{subsubsec:TSID_CUR_driver}).
In the process, we take care to prevent misunderstandings in what we mean by a \textit{\Nystrom{} approximation} of a positive semidefinite matrix.
Pseudocode is provided for at least one algorithm for each of these problems.

\paragraph{Computational routines.}
As is typical for surveys on this topic, we identify \textit{QB decomposition} (\S \ref{subsubsec:qb_alg}) and \textit{column subset selection (CSS) / one-sided ID} (\S \ref{subsec:CSS_CX_computational}) as the basic building blocks for most drivers.
We also isolate power iteration (\S \ref{subsubsec:data_aware}) and partial column-pivoted matrix decompositions (\S \ref{subsec:column_pivoted}) as subproblems with nontrivial design spaces that are important to low-rank approximation.

Some of the building blocks covered cumulatively from \cref{subsubsec:data_aware,subsubsec:qb_alg,subsec:column_pivoted,subsec:CSS_CX_computational} can be used to compute low-rank approximations iteratively.
If one seeks an approximation that is accurate to within some given tolerance, then these iterative algorithms require methods for estimating norms of linear operators; we cover such norm estimation methods briefly in \cref{subsubsec:norm_rank_est}.
\cref{app:lowrank:pseudocode} contains pseudocode for seven computational routines and details their dependency structure.

\phantomsection
\subsection*{Further Possibilities for Drivers (\cref{sec5:more_drivers})}
\label{subsec:sec5_summary}

This \nameCref{sec5:more_drivers} covers a handful of independent topics.

\cref{subsec:multipurpose_decomp} covers \textit{multi-purpose matrix decompositions}.
\cref{subsubsec:chol_qr} explains a simple algorithm for computing an unpivoted QR decomposition of a tall-and-thin matrix of full column rank; the algorithm uses randomization to precondition Cholesky QR for numerical stability.
\cref{subsec:fullrank_decomp:qrcp} first describes an existing algorithm from the literature for Householder QRCP of matrices with any aspect ratio, and then presents an extension of preconditioned Cholesky QR that incorporates pivoting and allows for rank-deficient matrices.
\cref{subsec:UTV_URV_QLP} summarizes methods for computing decompositions known by various names (UTV, URV, QLP) that all aim to serve as cheaper surrogates for the SVD.

\cref{subsec:general_linear_systems} addresses randomized algorithms for the solution of unstructured linear systems.
This includes direct methods based on accelerating (or safely bypassing) pivoting in matrix decompositions (\S \ref{subsec:factor_linsys}) as well as iterative methods (\S \ref{subsec:linsys_iterative}).
Some of these iterative methods were developed fairly recently and are a subject of considerable practical interest.

\cref{sec:trace_estimation} considers the problem of estimating the trace of a linear operator.
This problem is unique in the context of this monograph, since it makes no sense to consider in the shared-memory dense-matrix data model.
We have opted to cover it anyway since randomized methods are \textit{extremely} effective for it.
\cref{subsec:Girard_Hutchinson} introduces the elementary Girard--Hutchinson estimator developed in the late 1980s. \cref{subsec:trace_est_by_lowrank_approx} covers methods that benefit from contemporary developments on randomized algorithms for low-rank approximation.
Finally, \cref{subsec:trace_est_integral_quadrature} covers methods for computing the trace of $f(\mtx{B})$ where $\mtx{B}$ is a Hermitian matrix and $f$ is a matrix function.
We give a significant amount of background material to help the newcomer understand the methods in this last category.

\phantomsection
\subsection*{Advanced Sketching: Leverage Score Sampling (\cref{sec7:lev_scores}) }
\label{subsec:sec6_summary}

Leverage scores constitute measures of importance for the rows or columns of a matrix.
They can be used to define data-aware sketching operators that implement row or column sampling.

\cref{subsec:lev_scores} introduces three types of leverage scores: standard leverage scores, subspace leverage scores, and ridge leverage scores.
We explain how each type is suitable for sketching with different downstream tasks in mind.
For example, a proposition in \cref{subsec:lev_scores_standard} bounds the probability that a row-sampling operator satisfies a subspace embedding property for the range of a matrix $\mtx{A}$.
The bound shows that if rows are sampled according to a distribution $\vct{q}$, then it becomes more likely that the subspace embedding property holds as $\vct{q}$ approaches $\mtx{A}$'s standard leverage score distribution.

\cref{subsec:Estimation-of-leverage-scores} covers randomized algorithms for approximating leverage scores.
Such approximation methods are important since leverage scores are expensive to compute except when working with highly structured problem data.
The structure of these algorithms bears similarities to those seen in earlier \nameCrefs{sec7:lev_scores}.
For example, \cref{subsec:approx_subspace_leverage_scores} explains how a longstanding algorithm for approximating subspace leverage scores can be extended with QB approaches from \cref{subsubsec:qb_alg}. 

\phantomsection
\subsection*{Advanced Sketching: Tensor Product Structures (\cref{sec8:tensors})}
\label{subsec:sec7_summary}


Tensor computations are the domain of \textit{multilinear algebra}.
As such, it is reasonable to exclude them from the scope of a standard library from \RandNLA{}.
However, at the same time, it is reasonable for a \RandNLA{} library to support the core subproblems in tensor computations that are linear algebraic in nature.
Sketching implicit matrices with tensor product structure fits this description.

This \nameCref{sec8:tensors} reviews efficient methods for sketching matrices with Kronecker product or Khatri--Rao product structures (see \S \ref{subsubsec:def-Kronecker-Khatri--Rao-matrices} for definitions). 
The material in \ref{subsubsec:Row-structured-sketches}--\ref{subsubsec:Recursive_sketch} concerns data-oblivious sketching distributions that are similar to those from \cref{sec2:rblas} but modified for the tensor product setting.
\cref{subsubsec:tensor_product_lev_scores}, by contrast, concerns data-aware sketching methods based on leverage score sampling.
Notably, there are methods to efficiently sample from the \textit{exact} leverage score distributions of tall matrices with Kronecker and Khatri--Rao product structures without explicitly forming those matrices.

For completeness, \cref{subsec:partial-sketch-reuse} discusses motivating applications (specifically, tensor decomposition algorithms) that entail sketching matrices with these structures.


\section{Recommended reading}\label{subsec:recommended_reading}

This monograph is heavily influenced by a recent and sweeping survey by Martinsson and Tropp \cite{MT:2020}.
We draw detailed comparisons to that work in \cref{subsec:MT_survey_deltas}. But first, we give remarks on other resources of note for learning about \RandNLA{}.

\subsection{Tutorials, light on prerequisites}

\paragraph{RandNLA: randomized numerical linear algebra, \textup{by Drineas and Mahoney \cite{DM16_CACM}}.} \hfill
\vspace{0.25em}

\noindent Depending on one's background (and schedule!)\ this article can be read in one sitting.
It requires no knowledge of NLA or probability.
In fact, it does not even presume that the reader already cares about matrix computations. 
It starts with basic ideas of matrix approximation by subsampling, explains the effect of sampling in different data-aware ways, and frames general data-oblivious sketching as ``preprocessing followed by uniform subsampling.''
It summarizes, at a very high level, significant results of \RandNLA{} in least squares, low-rank approximation, and the solution of structured linear systems known as \textit{Laplacian systems}.

\paragraph{Lectures on randomized numerical linear algebra, \textup{by Drineas and Mahoney \cite{RandNLA_PCMIchapter_chapter}}.}\hfill\vspace{0.25em}

\noindent This book chapter is useful for those who want to see representative banner results in RandNLA \textit{with proofs}.
It covers algorithms for least squares and low-rank approximation.
Its proofs emphasize decoupling deterministic and probabilistic aspects of analysis.
Among resources that engage with the theory of \RandNLA{}, it is notable for its brevity and its self-contained introductions to linear algebra and probability.

\subsection{Broad and proof-heavy resources}

\paragraph{Sketching as a tool for numerical linear algebra, \textup{by Woodruff \cite{Woodruff:2014}}.}\hfill\vspace{0.25em}

\noindent This monograph proceeds one problem at a time, starting with $\ell_2$ regression, then on to $\ell_1$ regression, then low-rank approximation, and finally graph sparsification.
It develops the technical machinery needed for each of these settings, at various levels of detail.
Among resources that address \RandNLA{} theory, it is notable for its treatment of lower bounds (i.e., limitations of randomized algorithms).

\paragraph{An introduction to matrix concentration inequalities, \textup{by Tropp \cite{Tropp:2015:matrixConcentrationBook}}.}\hfill\vspace{0.25em}

\noindent This monograph gives an introduction to the theory of matrix concentration and its applications.
It is not about \RandNLA{} \textit{per se}, but several of its applications do focus on \RandNLA{}.
The course notes \cite{Tropp:2019:LecNotes} build on this monograph, exploring theory and applications of matrix concentration developed after \cite{Tropp:2015:matrixConcentrationBook} was written.

\paragraph{Lecture notes on randomized linear algebra, \textup{by Mahoney \cite{Mah16_RLA_TR}}.}\hfill\vspace{0.25em}

\noindent These notes are fairly comprehensive in their coverage of results in \RandNLA{} up to 2013.
They address matrix concentration, approximate matrix multiplication, subspace embedding properties of sketching distributions, as well as various algorithms for least squares and low-rank approximation.
These notes are distinct from \cite{Woodruff:2014} in that they address theory and practice. (Of course, being course notes, they are not suitable as a formal reference.)

\subsection{Perspectives on theory, light on proofs}


\paragraph{Randomized algorithms for matrices and data, \textup{by Mahoney \cite{Mah-mat-rev_BOOK}}.}\hfill\vspace{0.25em}

\noindent This monograph heavily emphasizes concepts, interpretations, and qualitative proof strategies.
It is a good resource for those who want to know what \RandNLA{} can offer in terms of theory for the least squares and low-rank approximation.
It is notable for the effort it expends to connect \RandNLA{} theory to theoretical developments in other disciplines.

\paragraph{Determinantal point processes in randomized numerical linear algebra, \textup{by Derezi\'{n}ski and Mahoney \cite{DM21_NoticesAMS}}.}\hfill\vspace{0.25em}

\noindent This article gives an overview of \RandNLA{} theory from the perspective of determinantal point processes and statistical data analysis.
Among the many resources for learning about \RandNLA{}, it is notable for offering a distinctly \textit{prospective} (rather than \textit{retrospective}) viewpoint.

\subsection{Deep investigations of specific topics}

\paragraph{Finding structure with randomness: probabilistic algorithms for constructing approximate matrix decompositions, \textup{by Halko, Martinsson, and Tropp \cite{HMT:2011}}.}\hfill\vspace{0.25em}

\noindent As of late 2022, this article is the single most influential resource on \RandNLA{}.
Its introduction includes a history of how randomized algorithms have been used in numerical computing, as well as a brief summary of (then) active areas of research in \RandNLA{}.
Following the introduction, it focuses exclusively on low-rank approximation.
It is extremely thorough in its treatment of both theory and practice.

This article is now somewhat out of date and is partially subsumed by \cite{MT:2020}.
However, it is still of distinct value for the fact that it proves all of its main results (and in certain cases, by novel methods).
It also includes some algorithms that are not found in \cite{MT:2020}.

\paragraph{Randomized algorithms in numerical linear algebra, \textup{by Kannan and Vempala \cite{KV:2017:survey}}.}\hfill\vspace{-1em}

\noindent This survey provides a detailed theory of row and column sampling methods.
It also includes methods for tensor computations.

\paragraph{Randomized methods for matrix computations, \textup{by Martinsson \cite{Martinsson:2018_ish}}.}\hfill\vspace{0.25em}

\noindent This book chapter focuses on practical aspects of randomized algorithms for low-rank approximation.
In this regard, it is important to note that while \cite{HMT:2011} provided thorough coverage of this topic \textit{at the time}, the more recent \cite{Martinsson:2018_ish} reviews important practical advances developed after 2011.
Among resources that provide an in-depth investigation into low-rank approximation, is notable for how it also includes algorithms for full-rank matrix decomposition.

\subsection[\textit{Randomized numerical linear algebra: Foundations and Algorithms}, by Martisson and Tropp]{Randomized numerical linear algebra: \\ Foundations and Algorithms}\label{subsec:MT_survey_deltas}
Martinsson and Tropp's recent \textit{Acta Numerica} survey, \cite{MT:2020}, covers a wide range of topics, each with substantial technical and historical depth.
We have benefited from it tremendously in developing our plans for \RandBLAS{} and \RandLAPACK{}.
Because we have found this resource so useful -- and, at the same time, because we have gone through the trouble of writing a distinct monograph that is just as long -- we think there is value in highlighting how it differs from our work.

\paragraph{Basic sketching.}
By comparison to \cite{MT:2020}, we focus more on implementation than on theory.
The outcome of this is the broadest-yet review of the literature relevant to the implementation of sketching methods.
In the appendices, we provide novel technical contributions to sketching theory and practice.
\begin{quote}
    See \cref{sec2:rblas,app:sketching}, \cite[\S 7 -- \S 9]{MT:2020}.
\end{quote}
    
\paragraph{Least squares and optimization.} 
Our coverage of these concepts is comprehensive, insofar as optimization can be reduced to linear algebra.
It also includes a number of novel technical contributions and a review of relevant software.
By comparison, \cite{MT:2020} provides very limited coverage of this area, as acknowledged in \cite[\S 1.6]{MT:2020}.
\begin{quote}
    See \cref{sec3:LS_and_optim,app:lstsq_details}, \cite[\S 10]{MT:2020}.
\end{quote}

\paragraph{Low-rank approximation.}
Our approach here is very different than that of \cite{MT:2020}.
It provides effective scaffolding for a reader to get a handle on the vast literature on low-rank approximation.
However, it comes at the price of creating fewer opportunities for mathematical explanations.
Separately, our coverage here is distinguished by providing an overview of software that implements randomized algorithms for low-rank approximation.
\begin{quote}
    See \cref{sec4:lowrank,app:lowrank}, \cite[\S 11 -- \S 15]{MT:2020}.
\end{quote}

\paragraph{Full-rank matrix decompositions.}
By comparison to \cite{MT:2020}, we emphasize a broader range of matrix decompositions and more algorithms for computing them.
One of the algorithms we cover is novel and is accompanied by proofs that characterize its behavior.
For the algorithms covered here \textit{and} in \cite{MT:2020}, the latter provides more mathematical detail.
\begin{quote}
    See \cref{subsec:multipurpose_decomp,subsec:factor_linsys}, \cref{app:cholqrcp}, \cite[\S 16]{MT:2020}.
\end{quote}

\paragraph{Kernel methods.}
Randomized methods have proven very effective in processing machine learning models based on \textit{positive definite kernels}.
They are also effective in approximating matrices from scientific computing induced by \textit{indefinite kernels}.
Both of these topics are addressed in \cite{MT:2020}.
We only address the former topic, and we do so in a way that emphasizes the resulting linear algebra problems.
\begin{quote}
    See Sections \ref{subsubsec:sketch_and_precond}, \ref{subsubsec:krr_sketch_and_solve}, and \ref{subsec:ridge_leverage_scores}, \cref{subapp:primer_krr}, \cite[\S 19, \S 20]{MT:2020}.
\end{quote}

\paragraph{Linear system solvers.}
We cover slightly more material for solving unstructured linear systems than \cite{MT:2020}.
However, we do not cover methods that are specific to sparse problems.
As a result, we do not cover a prominent method for approximate Cholesky decompositions of sparse graph Laplacians.
\begin{quote}
   See Sections \ref{subsubsec:nys_pcg} and \ref{subsec:general_linear_systems}, \cite[\S 17, \S 18]{MT:2020}.
\end{quote}

\paragraph{Trace estimation.}
We have the luxury of being able to cover recently-developed algorithms that were not available when \cite{MT:2020} was written.
This includes two methods that provide the first major advances in trace estimation since the late 1980s.
We provide less depth than \cite{MT:2020} on average, with the notable exception of stochastic Lanczos quadrature.
\begin{quote}
    See \cref{sec:trace_estimation}, \cite[\S 4, \S 6]{MT:2020}.
\end{quote}

\paragraph{Advanced sketching.}
Both this monograph and \cite{MT:2020} cover leverage score sampling and sketching operators with tensor product structures.
We cover these topics in substantially more detail, spending a full ten pages on each of them.
We do this partly because these topics complement one another: implicit matrices with tensor product structures are among the best candidates for practical leverage score sampling, nearly on par with kernel matrices from machine learning.
\begin{quote}
    See \cref{sec7:lev_scores,sec8:tensors}, \cite[\S 7.4, \S 9.4, \S 9.6, \S 19.2.3]{MT:2020}.
\end{quote}

\section{Notation and terminology}\label{subsec:notation}

Our notation is summarized in \cref{table:notation}; we also define some of this notation below as we explain basic concepts.

\subsubsection{Matrices and vectors}

Let $\mtx{A}$ be an $m \times n$ matrix or linear operator.
We use $\mtx{A}^*$ to denote its adjoint (transpose, in the real case) and $\mtx{A}^\dagger$ to denote its pseudo-inverse.
It is called \textit{Hermitian} if $\mtx{A}^{\trans} = \mtx{A}$ and \textit{positive semidefinite} if it is Hermitian and all of its eigenvalues are nonnegative.
We often abbreviate ``positive semidefinite'' with ``psd.''

We sometimes find it convenient to write $\mtx{A} \in \R^{m \times n}$.
However, it should be understood that the methods in this monograph generally apply to both real and complex matrices.
We therefore tend to define a matrix by phrases like ``$\mtx{A}$ is $m$-by-$n$'' or ``an $m$-by-$n$ matrix $\mtx{A}$.''
We often call a vector of length $n$ an \textit{$n$-vector}; vectors are oriented as columns by default.

For $m \geq n$, a \textit{QR decomposition} of $\mtx{A}$ consists of an $m \times n$ column-orthonormal matrix $\mtx{Q}$ and an upper-triangular matrix $\mtx{R}$ for which $\mtx{A} = \mtx{Q}\mtx{R}$.
Those familiar with the NLA literature will note that this is typically called the \textit{economic} QR decomposition.
If $\mtx{A}$ has rank $k < \min(m, n)$, then we also consider it valid for $\mtx{Q}$ to be $m \times k$ and for $\mtx{R}$ to be $k \times n$.
We also consider \textit{QR decomposition with column pivoting} (QRCP).
To describe QRCP, we say that if $J = (j_1,\ldots,j_n)$ is a permutation of $\idxs{n}$, then
\[
    \mtx{A}[:,J] = [\vct{a}_{j_1},\vct{a}_{j_2},\ldots,\vct{a}_{j_n}]
\]
where $\vct{a}_i$ is the $i^{\text{th}}$ column of $\mtx{A}$.
In this notation, QRCP produces an index vector $J$ and factors $(\mtx{Q},\mtx{R})$ that provide a QR decomposition of $\mtx{A}[:,J]$.

Now let $\mtx{A}$ have rank $k$.
Its \textit{singular value decomposition} (SVD) takes the form $\mtx{A} = \mtx{U}\mtx{\Sigma}\mtx{V}^*$, where the matrices $(\mtx{U},\mtx{V})$ have $k$ orthonormal columns and $\mtx{\Sigma} = \diag(\sigma_1,\ldots,\sigma_k)$ is a square matrix with sorted entries $\sigma_1 \geq \cdots \geq \sigma_k > 0$.
The SVD can also be written as a sum of rank-one matrices: $\mtx{A} = \textstyle\sum_{i=1}^{k} \sigma_i \vct{u}_i\vct{v}_i^*$,
where $(\vct{u}_i,\vct{v}_i)$ are the $i^{\text{th}}$ columns of $(\mtx{U},\mtx{V})$ respectively.
Those familiar with the NLA literature will note that this is typically called the \textit{compact SVD}.

\subsubsection{Probability and our usage of the term ``random.''}

A \textit{Rademacher} random variable uniformly takes values in $\{+1,-1\}$.
The initialism ``iid'' expands to \textit{independent and identically distributed}.

We often abuse terminology and say that a matrix ``randomly'' performs some operation.
In reality, matrices only perform deterministic calculations, and randomness only comes into play when the matrix is first constructed.
This convention extends to ``matrices'' that are abstract linear operators, in which case randomness is only involved in constructing the data that defines the operator.

Unqualified use of the term ``random'' before performing an action with a finite set of outcomes (such as sampling components from a vector, applying a permutation, etc...) means the randomness is uniform over the space of possible actions.

%
%
%
\begin{table}[htb]
\caption{Notation}\label{table:notation}
{\normalsize \begin{tabular}{ll}
\hline
\multicolumn{2}{l}{Arrays and indexing}
\\ \hline
$A_{ij}$ ~or~ $\mtx{A}[i, j]$                     & $(i,j)^{\mathrm{th}}$ entry of a matrix $\mtx{A}$             \\
 $\vct{a}_i$ ~~or~ $\mtx{A}[:,i]$                   & $i^{\mathrm{th}}$ column of $\mtx{A}$                       \\
$v_i$ ~~\,or~ $\vct{v}[i]$                           & $i^{\mathrm{th}}$ component of a vector $\vct{v}$          \\
$\idxs{m}$                                      & index set of integers from 1 to $m$                             \\
$I$ or $J$                                      & partial permutation vector for indexing into an array           \\
$|I|$                                           & length of an index vector        \\
$\mtx{A}[I,\fslice{}]$                          & submatrix consisting of (permuted) rows of $\mtx{A}$                       \\
$\mtx{A}[\fslice{},J]$                          & submatrix consisting of (permuted) columns of $\mtx{A}$                    \\
$\lslice{k}$                                    & index into the leading $k$ elements of an array,                \\
                                                & along an axis of length at least $k$                            \\
$\tslice{k}$                                    & index into the trailing $n-k+1$ elements of an array,             \\
                                                & along an axis of length $n \geq k$                              \\ \hline 
\multicolumn{2}{l}{Reserved symbols}                                                                              \\ \hline
$\mtx{S}$                                       & sketching operator                                              \\
$\mtx{I}_k$                                     & identity matrix of size $k \times k$                            \\
$\vct{\delta}_i$                                & $i^{\mathrm{th}}$ standard basis vector of implied dimension    \\
$\vct{0}_n$                                     & zero vector of length $n$                                       \\
$\mtx{0}_{m \times n}$                          & zero matrix of size $m \times n$                                \\ \hline
\multicolumn{2}{l}{Linear algebra}                                                                                \\ \hline
$\|\vct{x}\|_2$ ~or~ $\|\vct{x}\|$              & Euclidean norm of a vector $\vct{x}$                            \\
$\|\mtx{A}\|_2$                                 & spectral norm of $\mtx{A}$                                      \\
$\|\mtx{A}\|_{\mathrm{F}}$                      & Frobenius norm of $\mtx{A}$                                     \\
$\cond(\mtx{A})$                                & Euclidean condition number of $\mtx{A}$                         \\
$\lambda_i(\mtx{A})$                            & $i^{\mathrm{th}}$ largest eigenvalue of $\mtx{A}$               \\
$\sigma_i(\mtx{A})$                             & $i^{\mathrm{th}}$ largest singular value of $\mtx{A}$           \\
$\mtx{A}^{\trans}$                              & adjoint (transpose, in the real case) of $\mtx{A}$              \\
$\mtx{A}^{\dagger}$                             & Moore--Penrose pseudoinverse of $\mtx{A}$                       \\
$\mtx{A}^{1/2}$                                 & Hermitian matrix square root                                    \\ 
$\mtx{A} \preceq \mtx{B}$                       & the matrix $\mtx{B} - \mtx{A}$ is positive semidefinite         \\ \hline
\multicolumn{2}{l}{Matrix decomposition conventions}                                                              \\ \hline
$\mtx{A} = \mtx{Q}\mtx{R}$                      & QR decomposition (economic, by default)                                       \\
$(\mtx{Q},\mtx{R},J) = \code{qrcp}(\mtx{A})$    & QR with column-pivoting; $\mtx{A}[:,J] = \mtx{Q}\mtx{R}$. 
  \\
$\mtx{A} = \mtx{U}\mtx{\Sigma}\mtx{V}^{\trans}$ & singular value decomposition (compact, by default)                           \\
$\mtx{R} = \code{chol}(\mtx{G})$                & upper triangular Cholesky factor of $\mtx{G} = \mtx{R}^{\trans}\mtx{R}$ \\ \hline
\multicolumn{2}{l}{Probability}                                                                                   \\ \hline
$X \sim \mathcal{D}$                            & $X$ is a random variable following a distribution $\mathcal{D}$ \\
$\E[\mtx{X}]$                                   & expected value of a random matrix $\mtx{X}$                     \\
$\operatorname{var}(X)$                                       & variance of a random variable $X$                               \\
$\Pr\{E\}$                                      & probability of the event $E$                                   
\end{tabular}
}
\end{table}

\chapter{Basic Sketching}
\chaptermark{Basic Sketching}
\label{sec2:rblas}

\minitoc
\bigskip

The \BLAS{} (Basic Linear Algebra Subprograms) were originally a collection of Fortran routines for computations including vector scaling, vector addition, and applying Givens rotations \cite{LHKK:1979:BLAS1}.
They were later extended to operations such as matrix-vector multiplication and triangular solves \cite{DDcHH:1988:BLAS2} as well as matrix-matrix multiplication, block triangular solves, and symmetric rank-$k$ updates \cite{DCHD:1990:BLAS3}.
These routines have subsequently been organized into three \textit{levels} called \BLASlev{1}, \BLASlev{2}, and \BLASlev{3}.

Over the years the \BLAS{} have evolved into a \textit{community standard}, with implementations targeting different machine architectures in many programming languages.
This standardization has been instrumental in the development of linear algebra libraries -- from the early days of \LINPACK{}, through to \LAPACK{}, and on to modern libraries such as \PLASMA{} and \SLATE{} \cites{LINPACK:1979,LAPACK:1987:Lawn01,LAPACK:1999,PLASMA:2009,PLASMA:2019,SLATE:2017:design,SLATE:2017:roadmap}.
It has also reduced the coupling between hardware and software design for NLA.
Indeed, the spirit of the \BLAS{} has been adapted to accommodate dramatic changes in prevailing architectures, such as those faced by \ScaLAPACK{} and \MAGMA{} \cites{PBLAS:1995,ScaLAPACK:1996,MAGMA:2010:general,MAGMA:2010:BLAS}.

This \nameCref{sec2:rblas} summarizes our progress on the design of a ``\RandBLAS{}'' library, which is to be to RandNLA as \BLAS{} is to classical NLA.
\cref{sec2:rblas:high-level-plan} begins by speaking to high-level scope and design considerations.
From there, \cref{subsec:sketching_prelims} summarizes sketching concepts that remain important throughout this monograph; we encourage the reader to not dwell too long on this \nameCref{subsec:sketching_prelims} and instead return to it as-needed later on.
Sections \ref{subsec:dense_skops} through \ref{subsec:multi_sketch} present our plans for sketching dense data matrices.
In brief: our near-term plans are for the \RandBLAS{} to support sketching operators which could naturally be represented by dense arrays or by sparse matrices with certain structures; we consider row sampling and column sampling as particular types of sparse sketching.

\section{A high-level plan}
\label{sec2:rblas:high-level-plan}

We begin with a simple premise.
\begin{quote}
    The \RandBLAS{}' defining purpose should be to facilitate implementation of high-level \RandNLA{} algorithms.
\end{quote}
This premise works to reduce the \RandBLAS{}' scope, as there are ``basic'' operations in \RandNLA{} which do not support this purpose.\footnote{For example, the problem of accepting two matrices and using randomization to approximate their product is certainly basic, and it is of conceptual value \cite{dkm_matrix1}.
However, it is rarely used as an explicit building block in higher-level \RandNLA{} algorithms.}
Another way that we reduce the scope of the \RandBLAS{} is to only consider sketching dense data matrices.
It may be reasonable to lift this restriction in the future, and consider methods for producing dense sketches of sparse data matrices.

Our premise for the \RandBLAS{} suggests that it should be concerned with \textit{data-oblivious sketching} -- that is, sketching without consideration to the numerical properties of a dataset.
We identify three categories of operations on this topic:
\begin{itemize}
    \item sampling a random sketching operator from a prescribed distribution,  
    \item applying a sampled sketching operator to a data matrix, and
    \item sketching that is not naturally expressed as applying a single a linear operator to a data matrix.
\end{itemize}
These categories are somewhat analogous to \BLASlev{1}, \BLASlev{2}, and \BLASlev{3}, insofar as their implementations admit more and more opportunities for machine-specific performance optimizations.
At this time, however, we do not advocate for any formalization of ``\RandBLAS{} levels.''

We note that data-oblivious sketching is not the only kind of sketching of value in \RandNLA{}.
Indeed, \textit{data-aware} sketching operators such as those derived from power iteration are extremely important for low-rank approximation (see \cref{subsubsec:data_aware}).
Methods for row or column sampling based on leverage scores  are also useful for kernel ridge regression and certain tensor computations; see \cref{sec7:lev_scores,sec8:tensors}. 
Although important, most of the functionality for producing or applying these sketching operators should be addressed in higher-level libraries.

In the material under the next two headings, we address the questions of how to handle random number generation and reproducibility in the \RandBLAS{}.

\subsection{Random number generation}
\label{subsec:randblas_rngs}

For reproducibility's sake it is important that the \RandBLAS{} include a specification for \emph{random number generators} (RNGs).

We believe the \RandBLAS{} should use \textit{counter-based random number generators} (CBRNGs), which were first proposed in \cite{SMDS:2011:Random123}.
A CBRNG returns a random number upon being called with two integer parameters: the \textit{counter} and the \textit{key}.
The time required for the CBRNG to return does not depend on either of these parameters.
A serial application can set the key at the outset of the program and never change it.
Parallel applications (particularly parallel simulations) can use different keys across different threads.
Sequential calls to the CBRNG with a fixed key should use different values for the counter.
For a fixed key, a CBRNG with a $p$-bit integer counter defines a stream of random numbers with period length $2^p$.

In our context, CBRNGs are preferable to traditional state-based RNGs such as the Mersenne Twister.
A key reason for this is that CBRNGs maximize flexibility in the order in which a sketching operator is generated.
For example, given a user-provided counter offset $c$ which acts as a random seed, the $(i,j)^{\text{th}}$ entry of a dense $d \times m$ sketching operator can be generated with counter $c + (i + d j)$.
The fact that these computations are embarrassingly parallel will be important for vendors developing optimized \RandBLAS{} implementations.
We note that this flexibility also provides an advantage over widely-used linear congruential RNGs, which have separate shortcomings of performing very poorly on statistical tests \cite[\S 2.2.1]{SMDS:2011:Random123}. 

Particular examples of CBRNGs include Philox, ARS, and Threefish, each of which was defined in \cite{SMDS:2011:Random123} and implemented in the \RandomOneTwo{} library.
These CBRNGs have periods of $2^{128}$, can support $2^{64}$ different keys, and pass standard statistical tests for random number generators.
\RandomOneTwo{} provides the core of the sketching layer of the \LibSkylark{} \RandNLA{} library \cite{libskylark}.
Implementations of Philox and ARS can also be found in MKL Vector Statistics \cite[\S 6.5]{IntelVSL}.

\subsubsection{Shift-register RNGs}

We have observed that the CBRNGs in \RandomOneTwo{} are significantly more expensive than the state-based shift-register RNGs developed by Blackman and Vigna \cite{BV:2021:xoroshiro}.
In fact, Blackman and Vigna's generators are so fast that we have been able to implement a method for applying a Gaussian sketching operator to a sparse matrix that beats Intel MKL's sparse-times-dense matrix multiplication methods.
However, in the application where we observed that performance, processing the sketch downstream was more expensive than computing the sketch in the first place.
Therefore, while CBRNGs were substantially more expensive in that application, their longer runtimes were inconsequential in that case.
This longer runtime can be viewed as a price we pay for prioritizing reproducibility of sketching across compute environments with different levels of parallelism.

The overall situation is this:
\begin{quote}
    State-based RNGs may be preferable to CBRNGs \textit{if} sketching is the bottleneck in a \RandNLA{} algorithm \textit{and} where the cost of random number generation decisively affects the cost of sketching.
    At this time we have no evidence that high-performance implementations of \RandNLA{} algorithms run into such bottlenecks.
    Such evidence may arise in the future and warrant reconsideration to fast state-based RNGs for the \RandBLAS{}, particularly if major advances are made in hardware-accelerated sketching algorithms.
\end{quote}

\subsection{Portability, reproducibility and exception handling}
\label{subsec:randblas_reproducibility}

We believe it is important that the \RandBLAS{} lends itself to portability across programming languages.
Therefore we plan for the \RandBLAS{} to have a procedural API and make use of no special data structures beyond elementary structs.
Higher-level libraries should take responsibility for exposing \RandBLAS{} functionality with sophisticated abstractions.
In particular, we plan for \RandLAPACK{} to expose \RandBLAS{} functionality through a suitable object-oriented linear operator interface.
A key goal of this interface will be to make it possible to implement high-level \RandNLA{} algorithms with minimal assumptions on the sketching operator's distribution.
Such an interface will also reduce the coupling between determining \RandBLAS{}'s procedural API and prototyping \RandLAPACK{}.

Debugging high-performance numerical code is notoriously difficult.
Care must be taken in the design of the \RandBLAS{} so as to not contribute to this difficulty.
Indeed, it is essential that the \RBLAS{} be reproducible to the greatest extent possible.
The actual extent of the reproducibility will depend on factors outside of our control.
For example -- the \RBLAS{} cannot offer bitwise reproducibility guarantees unless the \BLAS{} does the same (see \cref{rem:bitwise_reproducibility}).
Therefore the main challenge for reproducibility for \RandBLAS{} is in random number generation; this challenge can be resolved comprehensively through the aforementioned CBRNGs.

A key source of exceptions in NLA is the presence of \code{NaN}s or \code{Inf}s in problem data.
Extremely sparse sketching matrices (such as those from \cref{subsubsec:randblas:LongAxSparse}) might not even read every entry of a data matrix, and so they might miss a \code{NaN} or \code{Inf}.
Those routines will be clearly marked as carrying this risk.
The majority of routines in the \RBLAS{} and \RandLAPACK{} will \textit{not} carry this risk: they will propagate \code{NaN}s and \code{Inf}s.
(See \cite{ExceptionHandling:2022} for a more detailed
discussion of how the \BLAS{} and \LAPACK{} (should) deal
with exceptions.)
For any such routine the exact behavior will depend on how the random sketching operator interacts with the problem data.
For example, if a data matrix containing multiple \code{Inf}s is sketched twice using different random seeds, then it is possible that an entry of the first sketch is an \code{Inf} while the corresponding entry of the second sketch is a \code{NaN}. 

\begin{remark}\label{rem:bitwise_reproducibility}
    Making the BLAS bitwise reproducible is challenging because 
    floating-point addition is not associative, and the order of
    summation can vary depending on the use of parallelism, vectorization, and other matters \cite{ReproBLAS}.
    Summation algorithms that guarantee bitwise reproducibility do exist
    \cite{ReproSum}.
    These algorithms may become practical on hardware that implements the latest IEEE 754 floating point
    standard, which includes a recommended instruction for bitwise-reproducible summation \cite{IEEE754-2019}.
    However, we leave these matters to future work.
\end{remark}

\section{Helpful things to know about sketching}
\label{subsec:sketching_prelims}

The purpose of sketching is to enact 
dimension-reduction so that computations of interest can be performed on a smaller matrix called a \textit{sketch}.
While precise computations performed on the sketch can vary dramatically, the simple statement of sketching's purpose lets us deduce the following facts.
\begin{itemize}
    \item 
    Sketching operators applied to the \emph{left} of a data matrix must \textit{must be wide}  (i.e., they must have more columns than rows).
    \item
    Sketching operators applied to the \emph{right} of a data matrix \textit{must be tall} (i.e., they must have more rows than columns).
\end{itemize}
This is to say, in left-sketching we require that $\mtx{S}\mtx{A}$ has fewer rows than $\mtx{A}$, and in right-sketching we require that $\mtx{A}\mtx{S}$ has fewer columns than $\mtx{A}$.
These facts are true regardless of the aspect ratio of the data matrix; see \cref{fig:sketch_left_right} for an illustration.
The facts are important because sketching operators in the literature are often defined under the assumption of left-sketching.

Before we proceed further, we reiterate some important advice.
\begin{quote}
    \textit{We encourage the reader to not dwell too long on this section (\mbox{\cref{subsec:sketching_prelims}}) and instead return to it as needed later on.}
\end{quote}
With that, \cref{subsec:geometric_interp_sketching} explains geometric interpretations of sketching from the left and right.
It also introduces the concepts of ``sketching in the embedding regime'' and ``sketching in the sampling regime.''
\cref{subsec:sketch_quality} covers concepts of subspace embedding distortion and the oblivious subspace embedding property -- these are \textit{central} to \RandNLA{} theory, but they play a modest role in this monograph.
\cref{subsec:inessential_properties_of_skops} states properties of sketching distributions that should hold as part of a `sanity check' for whether a proposed distribution is reasonable.

\begin{figure}[!htb]
	\centering
	\begin{overpic}[width=0.85\textwidth]{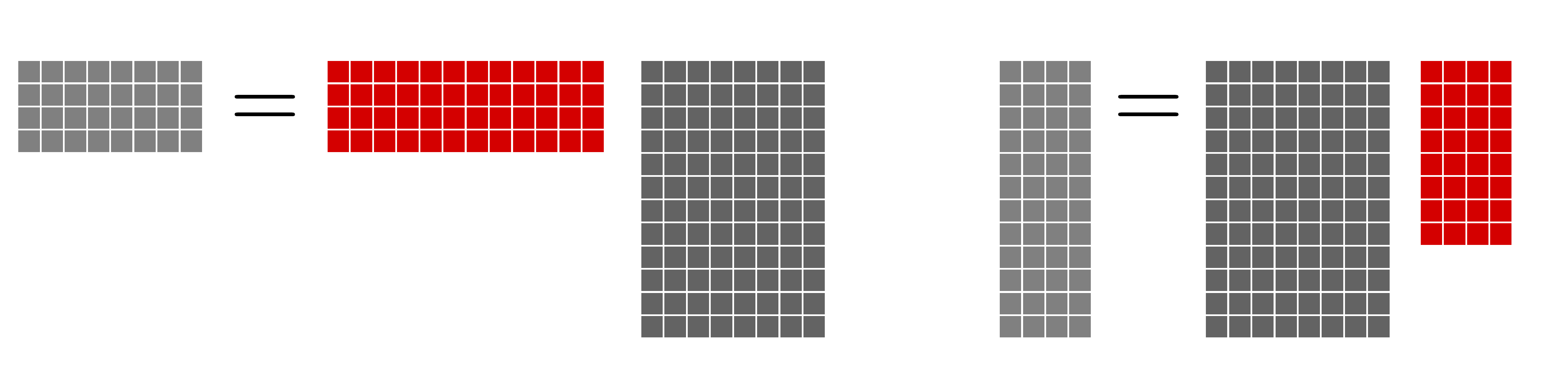} 
			\put(4,21){\color{black}{\large $\mtx{SA}$}}   
			\put(28,21){\color{black}{\large $\mtx{S}$}}   
			\put(45,21){\color{black}{\large $\mtx{A}$}}   

			\put(63.5,21){\color{black}{\large $\mtx{AS}$}}   
			\put(81,21){\color{black}{\large $\bf \mtx{A}$}}   
			\put(93.0,21){\color{black}{\large $\bf \mtx{S}$}}   	
			
			\put(-3,22){\color{black}{(a)}}   
			\put(58,22){\color{black}{(b)}}   
			
		\end{overpic}
	\caption{The left plot (a) shows that a sketching operator $\mtx{S}$ applied to the left of a matrix $\mtx{A}$ is wide, whereas the right plot (b) shows that a sketching operator $\mtx{S}$ applied to the right of a matrix $\mtx{A}$ is tall. 
	These stated properties hold universally; there are no exceptions for any kind of sketching.
	Separately, we note that both cases in the figure illustrate \textit{sketching in the sampling regime} in the sense of \cref{subsec:geometric_interp_sketching}.}\label{fig:sketch_left_right} 
\end{figure}
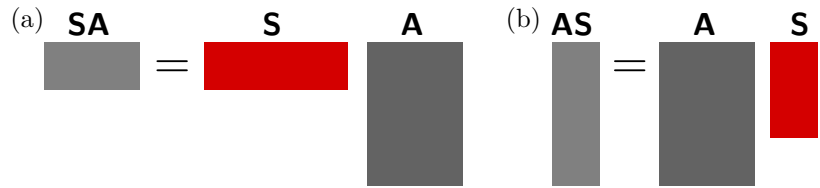

\subsection{Geometric interpretations of sketching}\label{subsec:geometric_interp_sketching}

\subsubsection{Prototypical left-sketching and right-sketching}

Sketching $\mtx{A}$ from the left preserves its number of columns.
Therefore, it is suitable for things such as estimating \textit{right} singular vectors.
We often interpret a left-sketch $\mtx{S}\mtx{A}$ as a compression of the range of $\mtx{A}$ to a space of lower ambient dimension.
In the special case when $\rank(\mtx{S}\mtx{A}) = \rank(\mtx{A})$, there is a sense in which the quality of the compression is completely independent from the spectrum of $\mtx{A}$.

Sketching $\mtx{A}$ from the right preserves its number of rows, and is suitable for things such as estimating \textit{left} singular vectors.
Conceptually, the right-sketch $\mtx{A}\mtx{S}$ can be interpreted as a sample from the range of $\mtx{A}$ using a test matrix $\mtx{S}$.
In the special case when $\rank(\mtx{A}\mtx{S}) \ll \rank(\mtx{A})$ then it is appropriate to think of this as a ``lossy sample'' from a much larger ``population,'' and it is natural to want this sample to capture as much information as possible for some sub-population of interest.

\subsubsection{Equivalence of left-sketching and right-sketching}

Left-sketching and right-sketching can be reduced to one another by replacing $\mtx{A}$ and $\mtx{S}$ by their adjoints.
For example, a left-sketch $\mtx{S}\mtx{A}$ can be viewed as a sample from the row space of $\mtx{A}$, equivalent to the right-sketch $\mtx{A}^{\trans}\mtx{S}^{\trans}$.
Conversely, a right-sketch $\mtx{A}\mtx{S}$ can be viewed as a compression of the row space of $\mtx{A}$, equivalent to the left-sketch $\mtx{S}^{\trans}\mtx{A}^{\trans}$.
Therefore it is artificial to strongly distinguish sketching operators by whether they are first defined for left-sketching or right-sketching.

This leads us to an important point.
\begin{quote}
    If $\mathcal{D}_{d, m}$ is a distribution over wide $d \times m$ sketching operators, it is \textit{canonically extended} to a distribution over tall $n \times d$ sketching operators by sampling $\mtx{T}$ from $\mathcal{D}_{d, n}$ and then returning the adjoint $\mtx{S} = \mtx{T}^{\trans}$.
\end{quote}
The notation in the statement above is carefully chosen: since our ``data matrices'' are typically $m \times n$, a typical left-sketching operator requires $m$ columns, and a typical right-sketching operator requires $n$ rows.

\subsubsection{The embedding and sampling regimes}
While it is artificial to associate a sketching distribution only with left-sketching or only with right-sketching, there are indeed families of sketching operators that are suited to qualitatively different situations.
The following terms help with our discussion of such families.
\begin{quote}
    Sketching in the \textit{embedding regime} is the use of a sketching operator that is \textit{larger} than the data to be sketched.
    Sketching in the \textit{sampling regime} is the use of a sketching operator that is \textit{far smaller} than the data to be sketched.
\end{quote}
In the above definitions one quantifies the size of an operator (or matrix) by the product of its number of rows and number of columns.

In \cref{sec3:LS_and_optim}, we will see that sketching in the embedding regime is nearly universal in randomized algorithms for least squares and related problems.
In \cref{sec4:lowrank}, we will see that sketching in the sampling regime is the foundation of randomized algorithms for low-rank approximation.
Over these \nameCrefs{sec3:LS_and_optim} we tend to see sketching in the embedding regime happen from the left, and sketching in the sampling regime happen from the right.
We stress that these tendencies are consequences of \textit{exposition}; they do not always hold when developing or using \RandNLA{} software.

\subsection{Sketch quality}\label{subsec:sketch_quality}

Let $L$ be a subset of some high-dimensional Euclidean space $\R^m$, $\mtx{S}$ be a sketching operator defined on $\R^m$, and consider the sketch $\mtx{S}L$.
Intuitively, $\mtx{S}L$ should be useful if its geometry is somehow ``effectively the same'' as that of $L$.
Here we discuss the preferred ways to quantify changes to geometry in \RandNLA{}.
We focus on methods suitable for when $L$ is a linear subspace, but we also consider when $L$ is a finite point set.

We acknowledge up-front that it only does so much good to measure the quality of an individual sketch.
Indeed, in order to make predictive statements about the behavior of algorithms, it is necessary to understand how the distribution of a sketching operator $\mtx{S} \sim \mathcal{D}$ induces a distribution over measures of sketch quality in a given application.
It is further necessary to analyze \textit{families of distributions} $\mathcal{D}_{d, m}$ parameterized by an embedding dimension $d$, since the size of a sketch is often a key parameter that a user can control.

\subsubsection{Subspace embeddings}

Let $\mtx{S}$ be a $d \times m$ sketching operator and $L$ be a linear subspace of $\R^m$.
We say that $\mtx{S}$ \textit{embeds $L$ into $\R^d$} and that it does so with \textit{distortion $\delta \in [0, 1]$} if $\vct{x} \in L$ implies
\begin{equation}\label{eq:subspace_embedding_condition}
    (1-\delta)\|\vct{x}\|_2 \leq \|\mtx{S}\vct{x}\|_2 \leq (1+\delta)\|\vct{x}\|_2.
\end{equation}
We often call such an operator an \textit{$\delta$-embedding}.

The concept of a subspace embedding was first used implicitly in RandNLA by \cite{DMM06};
the sketching operators used in \cite{DMM06} were based on a type of \textit{data-aware sketching} called leverage score sampling (discussed in \cref{sec7:lev_scores}).
The first explicit definition of subspace embeddings was given by \cite{Sarlos:2006}, who focused on \textit{data-oblivious} sketching.
We address data-oblivious subspace embeddings in detail momentarily.

The most transparent use of subspace embedding distortion arises when $L$ is the range of a matrix $\mtx{A}$.
In this context, $\mtx{S}$ is a $\delta$-embedding for $L$ if and only if the following two-sided linear matrix inequality holds:
\begin{equation}\label{eq:subspsace_embedding_as_lmi}
    (1-\delta)^2 \mtx{A}^{\trans}\mtx{A} \preceq (\mtx{S}\mtx{A})^{\trans}(\mtx{S}\mtx{A}) \preceq (1+\delta)^2\mtx{A}^{\trans}\mtx{A}.
\end{equation}
In other words, the distortion of $\mtx{S}$ as an embedding for $\range(\mtx{A})$ is a measurement of how well the Gram matrix of $\mtx{S}\mtx{A}$ approximates that of $\mtx{A}$.

Note that in order for $\mtx{S}$ to be a subspace embedding for $L$ it is necessary that $d \geq \dim(L)$.
Therefore if $L$ is the range of an $m \times n$ matrix of full-column-rank, the requirement that $d \geq \dim(L)$ means that subspace embeddings can only be achieved when ``sketching in the embedding regime,'' in the sense of \cref{subsec:geometric_interp_sketching}.
Furthermore, substantial dimension reduction can only be achieved in this framework when~$m \gg n$. 

\subsubsection{Effective distortion}
Subspace embedding distortion is the most common measure of sketch quality, 
but it is not without its limitations.
Its greatest limitation is that it is not invariant under scaling of $\mtx{S}$ (i.e., it is not invariant under replacing $\mtx{S} \leftarrow t\mtx{S}$ for $t \neq 0$).
This is a significant limitation since many \RandNLA{} algorithms \textit{are} invariant under scaling of $\mtx{S}$; existing theoretical analyses of \RandNLA{} algorithms simply do not take this into account.  

In \cref{subapp:effective_distortion} we explore a novel concept of \textit{effective distortion} that resolves the scale-sensitivity problem.
Formally, the effective distortion of a sketching operator $\mtx{S}$ for a subspace $L$ is
\begin{align}
    \mathscr{D}_e(\mtx{S};L) = \inf\{\, \delta  \,:~
        & 0 \leq \delta \leq 1, ~ 0 < t \label{eq:sec2:eff_dist}  \\
        & t\mtx{S} \text{ is a } \delta\text{-embedding for } L \}. \nonumber
\end{align}
In words, this is the minimum distortion that \textit{any} sketching operator $t\mtx{S}$ can achieve for $L$, optimizing over $t > 0$.
We briefly reference this concept in our discussion of algorithms for least squares and optimization (\S \ref{subsec:optim_drivers}).
\cref{subsubapp:we_sketch_subspaces} makes deeper connections between effective distortion and randomized preconditioning methods for least squares.
%
%
%
%

\subsubsection{Oblivious subspace embeddings}

Data-oblivious subspace embedding (OSEs) were first used in \RandNLA{} in \cite{Sarlos:2006} and were largely popularized by \cite{Woodruff:2014}.
There is a clean way to describe the ``reliability'' of a sketching distribution in this setting.
\begin{quote}
    Consider a distribution $\mathcal{D}$ over wide $d \times m$ matrices.
    We say that $\mathcal{D}$ has \textit{the OSE property with parameters $(\delta, n, p)$}  if, for every $n$-dimensional subspace $L \subset \R^m$, we have
    \[
        \Pr\{ \mtx{S} \sim \mathcal{D} \text{ is a }\delta\text{-embedding for } L \} \geq 1-p.
    \]
\end{quote}
Theoretical analyses of sketching distributions often concern bounding $d$ as a function of $(\delta,n,p)$ to ensure that $\mathcal{D}$ satisfies the OSE property.
Naturally, all else equal, we would like to achieve the OSE property for smaller values of $d$.

Theoretical results can be used to select $d$ in practice for very well-behaved distributions, particularly the Gaussian distribution.
Results for the more sophisticated distributions (such as those of sparse sketching operators) tend to be pessimistic compared to what is observed in practice.
Some of this pessimism stems from the existence of esoteric constructions which indeed call for large embedding dimensions.
Setting these constructions aside, we have reason to be optimistic since distortion is actually not the ideal measure of sketch quality in many settings.
Indeed, \textit{effective distortion} is far more relevant for least squares and optimization, and it will always be no larger than the standard notion of distortion.

All in all, there is something of an art to choosing the best sketching distribution for a particular \RandNLA{} task.
Luckily, for most \RandNLA{} algorithms it is far from necessary to choose the ``best'' sketching distribution; good results can be obtained even when setting distribution parameters by simple rules of thumb.

\subsubsection{Johnson--Lindenstrauss embeddings}

Let $\mtx{S}$ be a $d \times m$ sketching operator and $L$ be a finite point set in $\R^m$.
We say that $\mtx{S}$ is a \textit{Johnson--Lindenstrauss embedding} (or ``JL embedding'') for $L$ with distortion $\delta$ if, for all distinct $\vct{x},\vct{y}$ in $L$, we have
\[
    1 - \delta \leq \frac{\|\mtx{S}(\vct{x} - \vct{y})\|_2^2}{\|\vct{x} - \vct{y}\|_2^2} \leq 1 + \delta.
\]
This property is named for a seminal result by William Johnson and Joram Lindenstrauss, who used randomization to prove the existence of operators satisfying this property where $d$ is logarithmic in $|L|$ and linear in $1/\delta^2$ \cite{JL84}.
%
%
%

The JL Lemma (as the result is now known) is remarkable for two reasons.
First, the requisite value for $d$ did not depend on the ambient dimension $m$ and was only logarithmic in $|L|$.
Second, the construction of the transformation $\mtx{S}$ was \textit{data-oblivious} -- a scaled orthogonal projection.
This latter fact led to questions about how one might define alternative distributions over sketching operators, with the aim of 
\begin{enumerate}
    \item being simpler to implement than a scaled orthogonal projection, and
    \item attaining similar ``data-oblivious JL properties.''
\end{enumerate}
It so happened that \textit{many} constructions could achieve these goals.
For example, \cite{IM98} and~\cite{DG02} relaxed the condition of being a (scaled) orthogonal projector to $\mtx{S}$ having iid Gaussian entries, which still results in a rotationally-invariant distribution.
As another example, \cite{Ach03_JRNL} relaxed the rotational invariance by choosing the entries of $\mtx{S}$ to be scaled Rademacher random variables.

\subsection{(In)essential properties of sketching distributions}\label{subsec:inessential_properties_of_skops}

Distributions $\mathcal{D}$ over wide sketching operators are typically designed so that, for $\mtx{S} \sim \mathcal{D}$, the mean and covariance matrices are
\[
    \E[\mtx{S}] = \mtx{0} \qquad\text{and}\qquad \E[\mtx{S}^{\trans}\mtx{S}] = \mtx{I}.
\]
The property that $\E[\mtx{S}] = \mtx{0}$ is important -- if not ubiquitous -- in \RandNLA{}.
However, there is some flexibility in the latter property, as in most situations it suffices for the covariance matrix to be a scalar multiple of the identity.

To understand why we have flexibility in the scale of the covariance matrix, consider how $\E[\mtx{S}^{\trans}\mtx{S}] = \mtx{I}$ is equivalent to $\mtx{S}$ preserving squared Euclidean norms in expectation.
As it happens, the vast majority of algorithms mentioned in this monograph do not need sketching operators to preserve norms.
Rather, they rely on sketching preserving \textit{relative norms}, in the sense that $\|\mtx{S}\vct{u}\|_2 / \|\mtx{S}\vct{v}\|_2$ should be close to $\|\vct{u}\|_2 / \|\vct{v}\|_2$ for all vectors $\vct{u},\vct{v}$ in a set of interest.
Such a property is clearly unaffected if every entry of $\mtx{S}$ scaled by a fixed nonzero constant (i.e., if $\mtx{S}$ is replaced by $t\mtx{S}$ for some $t \neq 0$).

This \nameCref{sec2:rblas} uses scale-agnosticism to help describe sketching distributions with reduced emphasis on whether the operator is wide or tall.
For example, if the entries of $\mtx{S}$ are iid mean-zero random variables of finite variance, then both $\E[\mtx{S}^{\trans}\mtx{S}]$ and $\E[\mtx{S}\mtx{S}^{\trans}]$ are scalar multiples of the identity matrix.
Speaking loosely, the former property justifies using $\mtx{S}$ to sketch from the left and the latter property justifies using $\mtx{S}^{\trans}$ to sketch from the right.

With this observation in mind, this \nameCref{sec2:rblas} ignores most matters of scaling that is applied equally to all entries of a sketching operator.
This manifests in how we regularly describe sketching operators as having entries in $[-1, 1]$ even though it is more common to have entries in $[-v, v]$ for some positive $v$ (which is set to achieve an identity covariance matrix).
Note that this \textit{does not} confer freedom to scale individual rows, columns, or entries of a sketching operator separately from one another.

The main places where scaling matters are in algorithms for norm estimation (see \cref{subsubsec:norm_rank_est}) and algorithms which only sketch a portion of the data in a larger problem.
The subtleties in this latter situation warrant a detailed explanation.

\subsubsection{Scale sensitivity: partial sketching}

Let $\mtx{G}$ be an $n \times n$ psd matrix and $\mtx{A}$ be a very tall $m \times n$ matrix.
Suppose that we approximate
\[
        \mtx{H} = \mtx{A}^{\trans}\mtx{A} + \mtx{G}
\]
by a \textit{partial sketch}
\[
    \mtx{H}_{\mathrm{sk}} = (\mtx{S}_o\mtx{A})^{\trans}(\mtx{S}_o\mtx{A}) + \mtx{G}
\]
where $\mtx{S}_o$ is a $d \times m$ sketching operator.
How should we understand the statistical properties of $\mtx{H}_{\mathrm{sk}}$ as an estimator for $\mtx{H}$?

At the simplest level we can turn to the idea of subspace embedding distortion.
Using the characterization of distortion in \eqref{eq:subspsace_embedding_as_lmi}, we could study the distribution of the minimum $\delta \in (0, 1)$ for which 
\[
    (1-\delta)^2\mtx{H} \preceq \mtx{H}_{\mathrm{sk}} \preceq (1+\delta)^2 \mtx{H}.
\]
One can go beyond distortion by lifting to a higher-dimensional space.
Letting $\sqrt{\mtx{G}}$ denote the Hermitian square root of $\mtx{G}$, we define the augmented sketching operator and augmented data matrix
\[
    \mtx{S} = \begin{bmatrix}\mtx{S}_o & \mtx{0} \\ \mtx{0} & \mtx{I}\end{bmatrix} \qquad\text{and}\qquad \mtx{A}_{\mtx{G}} = \begin{bmatrix}\mtx{A} \\ \sqrt{\mtx{G}} \end{bmatrix}
\]
This lets us express $\mtx{H} =  \mtx{A}_{\mtx{G}}^{\trans}\mtx{A}_{\mtx{G}}^{}$ and $\mtx{H}_{\text{sk}} = \left(\mtx{S}\mtx{A}_{\mtx{G}}\right)^{\trans}\left(\mtx{S}\mtx{A}_{\mtx{G}}^{}\right)$.
Therefore the statistical properties of $\mtx{H}_{\text{sk}}$ as an approximation to $\mtx{H}$ can be understood in terms of how $\mtx{S}$ preserves (or distorts) the range of $\mtx{A}_{\mtx{G}}$.

\subsubsection{Scale sensitivity: row sampling from block matrices}

The concept of partial sketching can arise when sketching block matrices, which are indeed encountered in many applications.
For example, it is widely appreciated that a ridge regression problem with tall $m \times n$ data matrix $\mtx{A}$ and regularization parameter $\mu$ can be lifted to an ordinary least squares problem with data matrix $\mtx{A}_{\mu} := [\mtx{A};\sqrt{\mu}\mtx{I}]$.

Suppose we want to sketch $\mtx{A}_{\mu}$ by a row sampling operator $\mtx{S}$.
It is natural to treat the lower $n$ rows of $\mtx{A}_{\mu}$ differently than its upper $m$ rows.
In particular, it is natural for $\mtx{S}$ to be an operator that produces $\mtx{S}\mtx{A} = [\mtx{S}_o\mtx{A};\sqrt{\mu}\mtx{I}]$ with some other $d \times m$ row sampling operator $\mtx{S}_o$.
Here, even if $\mtx{S}_o$ sampled rows from $\mtx{A}$ uniformly at random, the map $\mtx{A}_{\mu} \mapsto \mtx{S}\mtx{A}_{\mu}$ would \textit{not} sample uniformly at random from $\mtx{A}_{\mu}$.
Therefore there is a sense in which partial sketching is a way of incorporating non-uniform row sampling into other sketching distributions; see \cite{dkm_matrix1} for the origins of this interpretation.
In the context of this specific example, the nonuniformity would necessitate that $\mtx{S}_o$ be scaled to have entries in $\{0,\pm 1/\sqrt{d}\}$.
We refer the reader to \cref{subsubsec:randblas:LongAxSparse} and \cref{subsec:lev_scores_standard} for more discussion on sketching operators that implement row sampling.

\section{Dense sketching operators}
\label{subsec:dense_skops}

The \RandBLAS{} should provide methods for sampling sketching operators with iid entries drawn from distinguished distributions.
Across this broad category, we believe the following types of operators stand out:
\begin{itemize}
    \item 
    \textit{Rademacher sketching operators}: entries are $\pm 1$ with equal probability;
    \item 
    \textit{uniform sketching operators}: entries are uniform over $[-1, 1]$;
    \item 
    \textit{Gaussian sketching operators}: entries follow the standard normal distribution.
\end{itemize}
We believe the \RandBLAS{} should also support sampling row-orthonormal or column-orthonormal matrices uniformly at random from the set of all such matrices.
%
%

The theoretical results for Gaussian operators are especially strong.
However, there is little practical difference in the performance of \RandNLA{} algorithms between any of the three entrywise iid operators given above.
This is reflected in implementations such as \cite{LLSKT:2017} that only use uniform sketching operators.
The practical equivalence between these types of sketching operators also has theoretical support through \textit{universality principles} in high-dimensional probability \cite{Vershynin:2018}, \cite{OT:2017}, \cite[\S 8.8]{MT:2020}, \cite{derezinski2020precise}.
In what follows we speak to implementation details and the intended use cases for these operators.

\subsubsection{Sampling iid-dense sketching operators}

Sampling from the Rademacher or uniform distributions is the most basic operation of random number generators.
Methods for sampling from the Gaussian distribution involve transforming random variables sampled uniformly from $[0, 1]$.
There are two transformations of interest for the \RandBLAS{}: 
Box-Muller \cite{BM:1958:sampling}; 
and the Ziggurat transform \cite{MT:2000:Ziggurat}.
The former should be included in the \RandBLAS{} because it is easy to implement and parallelizes well.
The latter method is far more efficient on a single thread, and it has been used within \RandNLA{} (see \cite{MSM:2014:LSRN}), but it does not parallelize well~\cite[\S 37.2.3]{Nguyen:2007:GPUgems3}.
We postpone any recommendation for whether it should be an option in the \RandBLAS{}.

\RandNLA{} algorithms tend to be very robust to the quality of the random number generator.
As a result, it is not necessary for us to sample from the Gaussian distribution with high statistical accuracy.
This is due in part to the aforementioned universality principles, and it can be seen through the success of \textit{sub-Gaussian distributions} as an analysis framework in high-dimensional probability \cite[\S 2]{Vershynin:2018}.
From an implementation standpoint, there is likely no
need to sample from the Gaussian distribution beyond
single precision \cite{gunnar_single_precision}.
It is worth exploring if even lower precisions (e.g., half-precision) would suffice for practical purposes.

\subsubsection{Applying iid-dense sketching operators}

If a dense sketching operator is realized explicitly in memory then it can (and should) be applied by an appropriate \BLAS{} function, most likely \code{gemm}.
Many \RandNLA{} algorithms provide good practical performance even with such simple implementations, although there is potential for reduced memory or communication requirements if a sketching operator is applied without ever fully allocating it in-memory.
There is a large design space for such algorithms with iid-dense sketching operators when using counter-based random number generators (see \cref{subsec:randblas_rngs}).
Such functionality could appear in an initial version of a \RandBLAS{} standard.
The reference implementations of such functions could start as mere wrappers around routines to generate a sketching operator and then apply that operator via~\code{gemm}.

\subsubsection{Sampling and applying Haar operators}

If we suppose left-sketching, then the Haar distribution is the uniform distribution over row-orthonormal matrices.
If we instead suppose right-sketching, then it is the uniform distribution over column-orthonormal matrices.
We call these operators ``dense'' because if one is sampled and then formed explicitly, it will be dense with probability one.

There are two qualitative approaches to sampling from this distribution.
The naive approach essentially requires sampling from a Gaussian distribution and performing a QR factorization, at a total cost of $O(d^2 m)$; see \cite[\S 1 - \S 4]{Li:1992} and more general methods in \cite{Mezzadri:2007}.
A more efficient approach -- which costs only $O(dm)$ time -- involves constructing the operator as a composition of suitable Householder reflectors \cite{Stewart:1980:sampleHaar}.
This approach has the secondary benefit of not needing to form the sketching operator explicitly.

Haar operators are of interest not just for sketching in \RandNLA{} algorithms but also for generating test data for evaluating other sketching operators.
As such, we believe they are natural to include in a first version of a \RandBLAS{} standard.

\subsubsection{Intended use-cases}

Using terminology from \cref{subsec:geometric_interp_sketching}, dense sketching operators are commonly used for ``sketching in the sampling regime.''
In particular, they are the workhorses of randomized algorithms for low-rank approximation.
They also have applications in certain randomized algorithms for ridge regression and some full-rank matrix decomposition problems.

These distributions are much less useful for sketching dense matrices ``in the embedding regime'' (again in the sense of \cref{subsec:geometric_interp_sketching}).
This is because they are more expensive to apply to dense matrices than many other types of sketching operators.
These types of sketching operators \textit{might} be of interest in the embedding regime if applied to sparse or otherwise structured data matrices. 

\section{Sparse sketching operators}
\label{subsec:sparse_skops}

The \RandNLA{} literature describes many types of sparse sketching operators, almost always under the convention of sketching from the left.
We think it is important to define sketching distributions in a way that is agnostic to sketching from the left or right.
Indeed, while we often focus on left-sketching for ease of exposition, asserting that this is ``without loss of generality'' ignores the plight of the user tasked with right-sketching.

In order to achieve our desired agnosticism, we use a taxonomy for sparse sketching operators which has not appeared in prior literature.
To describe it, we use the term \textit{short-axis vector} in reference to the columns of a wide matrix or rows of a tall matrix.
The term \textit{long-axis vector} is defined analogously, as the rows of a wide matrix or columns of a tall matrix.
In these terms, we have the following families of sparse sketching operators.
\begin{itemize}
    \item 
    \textit{Short-axis-sparse} sketching operators. 
    The short-axis vectors of these operators are independent of one another.
    Each short-axis vector has a fixed (and very small) number of nonzeros.
    Typically, the indices of the nonzeros in each short-axis vector are sampled uniformly without replacement.
    \item 
    \textit{Long-axis-sparse} sketching operators. 
    The long-axis vectors of these operators are independent of one another.
    For a given long-axis vector, the indices for its nonzeros are sampled with replacement according to a prescribed probability distribution (which can be uniform).
    The value of a given nonzero is affected by the number of times its index appears in the sample for that vector.
    \item 
    \textit{Iid-sparse} sketching operators.
    Mathematically, these can be described as starting with an iid-dense sketching operator and ``zeroing-out'' entries in an iid-manner with some high probability. (From an implementation standpoint this would work the other way around, randomly choosing a few entries to make nonzero.)
\end{itemize}
When abbreviations are necessary, we suggest that short-axis-sparse sketching operators be called \textit{SASOs} and that long-axis-sparse sketching operators be called \textit{LASOs}.
Most of our use of such abbreviations appears here and in \cref{subapp:SJLTs}.
A visualization of these types of sketching operators is given in \cref{fig:sparse_sketch_ops}.

Before proceeding further we should say that we are \textit{not} in favor of including iid-sparse sketching operators in the \RandBLAS{}.
Our first reason for this is that their theoretical guarantees are not as strong as either SASOs (see the discussion at the end of \cite[\S 7.4]{Tropp:2021:LecNotes} and remarks in \cite[\S 2.4]{Liberty:2009}) or LASOs~\cite{DLDM20_sparse_TR,DLPM21_newtonless_TR}.
Our second reason is that their lack of predictable structure makes it harder to implement efficient parallel algorithms for applying these operators.
Therefore in what follows we only give details on SASOs and LASOs.

\begin{figure}[!htb]
	\centering
	\begin{overpic}[width=0.85\textwidth]{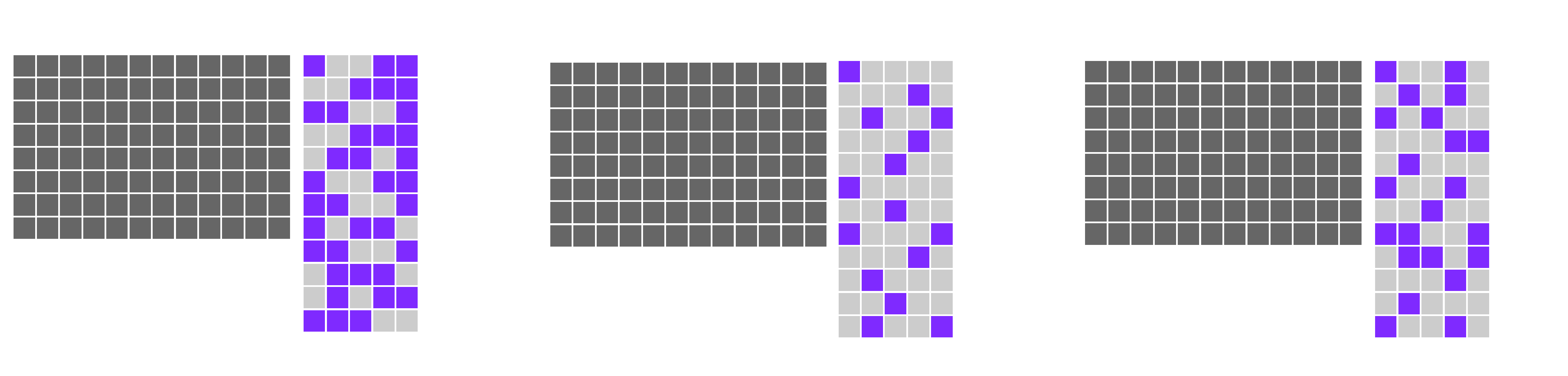} 
			
			\put(8,21){\color{black}{\large $\bf A$}}   
			\put(21,21){\color{black}{\large $\bf S$}}
			
			\put(43,21){\color{black}{\large $\bf A$}}   
			\put(55,21){\color{black}{\large $\bf S$}}  
			
			\put(77,21){\color{black}{\large $\bf A$}}   
			\put(90,21){\color{black}{\large $\bf S$}}   	
			
			\put(-3,22){\color{black}{(a)}}   
			\put(32,22){\color{black}{(b)}}   
			\put(66,22){\color{black}{(c)}}   
			
		\end{overpic}
	\caption{Illustration of a SASO (a) with 3 non-zero entries per row, LASO (b) with 3 non-zero entries per column, and an iid-sparse sketching operator (c) with iid non-zero entries.}\label{fig:sparse_sketch_ops}
\end{figure}
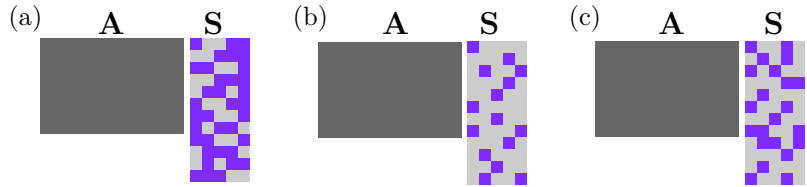

\subsection{Short-axis-sparse sketching operators}
\label{subsubsec:randblas:ShortAxSparse}

SASOs include sketching operators known as {sparse Johnson--Lindenstrauss transforms}, the {Clarkson--Woodruff transform}, {CountSketch}, and {OSNAPs} \cites{KN:2012:SJLTs,CW:2013:CWTransform,MM13_STOC:CountSketch,NN:2013:OSNAPs}.
These constructions are all described assuming we sketch from the left, and as such, they are all stated for \textit{wide} sketching operators.
They are described as having a fixed number of nonzeros within each column.
The more general notion (for sketching from the left or right) is to say there is a fixed number of nonzeros per short-axis vector.

The short-axis vectors of a SASO should be independent of one another.
One can select the locations of nonzero elements in different ways;
we are interested in two methods from \cite{KN:2012:SJLTs}.
For a wide $d \times m$ operator, we can
\begin{enumerate}
    \item sample $k$ indices uniformly from $\idxs{d}$ without replacement, once for each column, or 
    \item divide $\idxs{d}$ into $k$ contiguous subsets of equal size, and then for each column we select one index from each of the $k$ index sets.
\end{enumerate}
These definitions are extended from wide sketching operators to tall sketching operators in the natural~way.

For either method, the nonzero values in a SASO's short-axis vector are canonically independent Rademachers.
Alternatively, they can be drawn from other sub-Gaussian distributions.
For example, in the wide case, drawing the nonzeros independently and uniformly from a union of disjoint intervals, such as $[-2,-1]\cup[1,2]$, can protect against the possibility of a given row of $\mtx{S}$ being orthogonal to a column of a matrix to be sketched~\cite{tygert_disjoint_distribution}.

Details SASOs are provided in \cref{subapp:SJLTs}.
This includes implementation notes, a short historical summary of relevant theory, and remarks on setting the sparsity parameter $k$.
On the topic of theory, we note here that the state-of-the-art results for SASOs are due to Cohen \cite{Cohen:2016:SJLTs}.
More information can be found in the lecture notes \cite{Mah-mat-rev_BOOK,Mah16_RLA_TR,RandNLA_PCMIchapter_chapter}, \cite[\S 7.4]{Tropp:2021:LecNotes} and the surveys \cite{Woodruff:2014}, \cite[\S 9.2]{MT:2020}.

\begin{remark}[Naming conventions]
    The concept of what we would call a ``wide SASO'' is referred to in the literature as an \textit{OSNAP}.
    We have a slight preference for ``SASO'' over ``OSNAP'' for two reasons.
    First, it pairs naturally with the abbreviation \textit{LASO} for long-axis-sparse sketching operators, which is valuable for taxonomizing sparse sketching operators.
    Second, the literature consistently describes OSNAPs as having a fixed number of nonzeros per column.
    While this description is appropriate for left sketching, it is not appropriate for right sketching.
\end{remark}

\subsection{Long-axis-sparse sketching operators}
\label{subsubsec:randblas:LongAxSparse}

This category includes row and column sampling, LESS embeddings \cite{DLDM20_sparse_TR}, and LESS-uniform operators \cite{DLPM21_newtonless_TR}.

In the wide case, a LASO has independent rows and a fixed upper bound on the number of nonzeros per row.
All rows are sampled with reference to a distribution $\vct{p}$ over $\idxs{m}$ (which can be uniform) and a positive integer $k$.
Construction begins by sampling $t_1,\ldots,t_k$ from $\idxs{m}$ with replacement according to $\vct{p}$.
Then we initialize
\begin{equation}\label{eq:init_row_of_LASO}
  \mtx{S}[i,\fslice{}] =\frac{1}{\sqrt{d k}}\bigg(\sqrt{\frac{b_1}{p_1}},
  \ldots,\sqrt{\frac{b_m}{p_m}}\ \bigg),
\end{equation}
where $b_j$ is the number of times the index $j$ appeared in the sample $(t_1,\ldots,t_k)$.
We finish constructing the row by multiplying each nonzero entry by an iid copy of a mean-zero random variable of unit variance (e.g., a standard Gaussian random variable).
Such a LASO will have at most $k$ nonzeros per row and hence at most $dk$ nonzeros in total.
Note that this is much smaller than $mk$ nonzeros required by a SASO with the same parameters.

The quality of sketches produced by LASOs when $\vct{p}$ is uniform depends on the properties of the matrix to be sketched.
Specifically, it will depend on the \textit{leverage scores} of the matrix.
The leverage score concept, introduced in \cref{subsec:lev_scores}, is important for constructing data-aware sketching operators that implement row or column sampling.
If $\vct{p}$ is the leverage score distribution of some matrix then the sketching operator is known as a Leverage Score Sparsified (LESS) embedding for that matrix \cites{DLDM20_sparse_TR}.
The term \textit{LESS-uniform} has been used for long-axis-sparse operators that use the uniform distribution for $\vct{p}$ \cite{DLPM21_newtonless_TR}.

\begin{remark}[Scale]\label{rem:scaling_row_sampling_operators}
    The scaling factor $1/\sqrt{dk}$ appearing in the initialization \eqref{eq:init_row_of_LASO} is the same for all rows of $\mtx{S}$ (in the wide case, i.e., for each long-axis vector).
    This factor is necessary so that once the nonzeros in $\mtx{S}$ are multiplied by mean-zero unit-variance random variables, we have $\E[\mtx{S}^{\trans}\mtx{S}] = \mtx{I}_m$.
    This scaling matters when one cares about subspace embedding distortion or when one is only sketching a portion of the problem data (see \cref{subsec:inessential_properties_of_skops}).
    This scaling has no effect on effective distortion if $\vct{p}$ is uniform.
\end{remark}

\section{Subsampled fast trigonometric transforms}
\label{subsec:srfts}

\textit{Fast trigonometric transforms} (or \textit{fast trig transforms}) are orthogonal or unitary operators that take $m$-vectors to $m$-vectors in $O(m \log m)$ time or better.
The most important examples in this class are the Discrete Fourier Transform (for complex-valued inputs) and the Discrete Cosine Transform (for real-valued inputs).
The Walsh-Hadamard Transform is also notable; although it only exists when $m$ is a power of two, it is equivalent to a Kronecker product of $\log_2 m$ Discrete Fourier Transforms of size $2 \times 2$, and the standard algorithm for applying it involves no multiplications and entails no branching.

Traditionally, trig transforms are valued for their ability to map dense input vectors with a periodic structure into sparse output vectors.
Within \RandNLA{}, we are interested in them for the opposite reason: the fact that they map inputs that lack periodic structure to \textit{dense} outputs.
This behavior is useful because if we preprocess an input to destroy any periodic structure with high probability, the resulting output should be easier to approximate by random coordinate subsampling.
This leads to the idea of a \textit{subsampled randomized fast trig transforms} or \textit{SRFTs}.


\subsubsection{Traditional SRFTs}
Formally, a $d \times m$ SRFT takes the form
\[
\mtx{S} = \sqrt{m/d} \,\, \mtx{R}\mtx{F}\mtx{D} ,
\]
where $\mtx{D}$ is a diagonal matrix of independent Rademachers, $\mtx{F}$ is a fast trig transform that maps $m$-vectors to $m$-vectors, and $\mtx{R}$ randomly samples $d$ components from an $m$-vector~\cites{AC:2006,AC:2009}.
For added robustness one can define SRFTs slightly differently, replacing $\mtx{S}$ by $\mtx{S}\mtx{\Pi}$ for a permutation matrix $\mtx{\Pi}$ \cites{MT:2020}.

SRFTs are appealing for their efficiency and theoretical guarantees.
Speaking to the former aspect, a $d \times m$ SRFT can be applied to an $m \times n$ matrix in as little as $O(mn \log d)$ time by using methods for subsampled fast trig transforms \cite[\S 3.3]{WLRT:2008}, \cite[\S 3.3]{Liberty:2009}.
Theoretical guarantees for SRFTs are usually established assuming $\mtx{F}$ is the Walsh-Hadamard transform \cites{DMMS07_FastL2_NM10, Tropp:2011, BG:2013}.
These guarantees are especially appealing since they do not rely on tuning parameters such as sparsity parameters required by sketching operators from \cref{subsec:sparse_skops}.

The trouble with SRFTs is that they are notoriously difficult to implement efficiently.
Even their best-case $O(mn \log d)$ complexity is higher than the $O(mnk)$ complexity of a SASO that is wide with $k \ll \log d$ nonzeros per column.
However, SRFTs have an advantage when it comes to memory: if one overwrites $\mtx{A}$ by the $m \times n$ matrix $\mtx{B} :=\sqrt{m/d}\,\mtx{F}\mtx{D}\mtx{A}$ in $O(mn \log m)$ time, then $\mtx{S}\mtx{A}$ can be accessed as a submatrix of rows of $\mtx{B}$ without losing access to $\mtx{A}$ or $\mtx{A}^{\trans}$ as linear operators.
Further investigation is needed to determine the true value of this in-place nondestructive implementation.
For the time being, we do not believe that traditional SRFTs are essential for a preliminary \RandBLAS{} standard.

\subsubsection{Block SRFTs}
Let $p$ be a positive integer, $r = m/p$ be greater than $d$, and $\mtx{R}$ be a matrix that randomly samples $d$ components from an $r$-vector.
For each index $i \in \idxs{p}$, we introduce a $d \times r$ sketching operator
\[
    \mtx{S}_i =\sqrt{r/d}\, \mtx{D}_i^{\mathrm{post}}\mtx{R} \mtx{F} \mtx{D}_i^{\mathrm{pre}},
\]
where $\mtx{D}_i^{\mathrm{post}}$ and $\mtx{D}_i^{\mathrm{pre}}$ are diagonal matrices filled with independent Rademachers.
The \textit{block SRFT} \cite{grigori_et_al:2022:block_srft} is defined columnwise as $\mtx{S} =  [ \begin{matrix}  \mtx{S}_1 & \mtx{S}_2 & \hdots & \mtx{S}_p \end{matrix}]$.

Block SRFTs can effectively leverage parallel hardware using \textit{serial} implementations of the fast trig transform.
For concreteness, suppose $\mtx{A}$ is $m \times n$ and distributed block row-wise among $p$ processors.
We apply the block SRFT $\mtx{S}$ by the formula 
\begin{equation*} 
    \mtx{S} \mtx{A} =   \sum_{i \in \idxs{p}} \mtx{S}_i \mtx{A}_i,
\end{equation*}
where $\mtx{A}_i$ is the block of rows of $\mtx{A}$ stored on processor $i$.
The multiplication $\mtx{S}_i \mtx{A}_i$, computed locally on each processor, is followed by a reduction operation among processors to sum the local contributions.

We are undecided as to whether block SRFTs are appropriate for a preliminary \RandBLAS{} standard.
Their comparative ease of implementation is favorable.
However, they are problematic in that the definition of the distribution changes as we vary $p$, which complicates reproducibility across platforms.

\subsubsection{Historical remarks and further reading}


The development of SRFTs began with \textit{fast Johnson--Lindenstrauss transforms} (FJLTs) \cite{AC:2006}, which replace the matrix ``$\mtx{R}$'' in the SRFT construction by a particular type of sparse matrix.
FJLTs were first used in \RandNLA{} for least squares and low-rank approximation by \cite{Sarlos:2006}.
The jump from FJLTs to SRFTs was made independently in \cite{DMMS07_FastL2_NM10} and \cite{WLRT:2008} for usage in least squares and low-rank approximation, respectively.

For more background on this topic we refer the reader to the book \cite{Mah-mat-rev_BOOK}, the lecture notes \cite{Mah16_RLA_TR}, \cite{RandNLA_PCMIchapter_chapter}, \cite[\S 7.5]{Tropp:2021:LecNotes}, and the survey \cite[\S 9.3]{MT:2020}.
We also note that SRFTs are sometimes called \textit{randomized orthornomal systems} \cites{PW:2016:HessSketch, PW:2017:NewtSketch, OPA:2019} or (with slight abuse of terminology) FJLTs.
Finally, we point out that a type of  ``SRFTs without subsampling'' has been successfully used to approximate Gaussian matrices needed in random features approaches to kernel ridge regression \cite{LSS:2013:fastfood}.

\begin{remark}[Navigating the literature]
    The reader should be aware that \cite{AC:2009} is the journal version of \cite{AC:2006}.
    Additionally, while \cite{LWMRT:2007} also describes SRFTs, it was actually written after \cite{WLRT:2008}.
\end{remark}

\section{Multi-sketch and quadratic-sketch routines}
\label{subsec:multi_sketch}

For many years, the performance bottleneck in NLA algorithms has been data movement, rather than FLOPs performed on the data.
For example, a general matrix-matrix multiply with $n \times n$ matrices would do $O(n^3)$ data movement if implemented naively with three nested loops, but can be up to a factor of $\sqrt{M}$ smaller, where $M$ is the cache size, if appropriately implemented using loop tiling.

In our context of \RandNLA{}, the fastest randomized algorithms for low-rank matrix approximation involve computing multiple sketches of a data matrix.
Such \textit{multi-sketching} presents new challenges and opportunities in the development of optimized implementations with minimal data movement.

We believe the \RBLAS{} should include functionality for at least three types of multi-sketching, listed below.
The end of \cref{subsubsec:svd_algs} points to algorithms that use these primitives.
In all cases of which we are aware, these primitives are only used for sketching in the sampling regime.
\begin{enumerate}
    \item Generate $\mtx{S}$ and compute $\mtx{Y}_1 = \mtx{A}\mtx{S}$ and $\mtx{Y}_2 = \mtx{A}^{\trans}\mtx{A}\mtx{S}$.
    We illustrate the use of this primitive explicitly in \cref{alg:qb3}.
    \item 
    Generate independent $\mtx{S}_1,\mtx{S}_2$, and compute $\mtx{Y}_1 = \mtx{A}\mtx{S}_1$ and $\mtx{Y}_2 = \mtx{S}_2\mtx{A}$.
    Algorithms which use this primitive typically need to retain $\mtx{S}_1$ or $\mtx{S}_2$ for later use \cite[pg.~251]{HMT:2011}, \cite[Algorithm 2]{YGLLL:2017}, \cite[\S~1.4]{TYUC:2017:singlepass}. 
    \item 
    Generate independent $\mtx{S}_1,\mtx{S}_2,\mtx{S}_3,\mtx{S}_4$, and compute $\mtx{Y}_1 = \mtx{A}\mtx{S}_1$, $\mtx{Y}_2 = \mtx{S}_2\mtx{A}$, and $\mtx{Y}_3 = \mtx{S}_3\mtx{A}\mtx{S}_4$.
\end{enumerate}

\noindent 
Having identified these operations as basic building blocks, we arrive at the following question.
\begin{quote}
    What combination of sketching distributions should be supported in multi-sketching of types 2 and 3?
\end{quote}
For the former type (i.e., type 2), it is important to support at least the case that both $\mtx{S}_1$ and $\mtx{S}_2$ are dense sketching operators, and it may be useful to support when both are fast operators.
At this point, we do not see an advantage for one of $(\mtx{S}_1, \mtx{S}_2)$ to be a fast operator and for the other to be dense.
For the latter type (i.e., type 3), \cite[\S 7.3.2]{TYUC:2017:singlepass} suggests that $(\mtx{S}_1,\mtx{S}_2)$ be Gaussian and that $(\mtx{S}_3,\mtx{S}_4)$ be SRFTs.
We believe that it would be reasonable to use SASOs in place of SRFTs in this context.

The \RandBLAS{} should provide methods to compute sketches that are quadratic in the data matrix.
By ``quadratic sketch,'' we mean a linear sketch of $\mtx{A}^{\trans}\mtx{A}$ or $\mtx{A}\mtx{A}^{\trans}$.
This operation is ubiquitous in algorithms for low-rank approximation.
As with multi-sketching, all uses of quadratic sketching (of which we are aware) entail sketching in the sampling regime.
It is not possible to fundamentally accelerate this kind of sketching by using fast sketching operators.\footnote{This point has also been made in a related setting \cite[\S 11.6.1]{MT:2020}.}
Therefore it would be reasonable for \RandBLAS{}'s quadratic sketching methods to only support dense sketching operators.
(The preceding comments also apply to type 1 multi-sketching.)
In essence, this asks for a high-performance implementation of the composition of the \BLASlev{3} functions \code{syrk} and \code{gemm}: $(\mtx{A},\mtx{S}) \mapsto \mtx{A}\mtx{A}^{\trans}\mtx{S}$.
There is a substantial amount of structure in quadratic sketching that could be leveraged for reduced data movement, which suggests that the \RandBLAS{} would benefit significantly from having optimized routines for this functionality.

\chapter{Least Squares and Optimization}
\label{sec3:LS_and_optim}

\minitoc
\bigskip

Numerical linear algebra is the backbone of the most widely-used algorithms for continuous optimization.
Continuous optimization, in turn, is a workhorse for many scientific computing, machine learning, and data science applications. 

The connections between optimization and linear algebra are often introduced with \emph{least squares problems}.
Such problems have been used as a tool for curve fitting since the days of Gauss and Legendre over 200 years ago --- several decades before Cayley even defined linear algebraic concepts such as the matrix-inverse!
These problems are also remarkable because algorithms for solving them easily generalize to more complicated settings.
Indeed, one of this \nameCref{sec3:LS_and_optim}'s key messages is that, by adopting a suitable perspective, one can use randomization in essentially the same way to solve a wealth of different quadratic optimization problems.

Our perspective entails describing all least squares problems in terms of an $m \times n$ data matrix $\mtx{A}$ with at least as many rows as columns.
Specifically, we express the overdetermined problem as 
\begin{equation*}
\min_{\vct{x} \in \R^n}\|\mtx{A}\vct{x} - \vct{b}\|_2^2
\label{eqn:ls-overdet}
\end{equation*}
for a vector $\vct{b}$ in $\R^m$, 
while we express the underdetermined problem as
\begin{equation*}
\min_{\vct{y} \in \R^m}\{ \|\vct{y}\|_2^2 ~|~ \mtx{A}^{\trans}\vct{y} = \vct{c} \}
\label{eqn:ls-underdet}
\end{equation*}
for a vector $\vct{c}$ in $\R^n$.
Of course, both of these models could be expressed in the corresponding ``argmin'' formulation.
We generally prefer the ``min'' formulation for the optimization problem itself and use ``argmin'' only for the set of optimal solutions.

\cref{subsec:optim_problem_classes} introduces the problems we consider: minimization of regularized quadratics and various generalizations of least squares problems.
For each problem, it provides high-level comments on structures and desired outcomes that can make randomized algorithms preferable to classical ones.

\cref{subsec:optim_drivers} covers the drivers for these problems based on \RandNLA{}.
It details the problem structures that stand to benefit from a particular driver, and it highlights other linear algebra problems that largely reduce to solving problems amenable to these drivers.
\cref{subsec:optim_comp_routines} details some essential computational routines that would power the drivers.

The rest of the \cref{sec3:LS_and_optim} is largely supplemental.
\cref{subsec:other_opt_algs} reviews randomized optimization algorithms that we find notable but out-of-scope, as well as one type of deterministic computational routine that is potentially useful for (but not required by) the drivers.
We conclude by describing existing \RandNLA{} libraries for least squares and optimization in \cref{subsec:opt_libraries}.

\section{Problem classes}
\label{subsec:optim_problem_classes}

This \nameCref{sec3:LS_and_optim} covers drivers for two related classes of optimization problems: 
minimizing regularized positive definite quadratics (\S \ref{subsubsec:optim:min_quadratic_prob_class})
and certain generalizations of overdetermined and underdetermined least squares which we refer to as \textit{saddle point problems} (\S \ref{subsubsec:optim:saddle_prob_class}).
Problems in both classes can naturally be transformed to equivalent linear algebra problems.\footnote{As we explain later, the saddle point problems are equivalent to so-called \textit{saddle point systems}, which are well-studied in the NLA literature; see \cites{BGL:2005:saddle_sys,OA:2017:quasidef_book}} 
Functionality for solving these problems can easily provide the foundation for managing the core linear algebra kernels of larger optimization algorithms.

\subsubsection{How can we measure the accuracy of an approximate solution?}
The problem of quantifying the error of an approximate solution to a least squares or saddle point problem is very important.
While we would like to address this topic up-front, a proper discussion requires technical background that we only cover in \cref{subsec:optim_comp_routines}.
Furthermore, even given this background, there are various subtleties and special cases that would be laborious to describe here.
Therefore we defer the important topic of error metrics for least squares and related problems to \cref{subapp:error_metrics}.

\subsection{Minimizing regularized quadratics}
\label{subsubsec:optim:min_quadratic_prob_class}

Let $\mtx{G}$ be a positive semidefinite (psd) linear operator, and let $\mu$ be a positive regularization parameter.
One of the main topics of this \nameCref{sec3:LS_and_optim} is algorithms for computing approximate solutions to problems of the form
\begin{equation}\label{eq:unconstr_quadratic_opt}
    \min_{\vct{x}} \vct{x}^{\trans}\left(\mtx{G} + \mu\mtx{I}\right)\vct{x} - 2\vct{h}^{\trans}\vct{x}.
\end{equation}
Note that solving \eqref{eq:unconstr_quadratic_opt} is equivalent to solving $\left(\mtx{G} + \mu \mtx{I}\right)\vct{x} = \vct{h}$.
We refer to such problems in different contexts throughout this section.
In some contexts, we say that $\mtx{G}$ is $n \times n$, and in others, we say it is $m \times m$.

This \nameCref{sec3:LS_and_optim} covers algorithms for solving these problems to varying degrees of accuracy.
\begin{itemize}
    \item 
    Methods for solving to higher accuracy will access $\mtx{G}$ repeatedly by matrix-matrix and matrix-vector multiplication.
    \item
    Methods for solving to lower  accuracy may vary in how they access $\mtx{G}$: 
    they may only entail selecting a subset of its columns; or 
    they may perform a single matrix-matrix multiplication $\mtx{G}\mtx{S}$ with a tall and thin matrix $\mtx{S}$.
\end{itemize}
Note that the low accuracy methods may be useful in machine learning contexts such as kernel ridge regression (KRR), where an inaccurate solution to \eqref{eq:unconstr_quadratic_opt} can still be useful for downstream computational tasks.

\begin{remark}
    If the linear operator $\mtx{G}$ implements the action of an implicit Gram matrix $\mtx{A}^{\trans}\mtx{A}$ (with $\mtx{A}$ known) then it would be preferable to reformulate \eqref{eq:unconstr_quadratic_opt} as \eqref{eq:saddle_opt_x}, below, with $\vct{b} = \vct{0}$ and $\vct{c} = \vct{h}$.
\end{remark}

\subsubsection{Amenable problem structures}

The suitability of methods we describe for problem \eqref{eq:unconstr_quadratic_opt} will depend on how many eigenvalues of $\mtx{G}$ are larger than $\mu$.
Supposing $\mtx{G}$ is $n \times n$, it is desirable that the number of such eigenvalues is much less than $n$.
The data $(\mtx{G},\mu)$ that arise in practical KRR problems usually have this property.

In an ideal setting, the user would have an estimate for the number of eigenvalues of $\mtx{G}$ that are larger than $\mu$.
%
%
%
%
This is not a strong requirement when $\mtx{G}$ is accessible by repeated matrix-vector multiplication, in which case the accuracy of the estimate is unimportant.
The standard \RandNLA{} algorithm in this situation can easily be modified to recycle the work in solving \eqref{eq:unconstr_quadratic_opt} for one value of $\mu$ towards solving \eqref{eq:unconstr_quadratic_opt} for another value of $\mu$.

\subsection{Solving least squares and basic saddle point problems}
\label{subsubsec:optim:saddle_prob_class}

We are interested in certain generalizations of overdetermined and underdetermined least squares problems.
The generalizations facilitate natural specification of linear terms in composite quadratic objectives, which is a common primitive in many second-order optimization algorithms.

We frame these problems as complementary formulations of a common \textit{saddle point problem}.
The defining data for such a problem consists of a tall $m \times n$ matrix $\mtx{A}$, an $m$-vector $\vct{b}$, an $n$-vector $\vct{c}$, and a scalar $\mu \geq 0$.
For simplicity, our descriptions in this paragraph assume $\mtx{A}$ is full-rank.
The \textit{primal} saddle point problem is
\begin{equation}\label{eq:saddle_opt_x}
            \min_{\vct{x}\in\R^n}\|\mtx{A}\vct{x} - \vct{b}\|_2^2 + \mu\|\vct{x}\|_2^2 + 2\vct{c}^{\trans}\vct{x}.
\end{equation}
When $\mu$ is positive, the \textit{dual} saddle point problem is
\begin{equation}\label{eq:underdet_ridge}
            \min_{\vct{y} \in \R^m} \|\mtx{A}^{\trans}\vct{y} - \vct{c}\|_2^2 + \mu\|\vct{y} - \vct{b}\|_2^2.
\end{equation}
In the limit as $\mu$ tends to zero, \cref{eq:underdet_ridge} canonically becomes
\begin{equation}\label{eq:saddle_opt_y}
    \min_{\vct{y} \in \R^m}\{ \|\vct{y} - \vct{b} \|_2^2 \,:\, \mtx{A}^{\trans}\vct{y} = \vct{c} \}. 
\end{equation}
Note that the primal problem reduces to ridge regression when $\vct{c}$ is zero, and it reduces to overdetermined least squares when both $\vct{c}$ and $\mu$ are zero.
When $\vct{b}$ is zero, and depending on the value of $\mu$, the dual problem amounts to ridge regression with a wide data matrix or to basic underdetermined least squares.

\subsubsection{Pros and cons of this viewpoint}

Adopting this more general optimization-based viewpoint on least squares problems has two major benefits.
\begin{itemize}
    \item 
    It extends least squares problems to include linear terms in the objective. 
    The linear term in the primal problem is obvious.
    The linear terms in \eqref{eq:underdet_ridge} and \eqref{eq:saddle_opt_y} are obtained by expanding $ \|\vct{y} - \vct{b} \|_2^2 = \|\vct{y}\|_2^2 - 2 \vct{b}^{\trans}\vct{y} + \|\vct{b}\|_2^2$ and ignoring the constant term $\|\vct{b}\|_2^2$.
    \item 
    It renders the primal and dual problems equivalent for most algorithmic purposes.
    The equivalence is based on formulating the optimality conditions for these problems in a so-called \textit{saddle point system} over the variables $(\vct{x},\vct{y})$.
    \cref{subsubsec:tech_background_saddle} details this equivalence.
\end{itemize}
It must be noted that the saddle point problems we consider can be ill-posed when $\mu$ is zero and $\mtx{A}$ is rank-deficient.
Specifically, when $\mu = 0$ and $\vct{c}$ is not orthogonal to the kernel of $\mtx{A}$, the primal problem \eqref{eq:saddle_opt_x} has no optimal solution and the dual problem \eqref{eq:saddle_opt_y} has no feasible solution.
In this setting, we assign \textit{canonical solutions} by considering the limit as $\mu$ tends to zero.
\cref{app:lstsq_details:inconsistent} addresses the existence and form of these limiting solutions.
The outcome of the limiting analysis is that when $\mu = 0$, we obtain canonical solutions
\begin{equation}\label{eq:saddle_limiting_solutions}
\vct{x} = (\mtx{A}^{\trans}\mtx{A})^\dagger(\mtx{A}^{\trans}\vct{b} - \vct{c})\quad\text{ and }\quad\vct{y} =  (\mtx{A}^{\trans})^{\dagger}\vct{c} + (\mtx{I} - \mtx{A}\mtx{A}^\dagger)\vct{b},
\end{equation}
which are related through the identity $\vct{y} = \vct{b} - \mtx{A}\vct{x}$.

\subsubsection{Amenable problem structures}

This \nameCref{sec3:LS_and_optim} focuses on methods for solving these optimization problems to high accuracy.
Indeed, later in this \nameCref{sec3:LS_and_optim} we make the novel observation that methods for solving problems \eqref{eq:saddle_opt_x}--\eqref{eq:saddle_opt_y} to high accuracy can be used as the core subroutine in solving \eqref{eq:unconstr_quadratic_opt} to low accuracy at extremely large scales.
If $m \gg n$, then these methods are efficient regardless of numerical aspects of the problem data $(\mtx{A},\vct{b},\vct{c},\mu)$; problems such as poor numerical conditioning will be relevant insofar as they contribute to floating-point rounding errors in these efficient algorithms.
If $m$ is only slightly larger than $n$, then the methods we describe will only be effective when $\mtx{G} := \mtx{A}^{\trans}\mtx{A}$ and $\mu$ have the properties alluded to in \cref{subsubsec:optim:min_quadratic_prob_class}.
These properties are detailed later in this section.

\section{Drivers}
\label{subsec:optim_drivers}

Here we present four families of drivers for the problems described in \cref{subsec:optim_problem_classes}.
Two of the driver families belong to a paradigm in the RandNLA literature known as \textit{sketch-and-precondition}.
Algorithms in these families are capable of computing accurate approximations of a problem's true solution.
The other two driver families belong to a paradigm known as \textit{sketch-and-solve}.
They are less expensive than sketch-and-precondition methods (to varying degrees) but they are only suitable for producing rough approximations of a problem's true solution.
The sketch-and-solve drivers described in \cref{subsubsec:krr_sketch_and_solve} are novel in that they rely on separate sketch-and-precondition methods for their core subroutines.

\subsection{Sketch-and-solve for overdetermined least squares }\label{subsubsec:sketch_and_solve}

Sketch-and-solve is a broad paradigm within \RandNLA{}, and algorithms based on it have been central to early developments in the area~\cites{Mah-mat-rev_BOOK, Woodruff:2014, DM16_CACM, DM21_NoticesAMS}.
Its most notable manifestations have been for overdetermined least squares \cites{DMM06,Sarlos:2006, DMMS07_FastL2_NM10, CW:2013:CWTransform}, overdetermined $\ell_1$ and general $\ell_p$ regression \cites{DDHKM09_lp_SICOMP,YMM16_PIEEE}, and ridge regression \cites{avron2017sharper, wang2017sketched}.

We focus here on least squares for concreteness. 
In this case, one samples a sketching operator $\mtx{S}$, and returns
\begin{equation}\label{eqn:l2-sketch-and-solve}
(\mtx{SA})^{\dagger}(\mtx{S}\vct{b}) \in \argmin_{\vct{x}}\|\mtx{S}(\mtx{A}\vct{x} - \vct{b})\|_2^2
\end{equation}
as a proxy for the solution to 
$
\min_{\vct{x}}\|\mtx{A}\vct{x} - \vct{b}\|_2^2
$.
The quality of this solution can be bounded with the concept of subspace embeddings from \cref{subsec:sketch_quality}.
In particular, if $\mtx{S}$ is a subspace embedding for $V = \range([\mtx{A},\vct{b}])$ with distortion $\delta$, then
\begin{equation}\label{eq:sketch_and_solve_bound}
    \|\mtx{A}(\mtx{S}\mtx{A})^{\dagger}(\mtx{S}\vct{b}) - \vct{b}\|_2 \leq \left(\frac{1+\delta}{1-\delta}\right) \|\mtx{A}\mtx{A}^{\dagger}\vct{b} - \vct{b}\|_2.
\end{equation}
Note that \eqref{eqn:l2-sketch-and-solve} is invariant under scaling of $\mtx{S}$.
This implies that \eqref{eq:sketch_and_solve_bound} also holds when $\delta$ is the effective distortion of $\mtx{S}$ for $V$; see \eqref{eq:sec2:eff_dist} and \cref{subapp:effective_distortion}.

Implementation considerations and a viable application are given below.

\subsubsection{Methods for the sketched subproblem}

Direct methods for \eqref{eqn:l2-sketch-and-solve} require computing an orthogonal decomposition of $\mtx{S}\mtx{A}$, such as a QR decomposition or an SVD, in $O(dn^2)$ time.
In this context, sketch-and-solve can be used as a preprocessing step for sketch-and-precondition methods at essentially no added cost.
Indeed, this preprocessing step was used in \cite{RT:2008:SAP}.
Therefore if a direct method is being considered for sketch-and-solve, then sketch-and-precondition methods should also be viable when $m \in O(dn)$.

One can in principle apply an iterative solver to the problem defined by $(\mtx{S}\mtx{A}, \mtx{S}\vct{b})$. 
This strategy avoids the cost of factoring $\mtx{S}\mtx{A}$, and it reduces the per iteration cost relative to running the iterative solver on the original problem.
This is typically implemented without preconditioning (but see~\cite{YCRM18_JRNL}), in which case it leaves the dependence on the condition number of the original problem, and so it can only be recommended for problems where the condition number is known to be~small.

\subsubsection{Error estimation}

Since sketch-and-solve algorithms for overdetermined least squares are most suitable for computing rough approximations to a problem's true solution, it is important to have methods for estimating this error.
Such estimates can either be used to inform downstream processing of the approximate solution or to determine if a more accurate solution (computed by a more expensive algorithm) might be needed.
It is especially important that these methods work well in regimes where sketch-and-solve has a compelling computational profile, such as when $m \gg dn$.
\cref{subapp:bootstrap:least_squares} provides one such estimator based on the principle of \textit{bootstrapping} from statistics.

\subsubsection{Application to tensor decomposition}

The benefits of sketch-and-solve for least squares manifest most prominently when the following conditions are satisfied simultaneously: 
(1) $m$ is extremely large, so $\mtx{A}$ is not stored explicitly, and
(2) $\mtx{A}$ supports relatively cheap access to individual rows $\mtx{A}[i,:]$.
Among other places, this situation arises in alternating least squares approaches to tensor decomposition.
We touch upon that topic in \cref{subsubsec:fancy_sketch:cp_decomp_ols}, particularly in the remarks after \eqref{eq:normal-equation-update}.

\subsection{Sketch-and-precondition for least squares and saddle point problems}\label{subsubsec:sketch_and_precond}

The \textit{sketch-and-precondition} approach to overdetermined least squares was introduced by Rokhlin and Tygert \cite{RT:2008:SAP}.
When the $m \times n$ matrix $\mtx{A}$ is very tall, the method is capable of producing accurate solutions with less expense than direct methods. 
It starts by computing a $d \times n$ sketch $\mtx{A}^{\mathrm{sk}} = \mtx{S}\mtx{A}$ in the embedding regime (i.e., $d \gtrsim n$).
The sketch is decomposed by QR with column pivoting $\mtx{A}^{\mathrm{sk}}\mtx{\Pi} = \mtx{Q}\mtx{R}$, which defines a preconditioner $\mtx{M} = \mtx{\Pi}\mtx{R}^{-1}$.
If the parameters for the sketching operator distribution were chosen appropriately, then $\mtx{A}\mtx{M}$ will be nearly-orthogonal with high probability.\footnote{
    The condition number of $\mtx{A}\mtx{M}$ and the effective distortion of $\mtx{S}$ for $\range(\mtx{A})$ completely characterize one another; see \cref{subapp:effective_distortion,subsubapp:we_sketch_subspaces}.
}
The near-orthogonality of $\mtx{A}\mtx{M}$ ensures rapid convergence of an iterative method for the least squares problem's preconditioned normal equations.
If $T_{\mathrm{sk}}$ denotes the time complexity of computing $\mtx{S}\mtx{A}$, then the typical asymptotic FLOP count to solve to $\epsilon$-error is
\begin{equation}\label{eq:sap_runtime}
    O(T_{\mathrm{sk}} + dn^2 + mn \log(1/\epsilon))
\end{equation}
Importantly, this complexity has no dependence on the condition number of $\mtx{A}$.
\label{page:runtime_of_sap}

This approach was extended with stronger theoretical guarantees, support for more general least squares problems, and high-performance implementations   through \textit{Blendenpik} \cite{AMT:2010:Blendenpik} and \textit{LSRN} \cite{MSM:2014:LSRN}.
It has also recently been used to solve positive definite systems arising in linear programming algorithms \cite{CLAD:2020:SAP_linprog}.
All of these methods produce preconditioners $\mtx{M}$ where $\mtx{A}\mtx{M}$ is nearly-orthogonal and are intended for the regime where $\mtx{A}$ is very tall.

What constitutes ``very tall'' depends on the algorithm's implementation and the hardware that runs it.
It is easy to implement these algorithms in Matlab or Python so that, on a personal laptop, they are competitive with \LAPACK{}'s direct methods when $m \geq 50 n \geq 10^5$; see also \cref{subsec:opt_libraries}.


\subsubsection{Your attention, please!}

If a saddle point problem features regularization (i.e., if $\mu > 0$) and if $\mtx{A}$ has rapid spectral decay, then randomized methods can be used to find a good preconditioner in far less than $O(n^3)$ time, no matter the specific value of $m \geq n$. 
This is possible by borrowing ideas from \textit{\Nystrom{} preconditioning} \cite{FTU:2021:NystromRidge}, which we introduce in \cref{subsubsec:nys_pcg} for the related problem of minimizing regularized quadratics.
As a novel contribution, \cref{subsec:saddle_nys_precond} explains how \Nystrom{} preconditioning can naturally be adapted to saddle point problems.
Therefore while the material here (in \cref{subsubsec:sketch_and_precond}) focuses on the case $m \gg n$, one should be aware that this requirement can be relaxed.

\subsubsection{Algorithms}

Sketch-and-precondition algorithms can take different approaches to sketching, preconditioner generation, and choice of the eventual iterative solver.
\begin{itemize}
    \item 
    Blendenpik used SRFT sketching operators, obtained its preconditioner by unpivoted QR of $\mtx{A}^{\mathrm{sk}}$, and used LSQR \cite{PS:1982} as its underlying iterative method.
    
    \item
    LSRN used Gaussian sketching operators, obtained its preconditioner through an SVD of $\mtx{A}^{\mathrm{sk}}$, and defaulted to the Chebyshev semi-iterative method \cite{GV:1961} for its iterative solver.
\end{itemize}
These two examples hint at the huge range of possibilities for the implementation of sketch-and-precondition algorithms.
Indeed, we discuss preconditioners in detail over \cref{subsec:precond_gen,subsec:saddle_nys_precond}, and we review a suite of possible deterministic iterative methods in \cref{subsubsec:det_saddle_solve}.
For now, we give \cref{alg:ols_orth_sap,alg:saddle_to_ols_sap} (below) as footholds for understanding the various design considerations.

For simplicity's sake, both of these algorithms use a black-box function 
\[
\vct{z} = \code{iterative\_ls\_solver}(\mtx{F},\vct{g},\epsilon,L,\vct{z}_o)
\]
which computes an approximate solution to $\min_{\vct{z}}\|\mtx{F}\vct{z} - \vct{g}\|_2^2$.
The exact semantics of this function are unimportant for our present purpose.
Its general semantics are that the solver initializes an iterative procedure at $\vct{z}_o$ and that it runs until either an implementation-dependent error tolerance $\epsilon$ is met or an iteration limit $L$ is reached. 
Typical implementations would measure error with a suitably normalized version of the normal equation residual $\|\mtx{F}^{\trans}\left(\mtx{F}\vct{z} - \vct{g}\right)\|_2$.
If $\kappa$ denotes the condition number of $\mtx{F}$ then typical convergence rates are such that error $\|\mtx{F}(\vct{z} - \mtx{F}^{\dagger}\vct{g})\|_2$ decays multiplicatively by a factor of $(\kappa - 1)/(\kappa + 1)$ with each iteration.

\clearpage

Besides the use of a common iterative solver, both algorithms below initialize the iterative solver at the solution from a sketch-and-solve approach in the vein of \cref{subsubsec:sketch_and_solve}.
The time needed to perform this presolve step is negligible, but it should save several iterations when solving to a prescribed accuracy.
It also plays an important role in handling overdetermined least squares problems when $\vct{b}$ is in the range of $\mtx{A}$.
In such contexts, the sketch-and-solve result actually solves the least squares problem \textit{exactly} provided that $\rank(\mtx{S}\mtx{A}) = \rank(\mtx{A})$; this stands in contrast to using a preconditioned iterative method initialized at the origin, which would not be able to achieve relative error guarantees for $\|\mtx{A}\vct{x} - \vct{b}\|$ against $\|(\mtx{I} - \mtx{A}\mtx{A}^{\dagger})\vct{b}\| = 0$.

\begin{algorithm}[htb]
    \setstretch{1.0}
    \caption{\code{SPO1}: a Blendenpik-like approach to overdetermined least squares }\label{alg:ols_orth_sap}
    \begin{algorithmic}[1]
        \State \textbf{function} $\code{SPO1}(\mtx{A},\vct{b},\epsilon, L)$ \vspace{0.5pt} 
        \Indent
            \Statex \quad Inputs:
            \Statex \begin{quote}
                $\mtx{A}$ is $m \times n$ and $\vct{b}$ is an $m$-vector. We require $m \geq n$ and expect $m \gg n$.
                 The iterative solver's termination criteria are governed by $\epsilon$ and $L$: it stops if the solution reaches error $\epsilon \geq 0$ according to the solver's error metric, or if the solver completes $L \geq 1$ iterations.
            \end{quote}
            \Statex \quad Output:
            \Statex \qquad An approximate solution to \eqref{eq:saddle_opt_x}, with $\vct{c} = \vct{0}$ and $\mu = 0$.\vspace{2pt}
            \Statex  \quad Abstract subroutines and tuning parameters:
            \Statex \qquad $\code{SketchOpGen}$ generates an oblivious sketching operator.
            \Statex \qquad $\code{sampling{\_}factor} \geq 1$ is the size of the embedding dimension relative to~$n$.\vspace{4pt} 
            \setstretch{1.1}
            \State $d = \min\{\lceil n \cdot \code{sampling{\_}factor}\rceil,\, m \}$ 
            \State $\mtx{S} = \code{SketchOpGen}(d, m)$
            \State $[\mtx{A}^{\mathrm{sk}},\, \vct{b}_{\mathrm{sk}}] = \mtx{S}[\mtx{A},\, \vct{b}]$
            \State $\mtx{Q}, \mtx{R} = \code{qr\_econ}(\mtx{A}^{\mathrm{sk}})$
            \State $\vct{z}_o = \mtx{Q}^{\trans}\vct{b}_{\mathrm{sk}}$
            \quad~~~\, \codecomment{$\mtx{R}^{-1}\vct{z}_o$ solves $\min_{\vct{x}}\{\|\mtx{S}(\mtx{A}\vct{x} - \vct{b})\|_2^2\}$}
            \State $\mtx{A}_{\text{precond}} = \mtx{A}\mtx{R}^{-1}$ \codecomment{as a linear operator}
            \State $\vct{z} = \code{iterative\_ls\_solver}(\mtx{A}_{\text{precond}}, \vct{b}, \epsilon, L, \vct{{z}}_o)$ 
            \State \textbf{return} $\mtx{R}^{-1}\vct{z}$
        \EndIndent
    \end{algorithmic}
\end{algorithm}


While \cref{alg:ols_orth_sap} is standard, \cref{alg:saddle_to_ols_sap} is somewhat novel.
Using the same data that might be computed during a standard sketch-and-precondition algorithm for simple overdetermined least squares, it transforms any saddle point problem --- primal or dual --- into an equivalent primal saddle point problem with $\vct{c} = \vct{0}$.
To our knowledge, no such conversion routines have been described in the literature.
The conversion is advantageous because it opens the possibility of using iterative solvers with excellent numerical properties that are specific to least squares problems.
The validity of the algorithm's transformation is explained towards the end of \cref{subsubsec:tech_background_saddle}.

\begin{algorithm}[!htb]
    \setstretch{1.0}
    \caption{\code{SPS2} : sketch, transform a saddle point problem to least squares, and precondition. A more efficient version of this algorithm can be obtained using our observations on SVD-based preconditioning in \cref{subsec:precond_gen}.}\label{alg:saddle_to_ols_sap}
    \begin{algorithmic}[1]
        \State \textbf{function} $\code{SPS2}(\mtx{A},\vct{b},\vct{c},\mu,\epsilon, L)$ \vspace{0.5pt} 
        \Indent
            \Statex \quad Inputs:
            \Statex \begin{quote}
                $\mtx{A}$ is $m \times n$, $\vct{b}$ is an $m$-vector, $\vct{c}$ is an $n$-vector, and $\mu$ is a nonnegative
                regularization parameter. We require $m \geq n$ and expect $m \gg n$.
                 The iterative solver's termination criteria are governed by $\epsilon$ and $L$: it stops if the solution reaches error $\epsilon \geq 0$ according to its internal error metric, or if it completes $L \geq 1$ iterations.
            \end{quote}\vspace{4pt}
            \Statex \quad Output:
            \Statex \qquad Approximate solutions to \eqref{eq:saddle_opt_x} \textit{and} its dual problem.\vspace{4pt}
            \Statex  \quad Abstract subroutines and tuning parameters:
            \Statex \qquad $\code{SketchOpGen}$ generates an oblivious sketching operator.
            \Statex \qquad $\code{sampling{\_}factor} \geq 1$ is the size of the embedding dimension relative to~$n$.\vspace{4pt}
            \setstretch{1.1}
            \State $d = \min\{\lceil n \cdot \code{sampling{\_}factor}\rceil,\, m \}$ 
            \State $\mtx{S} = \code{SketchOpGen}(d, m)$
            \If{$\mu > 0$} 
                \State  \[
                \mtx{S} = \begin{bmatrix}\mtx{S} & \mtx{0} \\ \mtx{0} & \mtx{I}_{n}\end{bmatrix},
                \qquad \mtx{A} = \begin{bmatrix} \mtx{A} \\ \sqrt{\mu}\mtx{I}_{n} \end{bmatrix},
                \qquad \vct{b} = \begin{bmatrix} \vct{b} \\ \vct{0} \end{bmatrix}
                \]
            \EndIf
            \State $\mtx{A}^{\mathrm{sk}} = \mtx{S}\mtx{A}$
            \State $\mtx{U},\mtx{\Sigma},\mtx{V}^{\trans} = \code{svd}(\mtx{A}^{\mathrm{sk}})$
            \State $\mtx{M} = \mtx{V}\mtx{\Sigma}^{\dagger}$
            \State $\vct{b}_{\text{mod}} = \vct{b}$ 
            \If{ $\vct{c} \neq \vct{0}$} 
                \State $\vct{\hat{v}} = \mtx{U}\mtx{\Sigma}^{\dagger}\mtx{V}^{\trans}\vct{c}$
                ~\codecomment{$\vct{\hat{v}}$ solves $\min_{\vct{v}}\{\|\vct{v}\|_2^2 \,:\, \mtx{A}^{\trans}\mtx{S}^{\trans}\vct{v} = \vct{c}\}$}
                \State $\vct{b}_{\text{shift}} = \mtx{S}^{\trans}\vct{\hat{v}}$
                ~\,\,\codecomment{$\mtx{A}^{\trans}\vct{b}_{\text{shift}} = \vct{c}$}
                \State $\vct{b}_{\text{mod}} = \vct{b}_{\text{mod}} - \vct{b}_{\text{shift}}$
            \EndIf
            \State $\vct{z}_o = \mtx{U}^{\trans} \mtx{S}\vct{b}_{\text{mod}}$
                ~~~\,\codecomment{ $\mtx{M}\vct{z}_o$ solves $\min\{\|\mtx{S}\left(\mtx{A}\vct{x} - \vct{b}_{\text{mod}}\right)\|_2^2\}$ }
            \State $\mtx{A}_{\text{precond}} = \mtx{A}\mtx{M}$
                ~~\,\codecomment{ define implicitly, as a linear operator}
            \State $\vct{z} =  \code{iterative\_ls\_solver}(\mtx{A}_{\text{precond}}, \vct{b}_{\text{mod}}, \epsilon, L, \vct{z}_o)$
            \State $\vct{x} = \mtx{M}\vct{z}$
            \State $\vct{y} = \vct{b}[:m] - \mtx{A}[:m, :]\vct{x}$
            \State \textbf{return} $\vct{x},~ \vct{y}$
        \EndIndent
    \end{algorithmic}
\end{algorithm}

\clearpage

We wrap up our introduction to sketch-and-precondition algorithms by speaking to their tradeoffs with sketch-and-solve.
It is easy to see that if $m \ll n^2$ and we perform sketch-and-solve using a direct method for \cref{eqn:l2-sketch-and-solve}, then performing an additional constant number of steps of sketch-and-precondition's iterative phase does not increase the FLOP count by even so much as a constant factor. 
However, if we are in the regime where $m \geq n^2$, then even a single step of an iterative method in sketch-and-precondition can cost as much as an entire sketch-and-solve algorithm.
Therefore when an accurate solution is not required and $m \geq n^2$, it may be preferable to use sketch-and-solve rather than sketch-and-precondition.

\FloatBarrier

\subsubsection{Applications}


One application of these algorithms is to carry out the core subroutine in iterative methods for solving linear systems by block projection.
We explain the nature of this connection later on, in \cref{subsec:linsys_iterative}.

To explain the next application, we need some context.
Classical linear algebra techniques to solve a KRR problem with $m$ datapoints require $O(m^2)$ storage and $O(m^3)$ time.
Rahimi and Recht's \textit{random feature maps} provide a framework for replacing such a KRR problem with a more tractable ridge regression problem \cite{RR:2007:KRR}.
A data matrix in a \textit{random features ridge regression problem} is $m \times n$ (for a tuning parameter $n < m$) and is characterized by the KRR datapoints and functions $f_1,\ldots,f_n$ drawn from a suitable random distribution.
The $i^{\text{th}}$ row in this matrix is obtained by evaluating $f_1,\ldots,f_n$ on the $i^{\text{th}}$ KRR datapoint.

The randomness in random features ridge regression is not ``sketching'' in the sense meant by this monograph.
Still, this approach is notable in our context because it provides a source of models that are amenable to the methodology described above.
The \Nystrom{} preconditioning methodology (see \cref{subsubsec:nys_pcg,subsec:saddle_nys_precond}) has been reported to be especially effective for such problems when $n \lesssim m$ \cite{FTU:2021:NystromRidge}.

\subsubsection{When is sketch-and-precondition asymptotically faster than QR?}

Here we detail the runtime of sketch-and-precondition algorithms under the assumption of sketching with SRFTs.
These sketching operators were used in the original sketch-and-precondition paper \cite{RT:2008:SAP} and subsequently by \cite{AMT:2010:Blendenpik}.
We focus on them here because they have no tuning parameters besides their embedding dimension.
Minimizing the number of tuning parameters helps us make comparisons to direct solvers based on QR decomposition that run in time $O(mn^2)$.

Recall from \cref{subsec:srfts} that it takes $O(mn \log d)$ time to apply a $d \times m$ SRFT to an $m \times n$ matrix.
We can plug $T_{\text{sk}} = mn \log d$ into \eqref{eq:sap_runtime} to see that the ``typical'' runtime for sketch-and-precondition with an SRFT is
\begin{equation}\label{eq:sap_runtime_srft}
    O(mn \log d + dn^2 + mn \log(1/\epsilon)).
\end{equation}
This runtime is only ``typical'' because it does not address subtleties stemming from randomness.
In the algorithm's true runtime, there is a random multiplicative factor $F$ on the $mn \log(1/\epsilon)$ term in \eqref{eq:sap_runtime_srft}.
The distribution of $F$ depends on $(d, n)$ in a complicated way.
In formal algorithm analysis, one describes how to choose $d$ to upper-bound the probability that $F$ exceeds some universal constant $C$.
Then one can say that \eqref{eq:sap_runtime_srft} \textit{does} describe the algorithm's true runtime with some probability.
The convention in the field is to describe how to choose $d$ so the probability that $F \leq C$ tends to one as problem size increases.

\cite{RT:2008:SAP} observed that taking $d = s n$ for small constants $s$ (e.g., $s = 4$) sufficed for \eqref{eq:sap_runtime_srft} to accurately describe algorithm runtime in practice.
However, the theoretical analysis in \cite{RT:2008:SAP} needed to take $d \in \Omega(n^2)$
to bound $F$ with high probability.
Therefore the best theoretical runtime guarantee for sketch-and-precondition was originally obtained by plugging $d = n^2$ into \eqref{eq:sap_runtime_srft}.
The theoretical guarantees improved following developments in the analysis of SRFTs.
Specifically, \cite{AMT:2010:Blendenpik} observed that a transparent application of a result by \cite{NDT:2009:fast-low-rank-approx} could be used to prove that $d \in \Omega(n \log n)$ sufficed to bound $F$ with high probability.
Therefore one can plug $d = n \log n$ into \eqref{eq:sap_runtime_srft} to obtain a bound for algorithm runtime in terms of $(m,n,\epsilon)$ that holds with high probability.
This is the appropriate bound to use when comparing the theoretical asymptotic runtime of SRFT-based sketch-and-precondition to other algorithms.
However, in practice, it is still preferred to use $d = s n$ for some small $s > 1$, since the resulting preconditioned matrices tend to be extremely well-conditioned. 

\FloatBarrier

\subsection{\Nystrom{} PCG for minimizing regularized quadratics}
\label{subsubsec:nys_pcg}


\Nystrom{} preconditioned conjugate gradient (\Nystrom{} PCG) is a recently-proposed method for solving problems of the form \eqref{eq:unconstr_quadratic_opt} to fairly high accuracy \cite{FTU:2021:NystromRidge}. 
We describe it as a method to compute approximate solutions to linear systems $(\mtx{G} + \mu\mtx{I})\vct{x} = \vct{h}$ where $\mtx{G}$ is $n \times n$ and psd.

The randomness in \Nystrom{} PCG is encapsulated in an initial phase where it computes a low-rank approximation of $\mtx{G}$ by a so-called ``\Nystrom{} approximation.''
We defer discussion on such approximations (including the potentially-confusing naming convention) to \cref{subsubsec:herm_eig_algs}.
For our purposes, what matters is that a rank-$\ell$ \Nystrom{} approximation leads to a preconditioner $\mtx{P}$
which can be stored in $O(\ell n)$ space and applied in $O(\ell n)$ time.

Now let $\kappa$ denote the condition number of $\mtx{G}_p := \mtx{P}^{-1/2}(\mtx{G} + \mu\mtx{I})\mtx{P}^{-1/2}$.
It is well-known that each iteration of PCG requires one matrix-vector multiply with $\mtx{G}$, one matrix-vector multiply with $\mtx{P}^{-1}$, and reduces the error of the candidate solution to \eqref{eq:unconstr_quadratic_opt} by a multiplicative factor $(\sqrt{\kappa}-1)/(\sqrt{\kappa} + 1)$.
As we discuss below, one can expect that $\kappa$ will be $O(1)$ if the $\ell^{\text{th}}$-largest eigenvalue of $\mtx{G}$ is smaller than $\mu$.
Indeed, \Nystrom{} PCG is most effective for problems when this threshold is crossed at some $\ell \ll n$.
As a practical matter, users will not need to select the approximation rank parameter $\ell$ manually in order to use \Nystrom{} PCG; \cite[Algorithm E.2]{FTU:2021:NystromRidge} is a specialized adaptive method for \Nystrom{} approximation that can determine an appropriate value for $\ell$ given $(\mtx{G},\mu)$.

\subsubsection{Details on the preconditioner}

We presume access to a low-rank approximation
\begin{equation}
    \mtx{\hat{G}} = \mtx{V}\diag(\vct{\lambda})\mtx{V}^{\trans}
\end{equation}
where $\mtx{V}$ is a column-orthonormal $n \times \ell$ matrix that approximates the dominant $\ell$ eigenvectors of $\mtx{G}$ and $\lambda_1 \geq \cdots \geq \lambda_{\ell} > 0$ are the approximated eigenvalues.
The data $(\mtx{V},\vct{\lambda},\mu)$ is then used to define a preconditioner
\begin{equation}\label{eq:nys_precond}
\mtx{P}^{-1} = \mtx{V}\diag(\vct{\lambda} + \mu)^{-1}\mtx{V}^{\trans} + (\mu + \lambda_{\ell})^{-1}(\mtx{I}_{n} - \mtx{V}\mtx{V}^{\trans}).
\end{equation}
Alternatively, following \cite{FTU:2021:NystromRidge} to the letter, $\mtx{P}^{-1}$ can be the result of multiplying the expression above by $(\mu + \lambda_{\ell})$.
Under this latter convention, $\mtx{P}^{-1}$ acts as the identity on $\range(\mtx{V})^{\perp}$.

While the form of this preconditioner may appear mysterious, its appropriateness can be seen by considering a simple idealized setting.
To make a precise statement on this topic we adopt notation where $\lambda_i(\mtx{G})$ is the $i^{\text{th}}$-largest eigenvalue of $\mtx{G}$.
Assuming that $(\mtx{V},\vct{\lambda})$ are very good estimates for the top $\ell$ eigenpairs of $\mtx{G}$ \textit{and} that $\lambda_{\ell}(\mtx{G}) \approx \lambda_{\ell+1}(\mtx{G})$, the condition number of $\mtx{G}_p$ should be near
\[
\kappa_{\ell}(\mtx{G},\mu) := (\lambda_{\ell}(\mtx{G}) + \mu) / (\lambda_n(\mtx{G}) + \mu).
\]
Taking this for granted, the preconditioner \eqref{eq:nys_precond} can only be effective if $\ell \ll n$ is large enough so that
$\kappa_{\ell}(\mtx{G},\mu)$ is bounded by a small constant.
Using the fact that $\kappa_{\ell}(\mtx{G},u) \leq 1 + \lambda_{\ell}(\mtx{G})/\mu$, we can simplify the criteria and say that a good preconditioner is possible when $\lambda_{\ell}(\mtx{G}) / \mu$ is $O(1)$.

\begin{remark}
    The argument above can be made more rigorous by assuming that $\mtx{V}$ is an $n \times (\ell - 1)$ matrix that contains the \textit{exact} leading $\ell - 1$ eigenvectors of $\mtx{G}$, and that $\lambda_1,\ldots,\lambda_{\ell}$ are the \textit{exact} leading $\ell$ eigenvalues of $\mtx{G}$.
    In this case, the condition number of $\mtx{G}_p$ will be equal to $\kappa_{\ell}(\mtx{G},\mu)$, which will be at most $1 + \lambda_{\ell} / \mu$.
\end{remark}

\subsection{Sketch-and-solve for minimizing regularized quadratics}\label{subsubsec:krr_sketch_and_solve}

Randomization offers several avenues for solving problems of the form \eqref{eq:unconstr_quadratic_opt} to modest accuracy.
We describe two possible methods here through novel interpretations of existing work on KRR.
Our descriptions of the methods keep the focus on linear algebra, and we refer the reader to \cref{subapp:primer_krr} for information on the KRR formalism.
We note that our formulations of these methods are novel in how they apply sketch-and-precondition as the core subroutine in what is otherwise a sketch-and-solve style driver.
Such ``nested randomization'' is a relatively under-explored and potentially powerful algorithm design paradigm.

For notation, we shall say that $\mtx{G}$ is $m \times m$, that $\mu = m \lambda$ for some $\lambda > 0$, and that the optimization variable in \eqref{eq:unconstr_quadratic_opt} is denoted by ``$\vct{\alpha}$'' rather than ``$\vct{x}$.''

\subsubsection{A one-shot fallback on \Nystrom{} approximations}
Rather than solving \eqref{eq:unconstr_quadratic_opt} directly, it has been suggested that one solve 
\[
\left(\mtx{A}\mtx{A}^{\trans} + m \lambda\mtx{I}\right)\vct{\hat{\alpha}} = \vct{h},
\] 
where $\mtx{A}\mtx{A}^{\trans}$ is a \Nystrom{} approximation of $\mtx{G}$ \cite{AM:2015:KRR}.
The computation of $\mtx{A}$ only requires access to $\mtx{G}$ by a single sketch $\mtx{G}\mtx{S}$ for a tall $m \times n$ sketching operator $\mtx{S}$.
In the KRR context, it is especially popular for $\mtx{S}$ to be a column sampling operator \cites{WS:2000:KRR, KMT09c, GM15_NYSTROM_JRNL}.
\cref{subsec:ridge_leverage_scores} discusses how such column-selection sketches $\mtx{G}\mtx{S}$ can be computed adaptively using the concept of \textit{ridge leverage scores}.
Regardless of how the approximation is obtained, there is an equivalence between computing $\vct{\hat{\alpha}}$ and solving a dual saddle point problem with matrix $\mtx{A}$ and other data $(\vct{b},\vct{c},\mu) = (\vct{h},\vct{0},m\lambda)$.
That dual saddle point problem can naturally be approached by sketch-and-precondition methods from \cref{subsubsec:sketch_and_precond}.
The preconditioner generation steps in this context are subtle and addressed in 
\cref{subapp:KRR_AM15_SASAP}.

\subsubsection{Applying a random subspace constraint}
By taking the gradient of the objective function in \eqref{eq:unconstr_quadratic_opt} and multiplying the gradient by the positive definite matrix $\mtx{G}$, we can recast \eqref{eq:unconstr_quadratic_opt} as minimizing
\[
Q(\vct{\alpha}) = \vct{\alpha^{\trans}}(\mtx{G}^2 + m \lambda \mtx{G})\vct{\alpha} - 2\vct{h}^{\trans}\mtx{G}\vct{\alpha}.
\]
In \cite{yang2017randomized}, a sketch-and-solve approach to the problem of minimizing this loss function is proposed.
Specifically, one minimizes $Q(\vct{\alpha})$ subject to a constraint that $\vct{\alpha}$ is in the range of a very tall $m \times n$ sketching operator $\mtx{S}$.
The constrained minimization problem is equivalent to minimizing $\vct{z} \mapsto Q(\mtx{S}\vct{z})$ over $n$-vectors $\vct{z}$.
This in turn is equivalent to solving a highly overdetermined least squares problem, with an $(m + n) \times n$ data matrix $\mtx{A} = [\mtx{G}\mtx{S}; \sqrt{m \lambda}\mtx{R}]$ where $\mtx{R}$ is any matrix for which $\mtx{R}^{\trans}\mtx{R} = \mtx{S}^{\trans}\mtx{G}\mtx{S}$.
This problem can clearly be handled by our methods from \cref{subsubsec:sketch_and_precond}.

\begin{remark}
    We note that \cite{yang2017randomized} presumes access to the sketches $\vct{h}^{\trans}\mtx{G}\mtx{S}$, $\mtx{S}^{\trans}\mtx{G}\mtx{S}$, and $\mtx{S}^{\trans}\mtx{G}^2\mtx{S}$, and advocates for solving the resulting $n$-dimensional minimization problem by a direct method in $O(n^3)$ time.
    However, no guidance is given on how to compute the sketch $\mtx{S}^{\trans}\mtx{G}^2\mtx{S}$.
    From what we can tell, the most efficient way of doing this would be to form the Gram matrix at cost $O(mn^2)$ assuming access to the sketch $\mtx{G}\mtx{S}$.
    (Our usage of $(m, n)$ is swapped relative to \cite{yang2017randomized}.)
\end{remark}

\section{Computational routines}
\label{subsec:optim_comp_routines}

To contextualize the computational routines that follow, we begin in \cref{subsubsec:tech_background_saddle} with a brief discussion of optimality conditions for saddle point problems.
From there, we present in \cref{subsec:precond_gen,subsec:saddle_nys_precond} two families of methods for generating preconditioners needed by saddle point drivers; our presentation of both families includes novel observations that lead to improved efficiency and numerical stability.
Then in \cref{subsubsec:det_saddle_solve} we discuss deterministic preconditioned iterative methods for positive definite systems and saddle point problems.
Such iterative methods are applicable to all drivers from the previous \nameCref{subsec:optim_drivers} (although less so for \cref{subsubsec:sketch_and_solve}).

\subsubsection{Routines not detailed here}

The driver from \cref{subsubsec:nys_pcg} requires methods to compute \Nystrom{} approximations, which are described in \cref{sec4:lowrank}.
In addition, the drivers from \cref{subsubsec:krr_sketch_and_solve} would benefit from specialized data-aware methods for sketching kernel matrices,
which are discussed in \cref{sec7:lev_scores}.
We also note that this \nameCref{subsec:optim_comp_routines} does not describe computational routines for sketch-and-solve type drivers.
This is because those drivers are extraordinarily simple to implement and there is no need to isolate their building blocks into separate computational routines.

\subsection{Technical background: optimality conditions for saddle point problems}
\label{subsubsec:tech_background_saddle}

Here, we give a handful of characterizations of optimal solutions for saddle point problems.
Let us begin by calling an $n$-vector $\vct{x}$ \textit{primal-optimal} if it solves \eqref{eq:saddle_opt_x}.
Analogously, an $m$-vector $\vct{y}$ shall be called \textit{dual-optimal} if it solves \eqref{eq:underdet_ridge} when $\mu$ is positive or \eqref{eq:saddle_opt_y} when $\mu$ is zero.

Primal-dual optimal solutions can be characterized with \textit{saddle point systems}.
These are a class of $2 \times 2$ block linear systems that arise broadly in computational mathematics and especially in optimization.
General introductions to these systems can be found in the survey \cite{BGL:2005:saddle_sys} and the book \cite{OA:2017:quasidef_book}.
We are interested in saddle point systems of the form
    \begin{equation}\label{eq:saddle_sys}
        \begin{bmatrix}
                \mtx{I}                & ~\mtx{A} \\
                \mtx{A}^{\trans}      & -\mu\mtx{I}
        \end{bmatrix}
        \begin{bmatrix}
            \vct{y} \\ \vct{x}
        \end{bmatrix}
        =
        \begin{bmatrix}
            \vct{b} \\ \vct{c}
        \end{bmatrix}.
    \end{equation}
A solution to such a system always exists when $\mu$ is positive or when the tall matrix $\mtx{A}$ is full-rank.
Given that assumption, it can be shown that a point $\vct{\tilde{x}}$ is primal-optimal if and only if there is a $\vct{\tilde{y}}$ for which $(\vct{\tilde{x}},\vct{\tilde{y}})$ solve \eqref{eq:saddle_sys}.
Similarly, a point $\vct{\tilde{y}}$ is dual-optimal if and only if there is an $\vct{\tilde{x}}$ for which $(\vct{\tilde{x}},\vct{\tilde{y}})$ solve \eqref{eq:saddle_sys}.

Saddle point systems are often reformulated into equivalent positive semidefinite systems.
The reformulation takes the system's upper block to \textit{define} $\vct{y} = \vct{b} - \mtx{A}\vct{x}$, and then substitutes that expression into the system's lower block.
This gives us the \textit{normal equations}
\begin{equation}\label{eq:saddle_normal_eqs}
    (\mtx{A}^{\trans}\mtx{A} + \mu\mtx{I})\vct{x} = \mtx{A}^{\trans}\vct{b} - \vct{c}.
\end{equation}
Therefore one can solve \eqref{eq:saddle_sys} by first solving \eqref{eq:saddle_normal_eqs} and then setting $\vct{y} = \vct{b} - \mtx{A}\vct{x}$.
Such an approach to underdetermined least squares is suggested by Bj\"{o}rck in his books \cites{Bjorck:1996,Bjorck:2015}.

Thinking in terms of the normal equations helps with the design of preconditioners.
When accurate solutions are desired, however, it is preferable to employ reformulations that reduce the need for matrix-vector products with the linear operator $\mtx{A}^{\trans}\mtx{A}$.
Such reformulations start by defining an augmented data matrix $\mtx{A}_{\mu} = [\mtx{A}; \sqrt{\mu}\mtx{I}_n]$.
For dual saddle point problems, one solves
\begin{equation}\label{eq:dual_saddle_as_uls}
    \min\{\, \|\Delta\vct{y}\|_2^2 ~:~ \Delta\vct{y} \in \R^{m + n}, ~ ({\mtx{A}_{\mu}})^{\trans}\Delta\vct{y} = \vct{c} - \mtx{A}^{\trans}\vct{b} \},
\end{equation}
and subsequently recovers the dual-optimal solution $\vct{y} = [b_1 + \Delta y_1;\ldots;b_m + \Delta y_m]$.
For primal saddle point problems, one computes \textit{some} $\vct{b}_{\text{shift}} \in \R^{m + n}$ satisfying $(\mtx{A}_{\mu})^{\trans}\vct{b}_{\text{shift}} = \vct{c}$ and then defines $\vct{b}_{\mu} = [\vct{b}; \vct{0}_n] - \vct{b}_{\text{shift}}$. 
Any solution to the resulting problem
\begin{equation}\label{eq:primal_saddle_as_ols}
    \operatornamewithlimits{min}_{\vct{x} \in \R^n}\left\{\|\mtx{A}_{\mu}\vct{x}-\vct{b}_{\mu}\|_2^2\right\}
\end{equation}
is primal-optimal.
Of course, this reformulation is only useful if we have a cheap way to compute $\vct{b}_{\text{shift}}$.
As it happens, however, randomized methods for preconditioner generation provide methods to compute a near-minimum-norm solution to $\mtx{A}^{\trans}\vct{u} = \vct{c}$ in $O(mn)$ extra time compared to when $\vct{c} = \vct{0}$.
We illustrated this process earlier with an SVD-based preconditioner in \cref{alg:saddle_to_ols_sap}.

\subsubsection{Inconsistent saddle point systems}
    Suppose that $\mu$ is zero, so as to allow for the possibility that \eqref{eq:saddle_sys} is consistent.
    Under this assumption, \eqref{eq:saddle_sys} is inconsistent if and only if $\vct{c}$ is not in the range of $\mtx{A}^{\trans}$.
    When framed in this way, we have that \eqref{eq:saddle_sys} is inconsistent if and only if \eqref{eq:saddle_opt_y} has no feasible solution.
    What's more, since $\vct{c} \not\in\range(\mtx{A}^{\trans})$ is equivalent to $\vct{c} \not\in \ker(\mtx{A})^{\perp}$, we see that inconsistency of \eqref{eq:saddle_sys} is equivalent to \eqref{eq:saddle_opt_x} having no optimal solution.
    Therefore a saddle point system is consistent if and only if its associated saddle point problems are well-posed;
    for ill-posed problems, recall that we canonically assign solutions per~\eqref{eq:saddle_limiting_solutions}.

\subsection{Preconditioning least squares and saddle point problems: tall data matrices}
\label{subsec:precond_gen}

There is a simple unifying framework for preconditioner generation of the kind used in  \cites{RT:2008:SAP,AMT:2010:Blendenpik,MSM:2014:LSRN}.
The framework is applicable to any least squares or saddle point problem \eqref{eq:saddle_opt_x}--\eqref{eq:saddle_opt_y} in the regime $m \gg n$.
We describe its general form below and then turn to its concrete instantiations.

\subsubsection{Sketch and orthogonalize}
To describe our framework, begin by defining a sketch $\mtx{A}^{\mathrm{sk}} = \mtx{S}\mtx{A}$ where the sketching operator $\mtx{S}$ has $d \gtrsim n$ rows.
We also define the augmented matrices
\[
    \mtx{A}_{\mu} = \begin{bmatrix} \mtx{A} \\ \sqrt{\mu} \mtx{I} \end{bmatrix}
    \quad\text{and}\quad
    \mtx{A}^{\mathrm{sk}}_{\mu} = \begin{bmatrix} \mtx{A}^{\mathrm{sk}} \\ \sqrt{\mu}\mtx{I} \end{bmatrix}.
\]
These augmented matrices are only used as a formalism.
They reflect the influence of the normal equations \eqref{eq:saddle_normal_eqs} on preconditioner design. 
We emphasize that we specifically allow for $\mu = 0$ and one need not form these augmented matrices explicitly in memory.

Next, we introduce two key terms.
\begin{quote}
    We say that a matrix $\mtx{M}$ \textit{orthogonalizes} $\mtx{A}^{\mathrm{sk}}_{\mu}$
    if the columns of $\mtx{A}^{\mathrm{sk}}_{\mu}\mtx{M}$ are an orthonormal basis for the range of $\mtx{A}^{\mathrm{sk}}_{\mu}$.
    Such a matrix is called a \textit{valid preconditioner} for $\mtx{A}_{\mu}$ if, in addition, $\rank(\mtx{A}^{\mathrm{sk}}_{\mu}) = \rank(\mtx{A}_{\mu})$.
\end{quote}
We note that the rank requirement of a valid preconditioner is nearly universal in practice.
For example, it holds with probability one for uniform and Gaussian operators (\S \ref{subsec:dense_skops}).
We conjecture that it holds with exponentially-high  probability for suitable SASOs (\S \ref{subsubsec:randblas:ShortAxSparse}) and for SRFTs (\S \ref{subsec:srfts}).

\paragraph{How good are these preconditioners?}
In our context, $\mtx{M}$ is a good preconditioner if the spectrum of $\mtx{A}_{\mu}\mtx{M}$ can be divided into a small number of tightly clustered groups.
Given the tools at our disposal in \RandNLA{}, we mostly aim for the spectrum of this matrix to be tightly clustered into a single group, i.e., for its condition number to be small.
In this regard, we can provide the following principle.
\begin{quote}
    \emph{If $\mtx{M}$ is a valid preconditioner for $\mtx{A}_{\mu}$, 
    then the condition number of $\mtx{A}_{\mu}\mtx{M}$ does not depend that of $\mtx{A}_{\mu}$.}
\end{quote}
This principle can be formalized with the following proposition, which we state without regularization for the sake of clarity.
\begin{restatable}{proposition}{leftsketchprecond}\label{prop:left_sketch_precond}
    Let $\mtx{U}$ be a matrix whose columns form an orthonormal basis for the range of $\mtx{A}$.
    If $\mtx{M}$ is a valid preconditioner for $\mtx{A}$,
    then the spectrum of $\mtx{A}\mtx{M}$ is equal to that of $(\mtx{S}\mtx{U})^{\dagger}$.
\end{restatable}
\noindent
\cite[Theorem 1]{RT:2008:SAP} gives a very similar statement under the assumption that $\mtx{A}$ is full-rank.
\cite[Lemma 4.2]{MSM:2014:LSRN} improved upon \cite{RT:2008:SAP} by supporting the rank-deficient case, at the price of strong assumptions on the sketching operator and the form of the preconditioner.
In \cref{subsubapp:we_sketch_subspaces} we provide what to our knowledge is the first proof of \cref{prop:left_sketch_precond} in its general form; we also explain its application to regularized problems.

\paragraph{Up next.}
We now turn to how one can compute orthogonalizers.
To keep things at a reasonable length we only speak to QR-based and SVD-based methods, although others 
could also be used.
Our goal is to give a general overview that includes time and space complexity considerations.
As to the latter consideration, we must note that these preconditioners have insubstantial space requirements when $\mtx{A}$ is dense and $m \gg d \gtrsim n$.
Separately, we note that details of the preconditioner generation process can affect the sketch-and-solve preprocessing step in sketch-and-precondition algorithms.
For more information on theoretical properties of these preconditioners in a \RandNLA{} context, we refer the reader to \cite{CFS:2021:general_SAP_LS}.

\subsubsection{QR-based preconditioning in the full-rank case}

QR-based preconditioning when $\mu = 0$ is very simple; one need only run Householder QR on $\mtx{A}^{\mathrm{sk}}$ and return $\mtx{M} = \mtx{R}^{-1}$ as a linear operator.
We note that specialized methods for QR decomposition of very tall matrices would not be appropriate here, since the $d \times n$ matrix $\mtx{A}^{\mathrm{sk}}$ will have $d \gtrsim n$.
Householder-type representations of $\mtx{A}^{\mathrm{sk}}$'s QR decomposition are especially useful since they require a modest amount of added workspace on top of storing $\mtx{A}^{\mathrm{sk}}$.

The case with $\mu > 0$ is more complicated if we want to avoid forming $\mtx{A}^{\mathrm{sk}}_{\mu}$ explicitly.
To describe it, suppose we have an initial QR decomposition $\mtx{A}^{\mathrm{sk}} = \mtx{Q}_o\mtx{R}_o$.
It is easy to show that the factor $\mtx{R}$ from a QR decomposition of $\mtx{A}^{\mathrm{sk}}_{\mu}$ is the same as the triangular factor from a QR decomposition of $\mtx{\hat{R}} := [\mtx{R}_o; \sqrt{\mu}\mtx{I}]$.
This observation is useful because there are specialized algorithms for QR decomposition of matrices given by an implicit vertical concatenation of a triangular matrix and a diagonal matrix; these specialized algorithms only require $O(n)$ additional workspace.
The factor $\mtx{Q}$ from a QR decomposition of $\mtx{A}^{\mathrm{sk}}_{\mu}$ can also be recovered with this approach, although the representation would be somewhat complicated.

If $\mtx{A}$ is not too ill-conditioned then the same preconditioner can be obtained by Cholesky-decomposing the regularized Gram matrix
\[
(\mtx{A}^{\mathrm{sk}}_{\mu})^{\trans}(\mtx{A}^{\mathrm{sk}}_{\mu}) = (\mtx{A}^{\mathrm{sk}})^{\trans}(\mtx{A}^{\mathrm{sk}}) + \mu\mtx{I},
\]
since the upper-triangular Cholesky factor of that matrix is the same as the factor $\mtx{R}$ from the QR decomposition of $\mtx{A}^{\mathrm{sk}}_{\mu}$.
This approach is simple to implement, and its time and space requirements are unaffected by whether or not $\mu$ is zero.
A sophisticated implementation could even try to form the regularized Gram matrix without allocating $dn$ space for $\mtx{A}^{\mathrm{sk}}$ as an intermediate quantity.
Although, it is clear that unless such a sophisticated implementation is used, there is no material memory savings compared to the Q-less QR approach described above.
This approach also affects sketch-and-solve preprocessing by requiring that we solve the normal equations, which is not a numerically stable approach \cite{Bjorck:1996}.

\subsubsection{QR-based preconditioning in the rank-deficient case}

Suppose for ease of exposition that $\mu = 0$ and let $k = \rank(\mtx{A}^{\mathrm{sk}}) \lesssim n$.
One can use a variety of methods to compute preconditioners that are \textit{morally triangular} in the sense that they are of the form $\mtx{M} = \mtx{P}\mtx{R}^{-1}$ for an $n \times k$ partial-permutation matrix $\mtx{P}$ and a triangular matrix $\mtx{R}$.
As long as the preconditioner orthogonalizes $\mtx{A}^{\mathrm{sk}}$, we can postprocess $\vct{z}_{\text{sol}} = \operatorname{argmin}\|\mtx{A}\mtx{M}\vct{z} - \vct{b}\|_2^2$ to obtain $\vct{x}_{\text{sol}} = \mtx{M}\vct{z}$ which solves $\min\|\mtx{A}\vct{x} - \vct{b}\|_2^2$.

The subtlety here is that when $k < n$ there is a nontrivial affine subspace of optimal solutions to $\min\|\mtx{A}\vct{x} - \vct{b}\|_2^2$.
Our stated goal in the rank-deficient case is to find the minimum-norm solution to the least squares problem (see \cref{eq:saddle_limiting_solutions}).
Unfortunately, if we assume that $\vct{b}$ has no role in defining $\mtx{M}$, then it is clearly impossible to guarantee that the norm of the recovered solution is anywhere near the minimum possible among all minimizers of $\|\mtx{A}\vct{x} - \vct{b}\|_2^2$.

\subsubsection{SVD-based preconditioners}

Let us denote the SVD of $\mtx{A}^{\mathrm{sk}}$ by $\mtx{U}\diag(\vct{\sigma})\mtx{V}^{\trans}$.

First we consider preconditioner generation when $\mu = 0$.
In this case we must account for the fact that $\mtx{A}^{\mathrm{sk}}$ might be rank-deficient.
Letting $k$ denote the rank of $\mtx{A}^{\mathrm{sk}}$, the SVD-based preconditioner is the $n \times k$ matrix
\[
\mtx{M} = \left[\, \frac{\vct{v}_1}{\sigma_1},\,\ldots,\,\frac{\vct{v}_k}{\sigma_k} \,\right].
\]
This construction is important, because it can be shown that if $\vct{z}_\star$ solves
\begin{equation}\label{eq:precond_rankdef_prim_saddle}
    \min_{\vct{z}}\|\mtx{A}\mtx{M}\vct{z} - \vct{b}\|_2^2 + \vct{c}^{\trans}\mtx{M}\vct{z}
\end{equation}
then $\vct{x} = \mtx{M}\vct{z}_{\star}$ satisfies \eqref{eq:saddle_limiting_solutions}.
We note in particular that \eqref{eq:precond_rankdef_prim_saddle} has a unique optimal solution and so computing $\vct{z}_{\star}$ is a well-posed problem.

SVD-based preconditioning is conceptually simpler when $\mu$ is positive, since in that case it does not matter if $\mtx{A}^{\mathrm{sk}}$ is rank-deficient.
However, it is harder to efficiently implement compared to when $\mu = 0$.
Here we present an efficient construction based on the relationship between the SVD of a matrix and the eigendecomposition of its Gram matrix.
Specifically, recall that the right singular vectors of a matrix $\mtx{F}$ are the eigenvectors of $\mtx{F}^{\trans}\mtx{F}$, and that the singular values of $\mtx{F}$ are the square roots of the eigenvalues of $\mtx{F}^{\trans}\mtx{F}$.

When used in our context this fact implies that the right singular vectors of $\mtx{A}^{\mathrm{sk}}_{\mu}$ are equal to those of $\mtx{A}^{\mathrm{sk}}$, and that its singular values are
\[
    \hat{\sigma}_i = \sqrt{\sigma_i^2 + \mu}.
\]
These observations alone are sufficient to recover the preconditioner
\[
    \mtx{M} = \mtx{V}\diag\left(\frac{1}{\hat{\sigma}_1},\ldots,\frac{1}{\hat{\sigma}_n}\right)
\]
which orthogonalizes $\mtx{A}^{\mathrm{sk}}_{\mu}$.

As a final point we consider the problem of recovering the left singular vectors of $\mtx{A}^{\mathrm{sk}}_{\mu}$ given the SVD of $\mtx{A}^{\mathrm{sk}}$.
This is useful in settings such as \cref{alg:saddle_to_ols_sap} for presolve and problem transformation purposes.
Moreover, it can actually be done efficiently.
If we define
\[
\mtx{D}_1 = \diag\left(\frac{\sigma_1}{\hat{\sigma}_1},\ldots,\frac{\sigma_n}{\hat{\sigma}_n}\right)
\quad\text{and}\quad
\mtx{D}_2 = \diag\left(\frac{\sqrt{\mu}}{\hat{\sigma}_1},\ldots,\frac{\sqrt{\mu}}{\hat{\sigma}_n}\right)
\]
then by assumption on $\mtx{M}$ the left singular vectors of $\mtx{A}^{\mathrm{sk}}_{\mu}$ are given by
\[
\begin{bmatrix} \mtx{A}^{\mathrm{sk}} \\ \sqrt{\mu}\mtx{I}\end{bmatrix}\mtx{M} 
    = \begin{bmatrix}\mtx{A}^{\mathrm{sk}}\mtx{M} \\ \sqrt{\mu}\mtx{M}\end{bmatrix}
    = \begin{bmatrix}\mtx{U}\mtx{D}_1 \\ \mtx{V}\mtx{D}_2\end{bmatrix}.
\]
We note that the column-orthonormality of this matrix can easily be verified by its rightmost representation.

To recap, we have introduced three key benefits of SVD-based preconditioning for tall least squares and saddle point problems.
First, it can be used to find the minimum norm solutions in \eqref{eq:saddle_limiting_solutions} in the rank-deficient case.
Second, an SVD-based preconditioner can be computed in the presence of regularization given only the singular values and right singular vectors of $\mtx{A}^{\mathrm{sk}}$.
Third, the SVD of $\mtx{A}^{\mathrm{sk}}$ is sufficient to recover the SVD of $\mtx{A}^{\mathrm{sk}}_{\mu}$, which facilitates sketch-and-solve as a preprocessing step in sketch-and-precondition.


\begin{remark}[Computational complexity]
    The default algorithm for SVD is currently divide-and-conquer \cite{GE:1995}.
    Two somewhat-outdated algorithms for computing the SVD are described in \cite[\S 8.6]{GvL:2013:MatrixComputationsBook};
    \cite[Figure 8.6.1]{GvL:2013:MatrixComputationsBook} gives complexity estimates for these algorithms depending on whether the left singular vectors need to be computed.
\end{remark}

\subsection{Preconditioning least squares and saddle point problems: data matrices with fast spectral decay}
\label{subsec:saddle_nys_precond}

\paragraph{Interpreting the \Nystrom{} preconditioner.}
Recall from \cref{subsubsec:nys_pcg} that the \Nystrom{} preconditioning approach to solving $(\mtx{G} + \mu\mtx{I})\vct{x} = \vct{h}$ starts by constructing a low-rank approximation of $\mtx{G}$.
That approximation defines a preconditioner $\mtx{P}$ satisfying three properties:
\begin{enumerate}
    \item $\mtx{P}$ is positive definite.
    \item $(\mtx{G} + \mu \mtx{I})\mtx{P}^{-1}$ is well-conditioned on a subspace $L$ that contains $\mtx{G}$'s dominant eigenspaces.
    \item $\mtx{P}$ acts as the identity on $L^{\perp}$ (the orthogonal complement of $L$).
\end{enumerate}
Such a preconditioner will be effective when the action of $\mtx{G}$ on $L^{\perp}$ is ``not too pronounced'' compared to that of $\mu\mtx{I}$.
More formally, if we define the restricted spectral norm of $\mtx{G}$ on $L^{\perp}$
\[
    \|\mtx{G}\|_{L^{\perp}} = \max\{ \|\mtx{G}\vct{z}\|_2 \,:\, \vct{z} \in L^{\perp}, \|\vct{z}\|_2 = 1 \}
\]
then the preconditioner will be effective if $\|\mtx{G}\|_{L^{\perp}} / \mu$ is $O(1)$.

\paragraph{Adaptation to saddle point problems.}
\Nystrom{} preconditioners can naively be used for regularized saddle point problems by taking $\mtx{G} = \mtx{A}^{\trans}\mtx{A}$ and considering the normal equations \eqref{eq:saddle_normal_eqs}.
However, the numerical properties of iterative least squares solvers that only access $\mtx{A}^{\trans}\mtx{A}$ tend to be less robust than those of iterative solvers that access $\mtx{A}$ and $\mtx{A}^{\trans}$ separately (i.e., solvers such as LSQR).
This motivates having an extension of the \Nystrom{} preconditioner to be compatible with the latter type of solver.

Towards this end, let us express $\mtx{P}^{-1}$ with a (possibly non-symmetric) matrix square-root $\mtx{M}$, satisfying the relation $\mtx{P}^{-1} = \mtx{M}\mtx{M}^{\trans}$.
We appeal to the well-known fact that running PCG on $(\mtx{G} + \mu\mtx{I})\vct{z} = \vct{h}$ with preconditioner $\mtx{P}$ is equivalent to running the ``unpreconditioned'' CG algorithm on
\[
\mtx{M}^{\trans}(\mtx{G} + \mu \mtx{I})\mtx{M}\vct{z} = \mtx{M}^{\trans}\vct{h}.
\]
When framed in this way, we can ask how $\mtx{M}$ should relate to $\mtx{A}$ and $\mu$ so that it \textit{would} be a good preconditioner \textit{if} it were used on the normal equations.

To answer this question we work with the augmented matrix $\mtx{A}_{\mu} = [\mtx{A};\sqrt{\mu}\mtx{I}_n]$.
The basic criteria for a $\mtx{P}$ as a \Nystrom{} preconditioner for the normal equations \eqref{eq:saddle_normal_eqs} can be stated with $\mtx{M}$ as follows:
\begin{enumerate}
    \item $\mtx{A}_{\mu}\mtx{M}$ should be well-conditioned on a subspace $L$ that includes the dominant right singular vectors of $\mtx{A}_{\mu}$.
    \item we should have $\mtx{A}_{\mu}\mtx{M}\vct{x} = \mtx{A}_{\mu}\vct{x}$ for all $\vct{x} \in L^{\perp}$.
\end{enumerate}
Whether such a preconditioner will be effective can be stated with the ``restricted spectral norm'' as defined above.
Specifically, $\mtx{M}$ should be effective if the above conditions hold and $\|\mtx{A}\|_{L^{\perp}} / \sqrt{\mu}$ is $O(1)$.
We note that the requisite matrix $\mtx{M}$ can be constructed efficiently by similar principles as methods for low-rank SVD in \cref{sec4:lowrank}, and we leave the details to future work.

\subsection{Deterministic preconditioned iterative solvers}\label{subsubsec:det_saddle_solve}

Most of the drivers described in \cref{subsec:optim_drivers} amount to using randomization to obtain a preconditioner and then calling a traditional iterative solver that can make use of that randomized preconditioner.
Here we list some iterative solvers that could be of use for these drivers.
We note up front that many factors can affect the ideal choice of iterative method in a given setting.

\begin{itemize}
    \item \textbf{CG} \cite{HS:1952:CGLS} is the most broadly applicable solver in our context. It applies to the regularized positive definite system \eqref{eq:unconstr_quadratic_opt} and the normal equations of the primal saddle point problem \eqref{eq:saddle_normal_eqs}.
    \item \textbf{CGLS} \cite{HS:1952:CGLS} applies when $\vct{c}$ is zero. It is equivalent to CG on the normal equations in exact arithmetic, but is more stable than CG in finite-precision arithmetic. 
    \item \textbf{LSQR} \cite{PS:1982} applies when at least one of $\vct{c}$ or $\vct{b}$ is zero.
    When considered for overdetermined problems it is algebraically (but not numerically) equivalent to CGLS.
    It is more stable than CG \cite[\S~7.6.3]{Bjorck:1996}, \cite[\S~4.5.4]{Bjorck:2015} and CGLS \cite[\S~9]{PS:1982} for ill-conditioned problems. 
    \item \textbf{CS} (the Chebyshev semi-iterative method) \cite{GV:1961} applies to the same class of systems as CG. It has fewer synchronization points in each iteration and so can take better advantage of parallelism. It requires knowledge of an upper bound and a lower bound on the eigenvalues of the system matrix. We refer the reader to \cite[\S~7.2.5]{Bjorck:1996}, \cite[\S~4.1.7]{Bjorck:2015} for information on this method.
    \item \textbf{LSMR} \cite{FS:2011:LSMR} applies to the same problems as LSQR. For overdetermined least squares it is algebraically equivalent to MINRES on the normal equations. In that context, the residual of the normal equations will decrease with each iteration, which makes it safer to stop early compared to LSQR.
\end{itemize}

These algorithms vary in how they accommodate preconditioners.
Some require implicitly preconditioning the problem data, calling the ``unpreconditioned'' solver, then applying some (cheap) postprocessing to the returned solution.
We note that it is necessary to ``precondition'' any regularization term in the problem's objective when using such an algorithm.\footnote{That is, if we precondition a ridge regression problem, then it is necessary to precondition the \textit{augmented matrix} $[\mtx{A};\sqrt{\mu}\mtx{I}]$ in an unregularized version of the problem.}
For other algorithms, a preconditioner is supplied alongside the problem data, and the algorithm returns a solution that requires no postprocessing.
The difference between these situations is that different quantities end up being available for use in termination criteria (at least for off-the-shelf implementations).
We emphasize that appropriate choices of termination criteria can be crucial for iterative solvers to work effectively, and we refer the reader to  \cref{subapp:error_metrics} for discussion on this and other topics.

Any standard library implementing the drivers from \cref{sec3:LS_and_optim} should include computational methods for (preconditioned) CG and LSQR.
LSQR is most naturally applied to dual saddle point problems by reformulation to \eqref{eq:dual_saddle_as_uls} and to primal saddle point problems by reformulation to \eqref{eq:primal_saddle_as_ols}.
More full-featured \RandNLA{} libraries would do well to include implementations of CS or LSMR, and ``blocked'' versions of iterative solvers.
Such blocked methods apply to linear systems and least squares problems with multiple right-hand sides; they take better advantage of parallel hardware and have slightly faster convergence rates than their non-blocked counterparts.


\section{Other optimization functionality}\label{subsec:other_opt_algs}

Here, we briefly discuss other \RandNLA{} algorithms of note for least squares or optimization, often commenting on how they fit into our plans for \RandLAPACK{}.
Some of these algorithms are out-of-scope for a linear algebra library but can be directly facilitated by the drivers we described in \cref{sec3:LS_and_optim}.

\subsubsection{Facilitating second-order optimization algorithms}
\label{subsubsec:second_order_algs}

Many second-order optimization algorithms need to solve sequences of saddle point systems, where $\mtx{A},\vct{b},\vct{c}$ vary continuously from one iteration to the next.
\RandLAPACK{} will support such use-cases \textit{indirectly} through methods that help amortize the dominant computational cost of a single randomized algorithm across multiple saddle point solves.
See \cites{PW:2017:NewtSketch,fred_SSN_JRNL} for uses of \RandNLA{} for second-order optimization.

The most common way for $\mtx{A}$ to vary is by a reweighting: when $\mtx{A} = \mtx{W}\mtx{A}_o$ for a fixed matrix $\mtx{A}_o$ and an iteration-dependent matrix $\mtx{W}$.
The matrix $\mtx{W}$ is typically (but not universally) a matrix square root of the Hessian of some separable convex function.
The randomized algorithms described in this \nameCref{sec3:LS_and_optim} will only be useful for such problems when $\mtx{W}$ and its adjoint can be applied to $m$-vectors in $O(m)$ time.
This condition is satisfied in limited but important situations such as in algorithms for logistic regression, linear~programming, and iteratively-reweighted least squares.

\subsubsection{Stochastic Newton and subsampled Newton methods}

Newton Sketch is a prototype algorithm developed over two papers \cites{PW:2016:HessSketch, PW:2017:NewtSketch} which is closely related to subsampled Newton methods~\cites{XRM17_theory_TR, YXRM18_TR, fred_SSN_JRNL}.
Each is suited to optimization problems that feature non-quadratic objective functions or problems with constraints other than linear equations.
These methods entail sampling a new sketching operator (and applying it to a new data matrix) in each iteration, with the aim of approximating the Hessian of the objective at the given iterate.
The algorithms described in this \nameCref{sec3:LS_and_optim} can easily serve as the main subroutine in Newton Sketch and subsampled Newton methods.

Newton Sketch has a natural specialization for least squares which entails sampling and applying only one sketching operator.
This specialization can be viewed as sketch-and-precondition, where the iterative method for solving the saddle point system is based on preconditioned steepest-descent.
The asymptotic convergence of this approach can be established in various ways \cites{OPA:2019,LP:2019,Tropp:2021:LecNotes}.
It has been shown that ``traditional'' sketch-and-precondition methods (based on CG or the Chebyshev semi-iterative method) exhibit faster convergence \cite{LP:2019}. 
Therefore we do not expect to incorporate this method into \RandLAPACK{}.

There are two recently proposed extensions of Newton Sketch that may be suitable for solving the saddle point problems described in \cref{subsubsec:optim:saddle_prob_class}: Hessian averaging \cite{NDM:2022:HessianAvg} and stochastic variance reduced Newton (SVRN) \cite{Der22_Stochastic_TR}.
The performance profiles of these methods are better when $\mtx{A}$ is very tall.
When specialized to least squares, the former method amounts to preconditioned steepest-descent where the preconditioner is updated at each iteration.
By comparison (again in the least squares setting), SVRN amounts to steepest-descent with a fixed preconditioner that incorporates variance-reduced sketching methods (adapted from  \cite{JZ13_Accelerating}) to approximate the gradient at each iteration.

\subsubsection{Random features preconditioning for KRR}

A random-features approach for computing accurate solutions to problems of the form \eqref{eq:unconstr_quadratic_opt} in the context of KRR is proposed in \cite{ACW:2017:KRR}.
Specifically, \cite{ACW:2017:KRR} advocates for using random features to obtain a preconditioner for use in an iterative method such as PCG.
Since any such iterative solver requires access to $\mtx{G}$ by matrix-vector multiplication, \Nystrom{} PCG can be applied to the same problems as this \textit{random-features preconditioning}.
Empirical results strongly suggest that the \Nystrom{} approach has better performance than random-features preconditioning on shared-memory machines \cite{FTU:2021:NystromRidge}.
Thus, we do not plan for \RandLAPACK{} to support random-features preconditioning at this time.

\subsubsection{Utilities for iterative refinement}
\label{subsubsec:unti_iter_refine}

Iterative refinement can be used as a tool to compensate for rounding errors in otherwise reliable linear system solvers.
These methods typically work by computing residuals to higher precision than that used by the solver, running the solver with the updated residual, and then adding the new solution to the original solution \cite[\S 2.9.2]{Bjorck:1996}.\footnote{In some situations, it can suffice to recompute the residual with the same precision used by the underlying solver \cite[\S 2.9.3]{Bjorck:1996}.}
\LAPACK{} has some procedures of this kind.
See \cite{Higham:1997} and \cite{DHKLMR:2006} for theoretical and practical analyses.

The best way to use iterative refinement routines to support randomized algorithms in this \nameCref{sec3:LS_and_optim} is yet to be determined.
On the one hand, it may suffice to include methods similar to those in \LAPACK{}.
However, sketch-and-precondition algorithms might pose unique numerical problems that require different techniques.
See \cite[\S 5.7]{AMT:2010:Blendenpik} for some discussion of numerical issues in the sketch-and-precondition context.
It might also be natural for a \RandNLA{} library to have more iterative refinement methods than \LAPACK{} in order to better exploit low-precision arithmetic and accelerators.


\section{Existing libraries}\label{subsec:opt_libraries}

We know of four high-performance libraries with  sketch-and-precondition methods for least squares: Blendenpik \cite{AMT:2010:Blendenpik}, LSRN \cite{MSM:2014:LSRN}, \LibSkylark{} \cite{libskylark}, and \textsf{Ski-LLS} \cite{CFS:2021:general_SAP_LS}.\footnote{Note that ``Blendenpik'' and ``LSRN'' are names for algorithms \textit{and} libraries.}
To our knowledge,
\LibSkylark{} is the only \RandNLA{} library which supports least squares \textit{and} low-rank approximation (see \cref{subsec:lowrank_libraries}).
None of these libraries support saddle point problems of the kind we consider, and none of them make use of \Nystrom{} preconditioning.

\paragraph{Blendenpik.}
This library is written in C and callable from Matlab; it is currently available on the Matlab File Exchange.
It uses LSQR as the deterministic iterative solver, and obtains the preconditioner by running QR on a sketch $\mtx{A}^{\mathrm{sk}} = \mtx{S}\mtx{A}$, where $\mtx{S}$ is an SRFT.
Blendenpik also adaptively calls \LAPACK{} if a problem is deemed too poorly scaled or if the iterative method performs poorly.
It was shown to outperform an unspecified \LAPACK{} least squares solver on a machine with 8GB RAM and an AMD Opteron 242 processor \cite{AMT:2010:Blendenpik}.

\paragraph{LSRN.}
This comprises a C++ implementation callable from Matlab and a Python implementation.
The C++ implementation was shown to outperform \LAPACK{}'s \code{DGELSD} on large dense problems, and Matlab's backslash (\textsf{SuiteSparseQR}) on sparse problems.
The Python implementation has demonstrated that LSRN scales well on Amazon Elastic Compute Cloud clusters.
We note that the Python implementation relies on an auxiliary Python package with a custom C-extension for sampling from the Gaussian distribution via the ziggurat method.

\paragraph{\LibSkylark{}.}
This library is written in C++ and is available on GitHub.
Its support for least squares problems is very general and includes a few deterministic preconditioned iterative solvers.
Its sketch-and-precondition functionality includes implementations in the styles of Blendenpik and LSRN.
\LibSkylark{} has a Python interface, but only for Python 2.7.
Its linear algebra kernels are implemented partly in the \Elemental{} distributed linear algebra library \cite{Elemental}.
Unfortunately, \Elemental{} is no longer maintained.

\paragraph{\textsf{Ski-LLS}.}
This is a recently developed C++ library for solving dense and sparse highly overdetermined least squares problems.
It is distinguished by its flexibility in preconditioner generation.
In particular, it supports sketching by SRFTs, Gaussian operators, and SASOs.
It also supports factoring the sketch $\mtx{S}\mtx{A}$ by several methods, including a standard SVD algorithm, a randomized algorithm for full-rank column-pivoted QR (see \cref{subsec:fullrank_decomp:qrcp}), and a standard algorithm for sparse QR.
We record the following (adapted) quote from the GitHub repository that hosts this software:
\begin{quote}
    \textsf{Ski-LLS} is faster and more robust than Blendenpik and LAPACK on large over-determined data matrices, e.g., matrices having 40,000 rows and 4,000 columns. \textsf{Ski-LLS} is 10 times faster than Sparse QR and incomplete-Cholesky preconditioned LSQR on sparse data matrices that are ill-conditioned and sufficiently large, e.g., with 120,000 rows, 5,000 columns, and 1\% non-zeros.
\end{quote}

\paragraph{Falkon.}
Finally, we note the recently developed \textit{Falkon} library for sketch-and-solve approaches to KRR powered by multi-GPU machines \cites{MCRR:2020:KRRimplement,MCdVR:2022:KRRautotune}.
While this library works outside of our primary data model, it is of interest to anyone developing software for KRR based on \RandNLA{}.

\chapter{Low-rank Approximation}
\label{sec4:lowrank}

\minitoc
\bigskip

Modern scientific computing, machine learning, and data science applications generate massive matrices that need to be processed for reduced run time, reduced storage requirements, or improved interpretability.
\textit{Low-rank approximation} is a workhorse approach for achieving these goals.
Here, given a target matrix $\Ao$, the task is to produce a suitably factored representation of a low-rank matrix $\Aa$ of the same dimensions which approximates the matrix $\Ao$.

We can express the main aspects of a low-rank approximation as computing factor matrices $\mtx{E}$ and $\mtx{F}$ where
\begin{equation}
\begin{array}{cccccccc}
\Ao & \approx & \Aa & := & \mtx{E} & \mtx{F} \\
m\times n &   & m\times n & & m\times k & k\times n
\end{array} 
\end{equation}
for some $k \ll \min\{m, n\}$.
We note that it is very common to have a $k \times k$ ``inner factor'' that appears in between $\mtx{E}$ and $\mtx{F}$ above.

Such representations facilitate data interpretation by choosing the factors to have useful structure, such as having orthonormal columns or rows, or being submatrices of the target.
The extent of storage reduction from low-rank approximation depends on whether $\Ao$ is dense or sparse.
In the dense case, $\Aa$ is stored in $O(mk + nk)$ space.
In the sparse case, one representation consists of a dense $k \times k$ inner factor, a slice of $k$ rows of $\Ao$, and a slice of $k$ columns of $\Ao$.

The rank $k$ used in a low-rank approximation is a tuning-parameter that the user can control to trade-off between approximation accuracy and data compression.
The best choice of this parameter depends on context. 
For instance, one may want to choose $k$ small enough to graphically visualize coherent structure in the target.
In such a setting one would not expect that $\Aa$ is close to $\Ao$ in an absolute sense, but one can still ask that the distance is near the minimum among all approximations with the desired structure and rank.
Alternatively, one might know that $\Ao$ can be well-approximated by a low-rank matrix, and yet not know the rank necessary to achieve a good approximation.
Such matrices are called \textit{numerically low-rank} and arise in applications across the social, physical, biological, and ecological sciences. 
For example, they can arise as discretizations of differential operators, where the extent to which the matrix is numerically low-rank depends on the details of the operator and the discretization; and they can arise as noisy corruptions of general (hypothesized) data matrices with low exact rank.
When dealing with such matrices one can iteratively build $\Aa$ until a desired distance $\|\Ao - \Aa\|$ is small.
This \nameCref{sec4:lowrank} covers a variety of efficient and reliable low-rank approximation algorithms for both of these scenarios.

\section{Problem classes}
\label{subsec:supported_lowrank}

Low-rank approximation is naturally formalized as an optimization problem; one chooses $\Aa$ and its factors to minimize a loss function subject to some constraints.
The most common loss functions are distances $\Aa \mapsto \|\Ao - \Aa\|$ induced by the Frobenius or spectral norms.
Alternatively, one can use the discontinuous loss function $\Aa \mapsto \rank(\Aa)$ as a measure of the storage requirements for $\Aa$.
Constraints depend on the loss function in a complementary way.
When minimizing a norm-induced distance, one imposes rank constraints by limiting the dimensions of the factors.
When minimizing the rank of $\Aa$ (i.e., when seeking an approximation that admits the smallest-possible representation) one constrains the approximation error $\|\Ao - \Aa\|$ to be at most some specified value.
One can also impose \textit{structural} constraints on the factors of $\Aa$, such as being orthogonal, diagonal, or a submatrix of the target.

Our overview of randomized algorithms for low-rank matrix approximation is organized around such structural constraints.
Accordingly, we use the term \textit{problem class} for loose groups of low-rank approximation problems wherein the factors facilitate similar downstream tasks.
Currently, our two problem classes are the following.
\begin{itemize}
    \item 
    Spectral decompositions (\S \ref{subsubsec:class_spectral}): this consists of low-rank SVD and Hermitian eigendecomposition.
    \item
    Submatrix-oriented decompositions, i.e., decompositions with factors based on submatrices of the target matrix (\S \ref{subsubsec:class_submatrix}): this consists of so-called \textit{CUR} and \textit{interpolative decompositions}.
\end{itemize}
Optimal decompositions in the first class often serve as baselines in theoretical analyses of randomized algorithms for low-rank decomposition in both classes.
That is, such comparisons are made \textit{regardless} of whether the approximation is spectral or submatrix-oriented.
This fact can blur the distinction between the two problem classes, and the distinctions can blur even further when one considers methods for efficiently converting from one decomposition to another.
Still, keeping the problem classes separate is useful as an organizing principle for the most fundamental low-rank approximation problems in \RandNLA{}.

\begin{remark}
    Low-rank approximations that impose no requirements on $\Aa$'s representation are briefly addressed in  \cref{subsubsec:oblique_proj} in the context of computational routines.
    Methods for low-rank approximation with other representations (e.g., QR, UTV, LU, nonnegative factorization)
    are discussed in \cref{subsec:other_lowrank}.
\end{remark}

\subsection{Spectral decompositions}
\label{subsubsec:class_spectral}

In what follows we give an overview of the SVD and Hermitian eigendecomposition, with emphasis on the roles of these decompositions in low-rank approximation.
After covering these concepts we explain how they provide two perspectives on principal component analysis (PCA).
We advise the reader to at least skim this overview material even if they are already familiar with the relevant concepts; low-rank approximation is much more prominent in \RandNLA{} than it is in classical NLA.

\subsubsection{Singular value decomposition}

The SVD is widely used to compute low-rank approximations and as a workhorse algorithm for PCA. 
Given a $m \times n$ matrix $\Ao$, where $m \geq n$ (without loss of generality), its SVD is
\begin{equation}
\label{eqn:svd_xxx1}
\begin{array}{cccccccc}
\Ao & = & \mtx{U} & \mtx{\Sigma}  & \mtx{V}^{\trans} \\
m\times n &   & m\times n & n\times n & n\times n
\end{array},
\end{equation}
where $\mtx{U} = [ \vct{u}_1, \ldots, \vct{u}_n ]$ and $\mtx{V} = [ \vct{v}_1, \ldots , \vct{v}_n ]$ are column-orthonormal matrices that contain the left and right singular vectors of $\Ao$. 
The matrix $\mtx{\Sigma} = \diag(\sigma_1,\dots,\sigma_n)$ contains the corresponding singular values; we use the convention that they appear in decreasing order $\sigma_{1} \geq \ldots \geq  \sigma_{n} \geq 0$.
We can also think about the SVD as expressing $\Ao$ as the sum of $n$ rank-one matrices
\begin{equation}\label{eq:svd_sum}
	\Ao = \sum_{i=1}^{n} \sigma_i \vct{u}_i \vct{v}_i^{\trans}.
\end{equation}

In applications it is common to encounter data matrices with low-rank structure, i.e., matrices for which $r = \rank(\Ao)$ is smaller than the ambient dimensions $m$ and $n$ of $\Ao$.
In this case, the singular values $\{\sigma_{i}: i\geq r+1\}$ are zero, the corresponding singular vectors span the left and right null spaces, and it is natural to consider the \textit{compact SVD} where the sum in \eqref{eq:svd_sum} is truncated at $i = r$.
For a matrix $\Ao$ with \textit{approximate} low-rank structure, we can obtain approximations with low \textit{exact} rank by truncating this sum even earlier, at some $k < r$:
\begin{align}
	\Ao \approx \Aa_k :=& \sum_{i=1}^{k} \sigma_i \vct{u}_i \vct{v}_i^{\trans} \nonumber \\
        =& [ \vct{u}_1, \dots , \vct{u}_k ] 
	\diag(\sigma_1,\dots,\sigma_k) [ \vct{v}_1, \dots , \vct{v}_k ]^{\trans}= \mtx{U}_k \mtx{\Sigma}_k \mtx{V}_k^{\trans},
	\label{eqn:svd_trincated_xxx1}
\end{align}
Truncating trailing singular values yields optimal rank-constrained approximations, in the sense of solving
\begin{equation}\label{eq:svd_as_opt}
	\Aa_k \in \argmin_{\rank(\Aa') = k} \| \Ao - \Aa' \|.
\end{equation} 
This holds for every $k \in \idxs{r}$.
In other words, if $\Ao$ is approximated by a rank-$k$ matrix $\Aa_k$ given through its SVD, no further computation is needed to canonically obtain approximations of $\Ao$ with any rank $k \leq r$.

The optimality result of \eqref{eq:svd_as_opt} holds for any unitarily invariant matrix norm, and it is known as the \textit{Eckart-Young-Mirsky Theorem} when considered for the spectral norm or Frobenius norm.
The reconstruction errors according to these norms are
\begin{equation}\label{eq:eckartyoungbounds}
	\| \Ao - \Aa_k \|_2 = \sigma_{k+1}(\Ao) \quad \mbox{and} \quad \| \Ao - \Aa_k \|_{\text{F}} = \sqrt{\sum_{j > k} \sigma_j^2(\Ao)}.
\end{equation} 
These facts are important in applications, where it is common to see $\rank(\Ao) = \min\{m, n\}$ in exact arithmetic and yet many of the trailing singular values are so small that they can be presumed to be noise.
That is, the truncation introduced in \eqref{eqn:svd_trincated_xxx1} is often used as a denoising technique.
%

\subsubsection{Hermitian eigendecomposition}

A matrix is called \textit{Hermitian} if it is equal to its adjoint, i.e., if $\Ao = \Ao^{\trans}$.
For real matrices, being Hermitian is the same as being symmetric.
The eigendecomposition of a Hermitian matrix $\Ao$ is
\begin{equation}
\begin{array}{cccccccc}
\Ao & = & \mtx{V} & \mtx{\Lambda}  & \mtx{V}^{\trans} \\
n\times n &   & n\times n & n\times n & n\times n
\end{array},
\end{equation}
where $\mtx{V}$ is an orthogonal matrix of eigenvectors and $\mtx{\Lambda} = \diag(\lambda_1,\ldots,\lambda_n)$ is a real matrix containing the eigenvalues of $\Ao$.
A Hermitian matrix is further called \textit{positive semidefinite} (or ``psd'') if $\lambda_i \geq 0$ for all $i$.

We use the convention of sorting eigenvalues in decreasing order of absolute value: $|\lambda_1| \geq \cdots \geq |\lambda_n|$.
This allows for a more direct comparison to the SVD, since we obtain low-rank approximations
\begin{align}
	\Ao \approx \Aa_k :=& \sum_{i=1}^{k} \lambda_i \vct{v}_i \vct{v}_i^{\trans} \nonumber \\
        =& [ \vct{v}_1, \dots , \vct{v}_k ] 
	\diag(\lambda_1,\dots,\lambda_k) [ \vct{v}_1, \dots , \vct{v}_k ]^{\trans}= \mtx{V}_k \mtx{\Lambda}_k \mtx{V}_k^{\trans}
\end{align}
for which the spectral and Frobenius-norm distances to $\mtx{A}$ match those from \eqref{eq:eckartyoungbounds}.
Indeed, a (truncated) eigendecomposition can be converted to a (truncated) SVD by taking the columns of $\mtx{V}$ as the right singular vectors, setting the left singular vectors according to
\[
    \vct{u}_i = \begin{cases}
                    \vct{v}_i & \text{if } \lambda_i > 0 \\
                    -\vct{v}_i & \text{otherwise}
                \end{cases}
\]
and setting the singular values to $\vct{\sigma} = |\vct{\lambda}|$ (elementwise).

If a matrix is Hermitian then it is better to compute (and work with) its eigendecomposition, rather than its SVD.
The first reason for this is that a rank-$k$ eigendecomposition requires almost half the storage of a rank-$k$ SVD.
The second reason is that algorithms for computing low-rank eigendecompositions are able to leverage structure in the matrix for improved efficiency.
These efficiency improvements can be dramatic for psd matrices, where an eigendecomposition is technically also an SVD.

\subsubsection{Connections to principal component analysis}

PCA is a linear dimension reduction technique that is widely used in data science applications for extracting features, or for visualizing and summarizing complicated datasets.
The idea of PCA is to form $k$ new variables (components) $\mtx{Z} = [\vct{z}_1, \dots , \vct{z}_k]$ as linear combinations of the variables $\mtx{X} = [\vct{x}_1, \dots , \vct{x}_n] \in \R^{m \times n}$ (that are assumed to have been preprocessed to have column-wise zero empirical mean).
Specifically, given the data matrix $\mtx{X}$, one forms the variables as $\mtx{Z} = \mtx{X}\mtx{W}$ where the weights $\mtx{W} = [\vct{w}_1, \dots , \vct{w}_k] \in \R^{n \times k}$ 
are chosen so that the first component $\vct{z}_1$ accounts for most of the variability in the data, the second component $\vct{z}_2$ for most of the remaining variability, and so on.

Formally, we can formulate this problem as a variance maximization problem
\begin{equation}
    \vct{w}_1 := \argmax_{\|\vct{w}\|_2^2=1} \operatorname{Var}(\mtx{X}\vct{w})
\end{equation}
where we define the variance operator as $\operatorname{Var}(\mtx{X}\vct{w}):=\frac{1}{m-1}\|\mtx{X}\vct{w}\|^2_2$.
Defining the sample covariance matrix $\mtx{C}:=\frac{1}{m-1}\mtx{X}^{\trans} \mtx{X}$, this problem can be stated as
\begin{equation}
    \vct{w}_1 := \argmax_{\|\vct{w}\|_2^2=1} \vct{w}^{\trans}\mtx{C}\vct{w}.\label{eq:pca_first_component_covariance}
\end{equation}
We recognize \eqref{eq:pca_first_component_covariance} as the variational formulation of the dominant eigenvector of a Hermitian matrix.
That is, $\vct{w}_1$ satisfies
\begin{equation}
    \mtx{C}\vct{w}_1 = \lambda_1(\mtx{C}) \vct{w}_1. 
\end{equation}
More generally, PCA finds the weights $\vct{w}_1, \dots , \vct{w}_k$ by diagonalizing the empirical sample covariance matrix $\mtx{C}$ as $\mtx{C}=\mtx{W}\mtx{\Lambda}\mtx{W}^{\trans}$, and retaining only the top $k$ eigenvectors.

How one computes the PCA depends on how the data is presented and accessed.
There are two situations of interest.
In situations where the covariance matrix $\mtx{C}$ is given by the problem at hand---and thus where we access $\mtx{C}$ directly, but do \emph{not} directly access the data matrix $\mtx{X}$---one can directly employ a low-rank Hermitian eigendecomposition to compute the dominant $k$ eigenvectors. 
If instead we are presented with the variables in form of a mean-centered data matrix $\mtx{X}$, then the low-rank SVD becomes the preferable approach to computing the weights $\mtx{W}$.
This is because we can relate the eigenvalue decomposition of the inner product $\mtx{X}^{\trans}\mtx{X}$ to the SVD of $\mtx{X}=\mtx{U}\mtx{\Sigma}\mtx{V}^{\trans}$ by
\begin{equation}
    \mtx{X}^{\trans} \mtx{X} = (\mtx{V}\mtx{\Sigma}\mtx{U}^{\trans}) (\mtx{U}\mtx{\Sigma}\mtx{V}^{\trans}) = \mtx{V}\mtx{\Sigma}^2 \mtx{V}^{\trans}.
\end{equation}
Hence, we obtain the $k$ weights $\mtx{W} = [\vct{w}_1, \dots , \vct{w}_k]$ by computing the top $k$ right singular vectors $\mtx{V}_k = [\vct{v}_1, \dots , \vct{v}_k]$, and the eigenvalues of the sample covariance matrix $\mtx{C}$ are given by the diagonal elements $\frac{1}{m-1} \mtx{\Sigma}^2$.

Fast randomized algorithms for computing the SVD and eigenvalue decomposition enable scaling PCA to especially large problems.
However, one does not need a large problem to benefit from these randomized algorithms.
From 2016 until at least time of writing, the \code{pca} function \code{SciKit-Learn} defaults to a randomized algorithm when $d = \min\{m, n\} \geq 500$ and $k$ is less than $0.8 d$ \cite{SciKitLearn-PCA}.

\subsection{Submatrix-oriented decompositions}
\label{subsubsec:class_submatrix}

Here we describe four types of submatrix-oriented decompositions: a \textit{CUR decomposition} and three types of \textit{interpolative decompositions}.
Historically, these have been used far less often than spectral decompositions.
However, their value propositions have become much more compelling in recent years:
\begin{itemize}
    \item 
    They can offer reduced storage requirements compared to spectral decompositions, especially for sparse data matrices.
    This can be very valuable in processing massive data sets.
    \item 
    They provide for more transparent data interpretation.
    This is especially true when data is modeled as a matrix as a matter of convenience, rather than as a statement about the data defining a meaningful linear operator $\Ao \mapsto \Ao\vct{v}$.
\end{itemize}

\begin{remark}
     The following material is dense.
    We encourage the reader to return to it as needed while reading later parts of this~\nameCref{sec4:lowrank}.
\end{remark}

\subsubsection{CUR decomposition}

\paragraph{Definition.} A CUR decomposition is a low-rank approximation of the form
\begin{equation}\label{eq:cur}
\begin{array}{ccccc}
\Ao & \approx & \mtx{C} & \mtx{U} & \mtx{R}, \\
m\times n &   &  m\times k & k\times k & k\times n
\end{array} 
\end{equation}
where the factors $\mtx{C}$ and $\mtx{R}$ are formed by small subsets of actual columns and rows of $\Ao$, and the \textit{linking matrix} $\mtx{U}$ is chosen so that some norm of $\Ao - \mtx{C}\mtx{U}\mtx{R}$ is small.

\paragraph{Motivation.}

The literature on CUR decomposition traces back to work by Gore\u{\i}nov, Zamarashkin, and Tyrtyshnikov, who proved existential results for CUR decompositions with certain approximation error bounds \cites{GZT:1995:CUR,GTZ:1997:CUR}.
Gore\u{\i}nov et al.\ motivated their investigations by pointing out that CUR decompositions have far lower storage requirements than partial SVDs when $\Ao$ is sparse.
More specifically, they advocated for applying CUR decompositions for low-rank approximation of off-diagonal blocks in block matrices.

The usage of CUR as a data-analysis tool was popularized by Mahoney and Drineas \cite{mahoney2009cur}, following the development of efficient randomized algorithms for computing CUR decompositions with good approximation guarantees \cites{DMM:2008}.
The argument of Mahoney and Drineas was that experts often have a clear understanding of the actual meaning of certain columns and rows in a matrix, and this meaning is preserved by the CUR.
In contrast, SVD (or PCA) forms linear combinations of the columns or rows of the input matrix; these linear combinations can prove difficult to interpret and destroy structures such as sparsity or nonnegativity. 

\paragraph{Words of warning.}

Despite its simple definition, there are important subtleties in working with and understanding CUR decompositions.

First, CUR is unique among submatrix-oriented decompositions in that it involves taking products of submatrices of $\Ao$.
Therefore if $\Ao$ has low numerical rank and its CUR decomposition is computed to high accuracy, then all three factors ($\mtx{C}$, $\mtx{U}$, and $\mtx{R}$) will be ill-conditioned.
This can have detrimental effects on numerical behavior, particularly in the computation of $\mtx{U}$, which will behave qualitatively like inverting a matrix of low numerical rank.

Second, out of all the submatrix-oriented decompositions, CUR is of the least interest for providing exact decompositions of full-rank matrices.
For example, if $\Ao = \mtx{C}\mtx{U}\mtx{R}$ is a tall matrix of rank $n$, then we would necessarily have $\mtx{C} = \Ao\mtx{P}$ and $\mtx{U} = \mtx{P}^{\trans}\mtx{R}^{-1}$ for some permutation matrix $\mtx{P}$. 
Therefore the only real freedom in full-rank CUR of a tall matrix is in choosing the spanning set of rows that define $\mtx{R}$.
A similar conclusion applies when $\Ao$ is a wide matrix of rank $m$; an exact decomposition would necessarily have $\mtx{R} = \mtx{P}^{\trans}\Ao$ and $\mtx{U} = \mtx{C}^{-1}\mtx{P}$ for some permutation matrix $\mtx{P}$, and the only real freedom would be in choosing the spanning set of columns that define $\mtx{C}$.

The consequences of this second fact will be seen when we discuss randomized algorithms for computing CUR decompositions.
In particular, we will see that randomized algorithms for (low-rank) CUR generally \textit{do not} involve computing ``full-rank CUR decompositions'' on smaller matrices.
This will stand in contrast to randomized algorithms for (low-rank) SVD and interpolative decompositions, which usually \textit{do} involve computing the corresponding full-rank decomposition on smaller matrices.

\subsubsection{Interpolative decompositions}

Low-rank interpolative decompositions (IDs) 
come in three different flavors that share two common notes.
The first shared note of these flavors is that exactly one of the factors that define $\Aa$ is a submatrix of $\Ao$.
The second shared note concerns the factors of $\Aa$ that are not submatrices of $\Ao$.
These factors, called \textit{interpolation matrices}, are subject to certain regularity conditions.

Our crash course on ID comes in three parts.
First, we define versions of ID that only involve one interpolation matrix, so-called \textit{one-sided IDs}.
After that, we explain how accuracy guarantees of low-rank ID are affected by regularity conditions on interpolation matrices.
This explanation is important; we reference it repeatedly when we cover randomized algorithms for one-sided ID (\S \ref{subsec:CSS_CX_computational}).
We wrap up by introducing the \textit{two-sided ID}.

\paragraph{The one-sided IDs: column ID and row ID.}

In the low-rank case, a \textit{column ID} is an approximation of the form
\begin{equation}\label{eq:cid}
\begin{array}{ccccc}
\Ao & \approx & \mtx{C} & \mtx{X} \\
m\times n &   &  m\times k & k\times n
\end{array} 
\end{equation}
where $\mtx{C}$ is given by a small number of columns of $\Ao$ and $\mtx{X}$ is a wide interpolation matrix.
Full-rank column IDs can be of interest to us when $\Ao$ is very wide, in which case we have $k = m \ll n$ and $\mtx{X} = \mtx{C}^{-1}\Ao$.
Next, we consider \textit{row IDs}.
In the low-rank case, these are approximations of the form
\begin{equation}\label{eq:rid}
\begin{array}{ccccc}
\Ao & \approx & \mtx{Z} & \mtx{R}. \\
m\times n &   &  m\times k & k\times n
\end{array}
\end{equation}
where $\mtx{R}$ is given by a small number of rows of $\Ao$ and $\mtx{Z}$ is a tall interpolation matrix.
%
%
The submatrices $\mtx{C}$ and $\mtx{R}$ can be represented by ordered column and row index sets, which we denote by $J$ and $I$, that satisfy 
\[
    \mtx{C} = \mtx{A}[\fslice,J] \quad\text{and}\quad \mtx{R} = \mtx{A}[I,\fslice].
\]
These ordered index sets are called \textit{skeleton indices}.

Our definitions of row and column IDs are not complete until we specify the regularity conditions on $\mtx{X}$ and $\mtx{Z}$.
The skeleton indices $(J, I)$ play a central role here.
Most importantly, we require that the interpolation matrices satisfy
 \[
    \mtx{X}[\fslice, J] = \mtx{I}_k = \mtx{Z}[I,\fslice].
\]
These regularity conditions have two direct consequences, namely
\begin{align*}
    \mtx{A}[\fslice,J] &= \mtx{A}[\fslice,J]\mtx{X}[\fslice,J] \text{ , and} \\
    \mtx{A}[I,\fslice] &= \mtx{Z}[I,\fslice]\mtx{A}[I,\fslice],
\end{align*}
In addition to these conditions, the literature on ID typically requires that the entries of $\mtx{X}$ and/or $\mtx{Z}$ are bounded in modulus by a small constant $M$.
While we do not subscribe to this requirement in our definition of IDs, there is good motivation behind it.
We address this motivation next.


\paragraph{Quality-of-approximation in low-rank IDs.}

Suppose that $\Aa$ is a low-rank column ID of $\Ao$.
It is immediate from our definition of low-rank ID that there are at most ${n \choose k}$ possible values for the subspace $\range(\Aa)$.
In general, it can happen that none of these subspaces coincides with a dominant $k$-dimensional left singular subspace of $\Ao$.
When this happens, it will necessarily be the case that $\|\Ao - \Aa\|_2 > \sigma_{k+1}(\Ao)$, whereas a general rank-$k$ approximation could achieve an error equal to $\sigma_{k+1}(\Ao)$.
This raises a question.
\begin{quote}
    \textit{How much of a price do we pay by imposing that requirement that $\Aa$ be a low-rank ID?}
\end{quote}
The following proposition (which we prove in \cref{app:lowrank:ID} by standard techniques) gives a partial answer to this question.
The answer is notable in that it reveals the value of bounding the interpolation matrices.

\begin{proposition}\label{prop:regularity_ID_accuracy}
    Let $\mtx{\tilde{A}}$ be any rank-$k$ approximation of $\Ao$ that satisfies the spectral norm error bound $\|\Ao - \mtx{\tilde{A}}\|_2 \leq \epsilon$.
    If 
     $\mtx{\tilde{A}} = \mtx{\tilde{A}}[\fslice,J]\mtx{X}$ for some $k \times n$ matrix $\mtx{X}$ and an index vector $J$, then $\Aa = \Ao[\fslice,J]\mtx{X}$ is a low-rank column ID that satisfies
    \begin{equation}\label{eq:carryover_ID_bound}
        \|\Ao - \Aa \|_2 \leq (1 + \|\mtx{X}\|_2) \epsilon.
    \end{equation}
    Furthermore, if $|X_{ij}| \leq M$ for all $(i,j)$, then
    \begin{equation}\label{eq:interp_matrix_spectral_bound}
        \|\mtx{X}\|_2 \leq \sqrt{1 + M^2 k (n - k)}.
    \end{equation}
\end{proposition}

We note that the assumptions on the matrix $\mtx{\tilde{A}}$ from the proposition statement imply that $\mtx{\tilde{A}}[\fslice,J]$ is full column-rank.
This implies that $(\mtx{\tilde{A}}[\fslice,J])^{\dagger}(\mtx{\tilde{A}}[\fslice,J]) = \mtx{I}_k$ and hence that  $\mtx{X}$ is completely determined from the index vector $J$.
Since a rank-$k$ matrix will typically have multiple subsets of $k$ columns that span its range, we finally see that the matrix $\mtx{\tilde{A}}$ does not uniquely determine the interpolation matrix $\mtx{X}$ used in the proposition.
Indeed, the range of possibilities for the interpolation matrix is rather remarkable.
While it is known that there always exists an index set for which $\max_{ij}|X_{ij}| = 1$ \cite{Pan:2000:bound_ID_by_one}, it is NP-hard to compute this index set \cite{CM:2009:Opt-ID-is-NP-hard}.
At the same time, algorithms such as strong rank-revealing QR can be applied to $\mtx{\tilde{A}}$ with typical runtime $O(mnk)$, while ensuring $\max_{ij}|X_{ij}| \leq 2$ \cite{GE:1996}.

All-in-all, the main point of an interpolative decomposition is to provide a low-rank approximation that prominently features a submatrix of the target.
Therefore while an upper bound $M$ on the entries of an interpolation matrix gives the bound \eqref{eq:carryover_ID_bound} indirectly by way of \eqref{eq:interp_matrix_spectral_bound}, such a bound should not the end of the story.
The spectral norm $\|\mtx{X}\|_2$ is ultimately more informative for this purpose, even if it is harder to compute.

Of course, all of the points we have raised here for column IDs apply to row IDs with minor modifications.

\paragraph{Two-sided ID.}
The concept of a one-sided ID can be extended to a \textit{two-sided ID} by considering simultaneous row and column IDs.
In the low-rank case, a two-sided ID is an approximation of the form
\begin{equation}\label{eq:tsid}
\begin{array}{ccccc}
\Ao & \approx & \mtx{Z} & \mtx{A}[I,J] & \mtx{X}. \\
m\times n &   &  m\times k & k\times k & k \times n
\end{array}
\end{equation}
Methods for full-rank one-sided ID can be useful in computing low-rank two-sided IDs.
That is, if we first compute a low-rank column ID $\Aa = \mtx{C}\mtx{X}$ by some black-box method, then after obtaining a full-rank row ID of the tall matrix $\mtx{C} = \mtx{A}[:,J] = \mtx{Z}\mtx{A}[I,J]$ we have $\Aa = \mtx{Z}\mtx{A}[I,J]\mtx{X}$.

\subsubsection{On relationships between submatrix-oriented decompositions}

The landscape of methods for submatrix-oriented decompositions is very intertwined.
As we have already indicated, algorithms for one-sided ID are the main building blocks of algorithms for low-rank two-sided ID.
Later in this \nameCref{sec4:lowrank} we also mention several algorithms for CUR decomposition which depend on algorithms for one-sided ID.
In view of this, we emphasize the following point.

\begin{quote}
    For our purposes, it is helpful to introduce hierarchical relationships among submatrix-oriented decompositions.
    As such, we designate algorithms for one-sided ID as \textit{computational routines}, while methods for CUR and two-sided ID are designated as \textit{drivers}. 
\end{quote}

Next, let us point out a sense in which CUR and two-sided ID are ``dual'' to one another.
Both are submatrix-oriented decompositions that have three factors.
For CUR, the outer factors are submatrices of $\Ao$ and no requirements are placed on the inner factor (the linking matrix).
In particular, if we specify the outer factors by ordered index sets $J$ and $I$, then a CUR decomposition is expressed as
\begin{equation}\label{eq:cur_otherNotation_xxx1}
\begin{array}{ccccc}
\Ao & \approx & \mtx{A}[\fslice,J] & \mtx{U} & \mtx{A}[I,\fslice]. \\
m\times n &   &  m\times k & k\times k & k\times n
\end{array} 
\end{equation}
This can be contrasted with two-sided ID as defined in \eqref{eq:tsid}.
There, the inner factor is a submatrix of $\Ao$, and moderate requirements are placed on the outer factors (the interpolation matrices).

The properties of ``linking matrices'' and ``interpolation matrices'' are different enough to warrant their different names. 
The problem of computing a low-rank approximation via two-sided ID can be better numerically behaved, compared to low-rank approximation by CUR \cite[\S 13]{MT:2020}.
However, CUR offers far greater potential for storage reduction when dealing with sparse matrices.
The differences between two-sided ID and CUR can become less pronounced when one considers methods for losslessly converting one such representation to the other, as we mention in \cref{subsubsec:TSID_CUR_driver}.

\subsection{On accuracy metrics}

The problems from \cref{sec3:LS_and_optim} were mostly unconstrained minimization of convex quadratics.
Such problems are very nice, since the gradient of the quadratic loss function constitutes a canonical error metric that can be driven to zero.
Low-rank approximation problems can likewise be framed as optimization problems.
However, these formulations either involve constraints or a nonconvex objective function.
This distinction is important, since these structures rule out checking for a zero gradient as a cheap optimality condition.

The main error metrics in low-rank approximation are norm-induced distances.
For reasons that we give under the next two headings it is not appropriate to consider distances from a computed approximation $\Aa$ to some nominally ``optimal'' approximation.
Instead, one measures the distance from the approximation to the target, most often in the spectral or Frobenius norms.
%
%

\subsubsection{Distance to optimal approximations}

\paragraph{Non-unique solutions and sensitivity to perturbations.}
Recall from \cref{subsubsec:class_spectral} how truncating $\Ao$'s SVD at rank $k$ gives an optimal rank-$k$ approximation in any unitarily invariant norm.
Unfortunately, this truncation will be non-unique when $\Ao$ has more than one singular value equal to $\sigma_k$.
This is easiest to see when $\Ao = \mtx{I}$ is the identity matrix, in which case every diagonal $\{0,1\}$-matrix of rank $k$ is an optimal rank-$k$ approximation to $\Ao$.

More generally, if $\Ao$ has multiple singular values that are \textit{close to} $\sigma_k$, then extremely small perturbations to $\Ao$ can result in large changes to the singular vectors corresponding to these singular values; see \cite[\S 6 -- \S 8]{Bhatia:1997:MatrixAnalysis} for details.
This has a secondary complication: it is harder to estimate the dominant $k$ singular vectors of a matrix than it is to find a rank-$k$ approximation that is ``near optimal'' in the sense of \eqref{eq:svd_as_opt}.

\paragraph{Intractability of computing optimal approximations.} 
When working with submatrix-oriented decompositions, we do not even have the luxury of defining ``optimal'' approximations in the manner of truncated SVDs.
Indeed, the problem of finding an ``optimal'' ID necessitates specifying any regularity conditions such as the bound $M$ in a constraint $|X_{ij}| \leq M$.
As we mentioned before, even when $\Aa$ has exact rank $k$, a rank-$k$ column ID with $M=1$ always exists but is NP-hard to find \cite{CM:2009:Opt-ID-is-NP-hard}.

Going to another extreme, we could set aside the matter of $M$ and simply set $\mtx{X} = \mtx{C}^\dagger\Ao$ for a matrix $\mtx{C}$ containing $k$ columns of $\Ao$.
In this case it is not known if the columns can be chosen to minimize Frobenius- or spectral-norm error $\|\Aa - \Ao\|$ in time less than $O(n^k)$.
Still, there are theoretical guarantees for approximation quality by CUR relative to approximation quality achievable by SVD. 
We refer the reader to \cites{DMM:2008,BMD09_CSSP_SODA}, \cite[\S 1 - \S 2]{VM:2016:CUR}, and \cref{app:lowrank:ID} for more information about CUR and ID in our context.

\subsubsection{Distance relative to that of a reference approximation}

It is problematic to use a distance from $\Aa$ to $\Ao$ as an error metric for $\Aa$.
This is because there are situations when any such distance will be large even when $\Aa$ is close to an ``optimal'' approximation.
The simplest example of this is PCA, in which cases the approximation rank is $O(1)$, independent of the dimensions of the matrix or properties of its spectrum.
More generally, it can be hard to obtain a low-rank approximation that is very close to $\Ao$ when $\mtx{A}$ has \textit{slow spectral decay}, in the sense that the distribution of its singular values has a heavy tail.
Accurate approximations can also be hard to come by if the factors of $\Aa$ are highly constrained.

One handles this situation by considering the distance between $\Ao$ and $\Aa$ \textit{relative to} that between $\Ao$ and some reference matrix $\Ao_r$.
Formally, we concern ourselves with the smallest value of $\epsilon$ needed to achieve
\[
    \|\Ao - \Aa\| \leq (1 + \epsilon)\|\Ao - \Ao_r\|.
\]
The reference matrices $\Ao_r$ used in \RandNLA{} theory are not available to us when performing computations.
In fact, they are usually not optimal for the formal low-rank approximation problem at hand.
The most common source of non-optimality is that the reference is subject to a more stringent rank constraint: $\rank(\Ao_r) < \rank(\Aa) \leq \rank(\Ao)$.
Another source of non-optimality is that it may not be possible to decompose the reference into factors with the required~structure (e.g., the structure required by low-rank CUR).
For example, an approximation of $\Ao$ obtained by a rank-$k$ truncated SVD cannot (in general) be converted into a rank-$k$ CUR decomposition using submatrices of $\Ao$.

%

\section{Drivers}
\label{subsec:lowrank_drivers}


There exist many randomized algorithms for computing low-rank approximations of matrices.
This \nameCref{subsec:lowrank_drivers} focuses on low-rank approximation algorithms that take the two-stage approach popularized by \cite{HMT:2011}, because this approach has been demonstrated to be efficient and highly reliable over the years.
The high-level idea of the two-stage approach is the following: first one constructs a ``simple'' representation of $\Aa$ with the aid of randomization, and then one deterministically converts that representation of $\Aa$ into a more useful form.

In order to discuss these drivers for low-rank approximation, it is necessary to mention briefly the following two concepts (these are handled by computational routines, to be discussed in detail in \cref{subsec:lowrank_subroutines}):
%
%
%

\begin{itemize}
     \item 
     A \textit{QB decomposition} is a simple representation that is useful for SVD and eigendecomposition.
     The representation takes the form $\Aa = \mtx{Q}\mtx{B}$ for a tall matrix $\mtx{Q}$ with orthonormal columns and $\mtx{B} = \mtx{Q}^{\trans}\Ao$.
     The important point here is that the QB decomposition involves explicit construction of and access to both $\mtx{Q}$ and $\mtx{B}$.      
     We discuss QB algorithms in \cref{subsubsec:qb_alg}.
     \item 
     \textit{Column subset selection} (CSS) is the problem of selecting from a matrix a set of columns that is ``good'' in some sense. 
     CSS algorithms largely characterize algorithms for \emph{one-sided ID}.
     We discuss methods for these two problems in \cref{subsec:CSS_CX_computational}.
     They are important here because a one-sided ID can be used for the simple representation of $\Aa$ when working toward an SVD, eigendecomposition, two-sided ID, or CUR decomposition.
\end{itemize}

\subsection{Methods for SVD}\label{subsubsec:svd_algs}

The are several families of randomized algorithms for computing low-rank SVDs.
Here we describe a few families that give $\Aa = \mtx{U}\mtx{\Sigma}\mtx{V}^{\trans}$ through its compact SVD.

\subsubsection{A flexible method}

We begin with \cref{alg:svd}.
This algorithm uses randomization to compute a QB decomposition of $\Ao$, then deterministically computes $\mtx{Q}\mtx{B}$'s compact SVD, and finally truncates that SVD to a specified rank.

This algorithm assumes that the \code{QBDecomposer} is iterative in nature.
It assumes that each iteration adds some number of columns to $\mtx{Q}$ and rows to $\mtx{B}$, and that the algorithm can terminate once an implementation-dependent error metric for $\mtx{Q}\mtx{B} \approx \Ao$ falls below $\epsilon$ or once $\mtx{Q}\mtx{B}$ reaches a rank limit.
Here we have set the rank limit to $k + s$ where $s$ is a nonnegative ``oversampling parameter.''

\begin{algorithm}[H]
    \setstretch{1.0}
    \caption{\code{SVD1} : QB-backed low-rank SVD (see \cite{HMT:2011} and \cite{RST:2009:RandPCA}) }\label{alg:svd}
    \begin{algorithmic}[1]
        \State \textbf{function} $\code{SVD1}(\Ao, k, \epsilon, s)$\vspace{0.5pt} 
        \Indent
            \Statex \quad Inputs:
            \Statex \begin{quote}
                $\Ao$ is an $m \times n$ matrix. The returned approximation will have rank \textit{at most} $k$. The approximation produced by the randomized phase of the algorithm will attempt to $\Ao$ to within $\epsilon$ error, but will not produce an approximation of rank greater than $k + s$.
            \end{quote}\vspace{4pt}
            \Statex \quad Output:
            \Statex \qquad The compact SVD of a low-rank approximation of $\Ao$.\vspace{4pt}
            \Statex \quad Abstract subroutines:
            \Statex \begin{quote}
                $\code{QBDecomposer}$ generates a QB decomposition of a given matrix; it tries to reach a prescribed error tolerance but may stop early if it reaches a prescribed rank limit.
            \end{quote}\vspace{4pt}
            \setstretch{1.1}
            \State $\mtx{Q}, \mtx{B} = \code{QBDecomposer}(\Ao,k + s, \epsilon)$ \codecomment{$\mtx{Q}\mtx{B} \approx \Ao$}
            \State $r = \min\{k,  \text{ number of columns in } \mtx{Q}\}$
            \State $\mtx{U},\mtx{\Sigma},\mtx{V}^{\trans} = \code{svd}(\mtx{B})$
            \State $\mtx{U} = \mtx{U}[:\ ,\ :r]$
            \State $\mtx{V} = \mtx{V}[:\ ,\ :r]$
            \State $\mtx{\Sigma} = \mtx{\Sigma}[:r\ ,\ :r]$
            \State $\mtx{U} = \mtx{Q}\mtx{U}$
            \State \textbf{return} $\mtx{U},\mtx{\Sigma},\mtx{V}^{\trans}$
        \EndIndent 
    \end{algorithmic}
\end{algorithm}

The literature recommends setting $s$ to a small positive number (e.g., $s = 5$ or $s = 10$) to account for the fact that the trailing singular vectors of $\mtx{Q}\mtx{B}$ may not be good estimates for the corresponding singular vectors of $\Ao$.
However, using any positive oversampling parameter complicates the interpretation of the error tolerance $\epsilon$.
If a user deems this problematic then they can simply set $k \leftarrow k + s$ and $s \leftarrow 0$.
Such an approach can be reasonable if tuning parameters for the QB algorithm are chosen appropriately.
Specifically, if techniques such as power iteration are used (see \cref{subsubsec:data_aware}) then the trailing singular vectors of $\mtx{Q}\mtx{B}$ can be reasonably good approximations to the corresponding singular vectors of $\Ao$.

\subsubsection{Sacrificing accuracy for speed}

\paragraph{Converting from an ID.}
Setting our sights beyond \cref{alg:svd}, it is noteworthy that if $\Aa$ is given in \textit{any} compact representation then it can be losslessly converted into an SVD without ever accessing $\Ao$.
For example, given a column ID $\Aa = \mtx{C}\mtx{X}$, we would compute a QR decomposition $\mtx{C} = \mtx{Q}\mtx{R}$, set $\mtx{B} = \mtx{R}\mtx{X}$, and then proceed with $(\mtx{Q},\mtx{B})$ as in \cref{alg:svd}.
As another example, conversion from a row ID to an SVD is illustrated implicitly in \cite[Algorithm 5.2]{HMT:2011}.

Such approaches are potentially useful because one-sided ID can easily be implemented in a way that accesses $\Ao$ with a single matrix-matrix multiplication and then selects a row or column submatrix.
However, this comes at a cost of a much less accurate solution compared to typical QB methods.

\paragraph{Single-pass algorithms.}
For very large problems the main measure of an algorithm's complexity is the number of times it moves $\Ao$ through fast memory.
Besides the above ID-based method, there are three algorithms for low-rank SVD which move $\Ao$ through fast memory only once.
Each of them uses multi-sketching in the sense of \cref{subsec:multi_sketch}.
The first and second options simply use \cref{alg:svd}, but specifically with single-pass QB methods based on type 1 or type 2 multi-sketching.
Discussion of such QB algorithms is deferred to \cref{subsubsec:qb_alg}.
The third option is the algorithm described in \cite[\S 7.3.2]{TYUC:2017:singlepass}, which relies on type 3 multi-sketching.

\begin{remark}
     Algorithms designed to minimize the number of views of a matrix are usually analyzed in the \emph{pass efficient model} for algorithm complexity \cite{dkm_matrix1}.
     Early work on randomized pass-efficient and single-pass algorithms can be found in \cite{FKV04,dkm_matrix1,dkm_matrix2}.
\end{remark}

\paragraph{Error estimation in sketched one-sided SVD.}

Single-pass algorithms are unlikely to produce highly accurate approximations of singular vectors or singular values.
However, their results may be accurate enough to be useful in certain applications.
This motivates methods for estimating the errors of approximations returned by these algorithms.
\cref{subapp:bootstrap:svd} provides a bootstrap-based error estimator for a simple randomized algorithm that recovers approximate singular vectors from one side of the matrix.

\subsection{Methods for Hermitian eigendecomposition}
\label{subsubsec:herm_eig_algs}  

Each randomized algorithm for low-rank SVD has a corresponding version that is specialized to Hermitian matrices.
We recount those specialized algorithms here, and we mention an additional algorithm that is unique to the approximation of psd matrices.
In general, we shall say that $\Ao$ is $n \times n$ and that the algorithms represent $\Aa = \mtx{V}\diag(\vct{\lambda})\mtx{V}^{\trans}$, where $\mtx{V}$ is a tall column-orthonormal matrix and $\vct{\lambda}$ is a vector with entries sorted in decreasing order of absolute value.

\subsubsection{A flexible method for Hermitian indefinite matrices}

\cref{alg:evd1} is a variation of \cite[Algorithm 5.3]{HMT:2011}.
Its parameters $(k, \epsilon, s)$ have essentially the same interpretations as \cref{alg:svd}.
When $s = 0$, its output is simply is a compact eigendecomposition of $\mtx{Q}\mtx{C}\mtx{Q}^{\trans}$, where $\mtx{C} = \mtx{Q}^{\trans}\Ao\mtx{Q}$ and $\mtx{Q}$ is obtained from a black-box \code{QBDecomposer}.
The main difference between this method and \cref{alg:svd} is that $\epsilon$ is scaled down by a factor $1/2$ before being passed to \code{QBDecomposer}.
This change is needed to so that if $s = 0$ and $\|\mtx{Q}\mtx{B} - \Ao\| \leq \epsilon$ then the final approximation satisfies $\| \Aa - \Ao\| \leq \epsilon$; see \cite[\S 5.3]{HMT:2011}.
 
\begin{algorithm}[htb]
    \setstretch{1.0}
    \caption{\code{EVD1} : QB-backed low-rank eigendecomposition; see \cite{HMT:2011}}\label{alg:evd1}
    \begin{algorithmic}[1]
        \State \textbf{function} $\code{EVD1}(\Ao, k, \epsilon, s)$ \vspace{0.5pt}
        \Indent
            \Statex \quad Inputs:
            \Statex \begin{quote}
                $\Ao$ is an $n \times n$ Hermitian matrix.
                The returned approximation will have rank \textit{at most} $k$.
                The approximation produced by the randomized phase of the algorithm will attempt to $\Ao$ to within $\epsilon$ error, but will not produce an approximation of rank greater than $k + s$.
            \end{quote}\vspace{4pt}
            \Statex \quad Output:
            \Statex \qquad Approximations of the dominant eigenvectors and eigenvalues of $\Ao$.\vspace{4pt}
            \Statex  \quad Abstract subroutines:
            \Statex \begin{quote}
                $\code{QBDecomposer}$ generates a QB decomposition of a given matrix; it tries to reach a prescribed error tolerance but may stop early if it reaches a prescribed rank limit.
            \end{quote}\vspace{4pt}
            \setstretch{1.1}
            \State $\mtx{Q}, \mtx{B} = \code{QBDecomposer}(\Ao,\, k + s,\, \epsilon / 2)$ 
            \State $\mtx{C} = \mtx{B}\mtx{Q}$  ~\codecomment{since $\mtx{B} = \mtx{Q}^{\trans}\Ao$, we have $\mtx{C} = \mtx{Q}^{\trans}\Ao\mtx{Q}$}
            \State $\mtx{U},\vct{\lambda} = \code{eigh}(\mtx{C})$ \codecomment{full Hermitian eigendecomposition}
            \State $r = \min\{k, \text{ number of entries in } \vct{\lambda}\}$
            \State $P = \operatorname{argsort}(|\vct{\lambda}|)[:r]$ 
            \State $\mtx{U} = \mtx{U}[{:}\,,P]$
            \State $\vct{\lambda} = \vct{\lambda}[P]$
            \State $\mtx{V} = \mtx{Q}\mtx{U}$
            \State \textbf{return} $\mtx{V},\vct{\lambda}$
        \EndIndent 
    \end{algorithmic}
\end{algorithm}

\subsubsection{Sacrificing accuracy for speed with Hermitian indefinite matrices}

\paragraph{Converting from an ID.}
\cite[Algorithm 5.4]{HMT:2011} is a second approach to Hermitian eigendecomposition, based on postprocessing a low-rank row ID of $\Ao$.
We do not include pseudocode for this algorithm in this monograph.
However, the basic observation underlying the approach is that one can use the symmetry of $\Ao$ to canonically approximate an initial row ID $\mtx{\tilde{A}} = \mtx{Z}\Ao[I,:] \approx \Ao$ by the Hermitian matrix $\doublehat{\mtx{A}} = \mtx{Z}\Ao[I,I]\mtx{Z}^{\trans}$.
The compact representation of this Hermitian matrix makes it easy to compute its eigendecomposition by a lossless process.

When should one use \cite[Algorithm 5.4]{HMT:2011} over \cref{alg:evd1}?
Our answer is the same as for using \cite[Algorithm 5.2]{HMT:2011} over \cref{alg:svd}.
That is, the ID approach is only of interest when it moves $\Ao$ through fast memory \textit{once}, and it should be considered alongside other low-rank eigendecomposition algorithms with similar data movement patterns.

\paragraph{Single-pass algorithms.}
Just as with fast algorithms for low-rank SVD, one can obtain fast algorithms for low-rank Hermitian eigendecomposition by using \cref{alg:evd1} with the QB methods based on type 1 or type 2 multi-sketching.

Besides those approaches, we make note of \cite[Algorithm 5.6]{HMT:2011}, which accesses $\Ao$ \textit{exclusively} through a single sketch $\mtx{Y} = \Ao\mtx{S}$ and makes no assumptions on the representation of $\Ao$ in-memory.
This access pattern is possible because the algorithm solves a least squares problem involving $\mtx{Y}$, $\code{orth}(\mtx{Y})$, and $\mtx{S}$ to project a small $k \times k$ ``core matrix'' onto the set of Hermitian matrices.

\FloatBarrier

\subsubsection{\Nystrom{} approximations for positive semidefinite matrices}


Now we suppose that the $n \times n$ matrix $\Ao$ is psd, in which case we can define the \textit{Nystr\"{o}m approximation of $\Ao$ with respect to a matrix $\mtx{X}$} as
\begin{equation}\label{eq:nystrom}
    \Aa = (\Ao\mtx{X})\left(\mtx{X}^{\trans}\Ao\mtx{X}\right)^{\dagger}(\Ao\mtx{X})^{\trans}.
\end{equation}
When framed this way, the Nystr\"{o}m approximation is defined for any matrix $\mtx{X}$ with $k \leq n$ columns.
Indeed, it does not even presume that $\mtx{X}$ is random.
However, in \RandNLA{}, we ultimately set $\mtx{X}$ to a sketching operator and produce a compact spectral decomposition $\Aa = \mtx{V}\diag(\vct{\lambda})\mtx{V}^{\trans}$.
For any given type of sketching operator, low-rank approximation of psd matrices by \Nystrom{} approximations tend to be more accurate than approximations produced by comparable algorithms for general Hermitian eigendecomposition.

\paragraph{What's in a name? Disambiguating ``\Nystrom{} approximation.''}
The literature on randomized algorithms for \Nystrom{} approximation \textit{heavily} emphasizes using column selection operators \cites{WS:2000:KRR,dm_kernel_JRNL,Pla05_UAI,KMT09b,KMT09c,LKL10,Bac13,GM15_NYSTROM_JRNL,RCR15_TR,DKM20_CSSP_neurips}.
This stems from an analogy between sampling columns of kernel matrices in machine learning and the \Nystrom{} method from integral equation theory.
See \cite[\S 5]{dm_kernel_JRNL} for a detailed discussion.
In view of the prominence of column sampling in \Nystrom{} approximations, part of \cref{sec7:lev_scores} is dedicated to specialized methods for sampling columns from psd matrices.

When $\mtx{X}$ is a general sketching operator, the approximation \eqref{eq:nystrom} has been called a ``projection-based SPSD approximation'' \cite{GM15_NYSTROM_JRNL}.
The different terminology for general sketching operators $\mtx{X}$ is helpful for distinguishing the resulting approximations from those referred to as ``\Nystrom{} approximations'' in the machine learning literature.
However, it is not in line with our philosophy that \RandNLA{} concepts should be described with minimal assumptions on the nature of the sketching distribution.
This philosophy, first advocated for by Drineas and Mahoney~\cites{DMMS07_FastL2_NM10,MD16_chapter}, leads us to adopt the following convention.
\begin{quote}
    \textit{The term ``\Nystrom{} approximation'' shall be used for any approximation of the form \eqref{eq:nystrom}, even when $\mtx{X}$ is a general sketching operator.}
\end{quote}
We note that this also follows the convention used by El~Alaoui and Mahoney in their work on kernel ridge regression (see \cite[Theorem 1]{AM:2015:KRR}).

\paragraph{Algorithms.}
\cref{alg:evd2} is a practical approach to low-rank eigendecomposition by \Nystrom{} approximation.
It includes a function call $\code{TallSketchOpGen}(\mtx{A},k + s)$ which returns a sketching operator $\mtx{S}$ with $n$ rows and $k + s$ columns.
Our reason for specifying a sketching operator in this way is to provide flexibility in whether the sketching operator is data-oblivious or data-aware.
In this context, the main type of data-aware sketching operator would be based on so-called \textit{power iteration}; see \cref{subsubsec:data_aware} and \cite{GM15_NYSTROM_JRNL}.

\begin{algorithm}[htb]
\setstretch{1.0}
\caption{\code{EVD2} : for psd matrices only; adapts \cite[Algorithm 3]{TYUC:2017a}}\label{alg:evd2}
\begin{algorithmic}[1]
\State \textbf{function} $\code{EVD2}(\Ao,k,s)$ \vspace{0.5pt}
        \Indent
            \Statex \quad Inputs:
            \Statex \begin{quote}
                $\Ao$ is an $n \times n$ psd matrix.
                The returned approximation will have rank at most $k$, but the sketching operator used in the algorithm can have rank as high as $k + s$.
            \end{quote}\vspace{4pt}
            \Statex \quad Output:
            \Statex \qquad Approximations of the dominant eigenvectors and eigenvalues of $\Ao$.\vspace{4pt}
            \Statex  \quad Abstract subroutines:
            \Statex \begin{quote}
                $\code{TallSketchOpGen}$ generates a sketching operator with a prescribed number of columns, for use in sketching a given matrix from the right.
            \end{quote}\vspace{4pt}
            \setstretch{1.1}
    \State $\mtx{S} = \code{TallSketchOpGen}(\Ao, k+s)$
    \State $\mtx{Y} = \Ao\mtx{S}$
    \State $\nu=\sqrt{n}\cdot \epsilon_{\text{mach}} \cdot \|\mtx{Y}\|$  \codecomment{$\epsilon_{\text{mach}}$ is machine epsilon for current numeric type}
    \State $\mtx{Y} =\mtx{Y}+\nu\mtx{S}$ ~~~~\,\codecomment{ regularize for numerical stability}
    \State $\mtx{R} = \code{chol}(\mtx{S}^{\trans}\mtx{Y})$  \codecomment{ $\mtx{R}$ is upper-triangular and $\mtx{R}^{\trans}\mtx{R} = \mtx{S}^{\trans}\mtx{Y} = \mtx{S}^{\trans}(\Ao + \nu\mtx{I})\mtx{S}$}
    \State $\mtx{B} = \mtx{Y} (\mtx{R}^{\trans})^{-1}$ ~~\,\codecomment{ $\mtx{B}$ has $n$ rows and $k + s$ columns}
    \State $\mtx{V},\mtx{\Sigma},\mtx{W}^{\trans} = \code{svd}(\mtx{B})$ \codecomment{can discard $\mtx{W}$}
    \State $\vct{\lambda} = \diag(\mtx{\Sigma}^2)$ ~~\,\codecomment{extract the diagonal}
    \State $r = \min\{ k, \text{ number of entries in } \vct{\lambda} \text{ that are greater than } \nu \}$
    \State $\vct{\lambda} = \vct{\lambda}[\lslice{r}] - \nu$ ~~ \codecomment{undo regularization}
    \State $\mtx{V} = \mtx{V}[\fslice{} ,\lslice{r}]$.
    \State \textbf{return} $\mtx{V},\vct{\lambda}$
\EndIndent 
\end{algorithmic}
\end{algorithm}

The role of the parameter $s$ in \cref{alg:evd2} is analogous to that in \cref{alg:svd,alg:evd1} -- the algorithm effectively computes data needed for a rank $k + s$ eigendecomposition before truncating that approximation to rank $k$.
However, unlike \cref{alg:svd,alg:evd1}, \cref{alg:evd2} has no control of approximation error.
We refer the reader to \cite[Algorithm E.2]{FTU:2021:NystromRidge} for a more sophisticated version of this algorithm which can accept an error tolerance.

\subsection{Methods for CUR and two-sided ID}
\label{subsubsec:TSID_CUR_driver}

Here we describe two approaches to CUR and one approach to two-sided ID.
The descriptions are largely qualitative in that they are stated in terms of algorithms for low-rank column ID and CSS (which are detailed in \cref{subsec:CSS_CX_computational}).

\subsubsection{CUR by falling back on CSS}

Perhaps the simplest approach to CUR computes the row and column indices $(I, J)$ in one stage and then computes the linking matrix $\mtx{U}$ in a second stage.
The column indices $J$ are obtained by a randomized algorithm for CSS on $\Ao$, then the row indices $I$ are obtained by some CSS algorithm on $\mtx{C}^{\trans} = \Ao[:,J]^{\trans}$.\footnote{Of course, this process could be reversed to compute $I$ and then $J$.}
Because the matrix $\mtx{C}$ is so much smaller than $\Ao$, it is often practical to use a deterministic algorithm when performing CSS on $\mtx{C}^{\trans}$. 

There are two canonical choices for the linking matrix in this context: one obtained by projection
\[
    \mtx{U}_{\text{proj}} = \left(\Ao[:,J]\right)^\dagger \Ao \left(\Ao[I,:]\right)^\dagger
\]
and one obtained by submatrix inversion
\[
    \mtx{U}_{\text{sub}} = \Ao[I,J]^{\dagger}.
\]
It should be clear that the approximation error incurred by using the former matrix will never be larger than when using the latter.
Furthermore, the process of computing the former matrix is better conditioned than the process of computing the latter.
Therefore it is generally preferable to use $\mtx{U}_{\text{proj}}$ as the linking matrix when implementing CUR via randomized CSS.

\begin{remark}
    Randomized algorithms for CUR based on the pattern above were first proposed in \cites{DMM:2008,BMD09_CSSP_SODA}, particularly with linking matrices closer to the form $\mtx{U}_{\text{sub}}$.
    Deterministic analyses of CUR approximation quality with various linking matrices can be found in \cites{GZT:1995:CUR,GTZ:1997:CUR}.    
\end{remark}

\subsubsection{CUR by a combination of column ID and CSS}
Suppose we have access to data $(\mtx{X}, J)$ from a column ID of an initial low-rank approximation of $\Ao$.
Given this data, we can recover the row index set $I$ and $\mtx{U}$ for a CUR decomposition by running CSS on $\mtx{C}^* = \Ao[:,J]^{\trans}$ and setting $\mtx{U} = \mtx{X}\left(\Ao[I,:]\right)^\dagger$.

This approach only requires the application of one pseudo-inverse, which compares favorably to the two applications of pseudo-inverses needed to compute $\mtx{U}_{\text{proj}}$ in the first approach to CUR.
At the same time, if the randomized algorithm for computing $(\mtx{X},J)$ happens to return an interpolation matrix satisfying $\mtx{X} = \mtx{C}^\dagger\Ao$, then the resulting decomposition could have been obtained by the elementary CUR algorithm with linking matrix $\mtx{U}_{\text{proj}}$.
Therefore there is a sense in which this template algorithm generalizes the elementary approach to CUR.

We instantiate this template algorithm in \cref{alg:curd1}.
The CSS step of the algorithm calls a deterministic function for computing a QR decomposition with column pivoting, with the semantics indicated in \cref{table:notation}.
Whether the algorithm starts with a row ID or column ID depends on the aspect ratio of the data matrix; \cite[\S 13.3]{MT:2020} recommend starting with a row ID when $\mtx{A}$ is wide in the related context of computing two-sided IDs.

\begin{algorithm}[htb]
    \setstretch{1.0}
    \caption{\code{CURD1} : CUR by randomizing an initial ID \cites{VM:2016:CUR,DM:2021:CUR}}\label{alg:curd1}
    \begin{algorithmic}[1]
        \State \textbf{function} $\code{CURD1}(\mtx{A}, k, s)$ \vspace{0.5pt}
        \Indent
            \Statex \quad Inputs:
            \Statex \begin{quote}
                $\Ao$ is an $m \times n$ matrix.
                The returned approximation will have rank at most $k$. The \code{ColumnID} abstract subroutine can use sketching operators of rank up to $k + s$ in its internal calculations.
            \end{quote}\vspace{4pt}
            \Statex \quad Output:
            \Statex \qquad A low-rank CUR decomposition of $\Ao$.\vspace{4pt}
            \Statex  \quad Abstract subroutines:
            \Statex \begin{quote}
                $\code{ColumnID}$ produces a low-rank column ID of a given matrix, up to some specified rank.
            \end{quote}\vspace{4pt}
            \setstretch{1.1}
            \If{$m \geq n$}
                \State $\mtx{X}, J = \code{ColumnID}(\mtx{A},\, k,\, s)$\, \codecomment{$|J|=k$ and $\Ao[:,J]\mtx{X} \approx \Ao$}
                \State $\mtx{Q}, \mtx{T}, I = \code{qrcp}(\Ao[:, J]^{\trans})$ ~~\,\,\codecomment{only care about the indices $I$}
                \State $I = I[:k]$
                \State $\mtx{U} = \mtx{X}\left(\Ao[I,:]\right)^{\dagger}$
            \Else
            \State $\mtx{Z}^{\trans}, I = \code{ColumnID}(\Ao^{\trans},\, k, \,s)$ \,\codecomment{$|I| = k$ and $\mtx{Z}\Ao[I,:] \approx \Ao$.}
            \State $\mtx{Q},\mtx{T},J = \code{qrcp}(\Ao[I,:])$ \quad ~~~\,\codecomment{only care about the indices $J$}
            \State $J = J[:k]$
            \State $\mtx{U} = \left(\Ao[:, J]\right)^\dagger \mtx{Z}$
            \EndIf
        \State \textbf{return} $J, \mtx{U}, I$
        \EndIndent 
    \end{algorithmic}
\end{algorithm}

\subsubsection{Two-sided ID via one-sided ID}

Two-sided IDs are canonically computed by a simple reduction to one-sided ID: first obtain $(\mtx{X},J)$ by a column ID of $\Ao$ and then obtain $(I,\mtx{Z})$ by a row ID of $\Ao[:,J]$.
The initial column ID of $\Ao$ will be computed by a randomized algorithm and hence will always be low-rank.
However, it is not expensive to compute a \textit{full-rank} row ID $\Ao[:,J] = \mtx{Z}\Ao[I,J]$ by a deterministic method under the standard assumption that  $|J| \ll \min\{m, n\}$.
Such an approach is described in \cite[\S 2.4, \S 4]{VM:2016:CUR}.

Finally, we note that a two-sided ID can be naturally repurposed for CUR decomposition by either of the two qualitative approaches to CUR described above.
In the first case one only needs the index sets $(I, J)$ and computes the linking matrix by any desired method.
In the second case one needs the index sets \textit{and} one of the interpolation matrices (i.e., one of $\mtx{Z}$ or $\mtx{X}$).
The latter approach for converting from two-sided ID to CUR is used in \cite[\S 3 and \S 4]{VM:2016:CUR}.
General discussion on converting from two-sided ID to CUR can be found in \cite[\S 11.2]{Martinsson:2018_ish}, \cite[\S 13.2]{MT:2020}, and \cite[\S 4.1]{DM:2021:CUR}.

\FloatBarrier

\section{Computational routines}\label{subsec:lowrank_subroutines}

The last \nameCref{subsec:lowrank_drivers} explained how randomized algorithms for low-rank approximation exhibit a great deal of modularity. 
Here we summarize the design spaces for the constituent modules.

\cref{subsubsec:data_aware,subsubsec:qb_alg,subsec:CSS_CX_computational,subsec:column_pivoted} cumulatively cover QB, column ID, CSS, and building blocks for the same.
We acknowledge up-front that we treat column ID and CSS in far detail than we do QB.
This imbalance is not a statement about the importance of column ID or CSS over QB.
Rather, it stems from our desire to clarify broader concepts surrounding column ID that can be difficult to tease out from other literature.

From there, \cref{subsubsec:norm_rank_est} lists methods for norm estimation which are important for solving low-rank approximation problems to fixed accuracy.
Our last topic in the realm of computational routines for low-rank approximation is the notion of low-rank approximations from \textit{oblique projections} (\S \ref{subsubsec:oblique_proj}).
This framework motivates a type of low-rank approximation that is cheap to compute but that does not have meaningfully-structured factors.

We emphasize that this \nameCref{subsec:lowrank_subroutines} mentions \textit{many} algorithms for a wide variety of problems.
Due to practical constraints we only address a handful of these algorithms in detail.
Pseudocode for select algorithms can be found in \cref{app:lowrank}; these algorithms have been selected based on some combination of their conceptual significance and their practicality.

\subsection{Power iteration}\label{subsubsec:data_aware}

Given a matrix $\Ao$, suppose we sketch $\mtx{Y} = \Ao\mtx{S}$ using a very tall sketching operator $\mtx{S}$.
In a low-rank approximation context -- regardless of whether we work with spectral or submatrix-oriented decompositions -- it is generally preferable for $\range(\mtx{Y})$ to be well-aligned with the span of $\Ao$'s dominant left singular vectors.
This, in turn, is facilitated by having $\range(\mtx{S})$ be well-aligned with the span of $\Ao$'s dominant \textit{right} singular vectors.
To accomplish this, \RandNLA{} libraries should include methods for generating such sketching operators based on power iteration.

A basic approach to power iteration makes alternating applications of $\Ao$ and $\Ao^{\trans}$ to an initial data-oblivious sketching operator $\mtx{S}_o$, to obtain a data-aware sketching operators such as 
\begin{equation}\label{eq:power_iter_skop}
\mtx{S} = \left(\Ao^{\trans}\Ao\right)^{q}\mtx{S}_o 
\quad\mbox{or}\quad
\mtx{S} = \left(\Ao^{\trans}\Ao\right)^{q}\Ao^{\trans}\mtx{S}_o ,
\end{equation}
for some parameter $q \geq 0$.
Practical implementations need to incorporate some form of stabilization in between the successive applications of $\Ao$ and $\Ao^{\trans})$.
We give a general formulation of such a method in \cref{app:subsub:power_iter} with \cref{alg:TSOG1}. 
Notably, this general method allows an arbitrary number of passes over the data matrix (including zero passes, as an API convenience).

\begin{remark}
    The closest relative to \cref{alg:TSOG1} in the literature is probably \cite[Algorithm 3.3]{ZM:2020:LU}.
    However, the core idea behind this algorithm was explored earlier by Bjarkason \cite{Bjarkason:2019}.
\end{remark}

\subsection{Orthogonal projections: QB and rangefinders}
\label{subsubsec:qb_alg}

We begin with two definitions.

\begin{quote}
    Given a matrix $\Ao$, a \textit{QB decomposition} is given by a pair of matrices $(\mtx{Q},\mtx{B})$ where $\mtx{Q}$ is column-orthonormal and $\mtx{B} = \mtx{Q}^{\trans}\Ao$.
    It is intended that $\mtx{Q}\mtx{B}$ serve as an approximation of $\Ao$.
    An algorithm that computes only the factor $\mtx{Q}$ from a QB decomposition is called a \textit{rangefinder}.
\end{quote}

The value of QB decompositions stems from how they define approximations by orthogonal projection: $\Aa = \mtx{Q}\mtx{B} = \mtx{Q}\mtx{Q}^*\Ao$.
It is important to note that QB algorithms \textit{do not necessarily} first compute $\mtx{Q}$ in one phase and then set $\mtx{B} = \mtx{Q}^{\trans}\Ao$ in a second phase.
Indeed, the benefit of the rangefinder abstraction is that it considers an equivalent problem while setting aside the potentially-nuanced matter of computing $\mtx{B}$.

Before proceeding to algorithms, we note that 
these concepts are not useful in the ``full-rank'' setting.
Consider, for example, when $\Ao$ has full row-rank.
Here, \textit{any} orthogonal matrix $\mtx{Q}$ and accompanying $\mtx{B} = \mtx{Q}^{\trans}\Ao$ are valid outputs of rangefinders and QB decomposition algorithms.
Therefore full-rank QB decompositions can be \textit{entirely unstructured}.
Despite this caveat, QB decompositions are of fundamental importance in randomized algorithms for low-rank approximation.

\subsubsection{Rangefinder algorithms and basic QB} 
The simplest rangefinders are based on power iteration.
For example, one can prepare a data-aware sketching operator $\mtx{S}$ of the form \eqref{eq:power_iter_skop}, compute $\mtx{Y} = \Ao\mtx{S}$, and then return $\mtx{Q} = \code{orth}(\mtx{Y})$; this is formalized as \cref{alg:rf1} in \cref{app:subsub:qb}.
More advanced rangefinders use block Krylov subspace methods.
Specific examples of such rangefinders can be found in \cite{MM:2015:blockKrylov}, \cite[\S 7]{Bjarkason:2019}, and \cite[\S 11.7]{MT:2020}.

\cref{alg:qb1} in \cref{app:subsub:qb} is the simplest approach to QB.
It obtains $\mtx{Q}$ by calling an abstract rangefinder and obtains $\mtx{B}$ by explicitly computing $\mtx{B} = \mtx{Q}^{\trans}\Ao$.

\subsubsection{Iterative QB algorithms}

The most effective QB algorithms work by building $(\mtx{Q},\mtx{B})$ \textit{iteratively} \cite{MV:2016:QB}.
Generically, each iteration of such a QB method entails some number of matrix-matrix multiplications with $\Ao$ (as a rangefinder step), adds a specified number of columns to $\mtx{Q}$ and rows to $\mtx{B}$, and makes a suitable in-place update to $\Ao$.
Iterations terminate once some metric of approximation error $\Ao \approx \mtx{Q}\mtx{B}$ 
(e.g., $\| \Ao - \mtx{Q}\mtx{B} \|_2$ or an approximation thereof) falls below a certain level.

Iterative QB methods were improved by \cite{YGL:2018}.
In particular, Algorithm 2 of \cite{YGL:2018} does not modify $\Ao$, it uses a power-iteration-based rangefinder to compute new blocks for $(\mtx{Q},\mtx{B})$, and it efficiently updates the Frobenius error $\|\Ao - \mtx{Q}\mtx{B}\|_{\text{F}}$ as the iterations proceed.
This method is useful because it has complete control over the Frobenius norm error of the returned approximation.
\cref{alg:qb2} in \cref{app:subsub:qb} generalizes this method by allowing an abstract rangefinder in the iterative loop that updates $(\mtx{Q},\mtx{B})$.

 Meanwhile, Algorithm 4 of \cite{YGL:2018} performs power iteration \textit{before} entering its main iterative loop, it does not access $\Ao$ while in the iterative loop, and it can terminate early if a target accuracy is met before a pre-specified rank-limit.
This algorithm has the advantage of reducing the number of times $\Ao$ is moved through fast memory.
In fact, when power iteration is omitted, it can be implemented as a single-pass method based on Type 1 multi-sketching.
The downside is that this algorithm may waste a substantial amount of work if the rank limit is much higher than necessary.
This downside is compounded when power iteration is omitted.
We reproduce this method with slight modifications in \cref{app:subsub:qb} as \cref{alg:qb3}.

\subsubsection{Stopping criteria for iterative QB algorithms}

The Frobenius norm is easily computed for sparse matrices and dense matrices that are stored explicitly in memory.
However, it can be difficult to compute for abstract linear operators when the matrix is accessed only via matrix-vector multiplies, and this can pose problems in computing QB decompositions to specified accuracy.
One approach to address this situation is by careful application of a well-known randomized Frobenius norm estimator as part of the QB decomposition \cite[\S 3.4, \S 3.5, Eq. (3.26)]{GCGMRL:2018:stoppingCriteria}.
We also note that those looking for high-quality approximations often prefer that error be bounded in \textit{spectral norm} (and only use the Frobenius norm because it is usually very cheap to compute).
The problem of estimating spectral norms is well-studied in the NLA literature.
\cref{subsubsec:norm_rank_est} reviews randomized algorithms for estimating matrix norms. 

\subsubsection{Approximate single-pass QB via Type 2 multi-sketching}

Recall that a Type 2 multi-sketch of $\Ao$ is a sketch of the form $(\mtx{Y}_1,\mtx{Y}_2) = (\Ao\mtx{S}_1, \mtx{S}_2\Ao)$ for independent sketching operators $(\mtx{S}_1,\mtx{S}_2)$.
This sketch can be used to compute a QB decomposition that is \textit{approximate}, in the sense that $\mtx{Q}\mtx{B} \approx \Ao$ holds for column-orthonormal $\mtx{Q}$, but we drop the hard requirement that $\mtx{B} = \mtx{Q}^{\trans}\Ao$.

Put simply, the method is to compute $\mtx{Q} = \code{orth}(\mtx{Y}_1)$ and then $\mtx{B} = (\mtx{S}_2\mtx{Q})^\dagger \mtx{Y}_2$.
The intuition behind this approach is that if $\mtx{Q}\mtx{Q}^{\trans}\Ao$ is a good approximation for $\Ao$, then we would have $\mtx{Y}_2 \approx \mtx{S}_2\mtx{Q}\mtx{Q}^{\trans}\mtx{A}$, which would imply $\mtx{B} \approx \mtx{Q}^{\trans}\mtx{A}$.
We refer the reader to \cite[\S 4.2]{TYUC:2017:singlepass} for a proper explanation of this method.

\subsection{Column-pivoted matrix decompositions}
\label{subsec:column_pivoted}

Throughout this section we work with an $r \times c$ matrix $\mtx{G}$.
As before, we begin with some definitions.
\begin{quote}
    A \textit{column-pivoted decomposition} of $\mtx{G}$ is any decomposition of the form
    \begin{equation}\label{eq:generic_column_pivoted}
    \mtx{G}\mtx{P} = \mtx{F} \mtx{T}
    \end{equation}
    where $\mtx{P}$ is a permutation matrix and $\mtx{T}$ is upper-triangular.
    The permutation matrix is encoded by a vector of \textit{pivots}, $J$, so that $\mtx{G}[\fslice,J] = \mtx{G}\mtx{P}$.
\end{quote}
There are many ways of producing such decompositions.
We are only interested in the ways where rank-$k$ matrices obtained by truncation
\begin{equation}\label{eq:trunc_column_pivoted}
    \mtx{\hat{G}} := (\mtx{F}[\fslice,\lslice{k}])(\mtx{T}[\lslice{k},\fslice])\mtx{P}^{\trans}
\end{equation}
provide reasonably good rank-$k$ approximations of $\mtx{G}$.
The meaning of ``reasonably good'' is subjective.
It depends on the computational cost of the algorithm, and it depends on how well $\mtx{\hat{G}}$ approximates $\mtx{G}$ compared to the best approximation obtained by a truncated column-pivoted matrix decomposition.
Interestingly, some randomized algorithms for CSS (see \S \ref{subsec:CSS_CX_computational})
do not use the factors that appear in \eqref{eq:trunc_column_pivoted};
instead, these algorithms only care that the pivots $J$ \textit{could make} approximations form \eqref{eq:trunc_column_pivoted} reasonably accurate.

\paragraph{How much do we truncate?}
When a randomized algorithm uses this primitive for low-rank approximation, the matrix $\mtx{G}$ is usually a sketch of the target data matrix $\Ao$, and $k$ is close to $\min\{r, c\}$.
It helps to consider different situations when trying to build intuition for why these randomized algorithms work.
Specifically, it helps to consider when $\mtx{G}$ is equal to $\Ao$, or a low-rank approximation thereof.
In both such cases one should have $k \ll \min\{r, c\}$.
%
%

\subsubsection{Basics of column-pivoted decomposition algorithms}
There are two main families of algorithms for producing these decompositions: those based on QR with column pivoting (QRCP) and those based on LU with partial pivoting (LUPP).\footnote{
        The standard process of computing an LU decomposition with partial pivoting is called \textit{Gaussian elimination with partial pivoting} and is abbreviated as GEPP.
}
Algorithms in the former family define the decomposition \eqref{eq:generic_column_pivoted} in the natural way, where $\mtx{F}$ is column-orthonormal.
For algorithms in the latter family, one must consider how pivoted LU traditionally uses \textit{row pivoting}.
Therefore, to compute \eqref{eq:generic_column_pivoted} via LUPP, one must compute the transposed factors $(\mtx{P}^{\trans},\mtx{F}^{\trans},\mtx{T}^{\trans})$ in a row-pivoted decomposition,
\[
    \mtx{P}^{\trans}\mtx{G}^{\trans} = \mtx{T}^{\trans}\mtx{F}^{\trans},
\]
where $\mtx{T}^{\trans}$ is lower-triangular (with unit diagonal) and $\mtx{F}^{\trans}$ is upper-triangular.

Roughly speaking, QRCP-based algorithms prioritize accuracy, while LUPP-based algorithms prioritize speed.
The extent to which a specific algorithm does well on these performance metrics depends on the algorithm's pivoting rule.
The most widely used QRCP-based methods use the same pivoting rule as \LAPACK{}'s \code{GEQP3}.
Meanwhile, the most widely used LUPP-based methods use the pivoting rule from \LAPACK{}'s \code{GETRF}.
It is valid to rely on either of these functions for the column-pivoted decomposition steps that arise in randomized algorithms.
However, one should be aware of two potential sources of error when using \code{GETRF} for a column-pivoted decomposition rather than \code{GEQP3}.

\paragraph{What can we expect of \code{GETRF}'s pivots?}
When \code{GETRF} is applied to $\mtx{G}^{\trans}$, the process of computing the pivots up to and including the $\ell^{\text{th}}$ pivot makes decisions based only on the first $\ell$ rows of $\mtx{G}$.
Therefore it is unwise to use \code{GETRF} unless one has reason to believe that the information in $\mtx{G}$'s trailing $r - k$ rows would not drastically alter the columns chosen as pivots based on the first $k$ rows.
Similarly, it is unwise to use $\mtx{\hat{G}}$ as an approximation of $\mtx{G}$ when $k \ll \min\{r, c\}$, since this would suppose that $\mtx{G}$'s leading $k$ rows are significantly more important than all others.
One is most likely to find meaningful information in the column-pivoted LU decomposition when $\mtx{G}$ is very wide ($r \ll c$) and $k$ is close to $r$.

\paragraph{What isn't in the pivots?}

Suppose $\mtx{F}$ has $w = \min\{r, c\}$ columns.
It is easy to verify that for any nonsingular upper-triangular matrix $\mtx{U}$ of order $w$, the decomposition produced after a change-of-basis
\[
(\mtx{F},\mtx{T}) \leftarrow (\mtx{F}\mtx{U}^{-1}, \mtx{U}\mtx{T})
\]
will preserve \eqref{eq:generic_column_pivoted} for the same permutation matrix $\mtx{P}$.

It is informative to consider how such changes of basis affect $\mtx{\hat{G}}$.
For example, in simplest case, it is easy to see that $\mtx{\hat{G}}$ would not change if $\mtx{U}$ were diagonal.
This simple case shows that conditioning of the factors $(\mtx{F},\mtx{T})$ is unimportant in our formalism of column-pivoted decomposition.

To speak to a more interesting case, let us partition $\mtx{F}$ and $\mtx{T}$ into blocks $[\mtx{F}_1,\mtx{F}_2]$ and $[\mtx{T}_1;\mtx{T}_2]$
so that $\mtx{F}_1$ has $k$ columns and $\mtx{T}_1$ has $k$ rows.
Straightforward calculations show that 
\begin{equation}\label{eq:trunc_piv_error_1}
    \|\mtx{G} - \mtx{F}_1\mtx{T}_1\mtx{P}^{\trans}\| = \|\mtx{F}_2\mtx{T}_2\|
\end{equation}
holds in any unitarily-invariant norm.
Meanwhile, less straightforward calculations\footnote{See  \cref{prop:change_basis_column_piv} in \cref{app:theory_truncate_column_piv}.} show that there is always an upper-triangular nonsingular matrix $\mtx{U}$ for which
\begin{equation}\label{eq:trunc_piv_error_2}
    \|\mtx{G} - (\mtx{F}\mtx{U}^{-1})[\fslice,\lslice{k}] (\mtx{U}\mtx{T}[\lslice{k},\fslice])\| = \|(\mtx{I} - \mtx{F}_1^{}\mtx{F}_1^{\dagger})\mtx{F}_2\mtx{T}_2\| \leq \|\mtx{F}_2\mtx{T}_2\|. 
\end{equation}
Note that if there is substantial overlap between $\range(\mtx{F}_1)$ and $\range(\mtx{F}_2)$, then the inequality in \eqref{eq:trunc_piv_error_2} will be strict by a significant margin.
Therefore \textit{if} accuracy of the approximation \eqref{eq:trunc_column_pivoted} is important, then our decomposition should make sure that the columns of $\mtx{F}$ are orthogonal to one another.
This is ensured by QR-based methods, but it is not ensured by LU-based methods.

\subsubsection{Partial decompositions}

Standard algorithms for computing a column-pivoted decomposition of an $r \times c$ matrix require $\Theta(\min\{ r c^2, c r^2\})$ operations.
One can get away with spending less effort when only a \textit{partial} decomposition is needed.
Formally, in a ($k$-step) partial column-pivoted decomposition, we relax the requirement that $\mtx{T}$ be triangular.
Instead, we require that $\mtx{T}$ can be partitioned into a 2-by-2 block triangular matrix where the upper-left block is $k \times k$ and triangular in the proper sense.

The aforementioned standard algorithms for column-pivoted decomposition can be modified to compute $k$-step partial decompositions of $r \times c$ matrices in $\Theta(rck)$ operations.
There are plans for a version of \LAPACK{} subsequent to 3.10 to support this functionality as it pertains to QRCP.\footnote{See \url{https://github.com/Reference-LAPACK/lapack/issues/661}.}
At a practical level, it is certainly worth using this functionality when it is available.
However, this functionality is not critical, since randomized algorithms for low-rank approximation rarely need to compute $k$-step partial decompositions with $k \ll \min\{r, c\}$.

\subsubsection{More details on column-pivoted decomposition algorithms}

The pivots chosen by the strong rank-revealing QR (strong RRQR) algorithm from \cite{GE:1996} lead to the best theoretical guarantees for low-rank approximation by a partial column-pivoted QR decomposition.
In practice, it is more common to truncate the output of \LAPACK{}'s \code{GEQP3}, which is faster than strong RRQR.

Algorithms based on LUPP are typically faster than those based on pivoted QR.
While the LUPP approach comes with weaker guarantees (as explained above), these limitations are less significant in a randomized context where we seek nearly full-rank decompositions of \textit{wide sketches}.
Indeed, there is little practical difference in solution quality between LUPP-based and QRCP-based versions of some randomized algorithms for CSS and column ID \cite{gunnar_use_lupp,DM:2021:CUR}.

Other possibilities for column-pivoted matrix decomposition include LU or QR with tournament pivoting \cites{grigori10:_calu,demmel15:_commun_avoid_rank_reveal_qr_decom}.
Algorithms based on tournament pivoting exhibit reduced communication and hence can be more efficient without significant loss of accuracy.

Finally, we note that \cref{subsec:Rand_QRCP} includes a randomized algorithm for full-rank QRCP.
It is easy enough to modify that algorithm to support early termination.
Some variants of this algorithm specifically focus on low-rank approximation (e.g., the SRQR algorithm from \cite{XGL:2017:RandQRCP}).

\subsection{One-sided ID and CSS}
\label{subsec:CSS_CX_computational}

Column ID and CSS are nearly equivalent problems.
That is, a method for CSS can canonically be extended to a method for column ID by taking $\mtx{X} = (\Ao[:,J])^\dagger\Ao$.
Conversely, a method for column ID can be adapted to CSS by discarding any calculations that are only needed to form $\mtx{X}$.
This section covers deterministic and randomized algorithms for both of these problems.
Readers who are particularly interested in theoretical aspects of these algorithms should consult \cite{BMD09_CSSP_SODA}.\footnote{We frame all of our discussion of one-sided ID around column ID, rather than row ID.}
 

\subsubsection{Template deterministic algorithms}

Suppose we want to compute a rank-$k$ column ID of an $r \times c$ matrix $\mtx{G}$.
There is a template deterministic algorithm for handling this problem based on the notion of column-pivoted decompositions, as discussed in \cref{subsec:column_pivoted}.

The template algorithm works in two phases.
The first phase produces the decomposition $\mtx{G}\mtx{P} = \mtx{F}\mtx{T}$ where $\mtx{P}$ is a permutation matrix and $\mtx{T}$ is upper-triangular.
The second phase is a matter of simple postprocessing.
The postprocessing begins by partitioning $\mtx{T}$ into a 2-by-2 block triangular matrix
\[
    \mtx{T} = \begin{bmatrix} \mtx{T}_{11} & \mtx{T}_{12} \\ 
    \mtx{0} & \mtx{T}_{22} \end{bmatrix}.
\]
where $\mtx{T}_{11}$ is $k \times k$ and triangular in the usual sense.
From here, one sets the interpolation matrix to $\mtx{X} = [\mtx{I}_{k \times k}, \mtx{T}_{11}^{-1}\mtx{T}^{}_{12}]\mtx{P}^{\trans}$, 
and one sets the skeleton indices to the vector $J$ that provides $\mtx{X}[\fslice,J] = \mtx{I}_k$.

The importance of this algorithm stems from how its output can equivalently be analyzed as a truncated column-pivoted matrix decomposition.
That is, the column ID induced by $(J,\mtx{X})$ satisfies
\[
    \mtx{G}[\fslice,J]\mtx{X} = \mtx{F}[\fslice,\lslice{k}]\mtx{T}[\lslice{k},\fslice]\mtx{P}^{\trans}.
\]
We can therefore gain insights into the behavior of this column ID algorithm by appealing to results such as \cref{prop:change_basis_column_piv}, which we alluded to earlier in our discussion of LU and QR based pivoting methods.

This approach to column ID is illustrated more formally as
\cref{alg:osid_qrcp} in \cref{subapp:ID_and_CSS}.
It is easy to see that the CSS version of this algorithm does not need to compute $\mtx{X}$, since the definition of $\mtx{X}$ implies that $J$ can be determined from the first $k$ rows of $\mtx{P}$.



\subsubsection{Randomized algorithms}

We list five randomized algorithms for CSS and column ID below.
With some exception for the fifth algorithm, we do not comment on the theoretical guarantees of these methods.


\begin{enumerate}
    \item 
    For CSS, one can sample columns with probability proportional to their norms, where column norms are updated by projecting out selected columns as a QRCP-like factorization proceeds \cite{DV:2006:CSS}.
    Applying the standard post-processing scheme to the (partial) factorization yields the interpolation matrix $\mtx{X}$ needed for a column ID.
    \item 
    Also for CSS, one can sample columns according to a probability distribution related to so-called \textit{leverage scores} of the matrix under consideration.
    We discuss leverage score sampling in detail in \cref{sec7:lev_scores}.
    For now, we note that this approach especially useful for computing \Nystrom{} approximations.
    
    \item 
    The algorithm in \cite{BMD09_CSSP_SODA} approaches CSS with a combination of leverage score sampling and postprocessing by deterministic QRCP.
    The factorization produced by this postprocessing can be processed further to produce the interpolation matrix for a column ID.
    \item 
    \cite[\S V.D]{XGL:2017:RandQRCP} suggests solving CSS by taking the pivots from a randomized algorithm for QRCP.
    The output of the randomized algorithm for QRCP can, of course, be processed to recover the interpolation matrix for a column ID.
    \item 
    \cite[\S 5.1]{VM:2016:CUR} approaches low-rank column ID by computing a (nearly) full-rank column ID of a sketch  $\mtx{Y} = \mtx{S}\Ao$.
    The unmodified data $(\mtx{X},J)$ is used to define the low-rank column ID $\Ao[:,J]\mtx{X} \approx \Ao$.
\end{enumerate}

The last of these methods is simple and practical.
It appears with slight modifications in \cref{subapp:ID_and_CSS} as \cref{alg:osid1}, while the corresponding CSS version appears as \cref{alg:rocs1}.
For both the column ID and CSS versions, it is recommended that $\mtx{S}$ be a data-aware sketching operator based on power iteration.
To gain intuition for this method, one should first verify that if $(\mtx{X},J)$ defines a full-rank column ID of $\mtx{Y}$, then it also defines a full-rank column ID of $\mtx{\tilde{A}} = (\Ao\mtx{Y}^{\dagger})\mtx{Y}$.
With that given, we can apply \cref{prop:regularity_ID_accuracy} to see that the induced low-rank column ID satisfies an error bound
\[
    \|\Ao -  \Ao[\fslice,J]\mtx{X}\|_2 \leq (1 + \|\mtx{X}\|_2)\|\Ao - \mtx{\tilde{A}}\|_2.
\]
This bound is noteworthy for the following reason: using power iteration to prepare a data-aware sketching operator $\mtx{S}$ will drive $\mtx{\tilde{A}}$ closer to a rank-$k$ approximation of $\Ao$ obtained by a truncated SVD.
That, in turn, would give $\|\Ao - \mtx{\tilde{A}}\|_2 \approx \sigma_{k+1}(\Ao)$.

\begin{remark}
    The value of power iteration in the context of CSS / column ID \textit{and} in the context of rangefinders / QB is a key reason for considering power iteration as a basic primitive of \RandNLA{}.
\end{remark}

\subsubsection{On fixed-accuracy one-sided ID}

Standard implementations of deterministic QRCP-based algorithms for one-sided ID can compute approximations to specified accuracy.
\textit{Randomized} algorithms for low-rank one-sided ID do not possess this capability to the same extent.
In some respects, this is a principal disadvantage of low-rank approximation by ID compared to QB.
However, there are partial workarounds.
 
For example, suppose we approximate $\Ao$ via a QB decomposition, $\mtx{\tilde{A}} = \mtx{Q}\mtx{B}$.
If we computed $(\mtx{X},J)$ by a full-rank column ID of $\mtx{B}$, then we would also have a full-rank column ID of $\mtx{\tilde{A}}$.
If $\mtx{X}$ was obtained by the standard postprocessing of output from strong rank-revealing QR, then we would have $|X_{ij}| \leq 2$.
A straightforward application of \cref{prop:regularity_ID_accuracy} shows that if we had
\begin{equation}\label{eq:need_QB_for_fixed_acc_ID}
    \|\Ao - \mtx{Q}\mtx{B}\|_2 \leq \frac{\epsilon}{(1 + \sqrt{1 + 4k (n-k)})}
\end{equation}
for a rank-$k$ QB decomposition, then we could be certain that $ \|\Ao -  \Ao[\fslice,J]\mtx{X}\|_2$ was at most $\epsilon$.
Therefore, in principle, one could compute the QB decomposition iteratively, and only compute the ID of $\mtx{B}$ once \eqref{eq:need_QB_for_fixed_acc_ID} is satisfied.

The above approach is not without its shortcomings.
For one thing, reducing $\|\Ao - \mtx{Q}\mtx{B}\|_2$ entails increasing $k$, and so the termination criterion is a moving target.
As another issue, it needs to bound spectral norms of implicitly represented linear operators.
We address the problem of estimating matrix norms next.

\subsection{Estimating matrix norms}
\label{subsubsec:norm_rank_est}

Norm estimation plays an important role in stopping criteria for iterative low-rank approximation algorithms, particularly for QB and \Nystrom{} approximations.
Here we summarize methods that would be appropriate for expensive norms or norms of abstract linear operators that are only accessible by matrix-vector multiplications.

\begin{remark}
     The material presented here is covered in greater detail in \cite[\S 4.3 - \S 4.4]{HMT:2011} and \cite[\S 5 - \S 6, \S 12.0 - \S 12.4]{MT:2020}.
\end{remark}

\paragraph{A cheap spectral norm bound.}
Let the vectors $\vct{z}_1,\dots,\vct{z}_r\in\R^n$ be vectors with components drawn iid from the standard normal distribution and let $\beta>1$ be a tuning parameter.
Then, for any $\Ao$, it is known that the inequality
\begin{equation}\label{eqn:testop}
    \|\Ao\|_2 \ \leq \ \beta\sqrt{\frac{2}{\pi}}\max_{j\in\idxs{r}}\|\Ao\vct{z}_j\|_2
\end{equation}
holds with probability at least $1-\beta^{-r}$~\cites{HMT:2011, WLRT:2008}. 
Furthermore, this bound is easy to compute because the necessary vectors $\Ao\vct{z}_j$ can be formed with a single matrix-matrix product with $\Ao$.

\paragraph{A basic Frobenius norm estimator.} Let $\mtx{Z}\in\R^{n\times r}$ be the matrix whose columns are the random vectors $\vct{z}_1,\dots,\vct{z}_r$ mentioned above. Then, it turns out that the quantity $\frac{1}{r}\|\Ao\mtx{Z}\|_{\text{F}}^2$ is an unbiased estimate for the squared Frobenius norm, in the sense that
\begin{equation}\label{eqn:frobreln}
    \mathbb{E}\Big[\textstyle\frac{1}{r}\|\Ao\mtx{Z}\|_{\text{F}}^2\Big] \ = \ \|\Ao\|_{\text{F}}^2.
\end{equation}
In addition to being unbiased, the variance of the error estimate can also be controlled according to
\begin{equation}
    \text{var}\Big(\textstyle\frac{1}{r}\|\Ao\mtx{Z}\|_{\text{F}}^2\Big) \ \leq \ \textstyle\frac{2}{r}\|\Ao\|_2^2\|\Ao\|_{\text{F}}^2,
\end{equation}
as shown in~\cite{Girard:1989}.
Hence, as long as $r$ is sufficiently large, then the error estimate $\frac{1}{r}\|\Ao\mtx{Z}\|_{\text{F}}^2$ is likely to be close to $\|\Ao\|_{\text{F}}^2$.
From a computational standpoint, this error estimate is similar to the one described above for the spectral norm, insofar as it only requires $r$ matrix-vector products with $\Ao$.

\paragraph{A cheap Schatten $p$-norm estimator.}

Letting $\vct{\sigma}$ denote the vector of singular values of $\Ao$, the \textit{Schatten 2$p$-norm} of $\Ao$ is $\|\Ao\|_{(S,2p)} := \left(\sum_{i=1}^{\min\{m,n\}} \sigma_i^{2p}\right)^{1/2p}$.
Taking $p=1$ reduces to the Frobenius norm.
The spectral norm is obtained in the limit as $p \to \infty$.
In fact, deterministic bounds show that the spectral norm and Schatten $p$-norm more or less coincide when $p \gtrsim \log\min\{m,n\}$.

The Kong-Valiant estimator \cite{kong2017spectrum} can be used to cheaply estimate these norms.
It only accesses $\Ao$ by multiplication with an $n \times k$ data-oblivious sketching operator, where $k$ can be materially smaller than $\min\{m,n\}$.
See \cite[\S 5.4]{MT:2020} for a statement of the algorithm and remarks on its theoretical guarantees.

\paragraph{Accurate spectral norm estimators.} 

There is a large literature on deterministic and randomized algorithms for estimating spectral norms.
Much of this literature is based on methods designed for estimating the largest eigenvalue of a positive definite matrix (which can naively be applied since $\sqrt{\|\Ao^{\trans}\Ao\|_2} = \|\Ao\|_2$).
Most notably, Dixon was the first to study the randomized power method \cite{Dixon:1983:powerMethod}, and Kuczy\'{n}ski and Wo\'{z}niakowski were the first to study randomized Lanczos methods \cite{KW:1992:maxEig_powerLanczos}.
See \cite[Algorithm 5]{MT:2020} for a basic randomized Lanczos method and the subsequent remarks on block randomized Lanczos \cite[\S 6.5]{MT:2020}.


\subsection{Oblique projections}\label{subsubsec:oblique_proj}

Low-rank approximations can be expressed in a manner resembling the triple-sketch from \cref{subsec:multi_sketch}.
For sketching operators $\mtx{S}_1 \in \R^{n \times k}$ and $\mtx{S}_2 \in \R^{d \times m}$, we can define
\[
\Aa = \Ao\mtx{S}_1 (\mtx{S}_2\Ao\mtx{S}_1)^\dagger \mtx{S}_2\Ao  =  \mtx{Y}_1\mtx{Y}_3^\dagger\mtx{Y}_2,
\]
where
\[
    \mtx{Y}_1 = \Ao\mtx{S}_1,\quad \mtx{Y}_2 = \mtx{S}_2\Ao,\quad\text{and}\quad \mtx{Y}_3 = \mtx{S}_2\Ao\mtx{S}_1.
\]
This construction obtains each column of $\Aa$ by projecting the corresponding column of $\Ao$ onto the range of $\mtx{Y}_1$, where the projection is orthogonal with respect to the possibly degenerate inner product $(\vct{u},\vct{v}) \mapsto \langle\mtx{S}_2\vct{u},\mtx{S}_2\vct{v}\rangle$.
We call $\Aa$ an \textit{oblique projection of $\Ao$}.

The simplest oblique projections use column and row selection operators for $(\mtx{S}_1, \mtx{S}_2)$.
This provides a CUR decomposition where $\mtx{Y}_3^\dagger$ is the linking matrix $\mtx{U}$.
The connection to CUR foreshadows a more general fact: the sketching operators used in oblique projection are not necessarily independent of one another~\cite{DMM:2008}.
An example in this regard is that \Nystrom{} approximations amount to oblique projections that use $\mtx{S}_2 = \mtx{S}_1^{\trans}$.

It is natural to consider oblique projections where $\mtx{S}_1$ and $\mtx{S}_2$ are independent (e.g., independent Gaussian operators).
Such approximations can entail extremely ill-conditioned computations if one is not careful.
This ill-conditioning can be avoided  through the numerically stable approach described by Nakatsukasa \cite{Nakatsukasa:2020:genNystrom}.
These approximations employ oversampling for $\mtx{S}_2$ (relative to $\mtx{S}_1$) and split $\mtx{Y}_3$ (or a regularized variant thereof) into two factors.
The representation returned by this approach consists of four matrices.

\subsubsection{Historical notes}

Oblique projections for low-rank approximation are closely related to the rank reduction formula described in \cite{CFG:1995:RankReductionFormula}.
Drineas et al.\ first used oblique projections for low-rank approximation via CUR decomposition \cite{DMM:2008}, wherein $\mtx{S}_1,\mtx{S}_2$ are column and row selection matrices respectively.
Clarkson and Woodruff pioneered the use of general oblique projections in randomized algorithms for low-rank approximation \cite[Theorem 4.7]{CW:2009:streaming}.
Oblique projections have since been discussed in the context of a generalized LU factorization \cite{DGR:2019:GLU}.

\section{Other low-rank approximations}\label{subsec:other_lowrank}

Here we review a handful of other low-rank approximation problems and algorithms, particularly speaking to our development plans for \RandLAPACK{}.

\paragraph{Domain-specific representations.} Several low-rank approximation problems of interest involve specialized factorizations.
We plan for \RandLAPACK{} to eventually support nonnegative matrix factorization \cite{EMWK:2018}, dynamic mode decomposition (DMD) \cites{erichson2019randomized,erichson2019compressed}, and possibly sparse PCA \cite{erichson2020sparse}.
Among these methods, we expect that DMD will have highest priority, since full-rank DMD is slated for inclusion into \LAPACK{} in the near future \cite{Drmac:2022:DMD_LAWN}.
For a general introduction to DMD we refer the reader to \cite{TTLBK:2014:DMD}.

\paragraph{Low-rank Cholesky.}
As a separate topic, there is also a longstanding algorithm for ``low-rank Cholesky'' decompositions \cite{XG:2016:LowrankChol}.
We are unsure of its eventual role in \RandLAPACK{}, since a representation of the form $\Aa = \mtx{L}\mtx{L}^{\trans}$ for a very tall lower-triangular matrix $\mtx{L}$ offers almost no benefit over $\mtx{L}$ being dense.
Still, it will be considered in the near future alongside the recently proposed algorithm by \cite{CETW:2022:randPivCholesky} for randomly pivoted partial Cholesky decomposition.

\paragraph{Low-rank QR.}
Suppose $\Ao$ is a large full column-rank matrix with QR decomposition $\Ao = \mtx{Q}\mtx{R}$.
This decomposition has two especially prominent uses: (1) it facilitates application of a pseudoinverse $\Ao^\dagger\vct{v} = \mtx{R}^{-1}\mtx{Q}^{\trans}\vct{v}$ in $O(mn)$ time, and (2) it can be used as preprocessing for more complicated orthogonal decompositions such as SVD.
Unfortunately, low-rank QR decomposition, which is simply the economic QR decomposition of a rank-$k$ approximation of $\Ao$, does \emph{not} fully realize either of these use-cases.

The trouble with low-rank QR is that a $k \times n$ upper-triangular matrix with $k \ll n$ is effectively a full matrix.
That is, the mere \textit{representation} of a low-rank matrix by a QR decomposition is not much more useful than representation by QB decomposition.
Note also that unpivoted QR makes no effort to produce a rank-revealing representation, compared to pivoted QR.
Therefore \RandLAPACK{} will not offer methods for low-rank approximation by unpivoted QR.
 
 \paragraph{Low-rank UTV.}
A UTV decomposition $\Aa = \mtx{U}\mtx{T}\mtx{V}^{\trans}$ uses column-orthogonal matrices $\mtx{U},\mtx{V}$ and a triangular matrix $\mtx{T}$.
UTV (also called QLP) can be thought of as a cheaper alternative to SVD.
As we discuss in the next section, \RandLAPACK{} might include algorithms for UTV when $\Aa$ is full-rank \cites{GM:2018:URV,MQH:2019:URV,KCdL:2021:QLP_fullrank}.
Some of those algorithms (e.g., that in \cite{MQH:2019:URV}) proceed iteratively and can be terminated early.
If \RandLAPACK{} supports full-rank UTV by such an algorithm then it will expose the low-rank variant.

Several algorithms for producing low-rank approximations represented by UTV are given in \cites{DG:2017:QR,FXG:2019:FFSRQR_lowrank,WX:2020:QLP,RB:2020:QLP,KC:2021:QLP}.
We would need a better understanding of those methods, particularly how they compare to our planned methods for low-rank SVD, before making decisions on which of them to support.

\paragraph{Low-rank LU.}
LU is central to solving systems of linear equations in the full-rank case.
There is a small literature on low-rank LU within the field of \RandNLA{}: \cite{SSAA:2018:LU,DGR:2019:GLU,ZM:2020:LU}.
In \RandLAPACK{} we anticipate restricting our attention to algorithms that are related to a Gaussian elimination process (that is, where the error matrix can be expressed as a Schur complement of a block matrix), along the lines of \cite{DGR:2019:GLU}.
These algorithms are likely more useful for low-rank approximation with a fixed accuracy requirement rather than with a fixed rank requirement.
They are based on an oblique projection with $k=d$, that is $\mtx{S}_2\Ao\mtx{S}_1$ is square. 

RandLAPACK{} might include the LU algorithms from \cite{SSAA:2018:LU} if they can be proven to be significantly faster than high-quality implementations of QB algorithms.
If proven useful, we will consider in the future generalized LU-based low-rank approximation, as introduced in~\cite{DGR:2019:GLU}.
The algorithms for low-rank LU in \cite{ZM:2020:LU} are based on QB and so are unlikely to be included in \RandLAPACK{}.

\section{Existing libraries}
\label{subsec:lowrank_libraries}

Here we review established numerical libraries that support randomized low-rank approximation.
All of the libraries that focus on RandNLA (save for one) implement advanced sketching operators such as SRFTs.

\paragraph{NLA and data science packages that use RandNLA.}

There are a few packages for NLA that include a method for low-rank SVD based on QB decomposition:
\begin{itemize}
    \item \code{MLSVD{\_}RSI} in the \Tensorlab{} MATLAB toolbox
    \item \code{rsvd} in \SciKitCUDA{},
    \item \code{cusolverDnXgesvdr} in NVIDIA's \cuSOLVE{},
    \item and \code{randomized\_svd} in \textsf{SciKit-Learn}.
\end{itemize}
The last of these functions warrants special emphasis.
\textsf{SciKit-Learn}'s \code{pca} function actually \textit{defaults} to \code{randomized\_svd} for sufficiently large matrices \cite{SciKitLearn-PCA}.
In this way, one of the most important functions in the most widely-used Python package for data science already relies on \RandNLA{}.

\paragraph{\IDLib{}.}

\IDLib{} is a Fortran library for ID/CUR \cite{MRST:2014:IDsoftware}.
It is callable as part of the \SciPy{} Python library. 
\IDLib{} provides indirect support for SVD as part of its methods for converting one low-rank factorization into another.
It also includes routines for rank estimation and norm estimation.
\RandLAPACK{} will include many of these same utilities as \IDLib{} while expanding its scope of driver-level functions.

\paragraph{\RSVDPACK{}.}

\RSVDPACK{} is a C and MATLAB library for low-rank SVD and ID/CUR \cite{VM:2015:rsvdpack}.
It is callable after building from source code which is provided on GitHub.
Its SVD algorithms are based on a particular QB implementation \cite[\S 3.4]{VM:2015:rsvdpack} and its ID/CUR algorithms follow \cite{VM:2016:CUR}.
\RSVDPACK{} comes in different implementations which target different architectures.

By comparison, \RandLAPACK{} will take more general approaches to QB and ID/CUR, and it will include methods for other factorizations such as eigendecomposition via Nystr\"{o}m approximations.
\RandLAPACK{} will target different architectures by building on \LAPACKpp{} as a portability layer \cite{SLATE:2017:LAPACK++}.\footnote{This library is developed as part of \SLATE{} \cites{SLATE:2017:design,SLATE:2017:roadmap}.}

\paragraph{\Ristretto{}.}

\Ristretto{} is available on the Python Package Index. This library is based on the \textsf{rsvd} package implemented in R~\cite{EVBK:2019:Ristretto}. 
It supports low rank SVD, ID/CUR, LU, Nystr\"{o}m, PCA, Hermitian eigendecomposition, nonnegative matrix factorization~\cite{EMWK:2018}, dynamic mode decomposition~\cites{erichson2019randomized, erichson2019compressed,erichson2016randomized}, and sparse PCA~\cite{erichson2020sparse}.
One algorithm is provided for each distinct type of factorization. Many of these algorithms are based on QB \cite[\S 3.3]{EVBK:2019:Ristretto}, while its ID/CUR algorithms also follow \cite{VM:2016:CUR}. This library has also been demonstrated to be useful for finding patterns in large-scale climate data~\cite{velegar2019scalable}, and for providing routines for randomized tensor decompositions~\cite{erichson2020randomized}.

We plan for \RandLAPACK{} to eventually support the same range of factorizations as \Ristretto{} (with the exception of low-rank LU).
However, our priority is to focus on the factorizations in \cref{subsec:lowrank_drivers}, and to offer a range of algorithms for computing each of these decompositions.
Our longer-term plans include making \RandLAPACK{}'s C++ implementation callable from Python.

\paragraph{\LibSkylark{}.}

\LibSkylark{} \cite{libskylark} is written in C++ and callable after installing from source, which is available on GitHub.
To our knowledge, it is the only \RandNLA{} library that supports both least squares and low-rank approximation.
Its low-rank approximation functionality is restricted to SVD through a QB approach.
See \cref{subsec:opt_libraries} for its least squares functionality.

\paragraph{\LowRankApprox{}.}

\LowRankApprox{} is a Julia library for low-rank SVD, QR, ID, CUR, and Hermitian eigendecomposition.
It is callable after installation with the Julia package manager.
Most of its algorithms are based on first computing an ID, rather than a QB decomposition.
Note that this is quite different from the plans we have outlined for \RandLAPACK{} over \cref{subsec:lowrank_subroutines,subsec:lowrank_drivers}.

\paragraph{Other implementations.}

The many algorithms considered in \cite{Bjarkason:2019} are accompanied by Python implementations hosted on GitHub.
The \RandNLA{} tutorial \cite{Wang:2015:practical} covers a wide range of algorithms for low-rank approximation and hosts some MATLAB implementations on GitHub.
%
\chapter{Further Possibilities for Drivers}
\label{sec5:more_drivers}

\minitoc
\bigskip

This section covers multi-purpose matrix decompositions, the solution of unstructured linear systems, and trace estimation.
These are the last problems we cover that might be handled by ``drivers'' in a high-level \RandNLA{} library.
We emphasize that this monograph does not exhaust the set of prominent linear algebra problems that are amenable to randomization.
We make no effort to cover randomized algorithms for general eigenvalue problems, nor do we cover randomized algorithms for computing the action of matrices produced from matrix functions (i.e., computing $f(\mtx{A})\vct{b}$ for an analytic matrix function $f$), even though there are effective algorithms for both of these problems \cite{NT:2021:krylov,GS:2022:rand_Krylov_for_mat_funcs,CKN:2022:spectral_function}.

We have chosen the topics of this section because they require comparatively little background material to state, and we believe our summary of the relevant algorithms has some contribution
to the literature.
For example, the key contribution from \cref{subsec:multipurpose_decomp} is a novel algorithm for QR with column pivoting based on Cholesky QR.
The algorithm is notable for its ability to handle ill-conditioned or even outright rank-deficient matrices.
The contributions from \cref{subsec:general_linear_systems} include detailed introductions to recently-developed iterative methods for solving general linear systems.
Finally, our coverage of trace estimation in \cref{sec:trace_estimation}, includes state-of-the-art algorithms and implementations that were not available when earlier \RandNLA{} surveys were published.

\section{Multi-purpose matrix decompositions}\label{subsec:multipurpose_decomp}

Early in the year 2000, the IEEE publication \textit{Computing in Science \& Engineering} published a list of the top ten algorithms of the twentieth century.
Among this list was \textit{the decompositional approach to matrix computation}, on which G. W. Stewart gave the following remark.
\begin{quote}
    The underlying principle of the decompositional approach of matrix computation is that it is not the business of matrix algorithmists to solve particular problems but to construct computational platforms from which a variety of problems can be solved.
    This approach, which was in full swing by the mid-1960s, has revolutionized matrix computation.
\end{quote}

This section covers three decompositions that provide broad platforms for problem solving. 
They are addressed in an order where randomization offers increasing benefits over purely deterministic algorithms.
We note in advance that these randomized algorithms do not aim for an asymptotic speedup over deterministic methods.
Rather, the aim is to significantly reduce time-to-solution by taking better advantage of modern computing hardware.

\subsection{QR decomposition of tall matrices}\label{subsubsec:chol_qr}

Algorithms for computing unpivoted QR decompositions are true workhorses of numerical linear algebra.
They are the foundation for the preferred algorithms for solving least squares problems with full-rank data matrices.
They are also an important ingredient in preprocessing for more expensive algorithms.

For example, suppose we want to decompose a very tall $m \times n$ matrix $\mtx{A}$ via QR with column pivoting.
The instinctive thing to do here is to reach for the \LAPACK{} function \code{GEQP3}.
However, on modern machines, it is much faster to compute an unpivoted decomposition $\mtx{A} = \mtx{Q}\mtx{R}$, and then run \code{GEQP3} on $\mtx{R}$.
The final decomposition would be mathematically equivalent to calling \code{GEQP3} directly on $\mtx{A}$, just represented in a different format.
%
%


With this significance of unpivoted QR in mind, we briefly cover two types of randomized algorithms for computing such decompositions.

\subsubsection{Orthogonality in the standard inner product}
Cholesky QR is a method for computing unpivoted QR decompositions of matrices with linearly independent columns.
It is based on the following elementary observation: given a QR decomposition $\mtx{A} = \mtx{Q}\mtx{R}$ of a full-column-rank matrix $\mtx{A}$, the factor $\mtx{R}$ is simply the upper-triangular Cholesky factor of the Gram matrix $\mtx{A}^{\trans}\mtx{A}$.
Therefore in principle one can compute a QR decomposition as follows.
\begin{enumerate}
    \item Compute a Cholesky decomposition of the Gram matrix $\mtx{A}^{\trans}\mtx{A} = \mtx{R}^{\trans}\mtx{R}$.
    \item Perform a matrix-matrix triangular solve to obtain $\mtx{Q} = \mtx{A}\mtx{R}^{-1}$.
\end{enumerate}
Implementing Cholesky QR only requires three functions: \code{syrk} from \BLAS{}, \code{potrf} from \LAPACK{}, and \code{trsm} from \BLAS{}.
Standard implementations of these functions parallelize extremely well.
As a result, Cholesky QR can offer substantial speedups over Householder QR (and even Tall-and-Skinny QR \cite{demmel15:_commun_avoid_rank_reveal_qr_decom}) for very tall matrices on modern machines.

Despite the speed advantage of Cholesky QR, it is rarely used in practice, since it is unsuitable for even moderately ill-conditioned matrices.
Recently it has been shown that randomization can overcome this limitation by preconditioning Cholesky QR to ensure stability \cite{FGL:2021:CholeskyQR}.
For detailed analysis of this method we refer the reader to the results in \cite{Balabanov:2022:cholQR} on the algorithm called ``\code{RCholeskyQR2}.''

In \cref{subsec:cholesky_qrcp} we extend this methodology to rank-deficient matrices, and we connect it to an existing randomized algorithm for QRCP of general matrices.

\subsubsection{Orthogonality in a sketched inner product}

In \cite{BG:2021:GramSchmidt}, Balabanov and Grigori propose a randomized Gram--Schmidt (RGS) process that orthogonalizes $n$ vectors in $\R^m$ with respect to a sketched inner product
\begin{equation}\label{eq:sketched_inner_product}
    \langle \vct{u},\vct{v}\rangle_{\mtx{S}} = (\mtx{S}\vct{u})^{\trans}(\mtx{S}\vct{v}).
\end{equation}
We call such vectors \textit{$\mtx{S}$-orthogonal} or \textit{sketch-orthogonal}.
When \cite[Algorithm~2]{BG:2021:GramSchmidt} is run on the columns of a matrix $\mtx{A}$, values computed during sketched projections are assembled in an upper-triangular matrix $\mtx{R}$ so that $\mtx{A} = \mtx{Q}\mtx{R}$ and  $(\mtx{S}\mtx{Q})^{\trans}(\mtx{S}\mtx{Q}) = \mtx{I}_n$.
One can choose the distribution from which $\mtx{S}$ is drawn so that $\mtx{Q}$ will be nearly-orthonormal with respect to the standard inner product, with high probability.
Empirical and theoretical results show RGS is faster than classic Gram--Schmidt but as stable as modified Gram--Schmidt.

The idea of computing QR decompositions where $\mtx{Q}$ is sketch-orthogonal can be taken in several directions.
For example, a block version of RGS is proposed and analyzed in \cite{BG:2021:Randomized_Block_Gram_Schmidt}.
Taking this approach to the extreme where the block size is the number of columns in the matrix, one can compute the factor $\mtx{R}$ by Householder QR on $\mtx{S}\mtx{A}$ and then represent $\mtx{Q} = \mtx{A}\mtx{R}^{-1}$ as a linear operator.
Note that this procedure is the essence of sketch-and-precondition for least squares, as proposed in \cite{RT:2008:SAP}.
A detailed numerical analysis of this last method can be found in \cite{Balabanov:2022:cholQR}, where the algorithm is called \code{RCholeskyQR}.

\subsection{QR decomposition with column pivoting}\label{subsec:Rand_QRCP}

We recall the following reformulation of QR with column pivoting (QRCP) for the reader's convenience.
\begin{quote}
    Given a matrix $\mtx{A}$, produce a column-orthogonal matrix $\mtx{Q}$, an upper-triangular matrix $\mtx{R}$, and a permutation vector $J$ so that
    \[
    \mtx{A}[:, J] = \mtx{Q}\mtx{R}.
    \]
\end{quote} 
The diagonal entries of $\mtx{R}$ should approximate $\mtx{A}$'s singular values, and the columns of $\mtx{Q}$ should approximate $\mtx{A}$'s left singular vectors.
These stipulations reflect QRCP's main use cases: in low-rank approximation and in solving ill-conditioned least squares problems.
As usual, we say that our matrix $\mtx{A}$ is $m \times n$.

\paragraph{It's all in the pivots.}
We note that if $m \geq n$, then for any permutation vector $J$, the economic QR decomposition of $\mtx{A}[:,J]$ is unique.\footnote{Technically, it is only unique up to sign flips on the columns of $\mtx{Q}$ and rows of $\mtx{R}$. But it is clear how signs must be chosen if the diagonal of $\mtx{R}$ is to approximate the singular values
of $\mtx{A}$.
%
%
}
Therefore $J$ completely determines how well the columns of $\mtx{Q}$ (resp., diagonal entries of $\mtx{R}$) approximate the left singular vectors of $\mtx{A}$ (resp., singular values of $\mtx{A}$).

The method of choosing pivots that sees the widest use today (a simple method based on column norms) was first described in \cite{BG:1965:QRCP}.
The straightforward implementation of this method can have subtle failure cases in finite-precision arithmetic, however, this can be resolved by carefully restructuring norm calculations \cite{DB:2008:failure_of_RRQR_software}.

\subsubsection{An established randomized algorithm for general matrices}
\label{subsec:fullrank_decomp:qrcp}

Here, we outline a remarkable algorithm first developed by Martinsson \cite{Martinsson:2015:QR} and Duersch and Gu \cite{DG:2017:QR}, and then refined by Martinsson, Quintana-Ort\'{i}, Heavner, and van de Geijn \cite{MOHvdG:2017:QR}.
This refined algorithm was introduced with the name \textit{Householder QR with Randomization for Pivoting} or \textit{HQRRP}.
As this name implies, the factor $\mtx{Q}$ from HQRRP is an $m \times m$ operator defined by $n$ Householder reflectors.
The algorithm can run much faster than standard QRCP methods by processing the matrix in column blocks, which makes it possible to cast the overwhelming majority of its operations in terms of \BLASlev{3}, instead of about half \BLASlev{2}.

While a full description of HQRRP is beyond our scope, we can outline its structure.
As input, it requires that the user provide a block size parameter $b$ and an oversampling parameter $s$.
Typical values for these parameters are $b = 64$ and $s = 10$.
HQRRP starts by forming a thin $(b+s) \times n$ sketch $\mtx{Y} = \mtx{S}\mtx{A}$, and then it enters the following iterative loop.
\begin{enumerate}
    \item[1.] Use any QRCP method to find $P_{\text{block}}$: the first $b$ pivots for $\mtx{Y}$.
    \item[2.] Process the panel $\mtx{A}[:,P_{\text{block}}]$ by QRCP.
    \item[3.] Suitably update $(\mtx{A},\mtx{Y})$ and return to Step 1.
\end{enumerate}
The update to $\mtx{A}$ at Step 3 can be handled by standard methods, such as those used in blocked unpivoted Householder QR.
The update to $\mtx{Y}$ is more subtle.
If done appropriately (particularly, by Duersch and Gu's method \cite{DG:2017:QR}) then the leading term in the FLOP count for HQRRP is identical to that of unpivoted Householder QR.
The one downside of this algorithm is that the diagonal entries of $\mtx{R}$ are not guaranteed to decrease across block boundaries.

\paragraph{Implementation notes.}
We adapted the C implementation from \cite{MOHvdG:2017:QR} into C\texttt{++} code at
\begin{quote}
    \url{https://github.com/rileyjmurray/hqrrp}.
\end{quote}
\noindent Our main change was to access BLAS and LAPACK through \BLAS{}\texttt{++} and \LAPACK{}\texttt{++}.
The modified code also allows for matrix dimensions to be specified with either 32-bit or 64-bit integers and includes a small test suite.

We briefly point out two opportunities to improve the performance of this algorithm.
The first is to use mixed-precision arithmetic.
Specifically, both the sketch of $\mtx{A}$ and the call to deterministic QRCP on that sketch could use reduced precision.
Given that the real purpose of QRCP on the sketch is to select the block pivot indices for $\mtx{A}$, it might be that loss of accuracy in that phase does not compromise the accuracy of the larger algorithm.
%
%
The second opportunity is to call unpivoted QR on the very matrix $\mtx{A}_{\text{panel}}$ in the second phase of processing a block; if pivoting is used in the second phase then the pivots can be determined by deterministic QRCP on the $\mtx{R}$ factor from the unpivoted QR of $\mtx{A}_{\text{panel}}$.

\FloatBarrier
\subsubsection{A novel randomized algorithm for very tall matrices}\label{subsec:cholesky_qrcp}

The following algorithm overcomes the limitation of the preconditioned Cholesky QR methodology from \cite{FGL:2021:CholeskyQR} of requiring full-rank data matrices.
It does so by using a randomized preconditioner based on QRCP.

\begin{algorithm}[htb]
\setstretch{1.0}
\caption{QRCP via sketch-and-precondition and Cholesky QR.}
\label{new_notation_QRRQR}
\begin{algorithmic}[1]
\State \textbf{function} $[\mtx{Q}, \mtx{R}, J] = \code{sap\_chol\_qrcp}(\mtx{A}, d)$
\vspace{0.5pt}
\Indent
    \Statex \quad Inputs:
    \Statex \begin{quote}
    A matrix $\mtx{A} \in \R^{m \times n}$, an integer $d$ satisfying $n \leq d \ll m$
    \end{quote}\vspace{4pt}
    \Statex \quad Output:
    \Statex \begin{quote}
    Column-orthonormal $\mtx{Q} \in \R^{m \times k}$, upper-triangular $\mtx{R} \in \R^{k \times n}$, and a permutation vector $J$ of length $n$.
    \end{quote}\vspace{4pt}
    \Statex  \quad Abstract subroutines:
    \Statex \begin{quote}
        $\code{SketchOpGen}$ generates an oblivious sketching operator
    \end{quote}\vspace{4pt}
     \setstretch{1.2}
    \State $\mtx{S} = \code{SketchOpGen}(d, m)$ ~~\codecomment{$\mtx{S}$ is $d \times m$}
    \State $[\mtx{Q}^{\mathrm{sk}}, \mtx{R}^{\mathrm{sk}}, J] = \code{qrcp}(\mtx{S}\mtx{A})$\label{new_notation_qrcp} \codecomment{$\mtx{S}\mtx{A}[\fslice,J] = \mtx{Q}^{\mathrm{sk}}\mtx{R}^{\mathrm{sk}}$}
    \State $k = \rank(\mtx{R}^{\mathrm{sk}})$ \label{new_notation_k_def}
    \State $\mtx{A}^{\mathrm{pre}} = \mtx{A}[\fslice{},J[\lslice{k}]](\mtx{R}^{\mathrm{sk}}[\lslice{k},\lslice{k}])^{-1}$\label{new_notation_pre} 
    \State $[\mtx{Q}, \mtx{R}^{\mathrm{pre}}] = \texttt{chol\_qr}(\mtx{A}^{\mathrm{pre}})$\label{new_notation_pre_decomp}
    \State $\mtx{R} = \mtx{R}^{\mathrm{pre}}\mtx{R}^{\mathrm{sk}}[\lslice{k},\fslice{}]$\label{new_notation_r_def}
    \State \textbf{return} $\mtx{Q}$, $\mtx{R}$, $J$
\EndIndent    
\end{algorithmic}
\end{algorithm}

\begin{remark}
    This monograph was released as a technical report in November 2022.
    It has come to our attention that \cref{new_notation_QRRQR} was discovered slightly earlier by Balabanov; it is termed \code{RRRCholesyQR2} in arXiv:2210.09953:v2 \cite{Balabanov:2022:cholQR}.
\end{remark}

The following proposition states that  \cref{new_notation_QRRQR} produces correct output in exact arithmetic, under mild assumptions on $(\mtx{S},\mtx{A})$.
We prove the proposition in \cref{app:cholqrcp}.

\begin{proposition}\label{prop:chol_exactness}
 Consider the context of \cref{new_notation_QRRQR}.
If $\rank(\mtx{S}\mtx{A}) = \rank(\mtx{A})$ then  $\mtx{A}[:, J] = \mtx{Q}\mtx{R}$.
\end{proposition}

A practical implementation of \cref{new_notation_QRRQR} would need to consider aspects of finite-precision arithmetic.
One such aspect is that we cannot use the exact rank for $\mtx{R}^{\mathrm{sk}}$ on Line \ref{new_notation_k_def}.
Instead, some tolerance-based scheme would be needed. 

In analyzing the behavior of this algorithm our main concern is the condition number of $\mtx{A}^{\mathrm{pre}}$.
Indeed, if that matrix is not well-conditioned, then the factor $\mtx{Q}$ from Cholesky QR may not be orthonormal to machine precision. 
More generally, if $\cond(\mtx{A}^{\mathrm{pre}}) \geq \epsilon^{-1/2}$ (where $\epsilon$ is the working precision), then it is possible for Cholesky QR to fail outright.

Our next proposition says that if $\mtx{A}^{\mathrm{pre}}$ is formed in exact arithmetic then its condition number depends on neither the conditioning of $\mtx{A}$ nor that of $\mtx{A}^{\mathrm{sk}}$. 
Therefore if the distribution of the sketching operator is chosen judiciously, then the algorithm should return an accurate decomposition with extremely high probability.

\begin{proposition}\label{prop:chol_independence}
    Consider the context of \cref{new_notation_QRRQR} and let $\mtx{U}$ be an orthonormal basis for the range of $\mtx{A}$.
    If $\rank(\mtx{S}\mtx{A}) = \rank(\mtx{A})$, then the singular values of $\mtx{A}^{\mathrm{pre}}$ are the reciprocals of the singular values of $\mtx{SU}$. 
\end{proposition}

\cref{prop:chol_independence} follows easily from \cref{prop:left_sketch_precond}; we omit a formal proof.

\paragraph{Application to matrices with any aspect ratio.}
Although Cholesky QR only applies to very tall matrices, one could apply it to any $m \times n$ matrix $\mtx{A}$ (with $m \geq n$) by processing the matrix in blocks.

In fact, it would be natural to use Cholesky QR as the subroutine for processing a block of columns of $\mtx{A}$ in HQRRP.
Since each iteration of HQRRP performs QRCP on a sketch of $\mtx{A}$, the triangular factor from that run of QRCP can be used as the preconditioner in processing the subsequent panel of $\mtx{A}$.
However, there is a complication in this approach.
\begin{quote}
    HQRRP's update rule for $\mtx{A}$ requires that each panel's orthogonal factor is represented as a composition of $b$ Householder reflectors, where each reflector is $m \times m$.
    By contrast, Cholesky QR only returns an explicit $m \times b$ column-orthonormal matrix $\mtx{Q}$.
\end{quote}
This issue can be resolved by using a method to restore the full Householder representation of the explicit column-orthonormal matrix $\mtx{Q}$.
In \LAPACK{}, this is done with \code{sorhr\_col}, which amounts to unpivoted LU factorization.
While pairing Cholesky QR with \code{sorhr\_col} will reduce its speed benefit, it may still be faster than Householder QR (\code{GEQRF}) and Tall-and-Skinny QR (\code{GEQR}) in certain settings.
Detailed analysis of and benchmarks for this method are forthcoming. 

\FloatBarrier

\subsection{UTV, URV, and QLP decompositions}\label{subsec:UTV_URV_QLP}

If QRCP cannot be relied upon to provide an adequate surrogate for the SVD, then one can consider decompositions of the form
\[
    \mtx{A} = \mtx{U}\mtx{T}\mtx{V}^{\trans},
\]
where $\mtx{U}, \mtx{V}$ are column-orthogonal and $\mtx{T}$ is square and triangular.
This recovers the SVD when $\mtx{T}$ is the diagonal matrix of singular values of $\mtx{A}$.
It also recovers QRCP when $\mtx{V}$ is a permutation matrix.
These decompositions were first meaningfully studied by Stewart \cite{Stewart:1992:updating_URV,Stewart:1993:SIMAX:ULV,Stewart:1999:QLP}.
They are known by various names, including \textit{UTV}, \textit{URV}, and \textit{QLP}.
We have a slight preference for the name ``UTV'' for aesthetic reasons.

\subsubsection{Deterministic algorithms}

 Stewart's best-known algorithm for UTV (see \cite{Stewart:1999:QLP}) is as follows.
\begin{enumerate}
    \item Run QRCP on the original matrix: $\mtx{A} = \mtx{Q}_1\mtx{R}_1(\mtx{P}_1)^{\trans}$.
    \item Run QRCP on $(\mtx{R}_1)^{\trans}$, to obtain $\mtx{R}_1 = \mtx{P}_2(\mtx{R}_2)^{\trans}(\mtx{Q}_2)^{\trans}$.
    \item Grouping terms, we find the factors
    \[
        \mtx{A} = \big(\underbrace{\mtx{Q}_1\mtx{P}_2}_{ \mtx{U}}\big)\underbrace{(\mtx{R}_2)^{\trans}}_{\mtx{T}}\big(\underbrace{\mtx{P}_1\mtx{Q}_2}_{\mtx{V}}\big)^{\trans}.
    \]
    Note in particular that $\mtx{T}$ is \textit{lower} triangular.
\end{enumerate}
Assuming the standard pivoting scheme is used in the second call to QRCP, one can be certain that the diagonal entries of $\mtx{T}$ are in decreasing order: $T_{ii} \geq T_{jj}$ for $j \leq i$.
Numerical experiments show that the diagonal of $\mtx{T}$ can track the singular values of $\mtx{A}$ much better than the diagonal of $\mtx{R}_1$ (see, e.g., \cite[\S 3]{Stewart:1999:QLP}).
One can find intuition for this by considering the similarities between the successive calls to QRCP with the successive calls to QR in the well-known \textit{QR iteration}.
In \cite{FHH:1999:UTV}, Stewart's UTV algorithm is even described as ``half a QR iteration.''
Remarkably, this algorithm can be modified to interleave the computation of $\mtx{R}_1$ with factoring $\mtx{R}_1$ \cite[\S 5]{Stewart:1999:QLP}.
The resulting method, like QRCP, can be stopped early at a specified rank or once some accuracy metric is satisfied.

\paragraph{Complete Orthogonal Decomposition}

There is a notion of a UTV decomposition that is not the SVD, not QRCP, and yet predates Stewart's UTV by several decades.
It is called the \textit{complete orthogonal decomposition} (COD), and it is computed by one call to QRCP followed by one call to unpivoted QR \cite{HL:1969:COD}. 
The main use of a COD is to facilitate the application of a pseudoinverse $\mtx{A}^{\dagger}$ when $\mtx{A}$ is rank-deficient.
We note that this is only modestly in line with the ``spirit'' of UTV, which asks for a decomposition that can be used as a surrogate for the SVD more generally.
Still, the COD does have some historical importance in the development of randomized UTV algorithms.

\subsubsection{Randomized algorithms}

The first randomized algorithm for UTV was described in \cite[\S 5]{DDH:2007:URV}.
It used a random orthogonal transformation as a preconditioner for computing a COD, which made it safe to replace the usual call to QRCP with a call to unpivoted QR.
This approach does not produce good surrogates for the SVD on its own,
however, it has since been extended with power iteration ideas through the \textit{PowerURV} algorithm \cite[\S 3]{GM:2018:URV}.\footnote{
    We note that the authors of \cite{DDH:2007:URV} were not trying to develop a randomized algorithm for its own sake.
    Rather, they used randomization as a tool to reduce many linear algebra problems to a format amenable to recursive unpivoted QR, which can be accelerated by black-box fast matrix multiplication methods.
}
PowerURV is able to obtain better approximations of the SVD than Stewart's UTV without using any pivoted QR decompositions.

Much of the value in Stewart's algorithm for UTV is its ability to compute the decomposition incrementally.
The earliest randomized algorithm for UTV that enjoys this capability is given in \cite[Figure~4]{MQH:2019:URV}.
Qualitatively, this algorithm can be thought of as extending the ideas of HQRRP without relying on HQRRP as a black box.
In a historical context, it is notable because it is the first full-rank UTV algorithm to use sketching (i.e., random dimension reduction) rather than random rotations.

As we wrap up the discussion on this topic, we note that one can trivially incorporate randomization into Stewart's UTV by using HQRRP for the requisite QRCP calls.
There would be a downside to this approach in that the diagonal entries of $\mtx{T}$ would not be guaranteed to decrease across block boundaries.
However, that downside could be circumvented by using HQRRP for the initial QRCP of $\mtx{A}$ and then using a standard QRCP algorithm (e.g., \LAPACK{}'s \code{GEQP3}) for the QRCP of $(\mtx{R}_1)^{\trans}$.
The speedup of such an approach over Stewart's UTV would be fundamentally limited, but it should still be observable for $n \times n$ matrices even when $n$ is 
as small as 
a few thousand. 

\section{Solving unstructured linear systems}\label{subsec:general_linear_systems}


Two broad methodologies have emerged for incorporating randomization into general linear solvers.
The first aims to ameliorate the cost of common safeguards that are applied to fast but potentially unreliable direct methods.
The second aims to restructure computations in existing general-purpose iterative methods.
There is generally more excitement in the community for methods of this second kind, but methods of the first kind remain a subject of practical interest.

\subsection{Direct methods}\label{subsec:factor_linsys}

Direct methods for solving systems of linear equations center on finding a factored representation of the system matrix.
Most famously, we have the \textit{LU decomposition} of a general $n \times n$ matrix, which takes the form
\[
    \mtx{A} = \mtx{L}\mtx{U}
\]
for a lower-triangular matrix $\mtx{L}$ with unit diagonal ($L_{ii} = 1$ for all $i$) and an upper-triangular matrix $\mtx{U}$.
For Hermitian matrices, there is the \textit{\LDLt{} decomposition}
%
%
\[
    \mtx{A} = \mtx{L}\mtx{D}\mtx{L}^{\trans} ,
\]
where $\mtx{L}$ is unit lower-triangular and $\mtx{D}$ is block diagonal with blocks of size one and~two.
%

These are some of the most fundamental matrix decompositions.
Once in hand, they can be used to solve linear systems involving $\mtx{A}$ in $O(n^2)$ operations.
The standard methods for their computation exhibit good data locality and are naturally adapted to parallel processing environments.
However, these decompositions should be used cautiously; there are some nonsingular matrices for which they do not exist, or for which they cannot be computed stably in finite precision arithmetic.
Therefore these decompositions need to be carefully modified to ensure reliability without sacrificing too much speed.


\subsubsection{Stability through randomized pivoting}\label{subsec:fullrank_LU_pivoting}

Pivoting is the standard paradigm to modify LU and \LDLt{} for improved numerical stability.
For LU, we have partial pivoting and complete pivoting, which look like
\begin{equation}\label{eq:LU_pivoting}
    \mtx{P}\mtx{A} = \mtx{L}\mtx{U} \qquad\text{and}\qquad \mtx{P}_1\mtx{A}\mtx{P}_2 = \mtx{L}\mtx{U}
\end{equation}
respectively, where $\mtx{P},\mtx{P}_1,\mtx{P}_2$ are permutation matrices.

The standard algorithms for computing these decompositions are  Gaussian elimination with partial pivoting (GEPP) and Gaussian elimination with complete pivoting (GECP).
While GEPP is substantially faster than GECP, it has weaker theoretical guarantees than GECP when it comes to numerical behavior.
In \cite{MG:2015:LU}, Melgaard and Gu propose a randomized algorithm for partially pivoted LU that makes pivoting decisions in a manner similar to HQRRP (see page \pageref{subsec:fullrank_decomp:qrcp}). 
The randomized algorithm achieves efficiency comparable to that of GEPP, while also satisfying GECP-like element-growth bounds with high probability.

For \LDLt{}, pivoted decompositions take the form
\begin{equation}\label{eq:LDLt_pivoting}
    \mtx{A} = (\mtx{P}\mtx{L})\mtx{D}(\mtx{P}\mtx{L})^{\trans},
\end{equation}
where (again) $\mtx{D}$ is block-diagonal with blocks of size one and two.
%
There are a variety of ways to introduce pivoting into \LDLt{} decompositions.
The most notable are Bunch--Kaufman \cite{BK:1977:pivoted_LDLt} and bounded Bunch--Kaufman (which uses rook pivoting) \cite{AGL:1998:pivoted_LDLt}, both of which are available in \LAPACK{}.
In \cite{FXG:2018:symindef}, Feng, Xiao, and Gu propose a randomized algorithm for pivoted \LDLt{} that is as stable as GECP and yet only slightly slower than Bunch--Kaufman and bounded Bunch--Kaufman.
%
%

\subsubsection{Stability through randomized rotations}\label{subsec:fullrank_LU_rotating}

In \cref{subsec:UTV_URV_QLP}, we mentioned how the first randomized algorithm for UTV used randomized preconditioning to compute a COD-like factorization using only \textit{unpivoted} QR decompositions.
This was not the first use of randomization to remove the need for pivoting in matrix decompositions.
In fact, this idea was explored by Parker in 1995 to remove the need for pivoting in Gaussian elimination \cite{Parker:1995:butterfly}.
Here we summarize Parker's approach.

We begin by introducing some terms.
For an integer $d \geq 1$, a \textit{butterfly matrix} of size $2d \times 2d$ is a two-by-two block matrix, with $d \times d$ diagonal matrices in each of the four blocks.
Speaking loosely, a \textit{recursive butterfly transformation} (RBT) is a product of a chain of matrices, each with butterfly matrices as diagonal blocks.
RBTs of order $n$ (i.e., RBTs of size $n \times n$) are usually analyzed when $n$ is a power of two for the sake of simplicity.
The recursive structure in RBTs makes it possible to apply them with FFT-like methods.
In particular, an RBT of order $n = 2^\ell$ can be applied to an $n$-vector in $O(n \ell)$ time.
Detailed discussion on RBTs of general order can be found in \cite{Peca-Medlin:2021:random_RBTs}.

We are interested in RBTs that are \textit{orthogonal} and \textit{random}.
The orthogonality is useful since it means the same FFT-like algorithms used to apply an RBT can be used to apply its inverse.
The randomness in orthogonal RBT stems from how one chooses the entries in the diagonal matrices.
While there are a variety of ways that this can be done \cite{Parker:1995:butterfly}, we simply speak in terms of a distribution $\mathcal{D}_n$ over orthogonal RBTs of order $n$.

One of the major contributions of \cite{Parker:1995:butterfly} was to prove that for any nonsingular matrix $\mtx{A}$ of order $n$, one can sample $\mtx{B}_1, \mtx{B}_2$ iid from a certain distribution $\mathcal{D}_n$, so that matrix $\mtx{B}_1\mtx{A}\mtx{B}_2$ has an unpivoted LU decomposition with high probability.
Put another way, the decomposition
\[
    \mtx{A} = (\mtx{B}_1)^{\trans}\mtx{L}\mtx{U}(\mtx{B}_2)^{\trans}
\]
exists with high probability.

The high speed at which RBTs can be applied and the excellent data locality properties of unpivoted matrix decompositions have led to substantial interest in RBTs from the HPC community.
For example, implementation considerations for hybrid CPU/GPU machines were studied in \cite{BDHT:2013:butterfly} (in the single-node setting) and \cite{LLD:2020:butterfly} (in the distributed setting).

The idea of using RBTs to precondition an ``unsafe'' unpivoted method naturally applies to \LDLt{}.
In this case, one obtains factorizations of the form
\[
    \mtx{A} = (\mtx{B}\mtx{L})\mtx{D}(\mtx{B}\mtx{L})^{\trans}
\]
where $\mtx{B}$ is the random RBT.
Again, this methodology has received recent attention from the HPC community; see \cite{BBBDD:2014:butterfly} for work in the multi-core distributed-memory setting \cite{BBBDD:2014:butterfly} and \cite{BDRTY:2017:butterfly} for work in the setting of a single machine with a hybrid CPU/GPU architecture.

Remarkably, although the idea of RBTs seems predicated on destroying sparsity structure present in the matrix $\mtx{A}$, the random RBT methodology can be applied to sparse matrices without catastrophic fill-in.
See \cite{BLR:2014:butterfly} for work on this topic for both general matrices and symmetric/Hermitian indefinite matrices.

\subsection{Iterative methods}\label{subsec:linsys_iterative}

\subsubsection{Background on GMRES}
GMRES is a well-known iterative method for solving linear systems of the form $\mtx{A}\vct{x} = \vct{b}$ where $\mtx{A}$ is $n \times n$ and nonsingular.
The trajectory $(\vct{x}_p)_{p \geq 1}$ it generates has a simple variational characterization.
Specifically, $\vct{x}_p$ minimizes $L(\vct{x}) = \|\mtx{A}\vct{x} - \vct{b}\|_2^2$ over all vectors $\vct{x}$ in the $p$-dimensional \textit{Krylov subspace}
\begin{equation}\label{eq:def:krylov_subspace}
    K_p = \Span{\left\{\vct{b},\,\mtx{A}\vct{b},\,\ldots,\,\mtx{A}^{p-1}\vct{b}\right\}}.
\end{equation}
The standard implementation of GMRES uses the Arnoldi process.
This can be seen as a specialization of (modified) Gram--Schmidt to orthogonalize implicitly-defined matrices of the form 
$\mtx{K}_p = [\vct{b},\,\mtx{A}\vct{b},\,\ldots,\,\mtx{A}^{p-1}\vct{b}]$.
In particular, as iterations proceed, the Arnoldi process maintains a column-orthonormal matrix $\mtx{V}_p$ where $\range(\mtx{V}_p) = K_p$.
Optionally, it can also maintain an \textit{Arnoldi decomposition}, which represents $\mtx{A}\mtx{V}_p = \mtx{V}_{p+1}\mtx{H}_{p}$ in terms of an $n \times (p+1)$ column-orthonormal matrix $\mtx{V}_{p+1}$ and a $(p+1) \times p$ upper-Hessenberg matrix $\mtx{H}_p$.\footnote{A matrix is called upper-Hessenberg if all entries below the first subdiagonal are zero.}

Letting $T_{\text{mv}}(\mtx{A})$ denote the cost of a matrix-vector multiply with $\mtx{A}$, the Arnoldi decomposition up to step $p$ can be computed in time
\[
    O(p T_{\text{mv}}(\mtx{A}) + n p^2).
\]
If we are given this decomposition, then the least squares problem defining $\vct{x}_p$ can be solved in $O(np)$ time by applying a suitable direct method.
Strictly speaking, one does not need to compute $\vct{x}_{p-1}$ to compute $\vct{x}_p$.

We summarize some ways to introduce randomness into GMRES below.
They all work by relaxing the requirement that $\mtx{V}_p$ be column-orthonormal while retaining the requirement that $\range(\mtx{V}_p) = K_p$.
Some of them work by changing the loss function $L(\vct{x})$ to be minimized by $\vct{x}_p$.
These methods are of interest when the cost of the matrix-vector multiplies is dwarfed by the complexity of maintaining the Arnoldi decomposition.
We note that this situation can only arise when $\mtx{A}$ is a sparse or otherwise structured operator.

\subsubsection{Randomized GMRES: Arnoldi decompositions in a sketch-orthogonal basis}
The method from \cite[\S~4.2]{BG:2021:GramSchmidt} can be interpreted as using a ``sketched Arnoldi process'' based on sketched Gram--Schmidt.
It works by building up $\mtx{V}_p$ so that its columns are $\mtx{S}$-orthogonal in the sense of \eqref{eq:sketched_inner_product},
where $\mtx{S}$ is a $d \times n$ sketching operator ($p \lesssim d \ll n$).
Along the way, it maintains an \textit{Arnoldi-like decomposition} $\mtx{A}\mtx{V}_p = \mtx{V}_{p+1}\mtx{H}_p$, where $\mtx{V}_{p+1}$ is likewise $\mtx{S}$-orthogonal.
Access to this decomposition at step $p$ makes it possible to minimize the loss function $\|\mtx{S}(\mtx{A}\vct{x} - \vct{b})\|_2^2$ over all $\vct{x}$ in $K_p$ in only $O(np)$ added time.

To understand the quality of the solution obtained by this method it is helpful to consider the unconstrained formulation
\begin{equation}\label{eq:krylov_lstsq_subprob}
    \min_{\vct{z}}\|\mtx{A}\mtx{V}_p\vct{z} - \vct{b}\|_2^2.
\end{equation}
GMRES would return $\vct{x}_\star = \mtx{V}_p\vct{z}_\star$ where $\vct{z}_\star$ solves \eqref{eq:krylov_lstsq_subprob} exactly.
The sketched Arnoldi approach effectively approximates this solution by applying sketch-and-solve to \eqref{eq:krylov_lstsq_subprob}.
This puts us in a position to draw from our coverage of sketch-and-solve in \cref{subsubsec:sketch_and_solve}.
If $\delta$ is the effective distortion of $\mtx{S}$ for the subspace $K_{p+1}$, then the solution $\vct{x}_{\text{sk}}$ obtained by the sketched Arnoldi approach will satisfy $\|\mtx{A}\vct{x}_{\text{sk}} - \vct{b}\|_2 \leq (1+\delta)\|\mtx{A}\vct{x}_\star - \vct{b}\|_2$.

The big-$O$ time complexity of the sketched Arnoldi process is unchanged relative to the classic Arnoldi process.
However, the flop count for the sketched process can be up to a factor of two smaller.
The sketched process also makes better use of BLAS 2 over BLAS 1, and it has fewer synchronization points compared to the classic Arnoldi process based on modified Gram--Schmidt.
Taken together, using the sketched process can significantly reduce the wallclock time needed to obtain the decomposition of $\mtx{A}\mtx{V}_p$ while retaining the reliability of classic Arnoldi.

We note that a block version of this algorithm (for linear systems with multiple right-hand sides) is presented in the preprint \cite{BG:2021:Randomized_Block_Gram_Schmidt}.
For MATLAB implementations of these methods, see \cite{Balabanov:2022:randKylov}.

\subsubsection{Randomized GMRES: handling general non-orthogonal bases}
Both classic GMRES and the randomized variant given above maintain Arnoldi-like decompositions of matrices $\mtx{A}\mtx{V}_p$ at a cost of $O(np^2)$ time complexity.
Interestingly, this cost cannot be asymptotically reduced by forgoing the decomposition of $\mtx{A}\mtx{V}_p$.
The trouble is that building $\mtx{V}_p$ with full orthogonalization -- in the standard sense or the $\mtx{S}$-orthogonal sense -- already takes $O(np^2)$ time.

In \cite{NT:2021:krylov}, Nakatsukasa and Tropp identified that \eqref{eq:krylov_lstsq_subprob} has precisely the form needed to benefit from randomized algorithms, independent from how $\mtx{V}_p$ and $\mtx{A}\mtx{V}_p$ are generated.
Based on this observation they called attention to longstanding classical methods for computing non-orthogonal bases of Krylov subspaces.
For example, one can compute $\mtx{V}_p$ by a truncated $k$-step Arnoldi process for some $k \ll p$.
This can be done in $O(npk)$ time and can easily be implemented to provide a dense representation of $\mtx{A}\mtx{V}_p$ at no added cost.
Alternatively, it may be practical to use the Chebyshev method if one has knowledge of the spectrum of $\mtx{A}$.

\cite{NT:2021:krylov} primarily advocates for approximately solving \eqref{eq:krylov_lstsq_subprob} via sketch-and-solve, where the sketched subproblem is handled by factoring $\mtx{S}\mtx{A}\mtx{V}_p$.
Note that in exact arithmetic the solutions obtained from this method would coincide with those of GMRES based on sketched Arnoldi.
On the one hand this is very appealing, since the cost of running this method for $p$ iterations can \textit{easily} undercut the $O(np^2)$ cost of sketched Arnoldi.
On the other hand, the behavior of these methods can differ in finite-precision arithmetic.
If one is too lax in building the basis matrix $\mtx{V}_p$ then the condition number of $\mtx{A}\mtx{V}_p$ can explode as $p$ increases.

All in all, the design space for this methodology is large and worth navigating with care.
Valuable advice in this regard is given throughout \cite[\S 3 -- \S 5]{NT:2021:krylov}.
One particularly compelling comment is that one could simply solve \eqref{eq:krylov_lstsq_subprob} to high accuracy via a sketch-and-precondition method, such as \cref{alg:ols_orth_sap}.
The resulting solution in this case would be very close to that produced by GMRES.

\subsubsection{Nested randomization in block-projection and block-descent methods}

Having discussed GMRES at length, we now speak to a family of iterative solvers that do not use the Krylov subspace approach.

This family came into focus with the development of \textit{sketch-and-project} -- a template iterative algorithm for solving linear systems of the form $\mtx{F}\vct{z} = \vct{g}$, where $\mtx{F} \in \R^{M \times m}$ has at least as many rows as columns ($M \geq m$) \cite{GR:2015:iterative}.
%
%
Its special cases include randomized Kaczmarz \cite{SV:2008} and randomized block Kaczmarz \cite{NT:2014}.
It also has variants that are specifically designed for overdetermined least squares problems \cite{GIG:2021:RidgeSketch}.

Without getting into the mechanics of sketch-and-project in detail, we note that these methods share a significant weakness: their convergence rates worsen as one considers larger and larger problems.
We think they are most likely to be useful when one cannot fit an $m \times m$ matrix in memory.
While such situations fall outside our primary data model, the \textit{subproblems} encountered in sketch-and-project are amenable to methods we have covered.
Indeed, the subproblems are equivalent to problems of the form 
\begin{equation}
\min_{\vct{y} \in \R^m}\{ \|\vct{y} - \vct{b} \|_2^2 \,:\, \mtx{A}^{\trans}\vct{y} = \vct{c} \}, \tag{\eqref{eq:saddle_opt_y}, revisited}
\end{equation}
where the number of columns
$n$ in $\mtx{A}$ is a user-selected tuning parameter $n \ll m \leq M$.
Such problems are clearly amenable to \cref{alg:saddle_to_ols_sap}.

Recently, a general analysis framework for randomized linear system solvers based on block projection or block descent has been proposed~\cite{PJM:2022:randBlockLinearSys}.
We refer the reader to Table 3 of~\cite{PJM:2022:randBlockLinearSys} (and appendices A.15 -- A.26) for an extensive list of new and old randomized linear system solvers that are amenable to their proposed analysis framework.
Some of these methods are distinguished in their applicability to underdetermined problems.
As with sketch-and-project, the subproblems encountered in essentially all of these methods can be chosen to have a structure amenable to \cref{alg:saddle_to_ols_sap}.

\section{Trace estimation}
\label{sec:trace_estimation}

Many scientific computing and machine learning applications require estimating the trace of a square linear operator $\mtx{A}$ that is represented implicitly.
Randomized methods are especially effective for such problems.


\subsection{Trace estimation by sampling}\label{subsec:Girard_Hutchinson}
Let $\mtx{A}$ be $n \times n$ and $\{\vct{e}_1,\ldots,\vct{e}_n\}$ be the standard basis vectors in $\R^n$.
Clearly, one can compute the trace of $\mtx{A}$ with $n$ matrix-vector products by using the identity
\[
\trace(\mtx{A}) = \sum_{i=1}^n \vct{e}_i^{\trans}\mtx{A}\vct{e}_i.
\]
Randomization creates opportunities to estimate this quantity using $m \ll n$ matrix-vector multiplications.
The most basic method uses the fact that if $\vct{\omega} \sim \mathcal{D}$ is a random vector satisfying $\E[\vct{\omega}\vct{\omega}^{\trans}] = \mtx{I}_n$, then
\[
\trace(\mtx{A}) = \E\left[\vct{\omega}^{\trans}\mtx{A}\vct{\omega}\right].
\]
It is natural to approximate the expected value by the empirical mean.
That is, upon drawing $m$ independent vectors $\vct{\omega}_i \sim \mathcal{D}$, we estimate
\begin{equation}\label{eq:HutchinsonGirard_est}
    \trace(\mtx{A}) \approx \frac{1}{m}\sum_{i=1}^m \vct{\omega}_i^{\trans}\mtx{A}\vct{\omega}_i.
\end{equation}

The idea for this method goes back to 1987 with work by Girard \cite{Girard:1987:trace_est}, who proposed that $\mathcal{D}$ be the uniform distribution over the $\ell_2$ hypersphere with radius $\sqrt{n}$.
Shortly thereafter, Hutchinson proposed that one take $\mathcal{D}$ as a distribution over Rademacher random vectors \cite{Hutchinson:1990:traceEstim}.
Hutchinson's choice of $\mathcal{D}$ minimizes the variance of the estimator when $\mtx{A}$ is fixed, while Girard's choice minimizes the worst-case variance over sets of matrices that are closed under conjugation by unitary matrices; see \cite{Epperly:2023:blog} for an explanation of this point.

We call the right-hand side of \eqref{eq:HutchinsonGirard_est} a \textit{Girard--Hutchinson estimator}.
Such estimators require $m \in \Omega(1/\epsilon^2)$ samples to approximate $\trace(\mtx{A})$ to within $\epsilon$ error for some constant failure probability.

\subsection{Trace estimation with help from low-rank approximation}\label{subsec:trace_est_by_lowrank_approx}


\subsubsection{Compress and trace}

In \cite{SAI:2017:traceLogDet}, Saibaba, Alexanderian, and Ipsen propose two randomized algorithms for estimating the trace of a psd linear operator $\mtx{A}$.

When $\mtx{A}$ is accessible by matrix-vector products, the proposed method begins with a rangefinder step to find a column-orthonormal $n \times m$ matrix $\mtx{Q}$ where $\mtx{Q}\mtx{Q}^{\trans}\mtx{A}\mtx{Q}\mtx{Q}^{\trans} \approx \mtx{A}$.
The method then approximates
\[
    \trace(\mtx{A}) \approx \trace(\mtx{Q}\mtx{Q}^{\trans}\mtx{A}\mtx{Q}\mtx{Q}^{\trans}) = \trace(\mtx{Q}^{\trans}\mtx{A}\mtx{Q}).
\]
Whether or not this bound is accurate depends on the rate of $\mtx{A}$'s spectral decay and on how well-aligned $\mtx{Q}$ is with the dominant eigenvectors of $\mtx{A}$.
This method can provide for better relative error bounds than a Girard--Hutchinson estimator
if $\mtx{A}$'s spectral decay is sufficiently fast and $\mtx{Q}$ is obtained by power iteration.

Trace estimation is especially challenging when matrix-vector products with $\mtx{A}$ are expensive.
This often happens when $\mtx{A}$ is the image of another matrix $\mtx{B}$ under a matrix function, in the sense of \cite{Higham:2008:matrix_function_book}.
Saibaba~\textit{et~al}.\ consider the case where
\[
    \mtx{A} = \log(\mtx{I} + \mtx{B})
\]
for a psd matrix $\mtx{B}$.\footnote{For any positive definite matrix $\mtx{M}$, we use  $\log(\mtx{M})$ to denote the Hermitian matrix with the same eigenvectors as $\mtx{M}$ and whose eigenvalues are the logs of the eigenvalues of $\mtx{M}$. One can verify that $ \trace(\log(\mtx{M})) = \log\det\mtx{M}$ holds for any positive definite $\mtx{M}$.}
The idea here is again to find a tall $n \times m$ column-orthonormal $\mtx{Q}$ so that $\mtx{Q}\mtx{Q}^{\trans}\mtx{B}\mtx{Q}\mtx{Q}^{\trans}$ is a good low-rank approximation of $\mtx{B}$.
Then we approximate
\[
    \trace(\mtx{A}) \approx \sum_{i=1}^m \log\left(1 + \lambda_i(\mtx{Q}^{\trans}\mtx{B}\mtx{Q})\right)
\]
where $\lambda_i(\cdot)$ returns the $i^{\text{th}}$-largest eigenvalue of the given matrix.
Error bounds can be obtained for this estimate under suitable assumptions on the spectral decay of $\mtx{B}$.
We note that some of the techniques used to prove these bounds extend to any matrix function that is operator-monotone, such as the matrix square-root.

\subsubsection{Split, trace, and approximate}

In \cite{MMMW:2021:Hutchpp}, Meyer et al.~combined ideas from low-rank approximation with the Girard--Hutchinson estimator to obtain \code{Hutch++}.
This estimator starts by sampling a matrix $\mtx{Q}$ uniformly at random from the set of $n \times m$ column-orthonormal matrices.
It then defines the low-rank approximation $\mtx{\hat{A}} = \mtx{Q}\mtx{Q}^{\trans}\mtx{A}\mtx{Q}\mtx{Q}^{\trans}$ and computes the trace of this approximation by the formula
\[
    \trace(\mtx{\hat{A}}) = \trace(\mtx{Q}^{\trans}\mtx{A}\mtx{Q}).
\]
The last phase of \code{Hutch++} applies Girard--Hutchinson to the deflated matrix
\[
    \mtx{\Delta} = (\mtx{I} - \mtx{Q}\mtx{Q}^{\trans})\mtx{A}(\mtx{I} - \mtx{Q}\mtx{Q}^{\trans}),
\]
and adds this estimate to $\trace(\mtx{\hat{A}})$.

The basic validity of \code{Hutch++} follows by splitting the trace of $\mtx{A}$ into two parts:
\[
    \trace(\mtx{A}) = \trace(\mtx{\hat{A}}) + \trace(\mtx{A} - \mtx{\hat{A}})
\]
and verifying that $\trace(\mtx{A} - \mtx{\hat{A}}) = \trace(\mtx{\Delta})$.
As a splitting and deflation approach, this method is very effective in reducing the variance of the Girard--Hutchinson estimator.
Early results along these lines can be found in \cite{GSO:2017:trace_of_inv}, which investigated the use of deflation in estimating the trace of an inverted matrix.

The initial results proven for \code{Hutch++} applied only to psd matrices.
In that context, \code{Hutch++} can (with some small \textit{fixed} failure probability) compute $\trace(\mtx{A})$ to within $\epsilon$ relative error using only $O(1/\epsilon)$ matrix-vector products.
This is a substantial improvement upon the $O(1/\epsilon^2)$ matrix-vector products that are required by plain Girard--Hutchinson estimators.
In fact, the sample complexity of \code{Hutch++} cannot be improved when considering a large class of algorithms \cite[Theorems 4.1 and 4.2]{MMMW:2021:Hutchpp}.

Persson, Cortinovis, and Kressner have since extended \code{Hutch++} so that it can proceed adaptively, only terminating once some error tolerance has been achieved (up to a \textit{controllable} failure probability) \cite{PCK:2021:improvedHutch++}. 
The analysis of their modified \code{Hutch++} method notably accommodates symmetric indefinite matrices $\mtx{A}$.
We note that the accuracy guarantees of trace estimators for indefinite matrices cannot be as strong as those for positive definite matrices.
Indeed, relative error guarantees are essentially impossible when $\trace(\mtx{A}) = 0$.
Persson et al., therefore, provide additive error guarantees in this setting.

\subsubsection{Leveraging the exchangability principle}

In \cite{ETW:2023:XTrace}, Epperly, Tropp, and Webber develop a trace estimator based on the \textit{exchangability principle}.
In the context of trace estimation, this principle stipulates that if an algorithm computes its estimate based on $m$ pairs $\{(\vct{\omega}_i, \mtx{A}\vct{\omega}_i)\}_{i=1}^m$ where $\vct{\omega}_i$ are iid random vectors, then the minimum-variance unbiased estimator for $\trace(\mtx{A})$ must be invariant under relabelings $\{ \vct{\omega}_i \}_{i=1}^m \leftarrow \{ \vct{\omega}_{\sigma(i)} \}_{i=1}^m$ for permutations $\sigma$.

\code{Hutch++} does not respect the exchangability principle, since it uses randomness in two distinct stages: first to compute the matrix $\mtx{Q}$ and then to estimate the trace of the $\mtx{\Delta}$ by a a Girard--Hutchinson estimator.

The \code{XTrace} algorithm proposed in \cite{ETW:2023:XTrace} can be thought of as a symmetrized version of \code{Hutch++}.
Given $m$ samples $\{(\vct{\omega}_i, \mtx{A}\vct{\omega}_i)\}_{i=1}^m$, its estimate is an average of $m$ runs of \code{Hutch++}, where the $j^{\text{th}}$ run uses $\mtx{Q}_j = \code{orth}([\mtx{A}\vct{\omega}_i]_{i \neq j})$ and estimates $\trace(\mtx{\Delta}_j)$ by $\vct{\omega}_j^{\trans}\mtx{\Delta}_j \vct{\omega}_j$.
Implementing \code{XTrace} naively would be very expensive.
However, as explained in \cite[\S 2.1]{ETW:2023:XTrace}, a careful implementation can achieve the same asymptotic complexity as \code{Hutch++}.
\code{XTrace} also comes with adaptive-stopping and variance estimation methods analogous to those developed in \cite{PCK:2021:improvedHutch++}.

\subsection{Estimating the trace of $f(\mathrm{B})$ via integral quadrature}\label{subsec:trace_est_integral_quadrature}

\cref{subsec:trace_est_by_lowrank_approx} touched on a method for estimating the trace of $\mtx{A} = \log(\mtx{B} + \mtx{I})$, where $\mtx{B}$ is psd and $\log(\cdot)$ is the matrix logarithm.
This section covers powerful methods for a broader class of trace estimation problems.
The original method, now known as \textit{stochastic Lanczos quadrature} (SLQ), was introduced to the linear algebra community in \cite{BFG:1996:SLQ}, was popularized by \cite{UCS:2017:SLQ}, and has since extended in a few different ways 
\cite{CH:2022:trace_est_matrix_func,PK:2022:hutchpp_matrix_functions,CTU:2022:trace_est_monograph}.

To begin, consider how any function $f : \R \to \R$ can canonically be extended to act on a Hermitian matrix by acting separately on the eigenvalues of the matrix.
That is, if we expand $\mtx{B}$ in its eigenbasis
\[
    \mtx{B} = \sum_{i=1}^n \lambda_i \vct{u}_i\vct{u}_i^{\trans},
\]
then we can define
\begin{equation*}\label{eq:def_matrix_function}
    f(\mtx{B}) = \sum_{i=1}^n f(\lambda_i) \vct{u}_i\vct{u}_i^{\trans}.
\end{equation*}
Here we cover quadrature-based methods for approximating the trace of such matrices.
The concepts behind these methods apply whenever $f$ is sufficiently smooth and $\mtx{B}$ is Hermitian.
Theoretical guarantees for these methods are usually obtained under stronger assumptions, such as $f$ being analytic on $[\lambda_n, \lambda_1]$, or $\mtx{B}$ being psd.

\subsubsection{Technical background}

The concepts we summarize below are detailed in the book \cite{GM:2010:moments_matrices_quadrature}.

\paragraph{Riemann-Stieltjes integrals.}
Let $\mu$ be a real-valued function on $\R$.
The expression
\begin{equation}\label{eq:RS_integral}
    \int_{\R} f(t)\mathrm{d}\mu(t)
\end{equation}
is called the \textit{Riemann-Stieltjes integral} of $f$ against $\mu$.
We do not provide a formal definition of this integral.
Rather, we offer two footholds for understanding it.
 First, if $\mu$ is continuously differentiable, then \eqref{eq:RS_integral} is simply the Riemann integral of $t \mapsto f(t)(\tfrac{\mathrm{d}}{\mathrm{d}t}\mu(t))$.
 Second, recall the interpretation of Riemann integration in which one identifies $\mathrm{d}t \approx t_{\ell+1} - t_{\ell}$, where $t_{\ell} < t_{\ell+1}$ are consecutive points in a partition of the region of integration.
If the analogous interpretation is applied to \eqref{eq:RS_integral}, then we would say that $\mathrm{d}\mu(t) \approx \mu(t_{\ell+1}) - \mu(t_{\ell})$.

For our purposes we can assume that $\mu$ is nondecreasing.
We also assume that there are constants $L$ and $U$ where $\mu(t) = \mu(L)$ for all $t \leq L$ and $\mu(t) = \mu(U)$ for all $t \geq U$.
Under these assumptions, \eqref{eq:RS_integral} is well-defined whenever $f$ is continuous.

\paragraph{Quadrature and orthogonal polynomials.}

An \textit{$s$-point quadrature rule} for \eqref{eq:RS_integral} specifies $s$-vectors $\vct{w}$ and $\vct{\theta}$ (of weights and nodes respectively) to define an approximation
\begin{equation}\label{eq:quadrature_approx_error}
    \int_{\R} f(t) \mathrm{d}\mu(t) \approx \sum_{\ell = 1}^s w_\ell f(\theta_\ell).
\end{equation}
The nodes and weights selected by any reliable quadrature method will depend on~$\mu$.
One prominent approach to defining quadrature rules is to require that \eqref{eq:quadrature_approx_error} holds with equality whenever $f$ is a polynomial of degree $d$, where $d$ is suitably bounded in terms of $s$.
The idea behind this is that $s$ increases, we should be able to accommodate polynomials of higher degree.

\textit{Gaussian quadrature} achieves optimal sample complexity.
This method is exact for polynomials up to degree $2s - 1$, and there is no rule that can guarantee exactness for polynomials of degree higher than $2s-1$ with only $s$ samples.
There is a deep connection between Gaussian quadrature and orthogonal polynomials that make this quadrature rule viable in practice.
The connection uses the fact that, under our assumptions on $\mu$, it can be taken to define an inner product
\[
    \langle p, q \rangle_{\mu} = \int_{\R} p(t) q(t) \mathrm{d}\mu(t).
\]
This inner product can be used to define an orthonormal basis for the set of polynomials that have at most some prescribed degree.
When this orthonormal basis is sorted by degree, the resulting sequence of polynomials must satisfy a three-term recurrence relationship \cite[\S 2]{GM:2010:moments_matrices_quadrature}.
The coefficients of the recurrence relationship up to step $s$ can be assembled in a tridiagonal matrix $J$, called the Jacobi matrix, of size $s \times s$.
The nodes and weights of Gaussian quadrature against $\mu$ can be recovered from an eigendecomposition of the Jacobi matrix \cite[Theorem 6.2]{GM:2010:moments_matrices_quadrature}.

\subsubsection{Stochastic Lanczos quadrature : approximating Girard--Hutchinson}

A Girard--Hutchinson estimator for the trace of $f(\mtx{B})$ takes the form
\begin{equation*}\label{eq:hutchinson_for_matrix_func}
    T = \frac{1}{m}\sum_{i=1}^m T_i \quad\text{where}\quad T_i = \vct{\omega}_i^{\trans}f(\mtx{B})\vct{\omega}_i
\end{equation*}
for independent random vectors $\vct{\omega}_i$ drawn from a suitable distribution.
The naive way to compute this estimator would be to call a black-box function that implements the action of $f(\mtx{B})$; for each sample $i$ one would compute $\vct{v}_i = f(\mtx{B})\vct{\omega}_i$ and then take a dot product $T_i = \vct{\omega}_i^{\trans}\vct{v}_i$.
Here we describe an alternative approach which begins with an integral representation for $T_i$ and then approximates that integral via Gaussian quadrature \cite{BFG:1996:SLQ}.
The resulting method for trace estimation is now known as stochastic Lanczos quadrature (SLQ) \cite{UCS:2017:SLQ}.

\paragraph{Integral representation of a single sample.}

Let $u$ denote the piecewise constant function that is zero for $t \leq 0$ and one for $t > 0$.
This function can be used to define a Riemann-Stieltjes integral that samples $f$ at any prescribed point.
Specifically, for any scalar $z$, we have $f(z) = \int_{\R} f(t)\mathrm{d}u(t - z)$.
Therefore upon setting
\begin{equation}\label{eq:quadratic_estimator_measure}
    \mu_i(t) = \sum_{j=1}^n \left|\vct{\omega}_i^{\trans}\vct{u}_j\right|^2 u(t - \lambda_j),
\end{equation}
the following identity is immediate from the definition of $f(\mtx{B})$:
\begin{equation}\label{eq:quadratic_estimator_as_integral}
    T_i = \int_{\R} f(t) \mathrm{d}\mu_i(t).
\end{equation}
We note that this integral is written as being over all of $\R$, but it would suffice to integrate over the interval $[\lambda_n, \lambda_1]$.

\paragraph{Quadrature of a sample's integral representation.}
Integration is often described as a continuous analog of summation.
As such, one usually thinks of quadrature as an act of approximating a continuous operation by a discrete operation.
Quadrature of Riemann-Stieltjes integrals does not always follow this pattern.
Indeed, the integral \eqref{eq:quadratic_estimator_as_integral} can already be expressed as a weighted sum of $s = n$ point evaluations of $f$, with weights $w_{i\ell} = |\vct{\omega}_{i}^{\trans}\vct{u}_{\ell}|^2$ and nodes $\theta_{i\ell} = \lambda_\ell$.
The problem with this representation is that we do not know the weights or nodes a-priori.
Therefore in the setting of Riemann-Stieltjes integration it is possible that quadrature acts as a means of approximating an unknown integrator $\mu_i$ by a known integrator $\hat{\mu}_i$ for which we can efficiently compute $\int_{\R} f(t)\mathrm{d}\hat{\mu}_i(t)$.

Enter, \textit{Lanczos quadrature}.
This is a method for computing the Gaussian quadrature rule (or variations thereof) of Riemann-Stieltjes integrals with integrators of the form \eqref{eq:quadratic_estimator_measure} \cite[\S 7]{GM:2010:moments_matrices_quadrature}.
It uses the fact that the polynomials that are orthogonal with respect to the integrator $\mu_i$ are none other than the Lanczos polynomials associated with $(\mtx{B},\vct{\omega}_i)$ (see \cite[Theorem 4.2]{GM:2010:moments_matrices_quadrature}).
Hence, the Lanczos algorithm for computing an orthonormal basis for the $s$-dimensional Krylov subspace
\[
    \Span\{ \vct{\omega}_i, \mtx{B}\vct{\omega}_i, \ldots,\mtx{B}^{(s-1)}\vct{\omega}_i \}
\]
can be used to compute the Jacobi matrix.
Given that, standard tridiagonal eigensolvers can provide us with the nodes and weights needed for Gaussian quadrature.

\paragraph{Implementation notes.}

SLQ entails approximating $m$ samples of the form $\vct{\omega}_i^{\trans}f(\mtx{B})\vct{\omega}_i$, where each $\vct{\omega}_i$ is an independent random vector drawn from some distribution $\mathcal{D}$.
The quality of each approximate sample depends on the number of nodes allowed in the Gaussian quadrature rule, and hence on the number of steps in the Lanczos algorithm.

Taking $s$ steps of the Lanczos algorithm will always require $s-1$ matrix-vector products with $\mtx{B}$.
The arithmetic and storage complexity needed to compute each sample in SLQ depends on whether we run Lanczos proper or a version of Lanczos that only computes the data needed for Gaussian quadrature.
Indeed, in the latter case we have a substantial amount of freedom to make tradeoffs between computational complexity and numerical stability.
At one end this tradeoff, $s$ iterations of Lanczos with full orthogonalization costs $O(ns)$ storage and $O(ns^2)$ arithmetic.
At the other end of the tradeoff, performing no reorthogonalization reduces the costs to only $O(n)$ storage and $O(ns)$ arithmetic.
(We emphasize that these costs do not account for the $s-1$ matrix-vector products needed with $\mtx{B}$.)

SLQ is a powerful tool, with important applications in Gaussian process regression.
For an implementation of this method that scales to petascale problems by running on GPU farms, we refer the reader to the \textsf{IMATE} Python package \cite{Ameli:2022:IMATE}.

\subsubsection{Beyond stochastic Lanczos quadrature}

\paragraph{Accelerated quadrature-based methods.}
Let $\mtx{A} = f(\mtx{B})$.
The convergence rate of SLQ for estimating $\trace(\mtx{A})$ can only be as good as a Girard--Hutchinson estimator.
As such, one needs $m \in \Omega(1/\epsilon^2)$ samples in order to estimate $\trace(\mtx{A})$ to within $\epsilon$ error for some constant failure probability.
This leaves substantial improvement for SLQ in the case when $\mtx{A}$ is positive definite, where \code{Hutch++} could make do with $m \in \Omega(1/\epsilon)$ queries to $\mtx{A}$.
Luckily, it is possible to extend SLQ to use similar splitting techniques that \code{Hutch++} employs for its variance reduction; see \cite{CH:2022:trace_est_matrix_func} and \cite{PK:2022:hutchpp_matrix_functions} for details.

\paragraph{Spectral density estimation.}

One of SLQ's remarkable properties is that its quadrature rule for approximating \eqref{eq:quadratic_estimator_as_integral} does not depend on $f$.
As such, if the quadrature nodes and weights are computed to estimate $\trace(f(\mtx{B}))$ for one function $f$, then one can use those same nodes and weights to compute an estimate for $\trace(g(\mtx{B}))$ for another function $g$.
This gives some motivation for directly estimating the function
\begin{equation}\label{eq:spectral_density}
    \phi(t) = \textstyle\sum_{j=1}^n u(t - \lambda_i)
\end{equation}
which satisfies $\int_{\R} f(t)\mathrm{d}\phi(t) = \trace(f(\mtx{B}))$ for all continuous functions $f : \R \to \R$.
Note that $\phi$ is a nonnegative nondecreasing function with $\lim_{t \to \infty} \phi(t) = n$.
As such, $\phi / n$ is a cumulative probability distribution function that can be uniquely identified with the spectrum of $\mtx{B}$.

The problem of estimating a function of the form \eqref{eq:spectral_density} is a particular case of \textit{spectral density estimation}.
This problem, which has broader applications in the physical sciences than trace estimation, has been approached explicitly with randomized algorithms \cite{Lin:2016:traceEstim}.
Approaching the linear algebraic problem of trace estimation with this in mind can lead to new insights on how to leverage prior knowledge on the structure of $f$ or $\mtx{B}$ for algorithmic purposes.
In particular, \cite{CTU:2022:trace_est_monograph} provides a systematic treatment of quadrature-based trace estimation algorithms based on this perspective.

\chapter[Advanced Sketching: Leverage Score Sampling]{Advanced Sketching: \\ Leverage Score Sampling}
\chaptermark{Leverage Score Sampling}
\label{sec7:lev_scores}

\minitoc
\bigskip

Leverage scores quantify the extent to which a low-dimensional subspace aligns with coordinate subspaces.
They are fundamental to \RandNLA{} theory since they determine how well a matrix can be approximated through sketching by row or column selection, and thus indirectly how well a matrix can be approximated by sparse data-oblivious sketching methods \cite{DM16_CACM}.
They have algorithmic uses in least squares \cites{DMM06,DMMW12_JMLR} and low-rank approximation \cites{DMM:2008,BMD09_CSSP_SODA,MD16_chapter} among other topics.
More broadly, they play a key role in statistical regression diagnostics~\cite{ChatterjeeHadi88,MMY15}.

The computational value of leverage scores stems from how they induce data-aware probability distributions over the rows or columns of a matrix.
\textit{Leverage score sampling} refers to sketching by row or column sampling according to a leverage score distribution (or an approximation thereof).
The quality of sketches produced by leverage score sampling is relatively insensitive to numerical properties of the matrix to be sketched.
This can be contrasted with sketching by uniform row or column sampling, which can perform very poorly on certain families of matrices.

Leverage score distributions can be computed exactly with standard deterministic algorithms.
However, exact computation is expensive except in very specific cases (see \cref{sec8:tensors}).
Therefore in practice it is necessary to use randomized algorithms to \textit{approximate} leverage score distributions.
On the one hand, this point is significant since the costs of the approximation algorithms undermine the efficiency gains obtained from sketching by simple row or column selection, making the cost comparable to implementing data-oblivious random projection methods.
On the other hand, uniform sampling is clearly suboptimal in many cases, e.g., in that it can miss important nonuniformity structures needed to obtain data-aware subspace embeddings.
In general, the practical utility of leverage scores derives from when row or column selection of a matrix is \textit{required} by a particular application.
Leverage scores, therefore, compete with both uniform sampling and other methods for column (or row) selection as discussed in \cref{subsec:CSS_CX_computational}.

We emphasize that we have made no concrete plans regarding \RandLAPACK{}'s support for leverage score sampling methods.
We review them here since they are prominent and sophisticated sketching methods, and they \textit{might be} appropriate to support in \RandLAPACK{} via a suite of computational routines.

In what follows we introduce three flavors of leverage scores (\S \ref{subsec:lev_scores}) and methods for approximately computing them (\S \ref{subsec:Estimation-of-leverage-scores}).
We also cover three special topics: 
\cref{subsubsec:leverage-score-sparsified} explains how leverage scores can be used to define long-axis-sparse sketching operators (in the sense of \cref{subsubsec:randblas:LongAxSparse}), and \cref{subsubsec:determinantal-point-processes,subsubsec:leverage-score-inverse} discuss generalizations of leverage scores.

\section{Definitions and background}
\label{subsec:lev_scores}

Here we cover three types of leverage scores and corresponding approaches to leverage score sampling.
The first type of leverage score (which we mean by default) is applicable to sketching in the embedding regime.
As such, it is applicable primarily to highly overdetermined least squares problems or other saddle point problems with tall data matrices.
We spend more time on this first type of leverage score since it has theoretical value in understanding the behavior of \RandNLA{} algorithms. 
The second type is used for sketching in the sampling regime and has applications in a variety of low-rank approximation problems.
The third type is specifically for approximating psd matrices (typically kernel matrices) in the presence of explicit regularization.

\subsection{Standard leverage scores}
\label{subsec:lev_scores_standard}

Let $U$ be an $n$-dimensional linear subspace of $\R^m$ and $\mtx{P}_U$ be the orthogonal projector from $\R^m$ to $U$.
The $i^{\text{th}}$ \textit{leverage score} of $U$ is
\begin{equation}\label{eq:standard_lev_scores}
    \ell_i(U) = \|\mtx{P}_U \vct{\delta}_i\|_2^2 = \mtx{P}_{U}[i, i].
\end{equation}
where $\vct{\delta}_i$ is the $i^{\text{th}}$ standard basis vector.
Collectively, leverage scores describe how well the subspace $U$ aligns with the standard basis in $\R^m$.
They have algorithmic implications when we consider induced \textit{leverage score distributions}, defined by
\begin{equation}\label{eq:standard_lev_score_distribution}
    p_i(U) = \frac{\ell_i(U)}{\sum_{j=1}^m \ell_j(U)} = \frac{\ell_i(U)}{n}.
\end{equation}

Given a matrix $\mtx{A}$, one can associate as many sets of leverage scores to that matrix $\mtx{A}$, as one can associate subspaces to $\mtx{A}$.
Two of the most important such subspaces are $U = \range(\mtx{A})$ and $V = \range(\mtx{A}^{\trans})$.
In these contexts we say that the leverage score for the $i^{\text{th}}$ row of $\mtx{A}$ is $\ell_i(U)$, while the leverage score for the $j^{\text{th}}$ column is $\ell_j(V)$.
Such leverage scores provide leverage score distributions over the rows and columns of $\mtx{A}$, respectively.
Note that only one of these distributions can be nonuniform if $\mtx{A}$ is full-rank.
Therefore when speaking of leverage scores we typically assume the $m \times n$ matrix $\mtx{A}$ is tall, which allows for the possibility that $p(U)$ is nonuniform.

Moving forward, we routinely replace $U$ by $\mtx{A}$ in \eqref{eq:standard_lev_scores} and \eqref{eq:standard_lev_score_distribution}, with the understanding that $U = \range(\mtx{A})$.

\subsubsection{Probabilistic guarantees of sketching via row sampling}

Suppose $\mtx{S}$ is a wide $d \times m$ sketching operator that implements row sampling according to a probability distribution $\vct{q}$.
We are interested in evaluating the statistical quality of $\mtx{S}$ as a row sampling operator for an $m \times n$ matrix $\mtx{A}$.
Here, our measure of sketch quality the smallest $\epsilon \in (0, 1)$ where $\vct{y} \in \range(\mtx{A})$ implies
\begin{equation}\label{eq:raw_chernoff_target}
    (1-\epsilon)\|\vct{y}\|_2^2 \leq \|\mtx{S}\vct{y}\|_2^2 \leq (1+\epsilon)\|\vct{y}\|_2^2.
\end{equation}
Note that this metric is very similar to subspace embedding distortion.
In this monograph we have generally advocated for measuring sketch quality by a scale-invariant metric called \textit{effective distortion}.
Despite this, we care about \eqref{eq:raw_chernoff_target} since it provides for the following standard result (which we prove in \cref{subapp:row_selection_theory}).

\begin{restatable}{proposition}{}\label{prop:quality_of_general_row_sampling}
    Suppose $\mtx{A}$ is an $m \times n$ matrix of rank $n$.
    If
    \[
        r := \min_{j \in \idxs{m}}\frac{q_j}{p_j(\mtx{A})}
    \]
    then for all $0 < \epsilon < 1$, we have
    \begin{align}
        \Pr\left\{ \eqref{eq:raw_chernoff_target} \text{ fails for } (\mtx{S},\mtx{A},\epsilon)  \right\}
            & \leq 2n\left(\frac{\exp(\epsilon)}{(1+\epsilon)^{(1+\epsilon)}}\right)^{r d / n}\label{eq:row_sampling_failure_prob}
    \end{align}
    and $\exp(\epsilon) < (1+\epsilon)^{(1+\epsilon)}$.
\end{restatable}

The proposition's basic message is that the probability of $\mtx{S}\mtx{A}$ being a good sketch improves as $\vct{q}$ gets closer to the leverage score distribution $p(\mtx{A})$, where ``closer'' means that the value $r$ becomes larger.
This makes it desirable for  $\vct{q}$ to approximate the leverage score distribution.
In practice, such approximations would be obtained by first estimating leverage scores (e.g., via the method described in \cref{subsec:approx_std_lev_scores}) and then normalizing according to the estimates.
That is, we compute $\hat{\ell}$ as an estimate of $\ell(\mtx{A})$, then set
\[
    q_i = \frac{\hat{\ell}_i}{\sum_{j=1}^m \hat{\ell}_j}.
\]
With this in mind, we turn to our next question:
how large should $d$ be so that the failure probability \eqref{eq:row_sampling_failure_prob} tends to zero as $n$ tends to infinity?

As a short answer, it can be shown that taking $d \in O\left(n \log n / r \epsilon^2\right)$ is sufficient for \eqref{eq:row_sampling_failure_prob} to tend to zero as $n$ tends to infinity.\footnote{
Technically, this choice of $d$ also gives an explicit rate at which the probability tends to zero, but we do not dwell on that here.
}
%
%
With (exact) leverage score sampling we are fortunate to have $r = 1$, and so it suffices for the embedding dimension to satisfy
\begin{equation}
    d_{\text{lev}} \in O\left(\frac{n \log n}{\epsilon^2}\right).
\label{eqn:mm:lev_sampling_dim}
\end{equation}
To describe the bound with uniform sampling, we introduce the \textit{coherence of $\mtx{A}$} as
\[
\mathscr{C}(\mtx{A}) := m \max_{i \in \idxs{m}}\ell_i(\mtx{A}).
\]
It is easily be shown that coherence is bounded by $n \leq \mathscr{C}(\mtx{A}) \leq m$ and that uniform sampling leads to $r = n/\mathscr{C}(\mtx{A})$.
In view of these facts, the embedding dimension for uniform sampling should be on the order of
\[
    d_{\text{unif}} \in O\left(\frac{\mathscr{C}(\mtx{A})\log n}{\epsilon^2}\right).
\]
This is no better than leverage score sampling, and it can be \textit{much} worse.

As a final point on the effectiveness of sketching by row selection methods, consider the situation of using \textit{approximate} leverage scores where we have a bound $q_j \geq \beta p_j(\mtx{A})$ for all $j$.
In such a situation we would have $\beta \leq r$ and setting $d = d_{\text{lev}}/\beta$ would suffice to achieve the same guarantees as leverage score sampling.


\subsubsection{Preconditioned leverage score sampling, hidden in plain sight}
Many data-oblivious sketching operators can be described as applying a ``rotation'' and then performing coordinate subsampling.
Here are two such examples.
\begin{itemize}
    \item A wide $d \times m$ Haar sketching operator $\mtx{S}$ can be viewed as a composition of an $m \times m$ orthogonal matrix followed by a coordinate sampling operator.
    \item The diagonal sign flip and the fast trig transform in an SRFT amounts to a rotation, and the full action of the SRFT is just applying coordinate sampling to the rotated input. 
\end{itemize}
In both cases, the rotation acts as a type of preconditioner for sampling, i.e., as a transformation that converts a given problem into a related form that is more suitable for sampling methods~\cite{DM16_CACM}.
The example of SRFTs is especially informative, since using an embedding dimension $d \in O(n \log n)$ suffices for a $d \times m$ SRFT to be a subspace embedding with constant distortion (say, distortion 1/2) with high probability~\cite{AMT:2010:Blendenpik}.

\subsubsection{Formulas for leverage scores}

There are many concrete ways to express the leverage scores of a tall $m \times n$ matrix $\mtx{A}$.
Here is an expression that emphasizes the matrix itself, without making explicit reference to its range:
\begin{equation} \label{eq:leverage-score-dist-definition}
    \ell_j(\mtx{A}) = \mtx{A}[j,:] \, (\mtx{A}^{\trans} \mtx{A})^\dagger \, \mtx{A}[j,:]^{\trans}.
\end{equation}
We can obtain other concrete expressions for the leverage scores by considering \textit{any} matrix $\mtx{U}$ whose columns form an orthonormal basis for $U = \range(\mtx{A})$.
For example, this matrix $\mtx{U}$ could be the $\mtx{Q}$ from a QR decomposition or the $\mtx{U}$ from the SVD or any other such matrix.
Any such matrix suffices since $\mtx{P} = \mtx{U}\mtx{U}^*$, as the orthogonal projector onto $U$, satisfies
\[
\ell_j(\mtx{A}) = \| \mtx{U}[j,:] \|_2^2 = (\mtx{U}\mtx{U}^{\trans})[j, j].
\]
The subspace perspective is useful since it shows that leverage scores are unchanged if $\mtx{A}$ is replaced by $\mtx{A}\mtx{A}^{\trans}$.
More generally, if $\mtx{A} = \mtx{E}\mtx{F}$ and $\mtx{F}$ has full row-rank then the leverage scores of $\mtx{E}$ match those of $\mtx{A}$.

\subsection{Subspace leverage scores}

The standard leverage scores described in \cref{subsec:lev_scores_standard} are not suitable for low-rank approximation.
The first problem is that it is perfectly reasonable to ask for a low-rank approximation of a matrix that is invertible but has many small singular values.
In such situations both the row and column leverage scores will be uniform, and hence contain no information.
The second problem is that the map from a matrix to its leverage scores is not locally continuous at $\mtx{A}$ whenever $\mtx{A}$ is rank-deficient. (As a general rule, it is difficult to solve linear algebra problems where the map from problem data to the solution is discontinuous.)

These shortcomings can partially be addressed with the concept of \textit{subspace leverage scores}, which are also called \textit{rank-$k$ leverage scores} and \textit{leverage scores relative to the best rank-$k$ approximation}; see \cite[\S 5]{DMMW12_JMLR} along with \cite{DMM:2008} as an earlier conference version of the same.

Expressing the $m \times n$ matrix $\mtx{A}$ by its compact SVD, $\mtx{A} = \mtx{U}\mtx{\Sigma}\mtx{V}^{\trans}$, the \textit{rank-$k$ leverage scores} for its range are
\[
    \ell_j^{k}(\mtx{A}) = \|\mtx{U}[j,\lslice{k}]\|_2^2.
\]
Note that the rank-$k$ leverage scores can be nonuniform regardless of the aspect ratio of the matrix.
Indeed, so long as $k < \rank(\mtx{A})$, the rank-$k$ leverage scores of both $\range(\mtx{A})$ and $\range(\mtx{A}^{\trans})$ can be nonuniform.
The problem of discontinuity of the map from a matrix to its rank-$k$ subspace leverage scores can still persist.
More generally, there is a problem that a matrix may admit multiple distinct ``best rank-$k$ approximations'' for a given value of $k$.
These problems are less troublesome if one assumes that the $k^{\text{th}}$ spectral gap $\sigma_k(\mtx{A}) - \sigma_{k+1}(\mtx{A})$ is bounded away from zero.
(This assumption is perhaps more often made than well-justified.)
Alternatively, one can consider how well the computed scores approximate the leverage scores for some ``nearby'' rank-$k$ space~\cite{DMMW12_JMLR}.

Let us turn to how subspace leverage scores are \textit{used}.
Continuing to focus on the case of row sampling, we are interested in the rank-$k$ leverage score distribution
\[
    p_j^k(\mtx{A}) = \frac{\ell_j^k(\mtx{A})}{\sum_{i=1}^m \ell_i^k(\mtx{A})}.
\]
If $\mtx{S}$ denotes a $d \times m$ row-sampling operator induced by $\vct{p}^k(\mtx{A})$, then the sketch $\mtx{Y} = \mtx{S}\mtx{A}$ leads naturally to the approximation $\Aa = \mtx{A}\mtx{Y}^{\dagger}\mtx{Y}$.
Letting $\mtx{A}_k$ denote some best-rank-$k$ approximation of $\mtx{A}$ in a unitarily invariant matrix norm ``$\|\cdot\|$,'' it is possible to choose $d$ sufficiently large so that
\begin{equation}\label{eq:lev_score_lowrank_approx_want}
    \|\Ao - \Aa\| \lesssim \|\Ao - \Ao_k\|
\end{equation}
holds with high probability.
Note that if $\mtx{Y}$ were an arbitrary matrix then it would be possible to choose $\mtx{Y}$ so that the projection $\Aa = \Ao\mtx{Y}^{\dagger}\mtx{Y}$ was equal to some best-rank-$k$ approximation of $\Ao$.
However, the restriction that the rows of $\mtx{Y}$ are scaled rows of $\mtx{A}$ significantly limits the projectors that could be used to define $\Aa$.
Because of this limitation, one may need $d \gg k$ to have any chance that \eqref{eq:lev_score_lowrank_approx_want} holds.

\begin{remark}
One rarely samples according to an exact rank-$k$ leverage score distribution in practice.
Rather, one uses randomized algorithms to approximate them.
The key fact that enables this approximation is that leverage scores (``standard'' or ``subspace'') are preserved if we replace $\mtx{A}$ by $\mtx{A}\mtx{A}^*$.
Moreover, as leverage scores quantify a notion of eigenvector localization, we should note that in many applications one has domain knowledge that eigenvalues should be localized \cite{SCS10}, and this could be used to construct approximations.
\end{remark}

 \subsection{Ridge leverage scores}\label{subsec:ridge_leverage_scores}

Ridge leverage scores are used to approximate matrices in the presence of explicit regularization.
That is, we are given an $m \times m$ psd matrix $\mtx{K}$ and a positive regularization parameter $\lambda$, and we approximate $\mtx{K} + \lambda\mtx{I}$ by $\mtx{\hat{K}} + \lambda \mtx{I}$ where $\mtx{K}$ is a psd matrix of rank $n \ll m$.
The low-rank structure in these approximations makes it much cheaper to apply $(\mtx{\hat{K}} + \lambda \mtx{I})^{-1}$ compared to $(\mtx{K} + \lambda \mtx{I})^{-1}$.
This motivates the following question. 
\begin{quote}
    What rank $n$ is needed for $(\mtx{\hat{K}} + \lambda \mtx{I})^{-1}$ to approximate $(\mtx{K} + \lambda \mtx{I})^{-1}$ up to some fixed accuracy?
\end{quote}
It turns out that this is determined by quantity $\trace(\mtx{K}(\mtx{K} + \lambda\mtx{I})^{-1})$, 
which is called the \textit{effective rank of $\mtx{K}$}.
Using $\mu_i$ to denote the $i^{\text{th}}$-largest eigenvalue of $\mtx{K}$, we can express the effective rank as
\[
    \trace(\mtx{K}(\mtx{K} + \lambda\mtx{I})^{-1}) = \sum_{i=1}^m \frac{\mu_i}{\mu_i + \lambda}.
\]

Since we are working with psd matrices it is natural to define $\mtx{\hat{K}}$ as a \Nystrom{} approximation of $\mtx{K}$ with respect to some sketching operator $\mtx{S}$ (see \cref{subsubsec:herm_eig_algs}). 
Taking that as given, this leaves the question of how to choose the distribution for $\mtx{S}$.
Here it is worth considering how many numerically-low-rank psd matrices arising in applications are defined implicitly through pairwise evaluations of a \textit{kernel function} on a given dataset.
Taking $\mtx{S}$ as a column-selection operator is especially appealing in these settings.

\cite{AM:2015:KRR} introduced ridge leverage scores as a framework for data-aware column sampling in this context.
Formally, the \textit{ridge leverage scores of $(\mtx{K},\lambda)$} are
\begin{equation}\label{eq:ridge_lev_scores_K}
    \ell_i(\mtx{K};\lambda) = \left(\mtx{K}\left(\mtx{K} + \lambda\mtx{I}\right)^{-1}\right)[i, i].
\end{equation}
In certain cases -- particularly for estimating ridge leverage scores -- it can be convenient to express these quantities in terms of a matrix $\mtx{B}$ that satisfies $\mtx{K} = \mtx{B}\mtx{B}^*$ and that has at least as many rows as columns.
Specifically, by expressing $\mtx{B}$ in terms of its compact SVD, one can show that
\begin{equation}\label{eq:ridge_lev_scores_B}
    \ell_i(\mtx{K};\lambda) = \vct{b}_i^{\trans}\left(\mtx{B}^{\trans}\mtx{B} + \lambda\mtx{I}\right)^{-1}\vct{b}_i
\end{equation}
where $\vct{b}_i^{\trans}$ is the $i^{\text{th}}$ row of $\mtx{B}$.
We note that the identity matrix appearing in \eqref{eq:ridge_lev_scores_B} will be smaller than that from \eqref{eq:ridge_lev_scores_K} if $\mtx{B}$ is not square.


\section{Approximation schemes}\label{subsec:Estimation-of-leverage-scores}

Computing leverage scores exactly is an expensive proposition.
If $\mtx{A}$ is a tall $m \times n$ matrix, then it takes $O(mn^2)$ time to compute the standard leverage scores exactly.\footnote{The preferred way to do this would be to take the row norms of the factor $\mtx{Q}$ from a thin QR decomposition of $\mtx{A}$.}
If one is interested in subspace leverage scores and $k$ is small, then one can in principle use Krylov methods to approximate the dominant $k$ singular vectors in far less than $O(mn^2)$ time.
Such methods are not very reliable for producing good approximations of the truncated SVD, but they might suffice for estimating leverage scores.
If we want to compute the ridge leverage scores of an $m \times m$ matrix $\mtx{K}$ exactly, then the straightforward implementation takes $O(m^3)$ time.

These facts necessitate the development of efficient and reliable methods for \textit{leverage score estimation}, which we discuss below.
While these methods are generally too sophisticated for the \RandBLAS{}, they may be appropriate for higher-level libraries such as \RandLAPACK{}.


\subsection{Standard leverage scores}\label{subsec:approx_std_lev_scores}

Suppose the $m \times n$ matrix $\mtx{A}$ is very tall, i.e., $m \gg n$.
Here we summarize a method by Drineas et al.\ that can compute approximate leverage scores, to within a constant multiplicative error factor, in $O(m n \log m)$ time, i.e., in roughly the time it takes to implement a random projection, with some constant failure probability bounded away from one \cite{DMMW12_JMLR}.
This can offer improved efficiency over straightforward $O(m n^2)$ approaches when $m \gg n$ and yet $m \in o(2^n)$.

We set the stage for this method by expressing leverage scores as follows
\begin{equation} \label{eq:leverage-score-estimation-step-1}
    \ell_j(\mtx{A}) = \| \vct{\delta}_j^{\trans} \mtx{U} \|_2^2 = \| \vct{\delta}_j^{\trans} \mtx{U} \mtx{U}^{\trans} \|_2^2 = \| \vct{\delta}_j^{\trans} \mtx{A} \mtx{A}^\dagger \|_2^2
\end{equation}
where we note that the second equality in the above display follows from unitary invariance of the spectral norm.
The method proceeds by approximating two operations in the right-most expression in \eqref{eq:leverage-score-estimation-step-1}. First we approximate the pseudoinverse of $\mtx{A}$ and then we approximate the matrix-matrix product $\mtx{A} \mtx{A}^\dagger$.
It is important to note that using approximations in both steps is essential for asymptotic complexity improvements, since traditional methods would take $O(mn^2)$ for the first step and $O(m^2 n)$ time for the second step.
(In extreme situations, depending on the hardware that would be used, it may be worth performing the matrix-matrix product of the second step explicitly.)

The pseudoinverse computation is approximated by applying a wide $d_1 \times m$ SRFT $\mtx{S}_1$ to the left of $\mtx{A}$. 
Letting $\mtx{U}_1 \mtx{\Sigma}_1 \mtx{V}_1^{\trans}$ be an SVD of this $d_1 \times n$ sketched matrix $\mtx{S}_1\mtx{A}$, we approximate
\begin{align*} \label{eq:leverage-score-estimation-step-2}
    \ell_j(\mtx{A}) \approx \hat{\ell}_j(\mtx{A}) 
    &= \| \vct{\delta}_j^{\trans} \mtx{A} (\mtx{S}_1 \mtx{A})^\dagger \|_2^2 \\
    &= \| \vct{\delta}_j^{\trans} \mtx{A} \mtx{V}_1 \mtx{\Sigma}_1^{-1} \mtx{U}_1^{\trans} \|_2^2 \\
    &= \| \vct{\delta}_j^{\trans} \mtx{A} \mtx{V}_1 \mtx{\Sigma}_1^{-1} \|_2^2
\end{align*}
at a cost of $O(d_1 n^2)$.
However, we are not out of the woods yet, since multiplying $\mtx{A}$ with $\mtx{V}_1 \mtx{\Sigma}_1^{-1}$ would still cost $O(m n^2)$.
This is addressed by applying a tall sketching operator $\mtx{S}_2$ of size $n \times d_2$ to the right of $\mtx{V}_1 \mtx{\Sigma}_1^{-1}$ before multiplying it by $\mtx{A}$.
That is, we further approximate
\begin{equation}
    \hat{\ell}_j(\mtx{A}) \approx \hat{\hat{{\ell}}}_j(\mtx{A}) 
    = \| \vct{\delta}_j^{\trans} \mtx{A} (\mtx{V}_1 \mtx{\Sigma}_1^{-1} \mtx{S}_2) \|_2^2.
\end{equation}
This reduces the cost of the matrix multiplication to $O(m n d_2)$ and hence the cost of the overall procedure to $O(d_1 n^2 + d_2 m n)$.
\cite{DMMW12_JMLR} gives details on how large $d_1$ and $d_2$ must be to ensure useful accuracy guarantees for the approximate leverage scores; see also \cite[\S 5.2]{MMY15} for a related evaluation.

This estimation method can be adapted to efficiently compute ``cross-leverage scores,'' as well as subspace leverage scores; see \cite{DMMW12_JMLR} for details.
It also has natural adjustments to make it faster.
For example, \cite{CW:2017:nnztime} suggest replacing the SRFT $\mtx{S}_1$ by $\tilde{\mtx{S}}_1 = \mtx{F} \mtx{C}$ where $\mtx{C}$ is a CountSketch and $\mtx{F}$ is an SRFT that further compresses the output of $\mtx{C}$; \cite{NN:2013:OSNAPs} propose replacing $\mtx{S}_1$ by a SASO (recall from \cref{subsubsec:randblas:ShortAxSparse} that a SASO is generalized CountSketch), which yields a similar speed-up as that achieved in \cite{CW:2017:nnztime}. 

\subsection{Subspace leverage scores}\label{subsec:approx_subspace_leverage_scores}

There is a wide range of possibilities for estimating subspace leverage scores.
We describe two such methods here (slightly adapted) from \cite{DMMW12_JMLR}.
Let us say that we want to estimate the rank-$k$ leverage scores of $\mtx{A}$ for some $k \ll \min\{m, n\}$.
Both of the algorithms below work by finding the \textit{exact} leverage scores of an implicit rank-$k$ matrix $\mtx{\hat{A}}$, for which a distance $\|\mtx{\hat{A}} - \mtx{A}\|$ is near-optimal among all rank-$k$ approximations.

\subsubsection{An adaptation of \cite[Algorithm 5]{DMMW12_JMLR}}

The original goal of this algorithm was to return the leverage scores of a rank-$k$ approximation of $\mtx{A}$ that was near-optimal in Frobenius norm.
Framing things more abstractly, the approach requires that the user specify an oversampling parameter $s \in O(k)$.
Its first step is to compute a rank-$(k+s)$ QB decomposition of $\mtx{A}$ (e.g., by some method from \cref{subsubsec:qb_alg}) $\mtx{A} \approx \mtx{Q}\mtx{B}$.
Next, it computes the top $k$ left singular vectors of $\mtx{B}$ by some traditional method.
Letting $\mtx{U}_k$ denote the $(k+s) \times k$ matrix of such leading left singular vectors, the algorithm takes the columns of $\mtx{Q}\mtx{U}_k$ to define approximations of the leading $k$ left singular vectors of $\mtx{A}$.
The row-norms of this matrix define the approximate rank-$k$ leverage scores.

In context, \cite[Algorithm 5]{DMMW12_JMLR} used an elementary QB decomposition with $\mtx{Q} = \code{orth}(\mtx{A}\mtx{S})$ for an $n \times (k+s)$ Gaussian operator $\mtx{S}$.
The analysis of this algorithm presumed that $s \geq \lceil k/\epsilon + 1\rceil$ for some tolerance parameter $\epsilon$.
The meaning of $\epsilon$ was as follows: when viewed as random variables, the returned leverage scores coincide with those of a rank-$k$ approximation $\Aa$ where
\[
\E\|\Aa - \Ao\|_{\text{F}}^2 \leq (1 + \epsilon)\sum_{j > k}\sigma_j(\mtx{A})^2.
\]

Looking back at this error bound from our present perspective, it is clear that a huge variety of similar bounds can be obtained by using different methods for the QB decomposition.
One possibility on this front would be to use adaptive QB algorithms that approximate $\mtx{A}$ to some prescribed accuracy.
Subspace leverage scores obtained in this way may be well-suited for approximating $\mtx{A}$ by a low-rank submatrix-oriented decomposition up to prescribed accuracy.

\subsubsection{A description of \cite[Algorithm 4]{DMMW12_JMLR}}

This is a two-stage method to find the leverage scores of a rank-$k$ approximation to $\mtx{A}$ that is near-optimal in spectral norm.

To understand the first stage, recall that some of the simplest QB algorithms make use of power iteration as described in \cref{subsubsec:data_aware}.
That is, rather than setting $\mtx{Q} = \code{orth}(\mtx{A}\mtx{S})$ for Gaussian $\mtx{S}$, they set $\mtx{S} = (\mtx{A}^*\mtx{A})^q \mtx{S}_0$ for Gaussian $\mtx{S}_0$.
Practical implementations of QB based on power iteration introduce stabilization between successive applications of $\mtx{A}$ and $\mtx{A}^*$.
Such stabilization preserves the range of $\mtx{A}\mtx{S}$, but it may change its singular vectors.
If such stabilization is \textit{not} used, then the left singular vectors of $\mtx{A}(\mtx{A}^*\mtx{A})^q \mtx{S}_0$ for Gaussian $\mtx{S}_0$ would be reasonable approximations to the leading left singular vectors of $\mtx{A}$ (modulo numerical problems that are sure to arise for moderate $q$).

The observation above is the basis for \cite[Algorithm 4]{DMMW12_JMLR}.
In context, its first stage is to compute $\mtx{S}_{q+1} = (\mtx{A}\mtx{A}^{\trans})^q\mtx{A}\mtx{S}_0$ from an $n \times 2k$ Gaussian operator $\mtx{S}_0$.
In a second stage, approximate leverage scores of $\mtx{S}_{q+1}$ -- call them $\hat{\ell}_i$ -- are obtained from any method that ensures
\[
|\hat{\ell}_i - \ell_i(\mtx{S}_{q+1})| \leq \epsilon\ \ell_i(\mtx{S}_{q+1}).
\]
These approximations are the estimates for the rank-$k$ leverage scores of $\mtx{A}$.

\cite[Lemma 15 and Theorem 16]{DMMW12_JMLR} prescribe a value for $q$ (as a function of $m, n, k$, and $\epsilon$) that ensures an approximation guarantee for the leverage score estimates given above.
Specifically, for the prescribed $q$, the estimated leverage scores are within a factor $\tfrac{1-\epsilon}{2(1+\epsilon)}$ of the leverage scores of a rank-$k$ matrix $\Aa$ that satisfies
\[
    \E\|\Aa - \Ao\|_2 \leq (1+\epsilon/10)\sigma_{k+1}(\mtx{A}).
\]
As before, the randomness in this expectation is over the randomness used to estimate the leverage scores.


\subsection{Ridge leverage scores}

A wide variety of algorithms have been devised to estimate ridge leverage scores or carry out approximate ridge leverage score sampling.
The simplest such algorithm, proposed in \cite{AM:2015:KRR} alongside the definition of ridge leverage scores, proceeds as follows:
\begin{itemize}
    \item Start with a distribution $\vct{p} = (p_i)_{i \in \idxs{m}}$ over the column index set of $\mtx{K}$.
    \item Construct a column selection operator $\mtx{S}$ with $n$ columns, where each column is independently set to $\vct{\delta}_i \in \R^m$ with probability $p_i$.
    \item Compute the \Nystrom{} approximation of $\mtx{K}$ with respect to $\mtx{S}$. Suppose the approximation is represented as $\mtx{\hat{K}} = \mtx{B}\mtx{B}^{\trans}$ for an $m \times n$ matrix $\mtx{B}$.
    \item Using $\vct{b}_i \in \R^n$ for the $i^{\text{th}}$ row of $\mtx{B}$, take $\tilde{\ell}_i := \vct{b}_i^{\trans}(\mtx{B}^{\trans}\mtx{B} + \lambda \mtx{I})^{-1}\vct{b}_i$ as an approximation for the $i^{\text{th}}$ ridge leverage score of $\mtx{K}$ with regularization $\lambda$.
\end{itemize}
One can of course start with $\vct{p} = (1/m)_{i \in \idxs{m}}$ as the uniform distribution over columns of $\mtx{K}$.
An alternative starting point is the distribution $\vct{p} = \diag(\mtx{K})/\trace(\mtx{K})$.
While the latter distribution can lead to useful theoretical guarantees (see \cite[Theorem 4]{AM:2015:KRR}) it is not suitable for computing very accurate approximations.

Iterative methods should be used if accurate approximations to ridge leverage scores are desired.
Notably, most of the iterative methods in the literature \textit{simultaneously} estimate the ridge leverage scores and sample columns from $\mtx{K}$ according to the estimates \cites{MM:2017:lev_approx_for_nystrom,CLV:2017:estimate_Nys_ridge_leverage,RCCR:2018:estimating_Nys_ridge_leverage}.
This algorithmic structure blurs the distinction between approximating ridge leverage scores and producing a \Nystrom{} approximation of $\mtx{K}$ via column selection.
This precise nature of the blurring can also vary substantially from one algorithm to another.
For example, \cite[Algorithms 2 and 3]{MM:2017:lev_approx_for_nystrom} are very different from \cite[Algorithm 1]{CLV:2017:estimate_Nys_ridge_leverage}, which in turn is materially different from \cite[Algorithms 1 and 2]{RCCR:2018:estimating_Nys_ridge_leverage}.

The abundance and sophistication of these methods make it impractical for us to summarize them here.
We instead settle for stating their general qualitative conclusions.
Letting $d = \trace(\mtx{K}(\mtx{K} + \lambda\mtx{I})^{-1})$ denote the effective rank of $\mtx{K}$, one can construct an approximation $\mtx{\hat{K}}$ of rank $n \in O(d \log d)$ for which $\|\mtx{K} - \mtx{\hat{K}}\|_2 \leq \lambda$ holds with high probability.
Furthermore, these approximations can be constructed in time $O(mn^2)$ using only $n$ column samples from $\mtx{K}$.
We refer the reader to the cited works above for details on specific algorithms.

\section{Special topics and further reading}

Here, we mention a handful of generalizations and variants of leverage score sampling that, while not part of our immediate plans, may be of longer-term interest.
The interested reader should consult the source material for details of what we describe below.
In addition to those source materials, the interested reader is referred to Sobczyk and Gallopoulos' paper \cite{SG:2021:estimatingLevScores}, which is accompanied by a carefully developed C++ and Python library called \code{pylspack} \cite{SG:2022:pylspack}.
We also recommend Larsen and Kolda's recent work \cite[\S 4 and Appendix A]{LK:2020:leverageSampleTensors} -- which includes practical advice on leverage score sampling and theoretical results with explicit constant~factors.

\subsection{Leverage score sparsified embeddings}
\label{subsubsec:leverage-score-sparsified}

Our concept of long-axis-sparse operators from \cref{subsubsec:randblas:LongAxSparse} is based on the \textit{Leverage Score Sparsified} or \textit{LESS} embeddings of Derezi{\'n}ski et al. \cite{DLDM20_sparse_TR}.
Here we explain the role of leverage scores when using these sketching operators. 

Let $\mtx{S}$ be a a random $d \times m$ long-axis-sparse operator ($d \ll m$) with sparsity parameter $k$ and sampling distribution $\vct{p} = (p_1,\ldots,p_m)$.
The idea of LESS embeddings is that varying $k$ should provide a way to interpolate between the low cost of sketching by row sampling and the high cost of sketching by Gaussian operators, while still obtaining a sketch that is meaningfully Gaussian-like.
Indeed, if $k \approx n = \rank(\mtx{A})$, then \cite{Der22_Algorithmic_TR} showed that, with high probability, the resulting sketching operator is nearly indistinguishable from a dense sub-gaussian operator (such as Gaussian or Rademacher), despite the reduction from $O(dmn)$ time to $O(dn^2)$. This performance comparison was demonstrated for several estimation tasks involving the inverse covariance matrix $\mtx{A}^{\trans}\mtx{A}$ \cite{DLDM20_sparse_TR}, as well for the Newton Sketch optimization method \cite{DLPM21_newtonless_TR}.

As with other uses of leverage scores, approximate leverage scores suffice for LESS embeddings; and the computational cost of a LESS embedding is typically dominated by the cost of estimating the leverage scores of $\mtx{A}$.
The use of leverage scores in the sparsification pattern is essential for theoretically showing that a LESS embedding exhibits nearly identical performance to a Gaussian operator for all matrices $\mtx{A}$.
Good empirical performance observed in practice, to a varying degree, also when $\vct{p}$ is the uniform distribution and $k \ll \rank(\mtx{A})$ \cite[\S 5]{DLPM21_newtonless_TR}.

\subsection{Determinantal point processes}
\label{subsubsec:determinantal-point-processes}

In many data analysis applications, submatrix-oriented decompositions such as \Nystrom{} approximation via column selection are desirable for their interpretability.
In this context, we may wish to produce a very small but high-quality sketch of the matrix $\mtx{A}$, using a method more refined (albeit slower) than leverage score sampling.
Here we discuss Determinantal Point Processes (DPPs; \cite{DM21_NoticesAMS}) as one of many such methods from the literature.

Let $\mtx{A}$ be an $m \times m$ psd matrix.
A \textit{Determinantal Point Process} is a distribution over index subsets $J\subseteq \idxs{m}$ such that:
\begin{equation*}
  \mathbb{P}(J = S) = \frac{\det(\mtx{A}[S,S])}{\det(\mtx{A}+\mtx{I})}.
\end{equation*}
The above DPP formulation is known as an L-ensemble, and it is also sometimes called volume sampling~\cites{deshpande2006matrix,DM10_TR}.
Unlike leverage score sampling, individual indices sampled in a DPP are not drawn independently, but rather jointly, to minimize redundancies in the sampling process.
In fact, a DPP can be viewed as an extension of leverage score sampling that incorporates dependencies between the samples, inducing diversity in the selected subset~\cite{kulesza2012determinantal}. 
%
%
%
%

DPP sampling can be used to construct improved \Nystrom{} approximations $\mtx{\hat A} =(\mtx{A}\mtx{S})(\mtx{S}^{\trans}\mtx{A}\mtx{S})^\dagger(\mtx{A}\mtx{S})^{\trans}$ where the selection matrix $\mtx{S}$ corresponds to the random subset $J$.
 In particular, \cites{deshpande2006matrix,guruswami2012optimal,DKM20_CSSP_neurips} established strong guarantees for this approach in terms of the nuclear norm error relative to the best rank $k$ approximation: $\|\mtx{\hat A}-\mtx{A}\|_* \leq (1+\epsilon) \|\mtx{A}_k-\mtx{A}\|_*$, where $k$ is the target rank and the subset size $|J|$ is chosen to be equal or slightly larger than $k$.
 DPPs have also found applications in machine learning \cites{kulesza2012determinantal,DKM20_CSSP_neurips,DM21_NoticesAMS} as a method for constructing diverse and interpretable data representations.

It is challenging to implement efficient methods for sampling from a DPP, and this is an area of ongoing work.
One promising method has recently been proposed by Poulson \cite{Poulson:2020:DPPs}.
Two other classes of methods can be obtained by exploiting the connection between DPPs and ridge leverage scores.
\begin{enumerate}
    \item One can use intermediate sampling with ridge leverage scores to produce a larger index set $T$, which is then trimmed down to produce a smaller DPP sample $J\subseteq T$ \cites{derezinski2019fast,derezinski2019exact,calandriello2020sampling}.
    \item One can use iterative refinement on a Markov chain, where we start with an initial subset $J_1$, and then we gradually update it by swapping out one index at a time, producing a sequence of subsets $J_1,J_2,J_3,...,$ which rapidly converges to a DPP distribution \cites{anari2016monte,anari2020isotropy,anari2022domain}.
\end{enumerate}
The computational cost of these procedures is usually dominated by the cost of estimating the ridge leverage scores (recall that there are methods for doing this that do not need to access all of $\mtx{A}$).
However, these procedures carry additional overhead since some of the ridge leverage score samples must be discarded to produce the final DPP sample.

\subsection{Further variations on leverage scores}
\label{subsubsec:leverage-score-inverse}

In the case of tall data matrices $\mtx{A}$, leverage scores are useful for finding data-aware sketching operators $\mtx{S}$ so that the Euclidean norms of vectors in the range of $\mtx{S}\mtx{A}$ are comparable to the Euclidean norms of vectors in the range of $\mtx{A}$.
A related concept called \textit{Lewis weights} can be used for matrix approximation where we want $\mtx{S}$ to approximately preserve the $p$-norm of vectors in the range of $\mtx{A}$ for some $p \neq 2$ \cite{cohen2015lp}.
These are improved versions of leverage-like scores used by Dasgupta et al. for $\ell_p$ regression \cite{DDHKM09_lp_SICOMP}.
A more recently proposed concept samples according to the probabilities
\[
    p_i = \frac{\left\|\left(\mtx{A}^{\trans}\mtx{A}\right)^{-1}\mtx{A}[i,:]\right\|_2}{\sum_{j=1}^m \left\|\left(\mtx{A}^{\trans}\mtx{A}\right)^{-1}\mtx{A}[j,:]\right\|_2}
\]
in order to estimate the \textit{variability} of sketch-and-solve solutions to overdetermined least squares problems \cite{MZXMM22_JMLR}; see also \cite{MMY15} for related importance sampling probabilities that come with useful statistical properties.
Similar probabilities (where all norms in the above expression were squared) were studied in \cite{derezinski2019minimax}.

\chapter[Advanced Sketching: Tensor Product Structures]{Advanced Sketching:
\\ Tensor Product Structures}
\chaptermark{Tensor Product Structures}
\label{sec8:tensors}

\minitoc

\vspace{1em}

This \nameCref{sec8:tensors} considers efficient sketching of data with tensor product structure.
We specifically focus on implicit matrices with Kronecker and Khatri--Rao product structure.
These structures are of interest in \RandNLA{} due to their prominent role in certain randomized algorithms for tensor decomposition.
A secondary point of interest is that the operators discussed in this \nameCref{sec8:tensors} can also be used for sketching unstructured matrices. 
They may, for example, be used as alternatives to unstructured test vectors in norm and trace estimation \cite{bujanovic2021NormTraceEstimation}.
In this case, the main benefit would not be improved speed but reduced storage requirements for storing the sketching operator.

In \cref{subsubsec:def-Kronecker-Khatri--Rao-matrices} we define the Kronecker and Khatri--Rao matrix products.
\cref{subsec:Sketching-of-Kronecker-and-Khatri--Rao} presents four families of sketching operators that can be applied efficiently to matrices that are stored implicitly with these product structures.
\cref{subsec:partial-sketch-reuse} discusses implementation considerations for the structured sketching operators in this \nameCref{sec8:tensors}, with a focus on how they can be used in tensor decomposition algorithms.

\subsubsection{A note on scope}
We should emphasize that algorithms for general tensor computations are out-of-scope for \RandLAPACK{}.
The functionality described here would only be made available as utility functions (i.e., computational routines) for \textit{facilitating} certain tensor computations.
This is part of a broader idea that \RandLAPACK{} should facilitate advanced sketching operations of interest in \RandNLA{} that are outside the scope of the \RandBLAS{}.

\section{The Kronecker and Khatri--Rao products} \label{subsubsec:def-Kronecker-Khatri--Rao-matrices}

Suppose that $\mtx{B}$ is an $m \times n$ matrix and $\mtx{C}$ is a $p \times q$ matrix.
The Kronecker product of $\mtx{B}$ and $\mtx{C}$ is the $mp \times nq$ matrix
\begin{equation*}
    \mtx{B} \otimes \mtx{C} = 
    \begin{bmatrix}
        \mtx{B}[1,1] \cdot \mtx{C} & \mtx{B}[1,2] \cdot \mtx{C} & \cdots & \mtx{B}[1,n] \cdot \mtx{C} \\
        \mtx{B}[2,1] \cdot \mtx{C} & \mtx{B}[2,2] \cdot \mtx{C} & \cdots & \mtx{B}[2,n] \cdot \mtx{C} \\
        \vdots & \vdots & & \vdots \\
        \mtx{B}[m,1] \cdot \mtx{C} & \mtx{B}[m,2] \cdot \mtx{C} & \cdots & \mtx{B}[m,n] \cdot \mtx{C} \\
    \end{bmatrix}.
\end{equation*}
If $\mtx{B}$ and $\mtx{C}$ have the same number of columns (i.e., if $n=q$), then their Khatri--Rao product is the $mp \times n$ matrix
\begin{equation*}
    \mtx{B} \odot \mtx{C} = 
    \begin{bmatrix}
        \mtx{B}[:, 1] \otimes \mtx{C}[:, 1] & \mtx{B}[:, 2] \otimes \mtx{C}[:, 2] & \cdots & \mtx{B}[:, n] \otimes \mtx{C}[:, n]
    \end{bmatrix}.
\end{equation*}
The Khatri--Rao product is sometimes also referred to as the \emph{matching columnwise Kronecker product} for transparent reasons.
The Kronecker and Khatri--Rao products for more than two matrices are defined in the obvious way.
Note that for two vectors $\vct{x}$ and $\vct{y}$ we have that
\[
\vct{x} \otimes \vct{y} = \vct{x} \odot \vct{y} = \operatorname{vec}(\vct{x} \circ \vct{y})
\]
where $\circ$ denotes the outer product and $\operatorname{vec}(\cdot)$ is an operator that turns a matrix into a vector by vertically concatenating its columns.
We also use $\circledast$ to denote the elementwise (Hadamard) product.

Matrices with Kronecker and Khatri--Rao product structure tend to be \emph{very} large.
For example, consider matrices $\mtx{B}_1, \ldots, \mtx{B}_L$, all of size $m \times n$. 
Their Kronecker product $\mtx{B}_1 \otimes \cdots \otimes \mtx{B}_L$ is an $m^L \times n^L$ matrix and their Khatri--Rao product $\mtx{B}_1 \odot \cdots \odot \mtx{B}_L$ is an $m^L \times n$ matrix.
The exponential dependence on $L$ means that these products can become very large even if the matrices $\mtx{B}_1,\ldots,\mtx{B}_L$ are not especially large.
Even just forming and storing these products may therefore be prohibitively~expensive. 

Kronecker and Khatri--Rao product matrices feature prominently in algorithms for tensor decomposition (i.e., decomposition of multidimensional arrays into sums and products of more elementary objects, see \cref{subsec:partial-sketch-reuse}).
They also appear in a variety of other contexts when sketching techniques are helpful, such as for representation of polynomial kernels \cites{pham2013FastScalable, avron2014SubspaceEmbeddings, woodruff2020NearInputSparsity, woodruff2022LeverageScoreSampling}, when fitting polynomial chaos expansion models in surrogate modeling \cites{zhou2015WeightedDiscreteLeast, seshadri2017EffectivelySubsampledQuadratures, cheng2022QuadratureSampling}, multi-dimensional spline fitting \cite{diao2018SketchingKronecker}, and in PDE inverse problems \cite{chen2020StructuredRandomSketching}.


\section{Sketching operators}
\label{subsec:Sketching-of-Kronecker-and-Khatri--Rao}

\cref{subsubsec:Row-structured-sketches} introduces sketching operators that are distinguished by having rows with particular structures.
\cref{subsubsec:Kronecker_SRFTs} discusses a variant of the SRFT with an additional tensor-produce structure.
\cref{subsubsec:TensorSketch} discusses TensorSketch operators, which are analogous to CountSketch operators from \cref{subsubsec:randblas:ShortAxSparse}.
In \cref{subsubsec:Recursive_sketch} we describe sketching operators that are recursive and have multi-stage structure. 
These incorporate some of the sketching operators discussed in the previous subsections as stepping stones.
\cref{subsubsec:tensor_product_lev_scores} covers row sampling methods for tall matrices with tensor product structure.

We note that the sketching operators in Sections~\ref{subsubsec:Row-structured-sketches}--\ref{subsubsec:Recursive_sketch} are all oblivious, whereas the sampling-based methods in \cref{subsubsec:tensor_product_lev_scores} are not.
We also note that all of the oblivious sketching operators we discuss could be applied to \textit{unstructured} matrices.
This would yield no speed benefit compared to using their unstructured counterparts, but it would reduce the storage requirement compared to traditional dense sketching operators of the kind supported by the \RandBLAS{}.

\subsection{Row-structured tensor sketching operators} \label{subsubsec:Row-structured-sketches}

Here we describe three types of sketching operators whose rows can be applied to Kronecker and Khatri--Rao product matrices very efficiently.
The second of these methods requires notions of tensor representations such as the \textit{CP format}, which we will revisit in \cref{subsec:partial-sketch-reuse}.

\subsubsection{Khatri--Rao products of elementary sketching operators}

The most basic row-structured sketching operator takes the form 
\begin{equation} \label{eq:structured-row-sketch}
    \mtx{S} = (\mtx{S}_1 \odot \mtx{S}_2 \odot \cdots \odot \mtx{S}_L)^{\trans},
\end{equation}
where each $\mtx{S}_k$ is an appropriate random matrix of size $m_k \times d$ for $k \in \idxs{L}$.
Such an operator maps $(m_1\cdots m_L)$-vectors to $d$-vectors.
It can be efficiently applied to Kronecker product vectors, which in turn means that it can be applied efficiently (column-wise) to both Kronecker and Khatri--Rao product matrices. 
Consider vectors $\vct{x}_1, \ldots, \vct{x}_L$ where $\vct{x}_k$ is a length-$m_k$ vector.
The operator in \eqref{eq:structured-row-sketch} is then applied to a vector $\vct{v} = \vct{x}_1 \otimes \cdots \otimes \vct{x}_L$ via the formula
\begin{equation*}
    \mtx{S}
    \vct{v}
    = (\mtx{S}_1^{\trans} \vct{x}_1) \circledast (\mtx{S}_2^{\trans} \vct{x}_2) \circledast \cdots \circledast (\mtx{S}_L^{\trans} \vct{x}_L).
\end{equation*}

To the best of our knowledge, \cite{biagioni2015RandomizedInterpolative} were the first to use random matrices of the form \eqref{eq:structured-row-sketch} to accelerate tensor computations in the spirit of RandNLA.\footnote{Similar ideas were used earlier for applications in differential privacy; see \cites{kasiviswanathan2010PricePrivately,rudelson2012RowProducts}.}
They suggest drawing the entries of each $\mtx{S}_k$ independently from a distribution with mean zero and unit variance, but they do not provide any theoretical guarantees for the performance of such sketching operators.
Sun et al.\ \cite{sun2018TensorRandom} independently propose using operators of the form \eqref{eq:structured-row-sketch} where the submatrices $\mtx{S}_k$ are chosen to be either Gaussian or sparse operators.
They also propose a variance-reduced modification which is an appropriate rescaling of the sum of several maps of the form \eqref{eq:structured-row-sketch}.
They provide theoretical guarantees for sketching operators in \eqref{eq:structured-row-sketch} with $L=2$ (and its variance-reduced modification) when $\mtx{S}_1$ and $\mtx{S}_2$ have entries that are drawn independently from an appropriately scaled mean-zero sub-Gaussian distribution, leaving analysis for the case when $L > 2$ open for future work.

\subsubsection{Row-wise vectorized tensors}

Rakhshan and Rabusseau \cite{rakhshan2020TensorizedRandoma} propose a distribution of sketching operators for which the $i^{\text{th}}$ row is given by $\mtx{S}[i, :] = \vectorize(\mathcal{X}_i)^{\trans}$, where $\mathcal{X}_i$ is a tensor in some factorized format and $\vectorize$ is a function that returns a vectorized version of its input as a column vector ($\vectorize(\mathcal{X}) = \mathcal{X}(:)$ in Matlab notation).
More specifically, they consider two cases:
In the first case, $\mathcal{X}_i$ is in \textit{CP format} and defined elementwise via 
\begin{equation}
    \mathcal{X}_i[j_1, j_2, \ldots, j_L] = \sum_{r=1}^R \vct{a}_r^{(i, 1)}[j_1] \cdot \vct{a}_r^{(i, 2)}[j_2] \cdots  \vct{a}_r^{(i, L)}[j_L]
\end{equation}
where the vector entries $\vct{a}_r^{(i, n)}[j_n]$ are drawn independently from an appropriately scaled Gaussian distribution.
In the second case, $\mathcal{X}_i$ is in so-called \textit{tensor train format} and defined elementwise via
\begin{equation}
    \mathcal{X}_i[j_1, j_2, \ldots, j_L] = \mtx{A}_{j_1}^{(i, 1)} \mtx{A}_{j_2}^{(i, 2)} \cdots \mtx{A}_{j_L}^{(i, L)},
\end{equation}
where $L$ is the number of tensor modes, and each matrix $\mtx{A}_{j_n}^{(i, n)}$ is of size $R_{n} \times R_{n+1}$ where $R_{1} = R_{L+1} = 1$ to ensure that the product is a scalar.
The entries of $\mtx{A}_{j_n}^{(i, n)}$ are drawn independently from an appropriately scaled Gaussian distribution.

For both of the constructs described above, the inner product of $\vectorize(\mathcal{X}_i)$ and Kronecker product vectors can be computed efficiently due to the special structure of the CP and tensor train formats.
This makes efficient application of the operator to Kronecker and Khatri--Rao product matrices possible.
Theoretical guarantees are provided for these vectorized tensor sketching operators in \cite{rakhshan2020TensorizedRandoma}.
The follow-up work \cite{rakhshan2021RademacherRandom} shows that the results for the tensor train-based sketching operators also extend to the case when the cores are drawn from a Rademacher distribution.

\subsubsection{Two-stage operators}

Iwen et al.\ \cite{iwen2020LowerMemory} propose a two-stage sketching procedure for mapping $(m_1 \cdots m_L)$-vectors to $d$-vectors.
The first step consists of applying a row-structured matrix $(\mtx{S}_1 \otimes \cdots \otimes \mtx{S}_L)$, where each $\mtx{S}_k$ is a sketching operator of size $p_k \times m_k$.
This maps the $(m_1 \cdots m_L)$-vector to an intermediate embedding space of dimension $(p_1 \cdots p_L)$.
This is then followed by another sketching operator $\mtx{T}$ of size $d \times (p_1 \cdots p_L)$ which maps the intermediate representation to the final $d$-dimensional space.

\subsection{The Kronecker SRFT} \label{subsubsec:Kronecker_SRFTs}

Kronecker SRFTs are a variant of the SRFTs discussed in Section~\ref{subsec:srfts}.
They can be applied very efficiently to a Kronecker product vector without forming the vector explicitly.
They were first proposed by \cite{battaglino2018PracticalRandomized} for efficient sketching of the Khatri--Rao product matrices that arise in tensor CP decomposition.
Theoretical analysis of these sketching operators can be found in \cites{jin2020FasterJohnson,  malik2020GuaranteesKronecker, bamberger2021JohnsonlindenstraussEmbeddings}.

The Kronecker SRFT that maps $(m_1 \cdots m_L)$-vectors to $d$-vectors takes the form
\begin{equation} \label{eq:Kronecker-SRFT}
    \mtx{S} = \sqrt{\frac{m_1 \cdots m_L}{d}} \mtx{R} \Big( \bigotimes_{k=1}^L \mtx{F}_k \Big) \Big( \bigotimes_{k=1}^L \mtx{D}_k \Big),
\end{equation}
where each $\mtx{D}_k$ is a diagonal $m_k \times m_k$ matrix of independent Rademachers, each $\mtx{F}_k$ is an $m_k \times m_k$ fast trigonometric transform, and $\mtx{R}$ randomly samples $d$ components from an $(m_1 \cdots m_L)$-vector. 
The Kronecker SRFT replaces the $\mtx{F}$ and $\mtx{D}$ operators in the standard SRFT by Kronecker products of smaller operators of the same form.
With $\vct{x}_1, \ldots, \vct{x}_L$ defined as in \cref{subsubsec:Row-structured-sketches}, the operator in \eqref{eq:Kronecker-SRFT} can be applied efficiently to $\vct{x}_1 \otimes \cdots \otimes \vct{x}_L$ via the formula
\begin{equation*}
    \sqrt{\frac{m_1 \cdots m_L}{d}} \mtx{R} \Big( \bigotimes_{k = 1}^L \mtx{F}_k \Big) \Big( \bigotimes_{k=1}^L \mtx{D}_k \Big) \Big(\bigotimes_{k=1}^L \vct{x}_k\Big) 
    = \sqrt{\frac{m_1 \cdots m_L}{d}} \mtx{R} \Big( \bigotimes_{k = 1}^L \mtx{F}_k \mtx{D}_k \vct{x}_k \Big).
\end{equation*}
The formula shows that only those entries in $\bigotimes_k \mtx{F}_k \mtx{D}_k \vct{x}_k$ that are sampled by $\mtx{R}$ need to be computed.
From this, we can back out the indices of each vector $\mtx{F}_k \mtx{D}_k \vct{x}_k$ that need to be computed.
Given these indices one could compute the relevant entries of these vectors using subsampled FFT methods of the kind alluded to in \cref{subsec:srfts}.
We note that this formula is straightforwardly extended to Kronecker and Khatri--Rao product matrices.


\subsection{TensorSketch} \label{subsubsec:TensorSketch}

A TensorSketch operator is a kind of structured CountSketch that can be applied very efficiently to Kronecker product matrices.\footnote{Recall that a CountSketch is a SASO in the sense of \cref{subsubsec:randblas:ShortAxSparse}. Each short-axis vector in a CountSketch has a single nonzero entry,  sampled from the Rademacher distribution.} 
The improved computational efficiency of TensorSketch comes at the cost of needing a larger embedding dimension than CountSketch. 
TensorSketch was first proposed in \cite{pagh2013CompressedMatrix} for fast approximate matrix multiplication.
It was further developed in \cites{pham2013FastScalable, avron2014SubspaceEmbeddings, diao2018SketchingKronecker} where it is used for low-rank approximation, regression, and other tasks.

Let $\vct{x}_1,\ldots,\vct{x}_L$ be defined as in Section~\ref{subsubsec:Row-structured-sketches}.
A TensorSketch, which we denote by $\mtx{S}$ below, maps an $(m_1 \cdots m_L)$-vector $\vct{v} = \vct{x}_1 \otimes \cdots \otimes \vct{x}_L$ to a $d$-vector via the formula
\begin{equation} \label{eq:TensorSketch-vector-formula}
    \mtx{S}\vct{v} = \FFT^{-1} \Big( \startimes_{k=1}^L \FFT(\mtx{S}_k \vct{x}_k) \Big),
\end{equation}
where each $\mtx{S}_k$ is an independent CountSketch that maps $m_k$-vectors to $d$-vectors.
Here, $\FFT$ denotes the discrete Fourier transform which can be efficiently applied using fast Fourier transform (FFT) methods. 
TensorSketches use the fact that polynomials can be multiplied using the DFT, which is why DFT and its inverse appear in the formula above; see \cite{pagh2013CompressedMatrix} for details.

\begin{remark}
We have not investigated whether fast trig transforms other than the discrete Fourier transform (e.g., the discrete cosine transform) can be used for this type of sketching operator.
\end{remark}

\subsection{Recursive sketching}\label{subsubsec:Recursive_sketch}

In order to achieve theoretical guarantees, the sketching operators discussed so far require an embedding dimension $d$ which scales \emph{exponentially} with $L$ when embedding a vector of the form $\vct{x}_1 \otimes \cdots \otimes \vct{x}_L$.
Ahle et al.\ \cite{ahle2020ObliviousSketchingSIAM} propose sketching operators that are computed recursively and have the remarkable property that their requisite embedding dimensions scale \emph{polynomially} with $L$.
Since \cite{ahle2020ObliviousSketchingSIAM} are concerned with oblivious subspace embedding of polynomial kernels, they consider the case when all $\vct{x}_1, \ldots, \vct{x}_L$ are of the same length. 
However, their results should extend to the general case when the vectors are of different lengths (for example, see \cite[Corollary~18]{malik2021MoreEfficient}).

Suppose $\vct{x}_1, \ldots, \vct{x}_L$ are $m$-vectors and that $L = 2^q$ for a positive integer $q$.
The recursive sketching operator first computes
\[
\vct{y}^{(0)}_k = \mtx{T}_k \vct{x}_k \quad\text{for}\quad k \in \idxs{L}
\]
where $\mtx{T}_1,\ldots,\mtx{T}_L$ are independent SASOs (e.g., CountSketches, see  \cref{subsubsec:randblas:ShortAxSparse}) that map $m$-vectors to $d$-vectors.
The $d$-vectors $\vct{y}_1^{(0)},\ldots,\vct{y}_L^{(0)}$ are now combined pairwise into $L/2=2^{q-1}$ vectors.
This is done by computing
\[
\vct{y}_k^{(1)} = \mtx{S}_k (\vct{y}_{2k-1}^{(0)} \otimes \vct{y}_{2k}^{(0)}) \quad\text{for}\quad k \in \idxs{L/2}
\]
where $\mtx{S}_1,\ldots,\mtx{S}_{L/2}$ are independent sketching operators that map $d^2$-vectors to $d$-vectors.
If the initial $\mtx{T}_1,\ldots,\mtx{T}_L$ are CountSketches then the $\mtx{S}_i$ are canonically TensorSketches.
If instead $\mtx{T}_1,\ldots,\mtx{T}_L$ are more general SASOs then the $\mtx{S}_i$ are canonically Kronecker SRFTs.
Regardless of which configuration we use, the pairwise combination of vectors is repeated for a total of $q = \log_2(L)$ steps until a single $d$-vector remains, which is the embedding of $\vct{x}_1 \otimes \cdots \otimes \vct{x}_L$.
The case when $L$ is not a power of two is handled by adding additional vectors $\vct{x}_{k} = \vct{e}_1$ for $k = L+1, \ldots, 2^{\lceil \log_2(L) \rceil}$ where $\vct{e}_1$ is the first standard basis vector in $\R^m$.
Recursive sketching operators are linear despite their somewhat complicated description.


Song et al.\ \cite{song2021FastSketchingPolynomial} develop a similar recursive sketching operator which takes inspiration from the one discussed above and applies it to the sketching of polynomial kernels.
For the degree-$L$ polynomial kernel, this involves sketching of matrices of the form $\mtx{A}^{\odot L} = \mtx{A} \odot \cdots \odot \mtx{A}$, where the matrix $\mtx{A}$ appears $L$ times in the right-hand~side.

The recursive sketching operator by \cite{ahle2020ObliviousSketchingSIAM} can be described by a binary tree, with each node corresponding to an appropriate sketching operator.
Ma and Solomonik \cite{ma2022CostEfficientGaussian} generalize this idea by allowing for other graph structures, but limit nodes in these graphs to be associated with Gaussian sketching operators. 
Under this framework, they develop a structured sketching operator whose embedding dimension only scales \emph{linearly} with $L$.
These operators can be adapted for efficient application to vectors with general tensor network structure which includes Kronecker products of vectors as a special case.

\subsection[Leverage score sampling for implicit matrices with tensor product structures]{Leverage score sampling for implicit matrices with tensor product structures}
\subsectionmark{Leverage score samplers}
\label{subsubsec:tensor_product_lev_scores}

Consider the problem of sketching and solving a least squares problem
\begin{equation}\label{eq:tensor_least_squares}
    \min_{\mtx{x}}\| \mtx{A} \mtx{X} - \mtx{Y}\|_{\mathrm{F}}
\end{equation}
when the columns of $\mtx{A}$ have tensor product structure and $\mtx{Y}$ is a thin unstructured matrix.
The sketching operators discussed so far in this \nameCref{sec8:tensors} can be efficiently applied to $\mtx{A}$.
However, since $\mtx{Y}$ lacks structure, these sketching operators require accessing all nonzero elements of $\mtx{Y}$.
This can be prohibitively expensive in applications such as the following. 
\begin{itemize}
    \item In iterative methods for tensor decomposition, one typically solves a sequence of least squares problems for which $\mtx{A}$ is structured and $\mtx{Y}$ contains all the entries of the tensor being decomposed \cite{kolda2009TensorDecompositions}.
    When $\mtx{Y}$ has a fixed proportion of nonzero entries, the cost will therefore scale exponentially with the number of tensor indices---a manifestation of the \emph{curse-of-dimensionality}.
    \item When fitting polynomial chaos expansion functions in surrogate modeling \cites{zhou2015WeightedDiscreteLeast, seshadri2017EffectivelySubsampledQuadratures, cheng2022QuadratureSampling}, $\mtx{A}$ contains evaluations of a multivariate polynomial on a structured quadrature grid and $\mtx{Y}$ (which will now be a column vector) contains the outputs of an expensive data generation process (e.g., an experiment or high-fidelity PDE simulation).
\end{itemize}

In both example applications, it is clearly desirable to avoid using all entries of $\mtx{Y}$ when solving \eqref{eq:tensor_least_squares}.
As discussed in \cref{sec7:lev_scores}, leverage score sampling can be used to sketch-and-solve least squares problems without accessing all entries of the right-hand side $\mtx{Y}$ while still providing performance guarantees.
Here we discuss how to take advantage of the structure of $\mtx{A}$ to speed up leverage score sampling.


\subsubsection{Kronecker product structure}
Consider a Kronecker product $\mtx{A} = \mtx{A}_1 \otimes \cdots \otimes \mtx{A}_L$ of $m_k \times n_k$ matrices $\mtx{A}_k$.
It is possible to perform \emph{exact} leverage score sampling on $\mtx{A}$ without even forming it. 
Cheng et al.\ \cite{cheng2016SPALSFast} used this fact to approximately solve least squares problems with Kronecker product design matrices, which has applications in algorithms for Tucker tensor decomposition.
Formal statements and proofs of these results later appeared in~\cite{diao2019OptimalSketching}.

To see how the sampling works, let $(p_i)$ be the leverage score sampling distribution of $\mtx{A}$ and let $(p_{i_k})$ be the leverage score sampling distribution of $\mtx{A}_k$ for $k\in\idxs{L}$.
For any $i \in \idxs{\prod_{k=1}^L m_k}$ and corresponding multi-index $(i_1, \ldots, i_L)$ satisfying
\begin{equation} \label{eq:kronecker-row-sampling}
    \mtx{A}[i,:] = \mtx{A}_1[i_1, :] \otimes \cdots \otimes \mtx{A}_L[i_L,:],
\end{equation}
it holds that 
\begin{equation} \label{eq:product-lev-score-distribution}
    p_i = p_{i_1}^{(1)} p_{i_2}^{(2)} \cdots p_{i_L}^{(L)}.
\end{equation}
Therefore, instead of drawing an index $i$ according to $(p_i)$, one can draw the index $i_k$ according to $(p_{i_k}^{(k)})$ for each $k \in \idxs{L}$.
Due to \eqref{eq:kronecker-row-sampling}, the row corresponding to the drawn index can be computed and rescaled without constructing $\mtx{A}$.
This process can be easily adapted to drawing multiple samples.

Fahrbach et al.\ \cite{fahrbach2022SubquadraticKronecker} discuss how the sampling approach above can be adapted for use in ridge regression when the design matrix is a Kronecker product.
Malik et al.\ \cite{malik2022FastAlgorithms} show an approach for efficient sampling according to the exact leverage scores of matrices of the form $\mtx{A}[:,J]$ when $\mtx{A}$ is a Kronecker product and $J$ is an index vector that satisfies certain monotonicity properties.

\subsubsection{Khatri--Rao product structure}
Sampling according to the leverage scores of a Khatri--Rao product matrix $\mtx{A} = \mtx{A}_1 \odot \cdots \odot \mtx{A}_L$ is more challenging than it is for a Kronecker product matrix.
Still, several approaches for doing so have been proposed. 
We divide them into two categories.
The methods in the first category sample according to the leverage scores  of the \emph{Kronecker product} of $\mtx{A}_1, \ldots, \mtx{A}_L$ instead of the Khatri--Rao product since this allows for simple and efficient sampling.
This can be viewed as sampling from a coarse approximation of the Khatri--Rao product leverage scores.
The methods in the second category sample according to exact or high-accuracy approximations of the Khatri--Rao product leverage score distribution.

\paragraph{Sampling according to Kronecker product leverage scores}

As noted by \cites{cheng2016SPALSFast, battaglino2018PracticalRandomized}, the leverage scores of $\mtx{A}$ can be upper bounded by
\begin{equation} \label{eq:kr-leverage-score-upper-bound}
    \ell_i(\mtx{A}) \leq \prod_{k=1}^L \ell_{i_k}(\mtx{A}_k),
\end{equation}
where $(i_1,\ldots,i_L)$ is the multi-index corresponding to $i$.
The two papers \cites{cheng2016SPALSFast, LK:2020:leverageSampleTensors}  use the expression on the right-hand side of \eqref{eq:kr-leverage-score-upper-bound} as an approximation to the exact leverage scores on the left-hand side. 
By using the bound \eqref{eq:kr-leverage-score-upper-bound}, they are able to prove theoretical performance guarantees when this approach is used for sketch-and-solve in least squares problems.
More precisely, Cheng et al.\ \cite{cheng2016SPALSFast} sample according to a mixture of the distribution in \eqref{eq:product-lev-score-distribution} and a distribution which depends on the magnitude of the dependent variables (i.e., the entries in the ``right-hand sides'' in the least squares problem).
Larsen and Kolda \cite{LK:2020:leverageSampleTensors} sample with respect to only the distribution in \eqref{eq:product-lev-score-distribution}.
Bharadwaj et al.\ \cite{bharadwaj2022DistributedMemoryRandomized} extend the work by \cite{cheng2016SPALSFast, LK:2020:leverageSampleTensors} to a distributed-memory setting and provide high-performance parallel implementations.
Ideas similar to those in \cite{LK:2020:leverageSampleTensors} are developed for the more complicated design matrices that arise in algorithms for tensor ring decomposition in \cite{malik2021SamplingBasedMethod}.
Those matrices have columns that are sums of vectors with Kronecker product structure.

\paragraph{Sampling according to exact or high-quality approximations of leverage scores}

Malik \cite{malik2021MoreEfficient} proposes a different approach for the Khatri--Rao product least squares problem.
By combining some of the ideas for fast leverage score estimation (see  \S \ref{subsec:Estimation-of-leverage-scores}) and recursive sketching (see \S \ref{subsubsec:Recursive_sketch}) with a sequential sampling approach, he improves the sampling and computational complexities of \cite{LK:2020:leverageSampleTensors}.
Malik et al.\ \cite{malik2022SamplingBasedDecomposition} simplify and generalize the method by \cite{malik2021MoreEfficient} to a wider family of structured matrices.
An upshot of this work is a method for efficiently sampling a Khatri--Rao product matrix according to its \emph{exact} leverage score distribution without forming the matrix.


Motivated by applications in kernel methods, \cite{woodruff2020NearInputSparsity} develop a recursive leverage score sampling method for sketching of matrices of the form $\mtx{A}^{\odot L} = \mtx{A} \odot \cdots \odot \mtx{A}$.
Their method starts by sampling from a coarse approximation to the leverage score sampling distribution and then iteratively refining it.
These ideas are further refined in \cite{woodruff2022LeverageScoreSampling} where the method is also extended to general Khatri--Rao products of matrices that can all be distinct.

\section{Partial updates to Kronecker product sketches} \label{subsec:partial-sketch-reuse}


The structured sketching operators discussed in \cref{subsec:Sketching-of-Kronecker-and-Khatri--Rao} are notable in that they are defined in terms of multiple smaller sketching operators.
Here we discuss situations when it is advantageous to reuse some of these smaller sketches across multiple calls to the structured sketching operator.
The examples we discuss come from works that use sketching in tensor decomposition algorithms.
Our goal with this discussion is to bring to light some of the functionality we think is important for structured sketches to have in order to best support potential usage in the tensor~community. 

By tensor, we mean multi-index arrays containing real numbers.
A tensor with $L$ indices is called an \textit{$L$-way tensor}.
Vectors and matrices are one- and two-way tensors, respectively.
Calligraphic capital letters (e.g., $\mathcal{X}$) are used to denote tensors with three or more indices.
Much like matrix decomposition, the purpose of tensor decomposition is to decompose a tensor into some number of simpler components.
We only give minimal background material on tensor decomposition here; see the review papers \cites{kolda2009TensorDecompositions, cichocki2015TensorDecompositions, sidiropoulos2017TensorDecomposition} for further details.

\subsection{Background on the CP decomposition}\label{subsubsec:fancy_sketch:cp_decomp_ols}

We first consider the CANDECOMP/PARAFAC (CP) decomposition which is also known as the canonical polyadic decomposition \cite[\S 3]{kolda2009TensorDecompositions}.
It decomposes an $L$-way tensor $\mathcal{X}$ of size $m_1 \times m_2 \times \cdots \times m_L$ into a sum of $R$ rank-1 tensors:
\begin{equation} \label{eq:CP-decomposition}
    \mathcal{X} = \sum_{r=1}^R \vct{a}_r^{(1)} \circ \vct{a}_r^{(2)} \circ \cdots \circ \vct{a}_r^{(L)},
\end{equation}
where $\circ$ denotes the outer product.
The $m_n \times R$ matrices $\mtx{A}^{(n)} = \begin{bmatrix} \vct{a}_1^{(n)} & \cdots & \vct{a}_R^{(n)} \end{bmatrix}$ for $n \in\idxs{L}$ are called factor matrices. 
When $R$ is sufficiently large, we can express the factor matrices as the solution to
\begin{equation} \label{eq:CP-decomposition-opt}
    \argmin_{\mtx{A}^{(1)}, \ldots, \mtx{A}^{(L)}} \Big\| \mathcal{X} - \sum_{r=1}^R \vct{a}_r^{(1)} \circ \vct{a}_r^{(2)} \circ \cdots \circ \vct{a}_r^{(L)} \Big\|_{\text{F}},
\end{equation}
where $\|\cdot\|_{\text{F}}$ denotes the Frobenius norm as generalized to tensors in the obvious way.
The broader problem of tensor decomposition is concerned with approximately solving \eqref{eq:CP-decomposition-opt}.
In particular, it is common to seek locally optimal solutions to this problem even when $R$ is too small for an identity of the form \eqref{eq:CP-decomposition} to hold for $\mathcal{X}$.

It is computationally intractable to solve \eqref{eq:CP-decomposition-opt} in the general case.
However, the problem admits several heuristics that are effective in practice.
One of the most popular heuristics is alternating minimization, wherein one solves for only one factor matrix at a time while keeping the others fixed.
That is, one solves a sequence of problems of the form
\begin{equation} \label{eq:CP-decomposition-opt-ALS}
    \mtx{A}^{(n)} = \argmin_{\mtx{A}} \Big\| \mathcal{X} - \sum_{r=1}^R \vct{a}_r^{(1)}  \circ \cdots \circ \vct{a}_r^{(n - 1)} \circ \vct{a}_r \circ \vct{a}_r^{(n+1)} \circ \cdots \circ \vct{a}_r^{(L)} \Big\|_{\text{F}}
\end{equation}
for $n \in \idxs{L}$.
If we adopt appropriate notation then \eqref{eq:CP-decomposition-opt-ALS} can be stated as a familiar linear least squares problem.
Accordingly, this alternating minimization approach is called \textit{alternating least squares} (ALS).
The ALS approach cycles through the indices $n \in \idxs{L}$ multiple times until some termination criteria is met.
Typical termination criteria include reaching a maximum number of iterations or seeing that the improvement in the objective falls below some threshold.

\subsubsection{Formulating and solving the least squares problem}

We begin by introducing flattened representations of $\mathcal{X}$.
Specifically, for $n \in \idxs{L}$, the $m_n \times \left(\prod_{j \neq n} m_j\right)$ matrix $\mtx{X}_{(n)}$ is given by horizontally concatenating the mode-$n$ \emph{fibers} 
$\mathcal{X}[i_1, \ldots, i_{n-1}, :, i_{n+1}, \ldots, i_N]$ as columns.
Such a matrix can be expressed in Matlab notation as follows
\begin{equation} \label{eq:mode-n-unfolding-def}
    \mtx{X}_{(n)} = \code{reshape}\left(\code{permute}(\mathcal{X}, [n, 1, \ldots, n-1 , n+1, \ldots, L]), m_n, \prod_{j \neq n} m_j\right).
\end{equation}
Next, we introduce the following flattened tensorizations of the factor matrices:
\begin{equation} \label{eq:A-neq-n}
    \mtx{A}^{\neq n} 
    := \mtx{A}^{(L)} \odot \cdots \odot \mtx{A}^{(n+1)} \odot \mtx{A}^{(n-1)} \odot \cdots \odot \mtx{A}^{(1)} 
    =: \bigodot_{\substack{j = L \\ j \neq n}}^1 \mtx{A}^{(j)}.
\end{equation}
Where we emphasize that the order of the matrices in the above product is important; our notation for the Khatri--Rao product at right reflects how the order proceeds from $j=L$ to $j=1$, skipping $j = n$. 

In terms of these matrices, the ALS update rule for the $n^{\text{th}}$ factor matrix is
\begin{equation} \label{eq:CP-factor-matrix-update}
    \mtx{A}^{(n)} 
    = \argmin_{\mtx{A}} \big\| \mtx{A}^{\neq n} \mtx{A}^{\trans} - \mtx{X}_{(n)}^{\trans} \big\|_{\text{F}}.
\end{equation}
We note that the Gram matrix for this least squares problem can be computed efficiently by the formula 
\begin{align}\label{eq:gram_of_katrirao}
    \mtx{A}^{\neq n \trans} \mtx{A}^{\neq n} &= (\mtx{A}^{(L) \trans} \mtx{A}^{(L)}) \circledast \cdots \circledast (\mtx{A}^{(n+1) \trans} \mtx{A}^{(n+1)}) \nonumber \\ & \qquad \quad  \circledast (\mtx{A}^{(n-1) \trans} \mtx{A}^{(n-1)}) \circledast \cdots \circledast (\mtx{A}^{(1) \trans} \mtx{A}^{(1)}).
\end{align}
Therefore solving the least squares problem in \eqref{eq:CP-factor-matrix-update} via the normal equations can be very efficient \cite[\S 3.4]{kolda2009TensorDecompositions}.
Indeed, the ALS update rule for the $n^{\text{th}}$ factor matrix becomes
\begin{equation} \label{eq:normal-equation-update}
    \mtx{A}^{(n)} = \mtx{X}_{(n)} \mtx{A}^{\neq n} (\mtx{A}^{\neq n \trans} \mtx{A}^{\neq n})^\dagger.
\end{equation}
The most expensive part of this update is actually computing $\mtx{X}_{(n)} \mtx{A}^{\neq n}$ \cite[\S 3.1.1]{battaglino2018PracticalRandomized}, which is analogous to the vector $\mtx{F}^{*}\vct{h}$ in the normal equations for an overdetermined least squares problem $\min_{\vct{z}}\|\mtx{F}\vct{z} - \vct{h}\|_2^2$.
Therefore, the fact that computing this matrix is the computational bottleneck in solving \eqref{eq:CP-factor-matrix-update} is the opposite of what one would expect when not working with tensors.
This phenomenon is why row-sampling sketching operators have been successful in ALS algorithms that use sketch-and-solve for the least squares subproblems \cite{LK:2020:leverageSampleTensors}. 

\begin{remark}
Although it is cheap to form the Gram matrix \eqref{eq:gram_of_katrirao}, there is potential for \textit{very} bad conditioning even when $L$ is small.
We do not know how seriously the poor conditioning affects ALS approaches to CP decomposition in practice.
\end{remark}

\subsection{Sketching for the CP decomposition}


Battaglino et al.\ \cite{battaglino2018PracticalRandomized} apply the Kronecker SRFT from Section~\ref{subsubsec:Kronecker_SRFTs} to the least squares problem in \eqref{eq:CP-factor-matrix-update}.
Letting $\mtx{T}_j$ and $\mtx{F}_j$ be of size $m_j \times m_j$, the sketching operator used before solving for the $n^{\text{th}}$ factor matrix is
\begin{equation} \label{eq:CP-decomposition-Kronecker-SRFT}
    \mtx{S}_n = \sqrt{ \frac{\prod_{\substack{j = 1\\j \neq n}}^L m_j}{d}} \mtx{R} \Big( \bigotimes_{\substack{j = L \\ j \neq n}}^1 \mtx{T}_j \Big) \Big( \bigotimes_{\substack{j = L \\ j \neq n}}^1 \mtx{F}_j \Big).
\end{equation}
Our notation for the Kronecker product operator indexes from $j = L$ to $j = 1$ so as to mimic our earlier notation for the Khatri--Rao product (see \eqref{eq:A-neq-n}).

A by-the-book application of this operator would require drawing new $\mtx{R}$ and $(\mtx{F}_j)_{j \neq n}$ every time it is applied in \eqref{eq:CP-factor-matrix-update}, i.e., $L$ times for every execution of the for loop.
Battaglino et al.\ \cite[Alg.~4]{battaglino2018PracticalRandomized} instead propose drawing $\mtx{F}_1, \ldots, \mtx{F}_L$ once and then reusing them throughout the algorithm, only drawing $\mtx{R}$ anew for each least squares problem.
This reduces the computational cost considerably since it allows for greater reuse of various computed quantities.
More specifically, the expensive application of the full Kronecker SRFT to the unstructured matrix $\mtx{X}_{(n)}^{\trans}$ does not have to be computed for every least squares problem.

Larsen and Kolda \cite{LK:2020:leverageSampleTensors} also sketch the least squares problems in \eqref{eq:CP-factor-matrix-update}.
They use the efficient leverage score sampling scheme for Khatri--Rao products discussed in Section~\ref{subsubsec:tensor_product_lev_scores}.
This approach also allows for some reuse between subsequent sketches. 
When solving for $\mtx{A}^{(n)}$ in \eqref{eq:CP-factor-matrix-update}, a row with multi-index $(i_1, \ldots, i_{n-1}, i_{n+1}, \ldots, i_L)$ is sampled with probability $p^{(1)}_{i_1} \cdots p^{(n-1)}_{i_{n-1}} p^{(n+1)}_{i_{n+1}} \cdots p^{(L)}_{i_L}$, where $(p^{(k)}_{i_k})$ is the leverage score sampling distribution for $\mtx{A}^{(k)}$.
Since each $\mtx{A}^{(k)}$ only change for every $L^\text{th}$ least squares problem, the probability distribution $(p^{(k)}_{i_k})$ can be used in $L-1$ least squares problems before it needs to be recomputed.

\subsection{Background on the Tucker decomposition}

The Tucker decomposition \cite[\S 4]{kolda2009TensorDecompositions} is another popular method that decomposes an $L$-way tensor $\mathcal{X}$ of size $m_1 \times m_2 \times \cdots \times m_L$ into 
\begin{equation}
    \sum_{r_1 = 1}^{R_1} \sum_{r_2 = 1}^{R_2} \cdots \sum_{r_L = 1}^{R_L} \mathcal{G}[r_1, r_2, \ldots, r_L] \, \vct{a}_{r_1}^{(1)} \circ \vct{a}_{r_2}^{(2)}  \circ \cdots \circ \vct{a}_{r_L}^{(L)},
\end{equation}
where the so-called \textit{core tensor} $\mathcal{G}$ is of size $R_1 \times R_2 \times \cdots \times R_L$. 
The $m_n \times R_n$ matrices $\mtx{A}^{(n)} = \begin{bmatrix} \vct{a}_1^{(n)} & \cdots & \vct{a}_{R_n}^{(n)} \end{bmatrix}$ for $n \in \idxs{L}$ are called factor matrices.
Similarly to the CP decomposition, the core tensor and factor matrices can be initialized randomly and then updated iteratively via ALS:\footnote{
    The update rules in \eqref{eq:Tucker-factor-matrix-update} and \eqref{eq:Tucker-core-update} have been formulated as least squares problems in order to show where sketching can be applied in the ALS algorithm.
    A more standard formulation of the update rules can be found in \cite[\S 4.2]{kolda2009TensorDecompositions}.
}
\begin{align}
    &\text{For } n \text{ in } \idxs{L} : \;\;\;\; \mtx{A}^{(n)} = \argmin_{\mtx{A}} \| \mtx{B}^{\neq n} \mtx{G}_{(n)}^{\trans} \mtx{A}^{\trans} - \mtx{X}_{(n)}^{\trans}\|_\text{F}, \label{eq:Tucker-factor-matrix-update} \\
    &\mathcal{G} = \argmin_{\mathcal{Z}} \| \mtx{B} \vectorize(\mathcal{Z}) - \vectorize(\mathcal{X}) \|_\text{F} \label{eq:Tucker-core-update},
\end{align}
where
\begin{align*}
    &\mtx{B}^{\neq n} = \mtx{A}^{(L)} \otimes \cdots \otimes \mtx{A}^{(n+1)} \otimes \mtx{A}^{(n-1)} \otimes \cdots \otimes \mtx{A}^{(1)}, \\
    &\mtx{B} = \mtx{A}^{(L)} \otimes \cdots \otimes \mtx{A}^{(1)},
\end{align*}
and the unfolding $\mtx{G}_{(n)}$ is defined analogously to $\mtx{X}_{(n)}$ in \eqref{eq:mode-n-unfolding-def}.
The steps in \eqref{eq:Tucker-factor-matrix-update} and \eqref{eq:Tucker-core-update} are then repeated until some convergence criterion is met.
We note that the least squares problems  \eqref{eq:Tucker-factor-matrix-update} and \eqref{eq:Tucker-core-update} are highly overdetermined when $(R_n)_{n\in\idxs{L}}$ are small compared to $(m_n)_{n \in \idxs{L}}$.

\subsection{Sketching for the Tucker decomposition}

Malik and Becker \cite{malik2018LowrankTucker} apply the TensorSketch discussed in Section~\ref{subsubsec:TensorSketch} to these problems.
From a straightforward adaption of \eqref{eq:TensorSketch-vector-formula} to matrix Kronecker products, we have that the TensorSketch of the design matrix $\mtx{B}^{\neq n}$ is computed via the formula
\begin{equation*}
    \FFT^{-1}\bigg(\Big( \bigodot_{\substack{j = L \\ j \neq n}}^1 \big( \FFT(\mtx{S}_j \mtx{A}^{(j)}) \big)^\top \Big)^\top \bigg),
\end{equation*}
where $\mtx{S}_j$ is a $d \times m_j$ CountSketch, and where $\top$ denotes transpose without complex conjugation. 
The formula for sketching $\mtx{B}$ is the same except for that it also includes the $n^{\text{th}}$ term in the Khatri--Rao product.

Instead of drawing new CountSketches for every application of TensorSketch, \cite[Alg.~2]{malik2018LowrankTucker} draw two sets of independent CountSketches at the start of the algorithm:
$(\mtx{S}^{(1)}_j)_{j=1}^L$ where $\mtx{S}^{(1)}_j$ is of size $d_1 \times m_j$, and $(\mtx{S}^{(2)}_j)_{j=1}^L$ where $\mtx{S}^{(2)}_j$ is of size $d_2 \times m_j$.
These two sets of sketches are then reused throughout the algorithm: $(\mtx{S}^{(1)}_j)$ are used for sketching \eqref{eq:Tucker-factor-matrix-update} and $(\mtx{S}^{(2)}_j)$ are used for sketching \eqref{eq:Tucker-core-update}.
The latter least squares problems are much larger than the former.
Using two sets of sketching operators makes it possible to choose a larger embedding dimension for the latter problem, i.e., choosing $d_2 > d_1$.
By reusing CountSketches in this fashion, considerable improvement in run time is achieved.
Moreover, it is possible to compute all relevant sketches of unfoldings of $\mathcal{X}$ at the start of the algorithm, leading to an algorithm that requires only a single pass of $\mathcal{X}$ in order to decompose~it.

\subsection{Implementation considerations}

We deem it most appropriate to implement the structured sketches discussed in \cref{subsec:Sketching-of-Kronecker-and-Khatri--Rao} in \RandLAPACK{} rather than \RBLAS{}.
In order to facilitate the applications discussed in Section~\ref{subsec:partial-sketch-reuse}, it should be possible to update or redraw specific components of the sketching operator after it has been created.
For example, when applying the operator in \eqref{eq:CP-decomposition-Kronecker-SRFT} in an ALS algorithm for CP decomposition as in \cite[Alg.~4]{battaglino2018PracticalRandomized}, we want to keep the random diagonal matrices $\mtx{F}_1, \ldots, \mtx{F}_L$ fixed but draw a new sampling operator $\mtx{R}$ before each application of $\mtx{S}_n$.

In the applications above, components are shared across the $L$ different sketching operators that are used when updating the $L$ different factor matrices.
Rather than defining $L$ different sketching operators with shared components, it is better to define a single operator that contains all components and which allows ``leaving one component out'' when applied to a matrix or vector.
For example, consider a Kronecker SRFT from \eqref{eq:Kronecker-SRFT} but with reversed order in the Kronecker products.
It contains the components $\mtx{R}$ and $\mtx{F}_1, \ldots, \mtx{F}_L$.
A user should be able to supply a parameter $n$ which indicates that the $n^{\text{th}}$ term in the Kronecker products should be left out when computing the sketch, resulting in a sketch of the form \eqref{eq:CP-decomposition-Kronecker-SRFT}.

\appendix
\chapter{Details on Basic Sketching}\label{app:sketching}

\minitoc

This \nameCref{app:sketching} covers sketching theory and implementation of sketching operators.
Its contents are relevant to \cref{sec2:rblas,sec3:LS_and_optim,sec7:lev_scores}.

\section{Subspace embeddings and effective distortion}
\label{subapp:effective_distortion}

Let $\mtx{S}$ be a wide $d \times m$ sketching operator and $L$ be a linear subspace of $\R^m$.
Recall from \cref{subsec:sketch_quality} that $\mtx{S}$ \textit{embeds $L$ into $\R^d$ with distortion $\delta \in [0, 1]$} if
\[
    (1-\delta)\|\vct{x}\|_2 \leq \|\mtx{S}\vct{x}\|_2 \leq (1+\delta)\|\vct{x}\|_2
\]
holds for all $\vct{x}$ in $L$.
We often use the term \textit{$\delta$-embedding} for such an operator.
Note that if $\mtx{S}$ is a $\delta$-embedding and $\delta'$ is greater than $\delta$, then $\mtx{S}$ is also a $\delta'$-embedding.
It can be useful to speak of the smallest distortion $\delta$ for which $\mtx{S}$ is a $\delta$-embedding for $L$;
we call this \textit{the distortion of $\mtx{S}$ for $L$}, and denote it by
\begin{align*}
    \mathscr{D}(\mtx{S};L) = \inf\{\, \delta  \,:\,
        & 0 \leq \delta \leq 1 \\
        & \mtx{S} \text{ is a } \delta\text{-embedding for } L \}.
\end{align*}
In this notation, we have $\mathscr{D}(\mtx{S};L) \geq 1$ when $\ker\mtx{S} \cap L$ is nontrivial.
If there is a unit vector $\vct{x}$ in $L$ where $\|\mtx{S}\vct{x}\| > 2$, then $\mathscr{D}(\mtx{S};L) = +\infty$.

Subspace embedding distortion has a significant limitation in that it depends on the scale of $\mtx{S}$, while many \RandNLA{} algorithms have no such dependence.
This shortcoming leads us to define the \textit{effective distortion of $\mtx{S}$ for $L$} as
\begin{equation}\label{eq:effective_dist_def}
    \mathscr{D}_{\mathrm{e}}(\mtx{S};L) = \inf_{t > 0}\mathscr{D}(t\mtx{S};L).
\end{equation}
Here, the infimum is over $t > 0$ rather than $t \neq 0$ since $\mathscr{D}(\mtx{S};L) = \mathscr{D}(-\mtx{S};L)$.

There is a convenient formula for effective distortion using concepts of \textit{restricted singular values} and \textit{restricted condition numbers}.
Restricted singular values are a fairly general concept of use in random matrix theory; see, e.g., \cite{OT:2017}.
They are measures an operator's ``size'' when considered from different vantage points within a set of interest.
Formally, we define the largest and smallest restricted singular values of a sketching operator $\mtx{S}$ for a subspace $L$ as
\[
    \sigma_{\max}(\mtx{S};L) = \max_{\vct{x} \in L}\left\{ \|\mtx{S}\vct{x}\|_2 \,:~\|\vct{x}\|_2 = 1\right\}
\]
and
\[
    \sigma_{\min}(\mtx{S};L) = \min_{\vct{x} \in L}\left\{ \|\mtx{S}\vct{x}\|_2 \,:~\|\vct{x}\|_2 = 1\right\}
\]
Given these concepts, we define the restricted condition number of $\mtx{S}$ on $L$ as
\[
\cond(\mtx{S}; L) = \frac{\sigma_{\max}(\mtx{S}; L)}{ \sigma_{\min}(\mtx{S}; L)}, 
\]
where we take $c / 0 = +\infty$ for any $c \geq 0$.

We have formulated the concepts of restricted singular values and condition numbers in a way that reflects their geometric meaning.
More concrete descriptions can be obtained by considering any matrix $\mtx{U}$ whose columns provide an orthonormal basis for $L$.
With this one can see that $\sigma_{\min}(\mtx{S};L)$ and $\sigma_{\max}(\mtx{S};L)$ coincide with the extreme singular values of $\mtx{S}\mtx{U}$, and that $\cond(\mtx{S};L) = \cond(\mtx{S}\mtx{U})$.

Next, we provide the connection between restricted condition numbers and effective distortion.
\cref{subsubapp:eff_dist_in_sap} explores this connection more in the context of sketch-and-precondition algorithms for saddle point problems.
\begin{proposition}\label{prop:effective_distortion}
    Let $L$ be a linear subspace and $\mtx{S}$ be a sketching operator on $L$.
    The effective distortion of $\mtx{S}$ for $L$ is
    \[
        \mathscr{D}_{\mathrm{e}}(\mtx{S};L) = \frac{\kappa - 1}{\kappa + 1}
    \]
    where we take $(\infty-1)/(\infty+1) = 1$.
\end{proposition}
\begin{proof}
    The scaled sketching operator $t\mtx{S}$ is a $\delta$-embedding for $L$ if and only if
    \[
        (1-\delta)\|\vct{x}\|_2 \leq t\|\mtx{S}\vct{x}\|_2 \leq (1+\delta)\|\vct{x}\|_2\quad\text{for all }\vct{x} \text{ in } L.
    \]
    This is equivalent to
    \[
        \frac{1-\delta}{t} \leq \frac{\|\mtx{S}\vct{x}\|_2}{\|\vct{x}\|_2} \quad\text{and}\quad \frac{\|\mtx{S}\vct{x}\|_2}{\|\vct{x}\|_2} \leq \frac{1+\delta}{t} \quad\text{for all } \vct{x} \text{ in } L \setminus \{\vct{0}\}. 
    \]
    To simplify these bounds, we optimize over $\vct{x}$.
    Abbreviating $\sigma_1 := \sigma_{\max}(\mtx{S};L)$ and $\sigma_n := \sigma_{\max}(\mtx{S};L)$ for $n = \dim(L)$, we find that $t\mtx{S}$ is a $\delta$-embedding if and only if
    \[
        \frac{1-\delta}{t} \leq \sigma_n \quad\text{and}\quad \frac{1+\delta}{t} \leq \sigma_1.
    \]
    These identities can be rearranged to find the following constraints on $\delta$:
    \[
        1 - \sigma_1 t \leq \delta \quad\text{and}\quad t \sigma_n - 1 \leq \delta.
    \]
    The value of $t$ which permits minimum $\delta$ is that which makes the lower bounds coincide.
    That is, the optimal $t$ is $t_\star = 2 / (\sigma_1 + \sigma_n)$.
    Plugging this into the bounds above, our constraints on $\delta$ reduce to
    \[
        \delta \geq 1 - \sigma_1 t_\star = t_{\star} \sigma_n - 1 = \frac{\sigma_1 - \sigma_n}{\sigma_1 + \sigma_n} = \frac{\kappa - 1}{\kappa + 1},
    \]
    as desired.
\end{proof}

\subsection{Effective distortion of Gaussian operators}

It is informative to consider the concepts of restricted condition numbers and effective distortion for Gaussian sketching operators.
Therefore, let us suppose that our $d \times m$ sketching operator $\mtx{S}$ has iid mean-zero Gaussian entries, and consider an $n$-dimensional subspace $L$ in $\R^m$.
By rotational invariance of Gaussian distribution, we can infer that the distribution of $\cond(\mtx{S};L)$ coincides with that of $\cond(\mtx{\tilde{S}})$ for a $d \times n$ Gaussian matrix $\mtx{\tilde{S}}$.
Strong concentration results are available to understand the distribution of $\cond(\mtx{\tilde{S}})$.

Specifically, letting $d = sn$ for a constant $s > 1$, results by Silverstein \cite{Silverstein:1985:wishartEigval} and Geman \cite{Geman:1980:normOfGaussianMatrix} imply
\begin{equation}\label{eq:cond_of_gaussian}
    \cond(\mtx{S};L) \to \frac{\sqrt{s} + 1}{\sqrt{s} - 1} \quad\text{almost surely as}\quad n \to \infty. 
\end{equation}
This can be turned around using \cref{prop:effective_distortion} to obtain 
\begin{equation}\label{eq:eff_dist_of_gaussian}
    \mathscr{D}_{\mathrm{e}}(\mtx{S};L) \to  \frac{1}{\sqrt{s}} \quad\text{almost surely as}\quad n \to \infty. 
\end{equation}
We emphasize that \eqref{eq:cond_of_gaussian} and \eqref{eq:eff_dist_of_gaussian} hold for any fixed $n$-dimensional subspace $L$.
These facts justify aggressively small choices of embedding dimension when using Gaussian sketching operators in randomized algorithms for least squares.
Meng, Saunders, and Mahoney come to the same conclusion in their work on LSRN  \cite[Theorem 4.4]{MSM:2014:LSRN}.

\section{Short-axis-sparse sketching operators}\label{subapp:SJLTs}

In this appendix we make liberal use of the abbreviation \textit{SASO} (for ``short-axis-sparse sketching operator'') introduced in \cref{subsubsec:randblas:ShortAxSparse}.
Without loss of generality, we describe SASOs in the wide case, i.e., when $\mtx{S}$ is $d \times m$ with $d \ll m$.

\subsection{Implementation notes }\label{subapp:SASO_implement}

\subsubsection{Constructing SASOs column-wise}\label{subsubapp:sjlts_construct}

It is extremely cheap to construct and store a wide SASO.
The precise storage format depends on how one wants to apply the SASO later on, which can vary depending on context.
However, the construction is embarrassingly parallel across columns provided one uses CBRNGs (counter-based random number generators; see \S \ref{subsec:randblas_rngs}), and this structure leads to canonical methods for generating SASOs.

We first consider the SASO construction where row indices are partitioned into index sets $I_1,\ldots,I_k$ of roughly equal size.
Given such a partition, the indices of nonzeros for a given column are chosen by taking one element (independently and uniformly) from each of the index sets $I_j$.
The naive implementation can sample these row indices with $k$ parallel calls to the CBRNG.

Now consider the construction where the row indices for a column are chosen by selecting $k$ elements from $\idxs{d}$ uniformly without replacement.
This can be done in $O(km)$ time by using Fisher-Yates sampling and carefully re-using workspace.
For a concrete implementation, we refer the reader to
\begin{quote}
\url{https://github.com/BallisticLA/RandBLAS/blob/19sept22/src/sjlts.cc#L14-L78}.\footnote{This code was written when we used the term ``SJLT'' for what we now call a ``SASO.''}
\end{quote}
While the implementation above is sequential, it is easy to parallelize. Given $T$ threads, the natural modification to the algorithm takes $O(mk/T)$ time and requires $O(Td)$ workspace.
The constants in the big-$O$ notation are small.

\subsubsection{Remarks on storage formats}

It is reasonable for a standard library to restrict SASOs to only the most common sparse matrix formats.
We believe both compressed sparse row (CSR) and compressed sparse column (CSC) are worth considering.
CSC is the more natural of the two since (wide) SASOs are constructed columnwise.
If CSR format is preferred for some reason, then we recommend constructing $\mtx{S}$ columnwise while retaining extra data to facilitate conversion to CSR.

In principle, if the nonzero entries of $\mtx{S}$ are $\pm 1$ \textit{and} CSC is used as the storage format, then one could do away with storing the nonzero values explicitly; one could instead store the sign information using signed integers for the row indices.
We do not favor this approach since it precludes working with SASOs with more than two distinct nonzero values.

For the matrices $\mtx{A}$ and $\mtx{A}_{\text{sk}}$, we must consider whether they are in column-major or row-major format.
Indeed, both formats need to be supported since \cref{sec3:LS_and_optim} framed all least squares problems with tall data matrices.
While this was without loss of generality from a mathematical perspective, a user with an underdetermined least squares problem involving a wide column-major data matrix $\mtx{B}$ is effectively needing to sketch the tall row-major matrix $\mtx{A} = \mtx{B}^{\trans}$.

\subsubsection{Applying a wide SASO}\label{subsubapp:sjlts_apply}

There are four combinations of storage formats we need to consider for $(\mtx{S},\mtx{A})$.

\paragraph{$\mtx{S}$ is CSC, $\mtx{A}$ is row-major.}
Suppose we have $P$ processors.
Our suggested approach is to partition the row index set $\idxs{d}$ into $I_1,\ldots,I_P$ and to have each processor be responsible for its own block of rows.
The $p^{\text{th}}$ processor computes its row block by streaming over the columns of $\mtx{S}$ and rows of $\mtx{A}$, accumulating outer products as indicated below
\[
    \mtx{A}_{\text{sk}}[I_p, :] = \sum_{\ell\in\idxs{m}}\mtx{S}[I_p,\ell]\mtx{A}[\ell,:].
\]
An individual term $\mtx{S}[I_p,\ell]\mtx{A}[\ell,:]$ can cheaply be accumulated into $\mtx{A}_{\text{sk}}[I_p, :]$ by using the fact that $\mtx{S}[I_p,\ell]$ is extremely sparse.
If $R$ denotes the number of nonzeros in $\mtx{S}[I_p, \ell]$, then the outer-product accumulation can be computed with $R$ \code{axpy} operations involving the row $\mtx{A}[\ell,:]$.
Note that since $\mtx{S}$ has $k$ nonzeros per column (with row indices distributed uniformly at random), this value $R$ is a sum of $|I_p|$ iid Bernoulli random variables with mean $k/d$.
Therefore the expected number of \code{axpy}'s performed by processor $p$ for term $\ell$ is $|I_p|k/d$.

\paragraph{$\mtx{S}$ is CSR, $\mtx{A}$ is row-major.}
Here, we suggest that the $d$ rows of $\mtx{A}_{\text{sk}}$ be computed separately from one another.
An individual row is given by $\mtx{A}_{\text{sk}}[i,:] = \mtx{S}[i,:]\mtx{A}$.
Evaluating the product of this sparse vector and dense matrix can be done by taking a linear combination of a small number of rows of $\mtx{A}$.
Specifically, if $R$ is the number of nonzeros in $\mtx{S}[i,:]$ then computing $\mtx{A}_{\text{sk}}[i,:]$ only requires $R$ rows from $\mtx{A}$.
Since the columns of $\mtx{S}$ are independent, $R$ is a sum of $m$ iid Bernoulli random variables with mean $k/d$.
Therefore we expect to access $mk/d$ rows of $\mtx{A}$ in order to compute $\mtx{A}_{\text{sk}}[i,:]$.


\paragraph{$\mtx{S}$ is CSC, $\mtx{A}$ is column-major.}
Here, we suggest that the $n$ columns of $\mtx{A}_{\text{sk}}$ be computed separately from one another.
An individual column is given by $\mtx{A}_{\text{sk}}[:,j] = \mtx{S}\mtx{A}[:,j]$.
We evaluate this product by taking a linear combination of the columns of $\mtx{S}$, according to
\[
    \mtx{A}_{\text{sk}}[:,j] = \sum_{\ell\in\idxs{m}}\mtx{S}[:,\ell]\mtx{A}[\ell,j].
\]
Note that each of the $\ell$ terms in this sum is a sparse vector with $k$ nonzero entries.
Based on our preliminary experiments, this method has mediocre single-thread performance, but it has excellent scaling properties for many-core machines.

\paragraph{$\mtx{S}$ is CSR, $\mtx{A}$ is column-major.}
We were unable to determine a method that parallelizes well for this pair of data formats.
The most efficient algorithm may be to convert $\mtx{S}$ to CSC and then to apply the preferred method when $\mtx{S}$ is CSC and $\mtx{A}$ is column-major.

\subsection{Theory and practical usage}\label{subapp:SASO_usage}

\subsubsection{SASO theory}\label{subsubapp:sjlts_theory}

A precursor to the SASOs we consider is described in \cite{DKS:2010:SJLT}, which sampled row indices for nonzero entries from $\idxs{d}$ with replacement. 
The first theoretical analysis of the SASOs we consider was conducted in \cite{KN:2012:SJLTs} and concerned the distributional Johnson-Lindenstrauss property.
Shortly thereafter, \cite{CW:2013:CWTransform} and \cite{MM13_STOC:CountSketch} studied OSE properties for SASOs with a single nonzero per column; the latter referred to the construction as ``CountSketch.''

Theoretical analyses for OSE properties of general SASOs (i.e., those with more than one nonzero per column) were first carried out by \cites{NN:2013:OSNAPs,KN:2014:SJLTs} and subsequently improved by \cites{BDN:2015,Cohen:2016:SJLTs}.
Much of the SASO analysis has been through the lens of ``hashing functions,'' and does not require that the columns of the sketching operator are fully independent.
We do not know of any practical advantage to SASOs with partial dependence across the columns.

\begin{remark}[Navigating the literature]
    \cite{CW:2017:nnztime} is a longer journal version of \cite{CW:2013:CWTransform}.
    \cite{KN:2014:SJLTs} and \cite{KN:2012:SJLTs} have the same title, and 
    the former is considered a more developed journal version of the latter.
\end{remark}

\subsubsection{Selecting parameters for SASOs}\label{subsubapp:sjlts_params}

We are in the process of developing recommendations for how to set the parameters $d$ and $k$ for a distribution over SASOs.
So far we have observed that when $d$ is fixed the sketch quality increases rapidly with $k$ before reaching a plateau.
As one point of reference, we have observed that there is no real benefit in $k$ being larger than eight when embedding the range of a $100,000 \times 2,000$ matrix into a space with ambient dimension $d = 6,000$.
Furthermore, for the data matrices we tested, the restricted condition numbers of those sketching operators were tightly concentrated at $O(1)$.
Extensive experiments with parameter selection for SASOs in a least squares context are given in \cite{Urano:2013}.

\section{Theory for sketching by row selection}\label{subapp:row_selection_theory}

Here we prove \cref{prop:quality_of_general_row_sampling}.
Our proof is inspired by \cite[Problem 5.13]{Tropp:2021:LecNotes}, which begins with the following adaptation of \cite[Theorem 5.1.1]{Tropp:2015:matrixConcentrationBook}.

\begin{theorem}\label{thm:matrix_chernoff}
    Consider an independent family $\{\mtx{X}_1,\ldots,\mtx{X}_s\} \subset \Herm^n$ of random psd matrices that satisfy $\lambda_{\max}(\mtx{X}_i) \leq L$ almost surely.
    Let $\mtx{Y} = \sum_{i=1}^s \mtx{X}_i$, and define the mean parameters
    \[
            \mu_{\max} = \lambda_{\max}(\E\mtx{Y}) \quad\text{and}\quad \mu_{\min} = \lambda_{\min}(\E\mtx{Y}).
    \]
    One has that
    \begin{align*}
        &\Pr\left\{ \lambda_{\max}(\mtx{Y} - (1+t)\E\mtx{Y}) \geq 0 \right\} \leq n \left[\frac{\exp(t)}{(1+t)^{(1+t)}}\right]^{\mu_{\max}/L} \quad\text{for } t > 0,\quad\text{and}\\
        &\Pr\left\{ \lambda_{\max}((1-t)\E\mtx{Y} - \mtx{Y}) \geq 0 \right\} \leq n\left[\frac{\exp(-t)}{(1-t)^{(1-t)}}\right]^{\mu_{\min}/L}\quad\text{for } t \in (0, 1).
    \end{align*}
\end{theorem}

Here, we restate the result we aim to prove. 

\begin{proposition}[Adaptation of \cref{prop:quality_of_general_row_sampling}]\label{prop:quality_of_general_row_sampling_redux}
    Suppose $\mtx{A}$ is an $m \times n$ matrix of rank $n$, $\vct{q}$ is a distribution over $\idxs{m}$, and $t$ is in $(0, 1)$.
    Let $\mtx{S}$ be a $d \times m$ sketching operator with rows that are distributed iid as
    \[
        \mtx{S}[i, :] = \frac{\vct{\delta}_j}{\sqrt{d q_j}} \text{ with probability } q_j,
    \]
    and let $E(t,\mtx{S})$ denote the event that 
    \begin{equation*}
        (1-t)\|\vct{y}\|_2^2 \leq \|\mtx{S}\vct{y}\|_2^2 \leq (1+t)\|\vct{y}\|_2^2~~\forall\, \vct{y} \in \range\mtx{A}.
    \end{equation*}
    Using $r := \min_{j \in \idxs{m}}\frac{q_j}{p_j(\mtx{A})}$, we have
    \begin{align*}
        \Pr\left\{ E(t,\mtx{S}) \text{ fails} \right\}
            & \leq 2n\left(\frac{\exp(t)}{(1+t)^{(1+t)}}\right)^{r d / n}.
    \end{align*}
\end{proposition}
\begin{proof}
    The way that we use \cref{thm:matrix_chernoff} is along the lines of the hint in \cite[Problem 5.13, Part 3]{Tropp:2021:LecNotes}.
    We begin by considering the Gram matrices $\mtx{G} = \mtx{A}^{\trans}\mtx{A}$ and $\mtx{G}_{\text{sk}} = \mtx{A}^{\trans}\mtx{S}^{\trans}\mtx{S}\mtx{A}$.
    The event $E(t,\mtx{S})$ is equivalent to
    \[
        (1-t)\mtx{I}_n \preceq \mtx{G}^{-1/2}\mtx{G}_{\text{sk}}\mtx{G}^{-1/2} \preceq (1+t)\mtx{I}_n.
    \]
    
    The sketched Gram matrix can be expressed as a sum of $d$ outer products of rows of $\mtx{S}\mtx{A}$.
    Each of the $d$ outer products is conjugated by $\mtx{G}^{-1/2}$ to obtain our matrices $\{\mtx{X}_1,\ldots,\mtx{X}_d\}$.
    That is, we set
    \begin{equation}\label{eq:Xi_in_matrix_chernoff}
        \mtx{X}_i =  \mtx{G}^{-1/2}\left(\left(\mtx{S}\mtx{A}\right)[i,:]\right)^{\trans}\left(\left(\mtx{S}\mtx{A}\right)[i,:]\right)\mtx{G}^{-1/2}
    \end{equation}
    so that $\mtx{Y} = \sum_{i=1}^d \mtx{X}_i$ satisfies $\E\mtx{Y} = \mtx{I}_n$.
    A union bound provides
    \[
        \Pr\{E(t,\mtx{S}) \text{ fails}\} \leq \Pr\{ \lambda_{\max}(\mtx{Y}) \geq 1+t \} + \Pr\{1-t \geq \lambda_{\min}(\mtx{Y})\}.
    \]
    
    Note that the claim of this proposition only invokes \cref{thm:matrix_chernoff} in the special case when $t$ is between zero and one.
    Moreover, our particular choice of $\mtx{Y}$ leads to $\mu_{\min} = \mu_{\max} = 1$.
    Given these restrictions, it can be shown that the larger of the two probability bounds in the theorem is that involving the term $\exp(t) / (1+t)^{(1+t)}$.
    Therefore we have
    \[
        \Pr\{E(t,\mtx{S}) \text{ fails}\} \leq 2n\left(\exp(t) / (1+t)^{(1+t)}\right)^{1/L}.
    \]
    Next, we turn to finding the smallest possible $L$ given this construction, so as to maximize $1/L$.
    
    Let $i$ be an arbitrary index in $\idxs{d}$.
    By the definition of $\mtx{S}$, the following must hold for some $k \in \idxs{m}$:
    \[
        \mtx{X}_i = \frac{1}{d}\left(\frac{1}{q_k}\mtx{G}^{-1/2}\mtx{A}[k,:]^{\trans}\mtx{A}[k,:]\mtx{G}^{-1/2}\right).
    \]
    Our next step is to use the fact that the trace of a rank-1 psd matrix is equal to its largest eigenvalue.
    Cycling the trace shows that
    \begin{equation*}\label{eq:whats_Xi_lambda_max}
        \lambda_{\max}\left(\mtx{G}^{-1/2}\mtx{A}[k,:]^{\trans}\mtx{A}[k,:]\mtx{G}^{-1/2}\right) = \mtx{A}[k,:]\mtx{G}^{-1}\mtx{A}[k,:]^{\trans} = \ell_k(\mtx{A}),
    \end{equation*}
    and hence
    \[
        L = \frac{1}{d} \max_{j \in \idxs{m}}\left\{\frac{\ell_j(\mtx{A})}{q_j}\right\}
    \]
    is the smallest value that guarantees $\lambda_{\max}(\mtx{X}_i) \leq L$.
    
    To complete the proof we use the assumption that $\mtx{A}$ is of full rank $n$ to express the leverage score $\ell_j(\mtx{A})$ as $\ell_j(\mtx{A}) = n p_j(\mtx{A})$.
    This shows that
    \[
        L = \frac{n}{d} \max_{j \in \idxs{m}}\frac{p_j(\mtx{A})}{q_j} ,
    \]
    and the proposition's claim follows from just a little algebra.
\end{proof}

\chapter[Details on Least Squares and Optimization]{Details on Least Squares \\ and Optimization}
\chaptermark{Least Squares and Optimization}
\label{app:lstsq_details}

\minitoc

This \nameCref{app:lstsq_details} covers a few distinct topics.
\cref{subsubapp:we_sketch_subspaces} proves a novel result relevant to sketch-and-precondition algorithms for saddle point problems, and connects this result to the idea of effective distortion.
In \cref{subapp:error_metrics}, we provide background from classical NLA on what it means to compute an ``accurate'' solution to a least squares problem (overdetermined or underdetermined).
\cref{app:lstsq_details:inconsistent} derives limiting solutions of saddle point problems as the regularization parameter tends to zero from above.
These limiting solutions are important for treating saddle point problems as linear algebra problems even when their optimization formulations are ill-posed.
Finally, \cref{subapp:details_min_reg_quadratics} gives background on kernel ridge regression and details a new approach to sketch-and-solve of regularized quadratics.

\section{Quality of preconditioners}\label{subsubapp:we_sketch_subspaces}

Here we consider preconditioners of the kind described in \cref{subsec:precond_gen}.
These are obtained by sketching a tall $m \times n$ data matrix $\mtx{A}$ in the embedding regime, factoring the sketch, and using the factorization to construct an orthogonalizer.

\begin{proposition}[Adaptation of \cref{prop:left_sketch_precond}]\label{prop:left_sketch_precond_appendix}
    Consider a sketch $\mtx{A}_{\mathrm{sk}} = \mtx{S}\mtx{A}$ and a matrix $\mtx{U}$ whose columns are an orthonormal basis for $\range(\mtx{A})$.
    If $\rank(\mtx{A}_{\mathrm{sk}}) = \rank(\mtx{A})$ and the columns of  $\mtx{A}_{\mathrm{sk}}\mtx{M}$ are an orthonormal basis for the range of $\mtx{A}_{\mathrm{sk}}$, then singular values of $\mtx{A}\mtx{M}$ are the reciprocals of the singular values of $\mtx{S}\mtx{U}$.
\end{proposition}
Observe that this proposition is a linear algebraic result, i.e., there is no randomness.
When applied to randomized algorithms, the randomness enters only via the construction of the sketch.

This result can be applied to problems with ridge regularization by working with augmented matrices in the vein of \cref{subsec:precond_gen}.
In that context it is necessary to not only replace $(\mtx{A},\mtx{A}_{\mathrm{sk}})$ by $(\mtx{\hat{A}},\mtx{\hat{A}}_{\mathrm{sk}})$, but also to replace $\mtx{S}$ by the augmented sketching operator $\mtx{\hat{S}}$ that takes $\mtx{\hat{A}}$ to $\mtx{\hat{A}}_{\mathrm{sk}}$.
The augmented sketching operator in question was already visualized in \cref{alg:saddle_to_ols_sap}.

Our proof of \cref{prop:left_sketch_precond_appendix} requires finding an explicit expression for $\mtx{M}$.
Towards this end, we prove the following lemma.

\begin{lemma}\label{lem:tool_for_pinv_M}
    Suppose $\mtx{A}_{\mathrm{sk}}$ is a tall $d \times n$ matrix and that $\mtx{M}$ is a full-column-rank matrix for which the columns of $\mtx{A}_{\mathrm{sk}}\mtx{M}$ form an orthonormal basis for $\range(\mtx{A}_{\mathrm{sk}})$.
    If $\mtx{B}$ is a full-row-rank matrix for which $\mtx{A}_{\mathrm{sk}} = \mtx{A}_{\mathrm{sk}}\mtx{M}\mtx{B}$, then we have $\mtx{M} = \mtx{B}^{\dagger}$.
\end{lemma}
\begin{proof}[Proof of \cref{lem:tool_for_pinv_M}.]
    Let $k = \rank(\mtx{A}_{\mathrm{sk}}) = \rank(\mtx{A}_{\mathrm{sk}}\mtx{M})$.
    Since the columns of $\mtx{A}_{\mathrm{sk}}\mtx{M}$ are orthonormal we can infer that it has dimensions $d \times k$.
    Similarly, since $\mtx{M}$ is full column-rank we can infer that it is $n \times k$.
    We prove that $\mtx{B} = \mtx{M}^{\dagger}$, which amounts to showing four properties:
    \begin{enumerate}
        \item $\mtx{M}\mtx{B}\mtx{M} = \mtx{M}$,
        \item $\mtx{B}\mtx{M}\mtx{B} = \mtx{B}$,
        \item $\mtx{B}\mtx{M}$ is an orthogonal projector, and
        \item $\mtx{M}\mtx{B}$ is an orthogonal projector.
    \end{enumerate}
    By the lemma's assumption we have the identity $\mtx{A}_{\mathrm{sk}} = \mtx{A}_{\mathrm{sk}}\mtx{M}\mtx{B}$.
    Left multiply this expression through by $(\mtx{A}_{\mathrm{sk}}\mtx{M})^{\trans}$ to see that 
    \begin{equation}\label{eq:form_of_candidate_pinvM}
        \mtx{M}^{\trans}\mtx{A}_{\mathrm{sk}}^{\trans}\mtx{A}_{\mathrm{sk}} = \mtx{B}.
    \end{equation}
    Next, we right multiply both sides of \eqref{eq:form_of_candidate_pinvM} by $\mtx{M}$ and use column orthonormality of $\mtx{A}_{\mathrm{sk}}\mtx{M}$ to obtain $\mtx{B}\mtx{M} = \mtx{I}_k$ --- this is sufficient to show the first three conditions for the pseudoinverse.
    Showing the fourth and final condition takes more work.
    For that we left multiply \eqref{eq:form_of_candidate_pinvM} by $\mtx{M}$ so as to express
    \[
        \mtx{M}\mtx{M}^{\trans}\mtx{A}_{\mathrm{sk}}^{\trans}\mtx{A}_{\mathrm{sk}} = \mtx{M}\mtx{B}.
    \]
    Therefore our task is to show that $\mtx{M}\mtx{M}^{\trans}\mtx{A}_{\mathrm{sk}}^{\trans}\mtx{A}_{\mathrm{sk}}$ is an orthogonal projector.
   
    Consider the compact SVD $\mtx{A}_{\mathrm{sk}} = \mtx{U}_{\mathrm{sk}}\mtx{\Sigma}_{\mathrm{sk}}\mtx{V}_{\mathrm{sk}}$.
    Since $\mtx{A}_{\mathrm{sk}}$ is rank-$k$ we have that $\mtx{U}_{\mathrm{sk}}$ has $k$ columns and $\mtx{\Sigma}$ is a $k\times k$ invertible matrix.
    Since the columns of $\mtx{A}_{\mathrm{sk}}\mtx{M}$ form an orthonormal basis for the range of $\mtx{A}_{\mathrm{sk}}$, it must be that $\mtx{A}_{\mathrm{sk}}\mtx{M} = \mtx{U}_{\mathrm{sk}}\mtx{W}$ for some $k \times k$ orthogonal matrix $\mtx{W}$.
    Furthermore, this orthogonal matrix can be expressed as $\mtx{\Sigma}_{\mathrm{sk}}\mtx{V}_{\mathrm{sk}}^{\trans}\mtx{M} = \mtx{W}$, which implies
    \begin{equation}\label{eq:vvproject_expr_for_M}
            \mtx{V}_{\mathrm{sk}}\mtx{V}_{\mathrm{sk}}^{\trans}\mtx{M} = \mtx{V}_{\mathrm{sk}}\mtx{\Sigma}_{\mathrm{sk}}^{-1}\mtx{W}.
    \end{equation}
    We have reached a checkpoint in the proof.
    Our next task is to obtain an expression for $\mtx{M}$ by simplifying \eqref{eq:vvproject_expr_for_M}.
    
    Consider the subspaces $X = \range\mtx{V}_{\mathrm{sk}}$, $Y = \ker\mtx{A}_{\mathrm{sk}}$, and $Z = \range\mtx{M}$, all contained in $\R^n$.
    We know that $Y \cap Z$ is trivial since $\rank(\mtx{A}_{\mathrm{sk}}\mtx{M}) = \rank(\mtx{M})$.
    At the same time, since $Y$ and $Z$ are of dimensions $n-k$ and $k$ respectively, it must be that $Z = Y^{\perp}$.
    This fact can be combined with $Y = X^{\perp}$ (from the fundamental theorem of linear algebra) to obtain $Z = X$, which in turn implies $\mtx{V}_{\mathrm{sk}}\mtx{V}_{\mathrm{sk}}^{\trans}\mtx{M} = \mtx{M}$.
    Therefore \eqref{eq:vvproject_expr_for_M} simplifies to
    \[
        \mtx{M} = \mtx{V}_{\mathrm{sk}}\mtx{\Sigma}_{\mathrm{sk}}^{-1}\mtx{W}.
    \]
    This expression is precisely what we need; when the dust settles, it tells us that
    \[
        \mtx{M}\mtx{M}^{\trans}\mtx{A}_{\mathrm{sk}}^{\trans}\mtx{A}_{\mathrm{sk}} = \mtx{V}_{\mathrm{sk}}\mtx{V}_{\mathrm{sk}}^{\trans}.
    \]
\end{proof}

\begin{proof}[Proof of \cref{prop:left_sketch_precond_appendix}.]
    Let $k = \rank(\mtx{A})$.
    It suffices to prove the statement where $\mtx{U}$ is the $m \times k$ matrix containing the left singular vectors of $\mtx{A}$.
    Our proof involves working with the compact SVD $\mtx{A} = \mtx{U}\mtx{\Sigma}\mtx{V}^{\trans}$, where $\mtx{V}$ is $n \times k$ and $\mtx{\Sigma}$ is invertible.
    Noting that $\mtx{A}_{\mathrm{sk}} = \mtx{S}\mtx{U}\mtx{\Sigma}\mtx{V}^*$ holds by definition of $\mtx{A}_{\mathrm{sk}}$, we can replace $\mtx{S}\mtx{U}$ by its economic QR factorization $\mtx{S}\mtx{U} = \mtx{Q}\mtx{R}$ to see \begin{equation}\label{eq:prop:left_sketch_precond:factor1}
        \mtx{A}_{\mathrm{sk}} = \mtx{Q}\mtx{R}\mtx{\Sigma}\mtx{V}^*.
    \end{equation}
    Furthermore, since $\rank(\mtx{A}_{\mathrm{sk}}) = k$ it must be that $\rank(\mtx{S}\mtx{U}) = k$.
    This tells us that $\mtx{R}$ is invertible and that $\mtx{Q}$ provides an orthonormal basis for the range of $\mtx{A}_{\mathrm{sk}}$.
    
    By assumption on $\mtx{M}$, the matrix $\mtx{A}_{\mathrm{sk}}\mtx{M}$ is \textit{also} an orthonormal basis for the range of $\mtx{A}_{\mathrm{sk}}$.
    Therefore there exists a $k \times k$ orthogonal matrix $\mtx{P}$ where $\mtx{Q}\mtx{P} = \mtx{A}_{\mathrm{sk}}\mtx{M}$.
    We can rewrite \eqref{eq:prop:left_sketch_precond:factor1} as
    \begin{equation*}
        \mtx{A}_{\mathrm{sk}} = \left(\mtx{Q}\mtx{P}\right)\left(\mtx{P}^*\mtx{R}\mtx{\Sigma}\mtx{V}^*\right).
    \end{equation*}
    Since the left factor in the above display is simply $\mtx{A}_{\mathrm{sk}}\mtx{M}$, we have
     \begin{equation}\label{eq:suggest_M_dagger}
     \mtx{A}_{\mathrm{sk}} = \mtx{A}_{\mathrm{sk}}\mtx{M}\left(\mtx{P}^*\mtx{R}\mtx{\Sigma}\mtx{V}^*\right)\textcolor{blue}{.}
    \end{equation}
    The next step is to abbreviate $\mtx{B} = \mtx{P}^*\mtx{R}\mtx{\Sigma}\mtx{V}^*$ and apply \cref{lem:tool_for_pinv_M} to infer that $\mtx{B} = \mtx{M}^{\dagger}$.
    Invoking the column-orthonormality of $\mtx{V}$ and invertibility of $(\mtx{\Sigma},\mtx{R},\mtx{P})$ we further have $\mtx{B}^{\dagger} = \mtx{M} = \mtx{V}\mtx{\Sigma}^{-1}\mtx{R}^{-1}\mtx{P}$.
    Plug in this expression for $\mtx{M}$ to see that
    \begin{equation}\label{eq:form_of_preconditioned_A}
        \mtx{A}\mtx{M} = \left(\mtx{U}\mtx{\Sigma}\mtx{V}^*\right)\left(\mtx{V}\mtx{\Sigma}^{-1}\mtx{R}^{-1}\mtx{P}\right) = \mtx{U}\mtx{R}^{-1}\mtx{P}.
    \end{equation}
    The proof is completed by noting that the singular values of $\mtx{R}^{-1}$ are the reciprocals of the singular values of $\mtx{Q}\mtx{R} = \mtx{S}\mtx{U}$.
\end{proof}

\subsection{Effective distortion in sketch-and-precondition}\label{subsubapp:eff_dist_in_sap}

Recall from \cref{subapp:effective_distortion} that if the columns of $\mtx{U}$ are an orthonormal basis for a linear subspace $L$, then the restricted condition number of $\mtx{S}$ on $L$ is
\[
    \cond(\mtx{S};L) = \cond(\mtx{S}\mtx{U}).
\]
This identity combines with \cref{prop:left_sketch_precond_appendix} to make for a remarkable fact.
Namely, if $L = \range(\mtx{A})$ and $\mtx{M}$ is an orthogonalizer of a sketch $\mtx{S}\mtx{A}$, then
\begin{equation}\label{eq:restricted_condnum_eq_cond_precond}
    \cond(\mtx{S};L) = \cond(\mtx{A}\mtx{M}).
\end{equation}
Let us contextualize this fact algorithmically.
\begin{quote}
    If $\mtx{A}$ is an $m \times n$ matrix ($m \gg n$) in a saddle point problem, and if that problem is approached by the sketch-and-precondition methodology from \cref{subsubsec:sketch_and_precond}, then the condition number of the preconditioned matrix handed to the iterative solver is equal to the restricted condition number of $\mtx{S}$ on $\range(\mtx{A})$.
\end{quote}
But we can take this one step further.
By invoking \cref{prop:effective_distortion} and applying \eqref{eq:restricted_condnum_eq_cond_precond}, we obtain the following expression for the effective distortion of $\mtx{S}$ for $L$:
\begin{equation}\label{eq:eff_dist_as_convergence_rate}
    \mathscr{D}_{\mathrm{e}}(\mtx{S};L) = \frac{\cond(\mtx{A}\mtx{M}) - 1}{\cond(\mtx{A}\mtx{M}) + 1}.
\end{equation}
Alarm bells should be going off in some readers' heads.
The right-hand-side of \eqref{eq:eff_dist_as_convergence_rate} is none other than the convergence rate of LSQR (or CGLS) for a least squares problem with data matrix $\mtx{A}\mtx{M}$!
This shows a deep connection between our proposed concept of effective distortion and the venerated sketch-and-precondition paradigm in \RandNLA{}.

\section[Basic error analysis]{Basic error analysis for least squares problems}
\sectionmark{Basic error analysis}
\label{subapp:error_metrics}

When solving a computational problem numerically it is inevitable that the computed solutions deviate from the problem's exact solution.
This is a simple consequence of working in finite-precision arithmetic, and it remains true even when using very reliable algorithms.
Furthermore, for large-scale computations it is often of interest to trade off computational complexity with solution accuracy; this has led to algorithms that produce approximate solutions even when run in exact arithmetic.

These facts were encountered in the earliest days of NLA.
Their consequence in applications has motivated the development of a vast literature on quantifying and bounding the error of approximate solutions to computational problems.
Since several of the randomized algorithms from \cref{subsec:optim_drivers} purport to solve problems to any desired accuracy, it is prudent for us to summarize key points from this vast literature here.
The material from Appendices \ref{subapp:backward_error} to \ref{subsec:posteriori_backward_error} is typically covered in an introductory course on numerical analysis.
\cref{subapp:advanced_lstsq_error_est} mentions important topics which might not be covered in such a course.

\begin{remark}
We have focused this \nameCref{subapp:error_metrics} strongly on basic least squares problems (overdetermined and underdetermined) to keep it a reasonable length.
\end{remark}

\subsection{Concepts: forward and backward error}\label{subapp:backward_error}

The \textit{forward error} of an approximate solution to a computational problem is its distance to the problem's exact solution.
Forward error is easy to interpret, but it is not without its limitations.
First, it can rarely be computed in practice, since it is presumed that we do not have access to the problem's exact solution.
This means that substantial effort is needed to approximate or bound forward error in different contexts.
Second, even if one algorithm's behavior with respect to forward error has been analyzed, it may not be feasible to repurpose the analysis for another algorithm.
These shortcomings motivate the ideas of \textit{backward error} and \textit{sensitivity analysis}, wherein one asks the following questions, respectively.


\begin{itemize}
    \item How much do we need to perturb the problem data so that the computed solution exactly solves the perturbed problem?
    \item How does a small perturbation to a given problem change that problem's exact solution?
\end{itemize}

The connection between the two concepts is clear: any bound on backward error can be combined with sensitivity analysis to obtain an estimate of forward error.
The idea of sensitivity analysis is especially powerful since it is agnostic to the source of the problem's perturbation; the perturbations might be due to rounding errors from finite-precision arithmetic, uncertainty in data (as might arise from experimental observations), or deliberate choices to only compute approximate solutions.
In any of these cases one can combine knowledge of an algorithm's backward-error guarantees to obtain forward error estimates.

This reasoning can be carried further to arrive at two major benefits of the ``backward error plus sensitivity analysis'' approach.
\begin{itemize}
    \item A large portion of algorithm-specific error analysis can be accomplished purely by understanding the algorithm's behavior with respect to backward error.
    \item For many problems one can cheaply compute upper bounds on a solution's backward error \textit{at runtime}.
\end{itemize}
The combination of backward error and sensitivity analysis can therefore be used to establish  \textit{a priori} guarantees on algorithm numerical behavior and \textit{a posteriori} guarantees on the quality of an approximate solution.
However, we do note that sensitivity analysis results require knowledge of problem data that may not be available, such as extreme singular values of a data matrix in a least squares problem.
Therefore it is still difficult to compute forward error bounds at runtime.

\subsection[Basic sensitivity analysis for least squares problems]{Basic sensitivity analysis for unregularized least squares problems}\label{subapp:basic_unreg_ls_sensitivity}

Here we paraphrase facts from \cite[\S 5.3 and \S 5.6]{GvL:2013:MatrixComputationsBook} on sensitivity analysis of basic least squares problems.
Our restatements adopt the notation we used for saddle point problems, wherein both overdetermined and underdetermined problems involve a tall $m \times n$ matrix $\mtx{A}$.
The overdetermined problem induced by $\mtx{A}$ and an $m$-vector $\vct{b}$ is
\begin{equation*}\label{eq:app:ols}
    \min_{\vct{x} \in \R^n} \|\mtx{A}\vct{x} - \vct{b}\|_2^2,
\end{equation*}
while the underdetermined problem induced by $\mtx{A}$ and an $n$-vector $\vct{c}$ is
\begin{equation*}\label{eq:app:uls}
    \min_{\vct{y} \in \R^m}\left\{ \|\vct{y}\|_2^2 \,:\, \mtx{A}^{\trans}\vct{y} = \vct{c}\right\}.
\end{equation*}
In the following theorem statements, the reader should bear in mind that a perturbation $\delta\mtx{A}$ can only satisfy $\|\delta\mtx{A}\|_2 < \sigma_n(\mtx{A})$ if $\sigma_n(\mtx{A})$ is positive.
Therefore these theorem statements only apply when $\mtx{A}$ is full-rank.

\begin{theorem}\label{thm:ols_sensitivity}
    Suppose $\vct{b}$ is neither in the range of $\mtx{A}$ nor the kernel of $\mtx{A}^{\trans}$, and let $\vct{x} = \mtx{A}^{\dagger}\vct{b}$ be the optimal solution of the overdetermined least squares problem with data $(\mtx{A},\vct{b})$.
    Consider perturbations $\delta\vct{b}$ and $\delta\mtx{A}$ where $\|\delta\mtx{A}\|_2 < \sigma_n(\mtx{A})$.
    Define
    \begin{equation}\label{eq:thm_ols_sensitivity:eps}
        \epsilon = \max\left\{ \frac{\|\delta\mtx{A}\|_2}{\|\mtx{A}\|_2},~ \frac{\|\delta\vct{b}\|_2}{\|\vct{b}\|_2} \right\}
    \end{equation}
    together with some auxiliary quantities
    \begin{equation}\label{eq:thm_ols_sensitivity:aux_quants}
        \sin\theta = \frac{\|\vct{b} - \mtx{A}\vct{x}\|_2}{\|\vct{b}\|_2}\quad\text{and}\quad \nu = \frac{\|\mtx{A}\vct{x}\|_2}{\sigma_n(\mtx{A}) \|\vct{x}\|_2}.
    \end{equation}
    The perturbation $\delta\vct{x}$ necessary for $\vct{x} + \delta\vct{x}$ to solve the least squares problem with data $(\mtx{A}+\delta\mtx{A},\vct{b} + \delta\vct{b})$ satisfies
    \begin{align}
        \frac{\|\delta\vct{x}\|_2}{\|\vct{x}\|_2} &\leq \epsilon \left\{ \frac{\nu}{\cos \theta} + \kappa(\mtx{A})(1 + \nu\tan\theta) \right\} + O(\epsilon^2).\label{eq:sensitivity_x_ols} 
    \end{align}
\end{theorem}

\cref{thm:ols_sensitivity} restates part of \cite[Theorem 5.3.1]{GvL:2013:MatrixComputationsBook}; following the proof of this result, the source material presents some simplified estimates for these bounds.
The first step in producing the simplified estimate is to note that $\nu \leq \kappa(\mtx{A})$ holds for all nonzero $\vct{x}$.
Then, under the modest geometric assumption that $\theta$ is bounded away from $\pi/2$, 
\eqref{eq:sensitivity_x_ols} suggests that
\begin{equation}\label{eq:ols_sensitivity_x_est}
    \frac{\|\delta\vct{x}\|_2}{\|\vct{x}\|_2} \lesssim \epsilon\left\{\kappa(\mtx{A}) + \frac{\|\vct{b} - \mtx{A}\vct{x}\|_2}{\|\vct{b}\|_2} \kappa(\mtx{A})^2\right\}.
\end{equation}
The significance of this bound is that it shows the dependence of $\|\delta\vct{x}\|_2$ on the \textit{square} of the condition number of $\mtx{A}$.
This dependence is a fundamental obstacle to solving least squares problems to a high degree of accuracy when measured by forward error.
The situation is different if we try to bound the perturbation $\|\mtx{A}(\delta\vct{x})\|$.
We provide the following result (which completes the restatement of \cite[Theorem 5.3.1]{GvL:2013:MatrixComputationsBook}) as a step towards explaining why.

\begin{theorem}\label{thm:sensitivity_ols_y}
    Under the hypothesis and notation of \cref{thm:ols_sensitivity},
    we have
    \begin{align}
    \frac{\|\mtx{A}(\delta\vct{x})\|_2}{\|\vct{b} - \mtx{A}\vct{x}\|_2} &\leq \epsilon\left\{ \frac{1}{\sin \theta} + \kappa(\mtx{A})\left(\frac{1}{\nu \tan \theta} + 1\right)\right\} + O(\epsilon^2).\label{eq:sensitivity_y_ols}
    \end{align}
\end{theorem}
\cref{thm:sensitivity_ols_y} can be seen as a sensitivity analysis result for a very specific class of dual saddle point problems.
Specifically, since we have assumed that $\mtx{A}$ is full rank, $\vct{y}$ solves the dual problem if and only if $\vct{y} = \vct{b} - \mtx{A}\vct{x}$ where $\vct{x}$ solves the primal problem.
In the same vein, if $\vct{x} + \delta\vct{x}$ solves a perturbed primal problem and we set $\delta\vct{y} = -\mtx{A}(\delta\vct{x})$, then $\vct{y} + \delta\vct{y}$ solves the perturbed dual problem.

As with the bound for $\delta\vct{x}$, \eqref{eq:sensitivity_y_ols} can be estimated under mild geometric assumptions; \cite[pg. 267]{GvL:2013:MatrixComputationsBook} points out that if $\theta$ is sufficiently bounded away from 0 and $\pi/2$, then we should have
\begin{equation}\label{eq:sensitivity_y_ols_est}
    \frac{\|\delta\vct{y}\|_2}{\|\vct{y}\|_2} \lesssim \epsilon\, \kappa(\mtx{A}).
\end{equation}
This shows there is more hope for solving dual saddle point problems to a high degree of forward error accuracy, at least by comparison to primal saddle point problems.
Indeed, the following adaptation of \cite[Theorem 5.6.1]{GvL:2013:MatrixComputationsBook} provides an even more favorable sensitivity analysis result for underdetermined least squares.

\begin{theorem}\label{thm:uls_sensitivity}
    Let $\vct{y} = (\mtx{A}^{\trans})^{\dagger}\vct{c}$ solve the underdetermined least squares problem with data $(\mtx{A},\vct{c})$ for a nonzero vector $\vct{c}$.
    Consider perturbations $\delta\vct{c}$ and $\delta\mtx{A}$ where 
    \[
        \epsilon = \max\left\{\frac{\|\delta\vct{c}\|_2}{\|\vct{c}\|_2},~ \frac{\|\delta\mtx{A}\|_2}{\|\mtx{A}\|_2}\right\} < \sigma_n(\mtx{A}).
    \]
    The perturbation $\delta\vct{y}$ needed for $\vct{y} + \delta\vct{y}$ to solve the underdetermined least squares problem with data $(\mtx{A} + \delta\mtx{A}, \vct{c} + \delta\vct{c})$ satisfies
    \begin{equation}\label{eq:sensitivity_y_uls}
        \frac{\|\delta\vct{y}\|_2}{\|\vct{y}\|_2} \leq 3\, \epsilon\,\kappa(\mtx{A}) + O(\epsilon^2).
    \end{equation}
\end{theorem}

\subsection{Sharper sensitivity analysis for overdetermined problems}\label{subapp:sharper_sensitivity}

The analysis results in \cref{subapp:basic_unreg_ls_sensitivity} have notable limitations: they hide constants in $O(\epsilon^2)$ terms.
Luckily there are a wealth of more precise results in the literature that work with different notions of error.
One good example along these lines for overdetermined least squares is given in \cite[Fact 5.14]{Ilse:2009:book}, which obtains a relative error bound normalized by the solution of the \textit{perturbed problem} rather than the original problem.
We paraphrase this fact below.

\begin{theorem}\label{thm:ilse_ols_sensitivity}
    Consider an overdetermined least squares problem with data $(\mtx{A}_o,\vct{b}_o)$ that is solved by $\vct{x}_o = (\mtx{A}_o)^\dagger(\vct{b}_o)$; consider also perturbed problem data $\mtx{A} = \mtx{A}_o + \delta\mtx{A}_o$ and $\vct{b} = \vct{b}_o + \delta\vct{b}_o$ together with solution $\vct{x} = \mtx{A}^{\dagger}\vct{b}$.
    If we have $\rank(\mtx{A}_o) = \rank(\mtx{A}) = n$ and define  
    \[
    \epsilon_A = \frac{\|\delta\mtx{A}_o\|_2}{\|\mtx{A}_o\|_2} \quad\text{and}\quad \epsilon_b = \frac{\|\delta\vct{b}_o\|}{\|\mtx{A}_o\|_2\|\vct{x}\|_2},
    \]
    then we have
    \begin{equation}\label{eq:ilse_sensitivity_x_ols}
        \frac{\|\vct{x}_o - \vct{x}\|_2}{\|\vct{x}\|_2} \leq (\epsilon_A + \epsilon_b) \kappa(\mtx{A}_o)+ \epsilon_A \frac{\left[\kappa(\mtx{A}_o)\right]^2\|\vct{y}\|_2}{\|\mtx{A}_o\|_2\|\vct{x}\|_2}
    \end{equation}
    for $\vct{y} = \vct{b} - \mtx{A}\vct{x}$.
\end{theorem}
Bounds of similar character are given for $\vct{y}$ on \cite[pg. 101]{Ilse:2009:book}.
These bounds are useful for understanding how solutions exhibit different sensitivity for perturbations to the data matrix compared to perturbations to the right-hand-side.
Even better bounds can be obtained by assuming \textit{structured perturbations}.
For example, if $\range(\mtx{A}_o) = \range(\mtx{A})$, then the sensitivity of the overdetermined least squares solution depends only linearly on $\kappa(\mtx{A}_o)$ \cite[Exercise 5.1]{Ilse:2009:book}.

Our discussion of sensitivity analysis has only considered \textit{normwise} perturbations to the problem data.
More informative bounds can be had by considering \textit{componentwise} perturbations.
For example, one can measure a perturbation of an initial matrix $\mtx{A}$ by the smallest $\alpha$ for which $|\delta A_{ij}| \leq \alpha|A_{ij}|$ for all $i,j$.
We refer the reader to \cite[\S 20.1]{Higham:2002:book} for a componentwise sensitivity analysis result on overdetermined least squares.

\subsection{Simple constructions to bound backward error}\label{subsec:posteriori_backward_error}

Here we describe two methods for constructing perturbations to problem data that render an approximate solution exact.
By computing the norms of these perturbations, we can obtain upper bounds on (normwise) backward error.
Such bounds are useful as termination criteria for iterative solvers.

The notation here matches that of \cref{thm:ilse_ols_sensitivity}.
That is, we say our original least squares problem has data $(\mtx{A}_o,\vct{b}_o)$ and that $\vct{x}$ is an \textit{approximate} solution to this problem.

\begin{remark}
    For discussion on \textit{componentwise} backward error bounds for overdetermined least squares we again refer the reader to \cite[\S 20.1]{Higham:2002:book}.
\end{remark}

\subsubsection{Inconsistent overdetermined problems}
Letting $\vct{r} = \vct{b}_o - \mtx{A}_o\vct{x}$, we define
\begin{equation}\label{eq:back_err_construct:inconsistent_ols:A}
    \delta\mtx{A}_o = \frac{\vct{r}\vct{r}^{\trans}\mtx{A}_o}{\|\vct{r}\|_2^2}\quad\text{and}\quad\mtx{A} = \mtx{A}_o + \delta\mtx{A}_o,
\end{equation}
and subsequently
\begin{equation}\label{eq:back_err_construct:inconsistent_ols:b}
\delta\vct{b}_o = -(\delta\mtx{A}_o)\vct{x} \quad\text{and}\quad \vct{b} = \vct{b}_o + \delta\vct{b}_o.
\end{equation}
Some simple algebra shows that $\vct{x}$ satisfies the normal equations 
\[
    \mtx{A}^{\trans}\left(\vct{b} - \mtx{A}\vct{x}\right) = \vct{0},
\]
therefore it solves the overdetermined least squares problem with data $(\mtx{A},\vct{b})$.

This construction was first given in \cite[Theorem 3.2]{Stewart:1977:LINPACK}.
It is especially nice since the perturbation is rank-1, and so its spectral norm
\[
\|\delta\mtx{A}_o\|_2 = \frac{\|\mtx{A}_o^{\trans}\vct{r}\|_2}{\|\vct{r}\|_2}
\]
is easily computed at runtime by an iterative solver for overdetermined least squares.
Furthermore, if the iterative solver in question is LSQR, and if we assume exact arithmetic, then
the perturbation will satisfy $\delta\mtx{A}_o\vct{x} = \vct{0}$ \cite[\S 6.2]{PS:1982}.
Therefore the vector $\delta\vct{b}_o$ in \eqref{eq:back_err_construct:inconsistent_ols:b} is zero  when running LSQR (or any equivalent method) in exact arithmetic.

\subsubsection{Consistent overdetermined problems: a word of warning}

The perturbations given in \eqref{eq:back_err_construct:inconsistent_ols:A} -- \eqref{eq:back_err_construct:inconsistent_ols:b} are not suitable for least squares problems where the optimal residual, $(\mtx{I} -\mtx{A}_o\mtx{A}_o^{\dagger})\vct{b}_o$, is zero or nearly zero.
In these situations one should use a perturbation designed for consistent linear systems.
We describe such a construction here based on termination criteria used in LSQR.

Let $\delta\vct{b}_o$ be an arbitrary $m$-vector.
Suppose we set $\delta\mtx{A}_o$ as a function of $\delta\vct{b}_o$ in the following way:
\[
    \delta\mtx{A}_o = \frac{\left(\delta\vct{b}_o + \vct{b}_o - \mtx{A}_o\vct{x}\right)\vct{x}^{\trans}}{ \|\vct{x}\|_2^2}.
\]
It can be seen that $\mtx{A}\vct{x} = \vct{b}$ upon taking $\vct{b} = \vct{b}_o + \delta\vct{b}_o$ and $\mtx{A} = \mtx{A}_o + \delta\mtx{A}_o$, and so $\vct{x}$ trivially solves the perturbed least squares problem with data $(\mtx{A},\vct{b})$.

One can obtain many backward-error constructions by considering different choices for $\delta\vct{b}_o$ as a function of $\vct{x}$, the problem data $(\mtx{A}_o,\vct{b}_o)$, and desired error tolerances.
The construction for LSQR considers two tolerance parameters $\epsilon_A, \epsilon_b \in [0, 1)$, and sets $\delta\vct{b}_o$ as follows \cite[\S 6.1]{PS:1982}:
\begin{equation}\label{eq:back_err_construct:consistent_ols:lsqr_b}
    \delta\vct{b}_o = \left(\frac{\epsilon_b \|\vct{b}_o\|_2 }{\epsilon_b\|\vct{b}_o\|_2 + \epsilon_A \|\mtx{A}_o\|\|\vct{x}\|_2}\right)(\mtx{A}_o\vct{x} - \vct{b}_o).
\end{equation}
The parameters $\epsilon_A$ and $\epsilon_b$ indicate the (relative) sizes of perturbations to $(\mtx{A}_o,\vct{b}_o)$ that a user deems allowable.
The authors of \cite{PS:1982} suggest that ``allowable'' be based on the extent to which $(\mtx{A}_o,\vct{b}_o)$ are not known exactly in applications.

It is natural to want to reduce the two tolerance parameters $(\epsilon_A,\epsilon_b)$ to a single tolerance parameter.
For example, one might take $\epsilon_A = \epsilon_b$.
Unfortunately, our experience is that taking $\epsilon_A = \epsilon_b$ can produce unreliable algorithm behavior for overdetermined problems.
Therefore we recommend that one sets $\epsilon_A = 0$ if one wants to think only in terms of a single tolerance for consistent overdetermined problems.
While this may seem like a blunt solution, it ensures that $\delta\mtx{A}_o = \mtx{0}$, which is useful in applying \cref{thm:ilse_ols_sensitivity}.
If setting $\epsilon_A = 0$ still feels too extreme then one might consider setting $\epsilon_A = (\epsilon_b)^2 \ll \epsilon_b$.

\begin{remark}
    As a minor detail, we point out that the norm of $\mtx{A}_o$ in \eqref{eq:back_err_construct:consistent_ols:lsqr_b} is deliberately ambiguous.
    While the spectral norm would probably be most natural, the formal LSQR algorithm replaces $\|\mtx{A}_o\|$ by an \textit{estimate} of $\|\mtx{A}_o\|_{\mathrm{F}}$ that monotonically increases from one iteration to the next; see \cite[\S 5.3]{PS:1982}.
\end{remark}

\subsection{More advanced concepts}\label{subapp:advanced_lstsq_error_est}

Some of the earliest work on backward-error analysis for solutions to linear systems focused on componentwise backward error for direct methods \cite{OP:1964:back_err_linsys}.
A principal shortcoming of componentwise error metrics is that they are expensive to compute, especially as stopping criteria for iterative solvers.
\cite{ADR:1992:stop-iter-solver} investigates metrics for componentwise backward error suitable for iterative solvers.
%
%

The ``backward error plus sensitivity analysis'' approach may overestimate forward error.
Alternative estimates are available for some Krylov subspace methods such as PCG, wherein one uses algorithm-specific recurrences to estimate forward error in the Euclidean norm or the norm induced by $\Alift := [\mtx{A};\sqrt{\mu}]$.
See, for example, \cites{AK:2001:error_est_pcg,ST:2002:error_est_PCG,ST:2005:error_est_PCG}.
These error bounds are more accurate when used with a good preconditioner, which we can generally expect to have when using the randomized algorithms described herein.

It is not easy to apply sensitivity analysis results to compute forward error bounds at runtime.
A primary obstacle in this regard is the need to have accurate estimates for the extreme singular values of $\mtx{A}_o$ or the perturbation $\mtx{A}$ (depending on the sensitivity analysis result in question).
On this topic we note that if $\mtx{M}$ is an SVD-based preconditioner then we will have computed the singular values and right singular vectors of a sketch $\mtx{S}\mtx{A}_o$.
Those singular values can be used as approximations to the reciprocals of the singular values of $\mtx{A}_o$.
It is conceivable that more accurate approximations could be obtained by applying iterative preconditioned eigenvalue estimation methods for $(\mtx{A}_o)^{\trans}\mtx{A}_o$.
Such iterative methods typically require initialization with an approximate eigenvector.
On this front one can use the leading (resp. trailing) left singular vector of $\mtx{M}$ to approximate the trailing (resp. leading) right singular vector of $\mtx{A}_o$.
One should not expect too much of such estimates, however.\footnote{Any ``cheap'' method for estimating the smallest singular value even of \textit{triangular} matrices can return substantial overestimates and underestimates  \cite{DDM:2001:error_bound_complexity}.}

Finally, we note that some Krylov-subspace iterative methods can estimate condition numbers.
For example, when LSQR is applied to a problem with data matrix $\mtx{L}$, it can estimate the Frobenius condition number $\|\mtx{L}\|_{\mathrm{F}}\|\mtx{L}^{\dagger}\|_{\mathrm{F}}$.
Bear in mind that in our context we call LSQR with the preconditioned augmented data matrix, $\mtx{L} = \Alift\mtx{M}$.
It would be useful to embed estimators for \textit{componentwise condition numbers} (which are known to be computable in polynomial time \cite{Demmel:1992:dist_to_singular}) into Krylov subspace solvers.

\section{Ill-posed saddle point problems }\label{app:lstsq_details:inconsistent}

Our saddle point formulations of least squares problems can be problematic when $\mtx{A}$ is rank-deficient and $\mu$ is zero, in which case our problems can actually be infeasible or unbounded below.
This \nameCref{app:lstsq_details:inconsistent} uses a limiting analysis to define \textit{canonical solutions} to saddle point problems in these settings.

We begin by recalling
\begin{align}
    & \min_{\vct{x}\in\R^n}\|\mtx{A}\vct{x} - \vct{b}\|_2^2 + \mu\|\vct{x}\|_2^2 + 2\vct{c}^{\trans}\vct{x}, \tag{\eqref{eq:saddle_opt_x}, revisited} \\
    & \min_{\vct{y} \in \R^m} \|\mtx{A}^{\trans}\vct{y} - \vct{c}\|_2^2 + \mu\|\vct{y} - \vct{b}\|_2^2, \tag{\eqref{eq:underdet_ridge}, revisited} \\
    \text{and}\qquad & 
    \min_{\vct{y} \in \R^m}\{ \|\vct{y} - \vct{b} \|_2^2 \,:\, \mtx{A}^{\trans}\vct{y} = \vct{c} \}. \tag{\eqref{eq:saddle_opt_y}, revisited}
\end{align}
We also note the following form of solutions to \eqref{eq:underdet_ridge}, when $\mu$ is positive
\begin{equation}\label{eq:parameterized_opt_y}
    \vct{y}(\mu) = \left(\mtx{A}\mtx{A}^{\trans} + \mu\mtx{I}\right)^{-1}\left(\mtx{A}\vct{c} + \mu\vct{b}\right).
\end{equation}

\begin{proposition}\label{prop:limiting_sol_y}
    For any tall $m \times n$ matrix $\mtx{A}$, any $m$-vector $\vct{b}$, and any $n$-vector $\vct{c}$, we have
    \begin{equation}\label{eq:limiting_sol_y}
        \lim_{\mu \downarrow 0} \vct{y}(\mu) = (\mtx{A}^{\trans})^{\dagger}\vct{c} + (\mtx{I} - \mtx{A}\mtx{A}^\dagger)\vct{b}.
    \end{equation}
\end{proposition}
\begin{proof}
    Let $k = \rank(\mtx{A})$. If $k = 0$ then the claim is trivial since \eqref{eq:parameterized_opt_y} reduces to $\vct{y}(\mu) = \vct{b}$ for all $\mu > 0$.
    Henceforth, we assume $k > 1$.
    To establish the claim, consider how the compact SVD $\mtx{A} = \mtx{U}\mtx{\Sigma}\mtx{V}^{\trans}$ lets us express
    \[
        \mtx{A}\mtx{A}^{\trans} + \mu\mtx{I}_m = \mtx{H}_{\mu} + \mtx{G}_{\mu}
    \]
    in terms of the Hermitian matrices
    \begin{align*}
        \mtx{H}_{\mu} &= \mtx{U}\left(\mtx{\Sigma}^2 + \mu\mtx{I}_k\right)\mtx{U}^{\trans} \\
        \text{and}\quad\mtx{G}_{\mu} &= \mu(\mtx{I}_m - \mtx{U}\mtx{U}^{\trans}).
    \end{align*}
    Since $\mtx{H}_\mu\mtx{G}_{\mu} = \mtx{G}_{\mu}\mtx{H}_{\mu} = \mtx{0}$, the following identity holds for all positive $\mu$:
    \begin{align*}
        \left(\mtx{A}\mtx{A}^{\trans} + \mu\mtx{I}\right)^{-1} 
            &= \mtx{H}_{\mu}^{\dagger} + \mtx{G}_{\mu}^{\dagger}.
    \end{align*}
    Furthermore, by expressing
    \begin{align*}
        \mtx{H}_{\mu}^{\dagger}\mtx{A}\vct{c}
            &= \mtx{U}\left(\mtx{\Sigma}^2 + \mu\mtx{I}_k\right)^{-1}\mtx{\Sigma}\mtx{V}^{\trans}\vct{c} \\ 
        \mtx{G}_{\mu}^{\dagger}\mtx{A}\vct{c} &= \vct{0} \\
        \mu\mtx{H}_{\mu}^{\dagger}\vct{b} &= \mtx{U}\left(\mtx{\Sigma}^2/\mu + \mtx{I}_k\right)^{-1}\mtx{U}^{\trans}\vct{b} \\ 
        \mu\mtx{G}_{\mu}^{\dagger}\vct{b} &=  (\mtx{I}_m - \mtx{U}\mtx{U}^{\trans})\vct{b}
    \end{align*}
    we find that
    \begin{align*}
        \vct{y}(\mu) 
            &= \mtx{H}_\mu^{\dagger}(\mtx{A}\vct{c} + \mu\vct{b}) + \mtx{G}_{\mu}^{\dagger}(\mtx{A}\vct{c} + \mu\vct{b}) \\
            &\rightarrow \mtx{U}\mtx{\Sigma}^{-1}\mtx{V}^{\trans}\vct{c} + (\mtx{I}_m - \mtx{U}\mtx{U}^{\trans})\vct{b}.
    \end{align*}
    This is equivalent to the desired claim since $(\mtx{A}^{\trans})^{\dagger} = \mtx{U}\mtx{\Sigma}^{-1}\mtx{V}^{\trans}$ and $\mtx{A}\mtx{A}^{\dagger} = \mtx{U}\mtx{U}^{\trans}$.
\end{proof}

In light of the above proposition, we take $\vct{y}(0) = (\mtx{A}^{\trans})^{\dagger}\vct{c} + (\mtx{I} - \mtx{A}\mtx{A}^\dagger)\vct{b}$ as our canonical solution to the dual problem when $\mu = 0$.

Now we let $\vct{x}(\mu)$ denote the solution to \eqref{eq:saddle_opt_x} parameterized by $\mu > 0$.
It is clear that this is given by
\[
    \vct{x}(\mu) = \left(\mtx{A}^{\trans}\mtx{A} + \mu\mtx{I}\right)^{-1}(\mtx{A}^{\trans}\vct{b} - \vct{c}).
\]
It is easy to show that if $\vct{c}$ is not orthogonal to the kernel of $\mtx{A}$, then the norms $\|\vct{x}(\mu)\|$ will diverge to infinity as $\mu$ tends to zero.
However, if $\vct{c}$ \textit{is} orthogonal to the kernel of $\mtx{A}$, then we have
\begin{equation}\label{eq:limiting_sol_x}
    \lim_{\mu \downarrow 0}\vct{x}(\mu) =(\mtx{A}^{\trans}\mtx{A})^\dagger(\mtx{A}^{\trans}\vct{b} - \vct{c}) =: \vct{x}(0).
\end{equation}
We actually take the limit above as our canonical solution to the primal problem \eqref{eq:saddle_opt_x} regardless of whether or not $\vct{c}$ is orthogonal to the kernel of $\mtx{A}$.
Our reasons for this are two-fold. 
First, the values $\vct{x}(0), \vct{y}(0)$ given above are unchanged when $\vct{c}$ is replaced by its orthogonal projection onto range of $\mtx{A}^{\trans}$.
Second, the value $\vct{y}(0)$ is always the limiting solution to the dual problem. Meanwhile, the proposed value for $\vct{x}(0)$ relates to $\vct{y}(0)$ by $\vct{y}(0) = \vct{b} - \mtx{A}\vct{x}(0)$.

\section{Minimizing regularized quadratics}\label{subapp:details_min_reg_quadratics}

\cref{subapp:primer_krr} provides a brief introduction to kernel ridge regression (KRR). It covers the finite-dimensional linear algebraic formulation and the Hilbert space formulation of this regression model, and it explains how ridge regression can be understood in the KRR framework.
\cref{subapp:KRR_AM15_SASAP} presents a novel preconditioner-generation procedure for solving a \textit{sketch} of the regularized quadratic minimization problem \eqref{eq:unconstr_quadratic_opt}.

\subsection{A primer on kernel ridge regression}\label{subapp:primer_krr}



Kernel ridge regression (KRR) is a type of nonparametric regression for learning real-valued nonlinear functions $f : \mathcal{X} \to \R$.
It can be formulated as a linear algebra problem as follows: we are given $\lambda > 0$, an $m \times m$ psd ``kernel matrix'' $\mtx{K}$, and a vector of observations $\vct{h}$ in $\R^m$; we want to solve
\begin{equation}\label{eq:krr_finite_quadopt}
 \operatornamewithlimits{argmin}_{\vct{\alpha}\in\R^m} \textstyle\frac{1}{m}\|\mtx{K}\vct{\alpha} - \vct{h}\|_2^2 + \lambda\, \vct{\alpha}^{\trans}\mtx{K}\vct{\alpha}.
\end{equation}
\noindent Equivalently, we want to solve the \textit{KRR normal equations} $(\mtx{K} + m \lambda \,\mtx{I})\vct{\alpha} = \vct{h}$.
The normal equations formulation makes it clear that KRR is an instance of \eqref{eq:unconstr_quadratic_opt}.

A standard library for \RandNLA{} would be well-served to not dwell on how $\mtx{K}$ is defined; it should instead only focus on how $\mtx{K}$ can be accessed.
However, strictly speaking, \eqref{eq:krr_finite_quadopt} only encodes a KRR problem when the entries of $\mtx{K}$ are given by pairwise evaluations of a suitable two-argument \textit{kernel function} on some datapoints $\{\vct{x}_i\}_{i=1}^m \subset \mathcal{X}$.
Letting $k : \mathcal{X} \times \mathcal{X} \to \R$ denote this kernel function, 
the user will take $\vct{\alpha}$ that approximately solves \eqref{eq:krr_finite_quadopt} to define the learned model $g(\vct{z}) = \sum_{i=1}^m \alpha_i k(\vct{x}_i, \vct{z})$.

\subsubsection{A more technical description}

The kernel function $k$ induces a \textit{reproducing kernel Hilbert space}, $\mathcal{H}$, of real-valued functions on $\mathcal{X}$.
This space is (up to closure) equal to the set of real-linear combinations of functions $\vct{y} \mapsto k^{\vct{u}}(\vct{y}) := k(\vct{y},\vct{u})$ parameterized by $\vct{u} \in \mathcal{X}$.
Additionally, if the function
\[
        \vct{y} \mapsto f(\vct{y}) = \sum_{i=1}^m \alpha_i k(\vct{y},\vct{x}_i)
\]
is parameterized by $\vct{\alpha} \in \R^m$ and $\{\vct{x}_i\}_{i=1}^m \subset \mathcal{X}$, then its squared norm is given by
\[
\|f\|_{\mathcal{H}}^2 = \sum_{i=1}^m\sum_{j=1}^m \alpha_i\alpha_j k(\vct{x}_i,\vct{x}_j).
\]
Using the kernel matrix $\mtx{K}$ with entries $K_{ij} = k(\vct{x}_i,\vct{x}_j)$, we can express that squared norm as $\|f\|_{\mathcal{H}}^2 = \vct{\alpha}^{\trans}\mtx{K}\vct{\alpha}$. 
Furthermore, for any $\vct{u} \in \mathcal{X}$ and any $f \in \mathcal{H}$ we have $f(\vct{u}) = \langle f, k^{\vct{u}}\rangle_{\mathcal{H}}$.
For details on reproducing kernel Hilbert spaces we refer the reader to \cite{aronszajn:1950:reprod_kern_hilbert_space}.

KRR problem data consists of observations $\{(\vct{x}_i, h_i)\}_{i=1}^m \subset \mathcal{X} \times \R$ and a positive regularization parameter $\lambda$.
We presume there are functions $g$ in $\mathcal{H}$ for which $g(\vct{x}_i) \approx h_i$, and we try to obtain such a function by solving
\begin{equation}\label{eq:krr_funcspace}
\min_{g \in \mathcal{H}}~ \frac{1}{m}\sum_{i=1}^m (g(\vct{x}_i) - h_i)^2 + \lambda \| g \|_{\mathcal{H}}^2.
\end{equation}
It follows from \cite{KW:1970:representer_thm} that the solution to \eqref{eq:krr_funcspace} is in the span of the functions $\{k^{\vct{x}_i}\}_{i=1}^m$.
Specifically, the solution is
$g_\star = \sum_{i=1}^m \alpha_i k^{\vct{x}_i}$ 
where $\vct{\alpha}$ solves \eqref{eq:krr_finite_quadopt}.

\subsubsection{Why is ridge regression a special case of kernel ridge regression?}

Suppose we have an $m \times n$ matrix $\mtx{X} = [\vct{x}_1,\ldots,\vct{x}_n]$ with linearly independent columns, and that we want to estimate a linear functional $\hat{g} : \R^m \to \R$  given access to the $n$ point evaluations $(\vct{x}_i, \hat{g}(\vct{x}_i))_{i=1}^n$.

Given a regularization parameter $\lambda > 0$, ridge regression concerns finding the linear function $g : \R^m \to \R$ that minimizes
\[
    L(g) = \left\|\begin{bmatrix} g(\vct{x}_1) \\ \vdots \\ g(\vct{x}_n) \end{bmatrix} - \begin{bmatrix} \hat{g}(\vct{x}_1) \\ \vdots \\ \hat{g}(\vct{x}_n)\end{bmatrix}\right\|_2^2 + n \lambda  \|g\|^2.
\]
To make this concrete, let us represent $\hat{g}$ and $g$ by $m$-vectors $\vct{\hat{g}}$ and $\vct{g}$ respectively, and set $\vct{h} = \mtx{X}^{\trans}\vct{\hat{g}}$.
We also adopt a slight abuse of notation to write $L(g) = L(\vct{g})$, so that the task of ridge regression can be framed as minimizing
\[
    L(g) = \|\mtx{X}^{\trans}\vct{g} - \vct{h}\|_2^2 + \lambda n \|\vct{g}\|_2^2.
\]

\begin{remark}
    We pause to emphasize that this is a KRR problem with $n$ datapoints that define functions on $\mathcal{X} = \R^m$.
    The parameter ``$m$'' here has nothing to do with the number of datapoints in the problem; our notational choices for $(m, n)$ here are for consistency with \cref{sec3:LS_and_optim}.
\end{remark}

The essential part of framing ridge regression as a type of KRR is showing that the optimal estimate $\vct{g}$ is in the range of $\mtx{X}$.
To see why this is the case, let $\mtx{P}$ denote the orthogonal projector onto the range of $\mtx{X}$.
Using $\mtx{X}^{\trans}\mtx{P}\vct{g} = \mtx{X}^{\trans}\vct{g}$, we have that
\begin{align*}
    L(\vct{g}) &= \|\mtx{X}^{\trans}\mtx{P}\vct{g} - \vct{h}\|_2^2 + \lambda n \left(\|\mtx{P}\vct{g}\|_2^2 + \|(\mtx{I} - \mtx{P})\vct{g}\|_2^2\right) \\
    &\geq \|\mtx{X}^{\trans}\mtx{P}\vct{g} - \vct{h}\|_2^2 + \lambda n \|\mtx{P}\vct{g}\|_2^2 \\
    &= L(\mtx{P}\vct{g}),
\end{align*}
and so $g$ minimizes $L$ only if $L(\mtx{P}\vct{g}) = L(\vct{g})$.
Since $L(\mtx{P}\vct{g}) = L(\vct{g})$ holds if and only if $\mtx{P}\vct{g} = \vct{g}$, we have that $\vct{g} = \mtx{X}\vct{\alpha}$ for some $\vct{\alpha}$ in $\R^n$.
Therefore, under our stated assumption that the columns of $\mtx{X}$ are linearly independent, the following problems are equivalent
\begin{align*}
    & \argmin\{ L(g) \,:\, g \text{ is a linear functional on } \R^m\}, \\
    & \argmin\{ \|\mtx{X}^{\trans}\mtx{X}\vct{\alpha} - \vct{h}\|_2^2 + \lambda n\|\mtx{X}\vct{\alpha}\|_2^2 \,:\, \vct{\alpha} \in \R^n \}, \text{ and } \\
    & \argmin\{ \|\mtx{X}\vct{\alpha} - \vct{\hat{g}}\|_2^2 + \lambda n \|\vct{\alpha}\|_2^2 \,:\, \vct{\alpha} \in \R^n \}.
\end{align*}
The second of these problems is KRR with a scaled objective and the $n \times n$ kernel matrix $\mtx{K} = \mtx{X}^{\trans}\mtx{X}$.
The last of these problems is ridge regression in the familiar form.

Given this description of ridge regression, one obtains KRR by applying the so-called ``kernel trick'' (see, e.g., \cite[\S 14]{Murphy:2012:ML_book}).
That is, one replaces $h_j = \vct{x}_j^{\trans}\vct{\hat{g}}$ by
\[
    h_j = \langle k^{\vct{x}_j}, \hat{g}
    \rangle_{\mathcal{H}} = \hat{g}(\vct{x}_j)
\]
and expresses the point evaluation of $g = \sum_{i=1}^n \alpha_i k^{\vct{x}_i}$ at $\vct{x}_j$ by
\[
    g(\vct{x}_j) = \sum_{i=1}^n \alpha_i \langle k^{\vct{x}_i}, k^{\vct{x}_j}\rangle_{\mathcal{H}}.
\]
We note that within the KRR formalism it is allowed for $\mtx{K}$ to be singular, so long as it is psd.
This is because if $\vct{\beta}$ is any vector in the kernel of $\mtx{K}$ then the function $\vct{u} \mapsto \sum_{i=1}^n \beta_i k(\vct{x}_i, \vct{u})$ is identically equal to zero.

\subsection{Efficient sketch-and-solve for regularized quadratics}\label{subapp:KRR_AM15_SASAP}

Let $\mtx{G}$ be an $m \times m$ psd matrix and $\mu$ be a positive regularization parameter.
The sketch-and-solve approach to KRR from \cite{AM:2015:KRR} can be considered generically as a sketch-and-solve approach to the regularized quadratic minimization problem \eqref{eq:unconstr_quadratic_opt}.
The generic formulation is to approximate $\mtx{G} \approx \mtx{A}\mtx{A}^*$ with an $m \times n$ matrix $\mtx{A}$ ($m \gg n$) and then solve
\begin{equation}\label{eq:krr_sketch_k_norm_eq}
    \left(\mtx{A}\mtx{A}^* + \mu \mtx{I}\right)\vct{z} = \vct{h}.
\end{equation}
Identifying $\vct{b} = \vct{h} / \mu$, $\vct{c} = \vct{0}$, and $\vct{y} = \vct{z}$ shows that this amounts to a dual saddle point problem of the form \eqref{eq:underdet_ridge}.
Here we explain how the sketch-and-precondition paradigm can efficiently be applied to solve \eqref{eq:krr_sketch_k_norm_eq} under the assumption that $\mtx{A}\mtx{A}^{\trans}$ defines a \Nystrom{} approximation of $\mtx{G}$.

Let $\mtx{S}_o$ be an initial $m \times n$ sketching operator.
The resulting sketch $\mtx{Y} = \mtx{G}\mtx{S}_o$ and factor $\mtx{R} = \texttt{chol}(\mtx{S}_o^{\trans}\mtx{Y})$ together define $\mtx{A} = \mtx{Y}\mtx{R}^{-1}$.
This defines a \Nystrom{} approximation since
\[
    \mtx{A}\mtx{A}^{\trans} = \left(\mtx{K}\mtx{S}_o\right)\left(\mtx{S}_o^{\trans}\mtx{K}\mtx{S}_o^{}\right)^{\dagger}\left(\mtx{K}\mtx{S}_o\right)^{\trans}.
\]
Recall that the problem of preconditioner generation entails finding an orthogonalizer of a sketch of $\mtx{A}_{\mu} = [\mtx{A}; \sqrt{\mu}\mtx{I}]$.
The fact that $\mtx{A}$ is only represented implicitly makes this delicate, but it remains doable, as we explain below.

For the sketching phase of preconditioner generation, we sample a $d \times m$ operator $\mtx{S}$ (with $d \gtrsim n$) and set
\[
    \mtx{A}_{\mu}^{\text{sk}} = \begin{bmatrix} \mtx{S} & \mtx{0} \\ \mtx{0} & \mtx{I} \end{bmatrix}\mtx{A}_{\mu} =  \begin{bmatrix} \mtx{S}\mtx{Y} \\ \sqrt{\mu}\mtx{R} \end{bmatrix}\mtx{R}^{-1}.
\]
We then compute the SVD of the augmented matrix
\[
    \begin{bmatrix} \mtx{S}\mtx{Y} \\ \sqrt{\mu}\mtx{R} \end{bmatrix} = \mtx{U}\mtx{\Sigma}\mtx{V}^{\trans}
\]
and find that setting $\mtx{M} = \mtx{R}\mtx{V}\mtx{\Sigma}^{-1}$ satisfies $\mtx{A}_{\mu}^{\text{sk}}\mtx{M} = \mtx{U}$.
The preconditioned linear operator $\mtx{A}_{\mu}\mtx{M}$ (and its adjoint) should be applied in the iterative solver by noting the identity
\[
    \begin{bmatrix}\mtx{A} \\ \sqrt{\mu}\mtx{I}\end{bmatrix} \mtx{M} = \begin{bmatrix} \mtx{Y} \\ \sqrt{\mu}\mtx{R}\end{bmatrix}\mtx{V}\mtx{\Sigma}^{-1}.
\]
This identity is important since it shows that $\mtx{R}^{-1}$ need never be applied at any point in the sketch-and-precondition approach to \eqref{eq:krr_sketch_k_norm_eq}.

\chapter[Details on Low-Rank Approximation]{Details on Low-Rank \\ Approximation 
}
\chaptermark{Low-rank Approximation}
\label{app:lowrank}

\minitoc

\section{Theory for submatrix-oriented decompositions}\label{app:theory_submatrix_oriented}

\subsection{Approximation quality in low-rank ID}\label{app:lowrank:ID}

\begin{proposition}[Restatement of \cref{prop:regularity_ID_accuracy}]
    Let $\mtx{\tilde{A}}$ be any rank-$k$ approximation of $\Ao$ that satisfies the spectral norm error bound $\|\Ao - \mtx{\tilde{A}}\|_2 \leq \epsilon$.
    If 
     $\mtx{\tilde{A}} = \mtx{\tilde{A}}[\fslice,J]\mtx{X}$ for some $k \times n$ matrix $\mtx{X}$ and an index vector $J$, then $\Aa = \Ao[\fslice,J]\mtx{X}$ is a low-rank column ID that satisfies
    \begin{equation}
        \|\Ao - \Aa \|_2 \leq (1 + \|\mtx{X}\|_2) \epsilon.\tag{\eqref{eq:carryover_ID_bound}\text{, restated}}
    \end{equation}
    Furthermore, if $|X_{ij}| \leq M$ for all $(i,j)$, then
    \begin{equation}
        \|\mtx{X}\|_2 \leq \sqrt{1 + M^2 k (n - k)}. \tag{\eqref{eq:interp_matrix_spectral_bound}\text{, restated}}
    \end{equation}
\end{proposition}
\begin{proof}

Proceeding in the grand tradition of adding zero and applying the triangle inequality, we have the bound
\begin{align}
     \|\Ao - \Aa \|_2 
        &= \| (\Ao - \mtx{\tilde{A}}) + (\mtx{\tilde{A}} - \Aa)\|_2 \nonumber \\
        &\leq \| \Ao - \mtx{\tilde{A}}\|_2 + \|\mtx{\tilde{A}} - \Aa\|_2. \label{eq:ID_quality_bound_triangle_ineq}
\end{align}
We prove \eqref{eq:carryover_ID_bound} by bounding the two terms in \eqref{eq:ID_quality_bound_triangle_ineq}.
The first term is trivial since we have already assumed $\| \Ao - \mtx{\tilde{A}}\|_2 \leq \epsilon$.
We bound the second term by using the identity $\mtx{\tilde{A}} - \Aa = (\mtx{\tilde{A}}[:, J] - \Ao[:, J])\mtx{X}$ and then invoking submultiplicivity of the spectral norm:
\begin{equation}\label{eq:ID_quality_bound_submultiplicative}
     \|\mtx{\tilde{A}} - \Aa\|_2  \leq \|\mtx{\tilde{A}}[:, J] - \Ao[:, J]\|_2 \|\mtx{X}\|_2.
\end{equation}
Finally, since the spectral norm of a matrix is at least as large as the spectral norm of any of its submatrices, we obtain $\|\mtx{\tilde{A}}[:, J] - \Ao[:, J]\|_2 \leq \|\mtx{\tilde{A}} - \Ao\|_2 \leq \epsilon$.
Combining this with \eqref{eq:ID_quality_bound_submultiplicative} shows that \eqref{eq:ID_quality_bound_triangle_ineq} implies \eqref{eq:carryover_ID_bound}.

Now we address \eqref{eq:interp_matrix_spectral_bound}.
For this, consider the $m \times k$ matrix $\mtx{C} = \mtx{\tilde{A}}[\fslice,J]$.
Because we have assumed $\mtx{\tilde{A}} = \mtx{C}\mtx{X}$ has rank $k$ and that $\mtx{X}$ is $k \times n$, we can infer that $\mtx{C}$ is full column-rank.
We can also extract the columns from both sides of $\mtx{\tilde{A}} = \mtx{C}\mtx{X}$ at indices $J$ to find the identity $\mtx{C} = \mtx{C}\mtx{X}[\fslice,J]$.
Multiplying this identity through on the left by $\mtx{C}^{\dagger}$, we can use  $\mtx{C}^{\dagger}\mtx{C} = \mtx{I}_k$ to obtain $\mtx{X}[\fslice,J] = \mtx{I}_k$.

Now if $|X_{ij}| \leq M$ for all $(i,j)$ \textit{in addition to} $\mtx{X}[\fslice,J] = \mtx{I}_k$, then there exists a permutation $P$ of the column index set $\idxs{n}$ where $\mtx{X}[\fslice,P] = [\mtx{I}_k, \mtx{V}]$ for a $k \times (n-k)$ matrix $\mtx{V}$ satisfying $|V_{ij}| \leq M$.
Since permuting the columns of $\mtx{X}$ does not change its spectral norm, it suffices to bound $\|[\mtx{I}_k, \mtx{V}]\|_2$.
Towards this end, we claim without proof that for any block matrix $\mtx{W} = [\mtx{U},\mtx{V}]$, one has
\[
    \|\mtx{W}\|_2 \leq \sqrt{\|\mtx{U}\|_2^2 + \|\mtx{V}\|_{\mathrm{F}}^2}.
\]
This, combined with $\|\mtx{I}_k\|_2 = 1$ and $\|\mtx{V}\|_{\mathrm{F}}^2 \leq M^2 k (n-k)$, gives \eqref{eq:interp_matrix_spectral_bound}.
\end{proof}

\begin{remark}
    The bound \eqref{eq:interp_matrix_spectral_bound} is not the best possible.
    Indeed, looking at the final steps in the proposition's proof, we see that it suffices for $M$ to bound the entries of $\mtx{X}$ that are not part of the identity submatrix.
\end{remark}

\subsection{Truncation in column-pivoted matrix decompositions}\label{app:theory_truncate_column_piv}

This part of the appendix follows up on \cref{subsec:column_pivoted}.
It examines how changing basis of a column-pivoted decomposition can affect approximation quality when truncating these decompositions.
To begin, we set forth some definitions.

Our matrix of interest, $\mtx{G}$, is $r \times c$ and given through a decomposition $\mtx{G}\mtx{P} = \mtx{F}\mtx{T}$ for a permutation matrix $\mtx{P}$ and upper-triangular $\mtx{T}$.
Let us say that $\mtx{F}$ has $w = \min\{c, r\}$ columns (and hence that $\mtx{T}$ has as many rows).
We use $\mathcal{U}$ to denote the set invertible upper-triangular matrices of order $w$.
For a positive integer $k < \rank(\mtx{G})$, we consider the matrix-valued function $g_k$ that is defined on $\mathcal{U}$ according to
\begin{equation*}
         g_k(\mtx{U}) =  \mtx{F}(\mtx{U}^{-1})[\fslice,\lslice{k}] \mtx{U}[\lslice{k},\fslice]\mtx{T}\mtx{P}^{\trans}.
\end{equation*}
Note that for every diagonal $\mtx{D} \in \mathcal{U}$ we have $g_k(\mtx{D}) = (\mtx{F}[\fslice,\lslice{k}])(\mtx{T}[\lslice{k},\fslice])\mtx{P}^{\trans}$.

\begin{proposition}\label{prop:change_basis_column_piv}
    Partition the factors $\mtx{F}$ and $\mtx{T}$ into blocks $[\mtx{F}_1,\mtx{F}_2]$ and $[\mtx{T}_1;\mtx{T}_2]$
    so that $\mtx{F}_1$ has $k$ columns and $\mtx{T}_1$ has $k$ rows.
    If $\mtx{U}_\star$ is an optimal solution to 
    \begin{equation}\label{eq:column_piv_change_basis_fancy}
        \operatornamewithlimits{min}_{\mtx{U} \in \mathcal{U}}\|\mtx{G} - g_k(\mtx{U})\|_{\mathrm{F}}.
    \end{equation}
    then the following identity holds in any unitarily invariant norm:
    \[
        \left\|\mtx{G} - g_k(\mtx{U}_\star)\right\| =
        \left\|\left(\mtx{I} - \mtx{F}^{}_1\mtx{F}_1^{\dagger}\right)\mtx{F}_2\mtx{T}_2\right\|.
    \]
    Furthermore, we have $\|\mtx{G} - g_k(\mtx{I}_{w \times w})\| = \|\mtx{F}_2\mtx{T}_2\|$, and the identity matrix is optimal for \eqref{eq:column_piv_change_basis_fancy} if and only if $\range(\mtx{F}_2)$ and $\range(\mtx{F}_1)$ are orthogonal.
\end{proposition}
\begin{proof}
    Our proof requires working with several block matrices.
    First, the matrices $\mtx{T}_1$ and $\mtx{T}_2$ are further partitioned so that
    \[
        \mtx{T} = \begin{bmatrix} \mtx{T}_1 \\ \mtx{T}_2 \end{bmatrix} = \begin{bmatrix} \mtx{T}_{11} & \mtx{T}_{12} \\ \mtx{0} & \mtx{T}_{22} \end{bmatrix}
    \]
    where $\mtx{T}_{11}$ is $k \times k$ and $\mtx{T}_{22}$ is $(w-k) \times (c-k)$.
    Next, we introduce $\mtx{U} \in \mathcal{U}$ and partition it twice:
    \[
        \mtx{U} = \begin{bmatrix} \mtx{U}_1 \\ \mtx{U}_2 \end{bmatrix} = \begin{bmatrix} \mtx{U}_{11} & \mtx{U}_{12} \\ \mtx{0} & \mtx{U}_{22} \end{bmatrix}.
    \]
    In the expressions above, $\mtx{U}_1$ is a wide matrix of shape $k \times w$, $\mtx{U}_{11}$ is a square matrix of order $k$, and $\mtx{U}_{22}$ a square matrix of order $w-k$.
    Note that since $\mtx{U}$ is upper-triangular, the same is true of $\mtx{V} = \mtx{U}^{-1}$.
    We partition $\mtx{V} = [\mtx{V}_1, \mtx{V}_2]$ into a leading block of $k$ columns and a trailing block of $w-k$ columns.
    
    The point of all this notation is to help us find useful expressions for $g_k(\mtx{U})$.
    Our first such expression is
    \begin{equation}\label{eq:expr_gkU1}
        g_k(\mtx{U}) = \mtx{F}\mtx{V}_1\mtx{U}_1\mtx{T}\mtx{P}^{\trans}.
    \end{equation}
    As a step towards finding the next expression, we apply block a matrix-inversion identity to compute
    \[
        \mtx{F}\mtx{V}_1 = \mtx{F}_1\mtx{U}_{11}^{-1}.
    \]
    Meanwhile, a simple block matrix multiply gives
    \[
    \mtx{U}_1\mtx{T} = \mtx{U}_{11}\mtx{T}_1 + \mtx{U}_{12}\mtx{T}_2.
    \]
    We plug these two identities into \eqref{eq:expr_gkU1} to obtain
    \begin{equation*}\label{eq:expr_gkU2}
        g_k(\mtx{U}) = \mtx{F}_1\left(\mtx{T}_1 + \mtx{U}_{11}^{-1}\mtx{U}_{12}\mtx{T}_2\right)\mtx{P}^{\trans}.
    \end{equation*}
    This expression is just what we need.
    By combining it with the identity $\mtx{G} = \mtx{F}\mtx{T}\mtx{P}^{\trans}$, we easily compute the difference
    \begin{equation}\label{eq:expr_G_gk_diff}
        \mtx{G} - g_k(\mtx{U}) = (\mtx{F}_2\mtx{T}_2 - \mtx{F}^{}_1\mtx{U}_{11}^{-1}\mtx{U}^{}_{12}\mtx{T}^{}_2)\mtx{P}^{\trans}.
    \end{equation}
    
    Having access to \eqref{eq:expr_G_gk_diff} marks a checkpoint in our proof.
    With it, we obtain the following identity for the distance from $\mtx{G}$ to $g_k(\mtx{U})$ in any unitarily invariant norm:
    \[
        \|\mtx{G} - g_k(\mtx{U})\| = \|\mtx{F}_2\mtx{T}_2 - \mtx{F}^{}_1\mtx{U}_{11}^{-1}\mtx{U}^{}_{12}\mtx{T}^{}_2\|.
    \]
    This implies our claim about $g_k(\mtx{I}_{w \times w})$, since if $\mtx{U} $ is diagonal, then $\mtx{U}_{12} = \mtx{0}$.
    Therefore all that remains is our claim about matrices that solve \eqref{eq:column_piv_change_basis_fancy}.
    The truth of this claim is easier to see after a change variables.
    Upon replacing $\mtx{U}_{11}^{-1}\mtx{U}^{}_{12}$ by a general $k \times (w-k)$ matrix, we can express
    \[
        \operatornamewithlimits{min}_{\mtx{U} \in \mathcal{U}}\|\mtx{G} - g_k(\mtx{U})\|_{\mathrm{F}}  = \min\left\{ \|\mtx{F}_2\mtx{T}_2 - \mtx{F}^{}_1\mtx{B}\mtx{T}^{}_2\|_{\mathrm{F}}~\big|~ \mtx{M} \in \R^{k \times (w-k)}\right\},
    \]
    and it is easily shown that the matrix $\mtx{M}_{\star} = \mtx{F}_1^{\dagger}\mtx{F}_2^{}$ is an optimal solution to the problem on the right-hand side of this equation.
\end{proof}

\section[Computational routines]{Computational routine interfaces and implementations}\label{app:lowrank:pseudocode}

As we explained in \cref{sec4:lowrank}, the design space for low-rank approximation algorithms is quite large.
Here we illustrate the breadth and depth of that design space with pseudocode for computational routines needed for four drivers: \code{SVD1}, \code{EVD1}, \code{EVD2}, and \code{CURD1} (Algorithms \ref{alg:svd} through \ref{alg:curd1}, respectively).
All pseudocode here uses Python-style zero-based indexing.
 
The dependency structure of these drivers and their supporting functions is given in \cref{fig:lowrank_approx_dependencies}.
From the figure we see that the following three interfaces are central to low-rank approximation.
\begin{itemize}
    \item $\mtx{Y} = \code{Orth}(\mtx{X})$ returns an orthonormal basis for the range of a tall input matrix; the number of columns in $\mtx{Y}$ will never be larger than that of $\mtx{X}$ and may be smaller. The simplest implementation of \code{Orth} is to return the orthogonal factor from an economic QR decomposition of $\mtx{X}$.
    \item $\mtx{S} = \code{SketchOpGen}(\ell, k)$ returns an $\ell \times k$ oblivious sketching operator sampled from some predetermined distribution.
    The most common distributions used for low-rank approximation were covered in \cref{subsec:dense_skops}.
    In actual implementations, this function would accept an input representing the state of the random number generator.
    \item $\mtx{Y} = \code{Stabilizer}(\mtx{X})$ has similar semantics $\code{Orth}$. It differs in that it only requires $\mtx{Y}$ to be better-conditioned than $\mtx{X}$ while preserving its range.
    The relaxed semantics open up the possibility of methods that are less expensive than computing an orthonormal basis, such as taking the lower-triangular factor from an LU decomposition with column pivoting.
\end{itemize}
We explain the remaining interfaces as they arise in our implementations.

\begin{figure}[!htb]
    \centering
    \includegraphics[width=.65\textwidth]{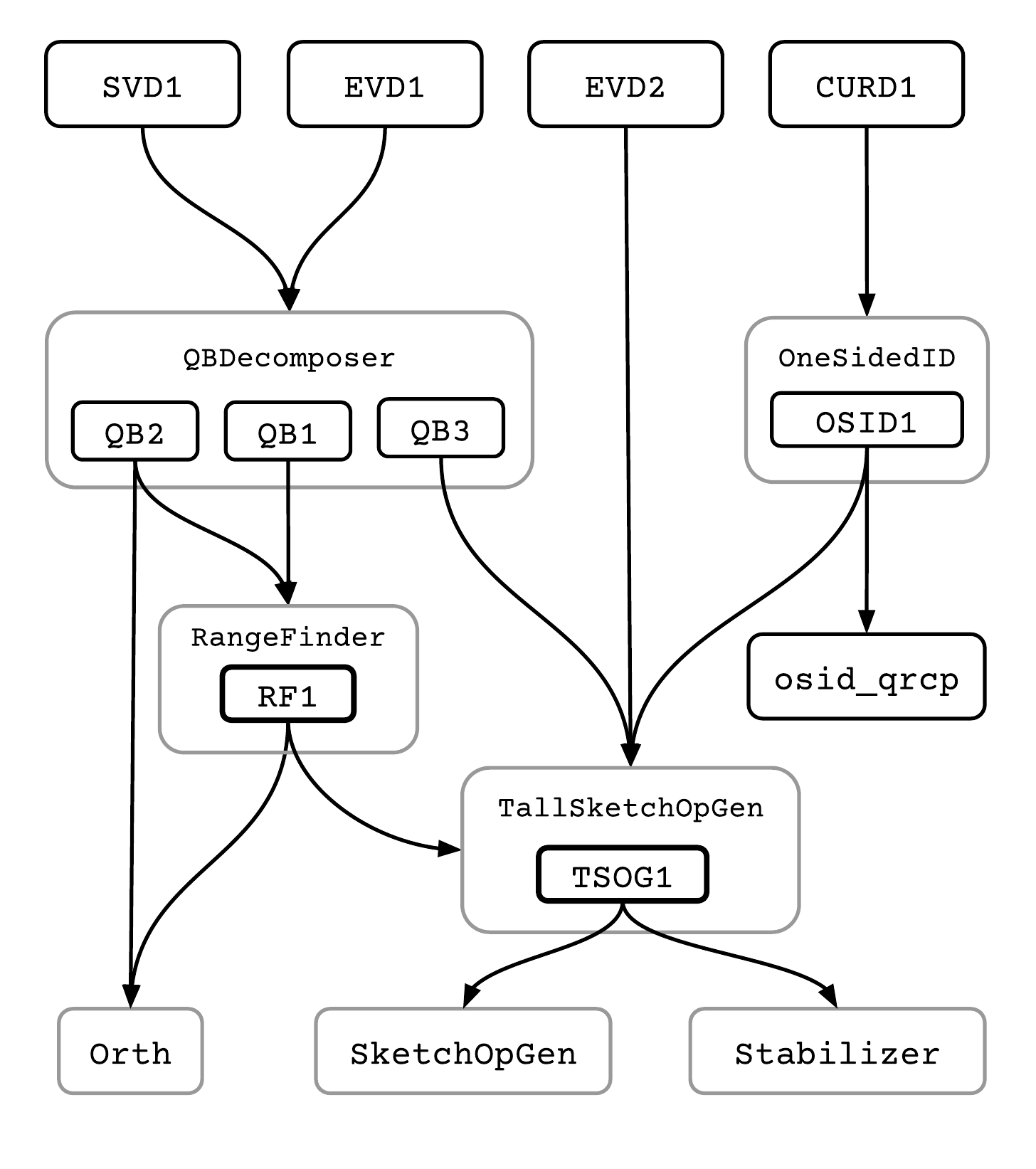}
    {\small \caption{Dependency illustration for low-rank approximation functionality.
    Lighter gray boxes correspond to abstract \textit{interfaces} which specify semantics. 
    Any interface can have many different \textit{implementations}. 
    To keep things at a reasonable length the only interface with multiple implementations is \code{QBDecomposer}. \cref{subsec:lowrank_subroutines} describes many algorithms that could be used for any such interface.
    }\label{fig:lowrank_approx_dependencies}}
\end{figure}

\FloatBarrier

The computational routines represented in \cref{fig:lowrank_approx_dependencies} include Algorithms \ref{alg:rf1} through \ref{alg:osid1}.
This appendix provides pseudocode for one additional function that is not reflected in the figure: \cref{alg:rocs1} shows one way to perform row or column subset selection.
We note that while \cref{alg:rocs1} is not used in the drivers mentioned above, it could
easily have been used in a different implementation of~\code{CURD1}.

\subsection{Power iteration for data-aware sketching}\label{app:subsub:power_iter}

When a \code{TallSketchOpGen} is called with parameters $(\mtx{A}, k)$, it produces an $n \times k$ sketching operator where $\range(\mtx{S})$ is reasonably well-aligned with the subspace spanned by the $k$ leading right singular vectors of $\mtx{A}$.
Here, ``reasonably'' is assessed with respect to the computational cost incurred by running \code{TallSketchOpGen}.
One extreme case of interest is to return an oblivious sketching operator without reading any entries of $\mtx{A}$.

This method uses a $p$-step power iteration technique.
When $p = 0$, the method returns an oblivious sketching operator.
It is recommended that one use $p > 0$ (e.g., $p \in \{2, 3\}$) when the singular values of $\mtx{A}$ exhibit ``slow'' decay.


\begin{algorithm}[H]
\caption{\code{TSOG1} : a \code{TallSketchOpGen} based on a power method, conceptually following \cite{ZM:2020:LU}. The returned sketching operator is suitable for sketching $\mtx{A}$ from the right for purposes of low-rank approximation.}\label{alg:TSOG1}
\setstretch{1.0}
\begin{algorithmic}[1]
\State \textbf{function} $\code{TSOG1}(\mtx{A}, k)$\vspace{0.5pt} 
        \Indent
            \Statex \quad Inputs:
            \Statex \begin{quote}
                $\mtx{A}$ is $m \times n$, and $k \ll \min\{m, n\}$ is a positive integer.
            \end{quote}\vspace{4pt}
            \Statex \quad Output:
            \Statex \begin{quote}
                $\mtx{S}$ is $n \times k$, intended for later use in computing $\mtx{Y} = \mtx{A}\mtx{S}$.
            \end{quote}\vspace{4pt}
            \Statex \quad Abstract subroutines:
            \Statex \begin{quote}
               \code{SketchOpGen} and \code{Stabilizer}
            \end{quote}\vspace{4pt}
            \Statex \quad Tuning parameters:
            \Statex \begin{quote}
                $p \geq 0$ controls the number of steps in the power method. It is equal to the total number of matrix-matrix multiplications that will involve either $\mtx{A}$ or $\mtx{A}^{\trans}$.
                If $p = 0$ then this function returns an oblivious sketching operator.\\
                $q \geq 1$ is the number of matrix-matrix multiplications with $\mtx{A}$ or $\mtx{A}^{\trans}$ that accumulate before the stabilizer is called.
            \end{quote}\vspace{4pt}
            \setstretch{1.1}
    \State $p_{\mathrm{done}} = 0$
    \If{$p$\text{ is even}}
        \State $\mtx{S} = \code{SketchOpGen}(n, k)$
    \Else
        \State $\mtx{S} = \mtx{A}^{\trans} \code{SketchOpGen}(m, k)$
        \State $p_{\mathrm{done}} = p_{\mathrm{done}} + 1$
        \If{$p_{\mathrm{done}} \mod q = 0$}
            \State $\mtx{S} = \code{stabilizer}(\mtx{S})$
        \EndIf
    \EndIf
    \While{$p - p_{\mathrm{done}} \geq 2$}
        \State $\mtx{S} = \mtx{A}\mtx{S}$
        \State $p_{\mathrm{done}} = p_{\mathrm{done}} + 1$
        \If{$p_{\mathrm{done}} \mod q = 0$}
            \State $\mtx{S} = \code{stabilizer}(\mtx{S})$
        \EndIf
        \State $\mtx{S} = \mtx{A}^{\trans} \mtx{S}$
        \State $p_{\mathrm{done}} = p_{\mathrm{done}} + 1$
        \If{$p_{\mathrm{done}} \mod q = 0$}
            \State $\mtx{S} = \code{stabilizer}(\mtx{S})$
        \EndIf
    \EndWhile
    \State \textbf{return} $\mtx{S}$
\EndIndent    
\end{algorithmic}
\end{algorithm}

\FloatBarrier


\subsection{RangeFinders and QB decompositions}\label{app:subsub:qb}

A general $\code{RangeFinder}$ takes in a matrix $\mtx{A}$ and a target rank parameter $k$, and returns a matrix $\mtx{Q}$ of rank $d = \min\{k, \rank(\mtx{A})\}$ such that the range of $\mtx{Q}$ is an approximation to the space spanned by $\mtx{A}$'s top $d$ left singular vectors. 

The rangefinder problem may also be viewed in the following way: given a matrix $\mtx{A} \in \mathbb{R}^{m \times n}$ and a target rank $k \ll \min(m, n)$, find a matrix $\mtx{Q}$ with $k$ columns such that the error $\|{\mtx{A} - \mtx{Q}\mtx{Q}^{\trans}\mtx{A}}\|$ is ``reasonably'' small.
Some \code{RangeFinder} implementations are iterative and can accept a target accuracy as a third argument.

The \code{RangeFinder} below, \code{RF1}, is very simple.
It relies on an implementation of the \code{TallSketchOpGen} interface (e.g., \code{TSOG1}) as well as the \code{Orth} interface.

\begin{algorithm}[H]
\setstretch{1.0}
\caption{\code{RF1} : a \code{RangeFinder} that orthogonalizes a single row sketch}\label{alg:rf1}
\begin{algorithmic}[1]
\State \textbf{function} $\code{RF1}(\mtx{A}, k)$\vspace{0.5pt} 
        \Indent
            \Statex \quad Inputs:
            \Statex \begin{quote}
                $\mtx{A}$ is $m \times n$, and $k \ll \min\{m, n\}$ is a positive integer
            \end{quote}\vspace{4pt}
            \Statex \quad Output:
            \Statex \begin{quote}
            $\mtx{Q}$ is a column-orthonormal matrix with $d = \min\{k, \rank\mtx{A}\}$ columns. We have $\range(\mtx{Q}) \subset \range(\mtx{A})$; it is intended that $\range(\mtx{Q})$ is an approximation to the space spanned by $\mtx{A}$'s top $d$ left singular vectors.
            \end{quote}\vspace{4pt}
            \Statex \quad Abstract subroutines and tuning parameters:
            \Statex \begin{quote}
                \code{TallSketchOpGen}
            \end{quote}\vspace{4pt}
            \setstretch{1.1}
    \State $\mtx{S} = \code{TallSketchOpGen}(\mtx{A},k)$ ~~\codecomment{$\mtx{S}$ is $n \times k$}
    \State $\mtx{Y} = \mtx{A}\mtx{S}$
    \State $\mtx{Q} = \code{orth}(\mtx{Y})$
    \State \textbf{return} $\mtx{Q}$
\EndIndent 
\end{algorithmic}
\end{algorithm}

The conceptual goal of QB decomposition algorithms is to produce an approximation $\|\mtx{A} - \mtx{QB}\| \leq \epsilon$ (for some unitarily-invariant norm), where $\rank(\mtx{QB}) \leq \min\{k, \rank(\mtx{A})\}$.
Our next three algorithms are different implementations of the \code{QBDecomposer} interface.
The first two of these algorithms require an implementation of the \code{RangeFinder} interface.
The ability of the implementation \code{QB1} to control accuracy is completely dependent on that of the underlying rangefinder.

\begin{algorithm}[H]
    \setstretch{1.0}
    \caption{\code{QB1} : a \code{QBDecomposer} that falls back on an abstract rangefinder}\label{alg:qb1}
    \begin{algorithmic}[1]
        \State \textbf{function} $\code{QB1}(\mtx{A}, k, \epsilon)$\vspace{0.5pt} 
        \Indent
            \Statex \quad Inputs:
            \Statex \begin{quote}
                $\mtx{A}$ is an $m \times n$ matrix and $k \ll \min\{m, n\}$ is a positive integer. \\
                $\epsilon$ is a target for the relative error
                $\|\mtx{A} - \mtx{Q}\mtx{B}\| / \|\mtx{A}\|$ measured in some unitarily-invariant norm.
                This parameter is passed directly to the \code{RangeFinder}, which determines its precise interpretation.
            \end{quote}\vspace{4pt}
            \Statex \quad Output:
            \Statex \begin{quote}
                $\mtx{Q}$ an $m \times d$ matrix returned by the underlying \code{RangeFinder} and $\mtx{B} = \mtx{Q}^{\trans}\mtx{A}$ is $d \times n$; we can be certain that $d \leq \min\{k, \rank(\mtx{A})\}$.
                The matrix $\mtx{Q}\mtx{B}$ is a low-rank approximation of $\mtx{A}$.
            \end{quote}\vspace{4pt}
            \Statex \quad Abstract subroutines and tuning parameters:
            \Statex \begin{quote}
                \code{RangeFinder}
            \end{quote}\vspace{4pt}
            \setstretch{1.1}
            \State $\mtx{Q} = \code{RangeFinder}(\mtx{A}, k, \epsilon)$
            \State $\mtx{B} = \mtx{Q}^{\trans}\mtx{A}$
            \State \textbf{return} $\mtx{Q}, \mtx{B}$
        \EndIndent 
    \end{algorithmic}
\end{algorithm}

The following algorithm builds up a QB decomposition incrementally.
It's said to be \textit{fully-adaptive} because it has fine-grained control over the error $\|\mtx{A} - \mtx{Q}\mtx{B}\|_{\text{F}}$.
If the algorithm is called with $k = \min\{m, n\}$, then its output will satisfy $\|\mtx{A} - \mtx{Q}\mtx{B}\|_{\text{F}} \leq \epsilon$.

\begin{algorithm}[H]
    \setstretch{1.0}
    \caption{\code{QB2} : a \code{QBDecomposer} that's fully-adaptive \\ (see \cite[Algorithm 2]{YGL:2018}) }\label{alg:qb2}
    \begin{algorithmic}[1]
        \State \textbf{function} $\code{QB2}(\mtx{A},k,\epsilon)$ \vspace{0.5pt} 
        \Indent
            \Statex \quad Inputs:
            \Statex \begin{quote}
                $\mtx{A}$ is an $m \times n$ matrix and $k \ll \min\{m, n\}$ is a positive integer. \\
                $\epsilon$ is a target for the relative error
                $\|\mtx{A} - \mtx{Q}\mtx{B}\|_{\mathrm{F}} / \|\mtx{A}\|_{\mathrm{F}}$.
                This parameter is used as a termination criterion upon reaching the desired accuracy. 
            \end{quote}\vspace{4pt}
            \Statex \quad Output:
            \Statex \begin{quote}
                $\mtx{Q}$ an $m \times d$ matrix combined of successive outputs from the underlying \code{RangeFinder} and $\mtx{B} = \mtx{Q}^{\trans}\mtx{A}$ is $d \times n$; we can be certain that $d \leq \min\{k, \rank(\mtx{A})\}$.
                The matrix $\mtx{Q}\mtx{B}$ is a low-rank approximation of $\mtx{A}$.
            \end{quote}\vspace{4pt}
            \Statex \quad Abstract subroutines:
            \Statex \begin{quote}
                 \code{RangeFinder}
            \end{quote}\vspace{4pt}
            \Statex \quad Tuning parameters:
            \Statex \begin{quote}
                $\mathrm{block{\_}size} \geq 1$ - at every iteration (except possibly for the final \\ 
                iteration), $\mathrm{block{\_}size}$ columns are added to the matrix $\mtx{Q}$.
            \end{quote}\vspace{4pt}
            \setstretch{1.1}
            \State $d = 0$
            \State $\mtx{Q} = [\ ] \in \R^{m \times d}$ \codecomment{Preallocation is dangerous; $k = \min\{m, n\}$ is allowed.}
            \State $\mtx{B} = [\ ] \in \R^{d \times n}$
            \State $\mathrm{squared{\_}error} = \|\mtx{A}\|_{\text{F}}^2$
            \While{$k > d$}
                \State $\mathrm{block{\_}size} = \min\{\mathrm{block{\_}size},~ k - d\}$
                \State $\mtx{Q}_i = \code{RangeFinder}(\mtx{A},\ \mathrm{block{\_}size})$
                 \label{alg:qb2:rangefinder} 
                \State $\mtx{Q}_i = \code{orth}(\mtx{Q}_i - \mtx{Q}\left(\mtx{Q}^{\trans}\mtx{Q}_i\right))$ \codecomment{for numerical stability}
                \State $\mtx{B}_i = \mtx{Q}_i^{\trans}\mtx{A}$ \codecomment{original matrix $\mtx{A}$ is valid here}
                \State $\mtx{B} = \begin{bmatrix} \mtx{B} \\ \ \mtx{B}_i \end{bmatrix}$
                \State $\mtx{Q} = \begin{bmatrix} \mtx{Q} & \mtx{Q}_i \end{bmatrix}$
                \State $d = d + \mathrm{block{\_}size}$
                \State $\mtx{A} = \mtx{A} - \mtx{Q}_i\mtx{B}_i$  \codecomment{modification can be implicit, but is required by Line \ref{alg:qb2:rangefinder} } 
                \State $\mathrm{squared{\_}error}  = \mathrm{squared{\_}error} - \|\mtx{B}_i\|_{\text{F}}^2$
                \codecomment{compute by a stable method}
                \If{$\mathrm{squared{\_}error} \leq \epsilon^2$}
                    \State \code{break}
                \EndIf
            \EndWhile
            \State \textbf{return} $\mtx{Q}, \mtx{B}$
        \EndIndent
    \end{algorithmic}
\end{algorithm}

Our third and final QB algorithm also builds up its approximation incrementally.
It is called \textit{pass-efficient} because it does not access the data matrix $\mtx{A}$ within its main loop (see \cite{dkm_matrix1} for the original definition of the pass-efficient model).
The algorithm can use a requested error tolerance as an early-stopping criterion.
This function should never be called with $k = \min\{m, n\}$.
We note that it takes a fair amount of algebra to prove that this algorithm produces a correct result.

\begin{algorithm}
    \caption{\code{QB3} : a \code{QBDecomposer} that's pass-efficient and partially adaptive (based on \cite[Algorithm 4]{YGL:2018}) }\label{alg:qb3}
    \begin{algorithmic}[1]
        \State \textbf{function} $\code{QB3}(\mtx{A}, k, \epsilon)$\vspace{0.5pt} 
        \Indent
            \Statex \quad Inputs:
            \Statex \begin{quote}
                $\mtx{A}$ is an $m \times n$ matrix and $k \ll \min\{m, n\}$ is a positive integer. \\
                $\epsilon$ is a target for the relative error
                $\|\mtx{A} - \mtx{Q}\mtx{B}\|_{\mathrm{F}} / \|\mtx{A}\|_{\mathrm{F}}$.
                This parameter is used as a termination criterion upon reaching the desired accuracy. 
            \end{quote}\vspace{4pt}
            \Statex \quad Output:
            \Statex \begin{quote}
                $\mtx{Q}$ an $m \times d$ matrix combined of successively-computed orthonormal bases $\mtx{Q}_i$ and $\mtx{B} = \mtx{Q}^{\trans}\mtx{A}$ is $d \times n$; we can be certain that $d \leq \min\{k, \rank(\mtx{A})\}$.
                The matrix $\mtx{Q}\mtx{B}$ is a low-rank approximation of~$\mtx{A}$.
            \end{quote}\vspace{4pt}
            \Statex \quad Abstract subroutines:
            \Statex \begin{quote}
                 \code{TallSketchOpGen}
            \end{quote}\vspace{4pt}
            \Statex \quad Tuning parameters:
            \Statex \begin{quote}
                $\mathrm{block{\_}size}$ is a positive integer; at every iteration (except possibly for the last), we add $\mathrm{block{\_}size}$ columns to $\mtx{Q}$.
            \end{quote}\vspace{4pt}
            \setstretch{1.1}
            \State $\mtx{Q} = [\ ] \in \R^{m \times 0}$ \codecomment{It would be preferable to preallocate.}
            \State $\mtx{B} = [\ ] \in \R^{0 \times n}$
            \State $\mathrm{squared{\_}error} = \|\mtx{A}\|_{\text{F}}^2$
            \State $\mtx{S} = \code{TallSketchOpGen}(\mtx{A}, k)$
            \State $\mtx{G} = \mtx{A}\mtx{S},~ \mtx{H} = \mtx{A}^{\trans}\mtx{G}$ \codecomment{Can be done in one pass over $\mtx{A}$}
            \State $\mathrm{max{\_}blocks} = \lceil k / \mathrm{block{\_}size} \rceil$
            \State $i = 0$
            \While{$i < \mathrm{max{\_}blocks}$}
                \State $b_{\text{start}} = i \cdot \mathrm{block{\_}size} +1$
                \State $b_{\text{end}} = \min\{ (i+1) \cdot \mathrm{block{\_}size},~ k\}$
                \State $\mtx{S}_i = \mtx{S}[:\ ,\ b_{\text{start}}:b_{\text{end}}]$
                \State $\mtx{Y}_i = \mtx{G}[:\ ,\ b_{\text{start}}:b_{\text{end}}] - \mtx{Q} (\mtx{B} \mtx{S}_i)$
                \State $\mtx{Q}_i, \mtx{R}_{i} = \code{qr}(\mtx{Y}_i)$ \codecomment{the next three lines are for numerical stability}
                \State $\mtx{Q}_i = \mtx{Q}_i - \mtx{Q}(\mtx{Q}^{\trans}\mtx{Q}_i)$
                \State $\mtx{Q}_i, \mtx{\hat{R}}_{i} = \code{qr}(\mtx{Q}_i)$
                \State $\mtx{R}_i = \mtx{\hat{R}}_i \mtx{R}_i$
                \State $\mtx{B}_i = \left(\mtx{H}[:\ ,\ b_{\text{start}}:b_{\text{end}}]\right)^{\trans} - (\mtx{Y}_i\mtx{Q})\mtx{B} - \left(\mtx{B}\mtx{S}_i\right)^{\trans}\mtx{B}$
                \State $\mtx{B}_i = \left(\mtx{R}_i^{\trans}\right)^{-1}\mtx{B}_i$ \codecomment{in-place triangular solve}
                \State $\mtx{B} = \begin{bmatrix} \mtx{B} \\ \ \mtx{B}_i \end{bmatrix}$
                \State $\mtx{Q} = \begin{bmatrix} \mtx{Q} & \mtx{Q}_i \end{bmatrix}$
                \State $\mathrm{squared{\_}error}  = \mathrm{squared{\_}error} - \|\mtx{B}_i\|_{\text{F}}^2$ \codecomment{compute by a stable method}
                \State $i = i + 1$
                \If{$\mathrm{squared{\_}error} \leq \epsilon^2$}
                    \State \code{break}
                \EndIf
            \EndWhile
            \State \textbf{return} $\mtx{Q}, \mtx{B}$
        \EndIndent
    \end{algorithmic}
\end{algorithm}

\FloatBarrier


\subsection{ID and subset selection}\label{subapp:ID_and_CSS}

As we indicated in \cref{subsubsec:TSID_CUR_driver,subsec:column_pivoted}, the collective design space of algorithms for ID, subset selection, and CUR is very large.
This appendix presents one randomized algorithm for one-sided ID (\cref{alg:osid1}) and an analogous randomized algorithm for subset selection (\cref{alg:rocs1}).
These algorithms are implemented in our Python prototype.
The Python prototype has two more randomized algorithms which are not reproduced here (one for one-sided and one for two-sided ID).

We need two deterministic functions in order to state these algorithms.
The first deterministic function -- called as $\mtx{Q},\mtx{R},J = \code{qrcp}(\mtx{F},k)$ -- returns data for an economic QR decomposition with column pivoting, where the decomposition is restricted to rank $k$ and may be incomplete.
The second deterministic function (\cref{alg:osid_qrcp}, below) is the canonical way to use QRCP for one-sided ID.
It produces a column ID when the final argument ``$\text{axis}$'' is set to one; otherwise, it produces a row ID. 
When used for column ID, it's typical for $\mtx{Y} \in \R^{\ell \times w}$ to be (very) wide and for $k$ to be only slightly smaller than $\ell$ (say, $\ell/2 \leq k \leq \ell$).

\begin{algorithm}[H]
    \setstretch{1.0}
    \caption{deterministic one-sided ID based on QRCP}\label{alg:osid_qrcp}
    \begin{algorithmic}[1]
        \State \textbf{function} $\code{osid{\_}qrcp}(\mtx{Y}, k, \text{axis})$\vspace{0.5pt} 
        \Indent
            \Statex \quad Inputs:
            \Statex \begin{quote}
                $\mtx{Y}$ is an $\ell \times w$ matrix, typically a sketch of some larger matrix. \\
                $k$ is an integer, typically close to $\min\{\ell, w\}$. \\
                $\text{axis}$ is an integer, equals $1$ for row ID and $2$ for column~ID. 
            \end{quote}\vspace{2pt}
            \Statex \quad Outputs:
            \Statex \begin{quote}
                When $\text{axis} = 1$:
                \begin{quote}
                    $\mtx{Z}$ is $\ell \times k$ and $I$ is a length-$k$ index vector.\\
                    Together, they satisfy $\mtx{Y}[I, :] = (\mtx{Z}\mtx{Y}[I, :])[I, :]$.
                \end{quote}
                When $\text{axis} = 2$: 
                \begin{quote}
                    $\mtx{X}$ is $k \times w$ and $J$ is a length-$k$ index vector.\\
                    Together, they satisfy $\mtx{Y}[\fslice,J] = (\mtx{Y}[\fslice,J]\mtx{X})[\fslice,J]$.
                \end{quote}
            \end{quote}\vspace{2pt}
            \Statex \quad Abstract subroutines:
            \Statex \begin{quote}
                \code{qrcp}
            \end{quote}\vspace{4pt}
            \setstretch{1.1}
            \If{axis == 2} 
            \State $(\ell, w) = $ the number of (rows, columns) in $\mtx{Y}$ 
            \State \textbf{assert} $k \leq \min\{\ell, w\}$
            \State $\mtx{Q}, \mtx{R}, J = \code{qrcp}(\mtx{Y}, k)$
            \State $\mtx{T} = \left(\mtx{R}[\lslice{k},\ \lslice{k}]\right)^{-1}\mtx{R}[\lslice{k}, \tslice{k+1}]$ \codecomment{use \code{trsm} from \BLASlev{3}}
            \State $\mtx{X} = \code{zeros}(k, w)$
            \State $\mtx{X}[:, J] = [\mtx{I}_{k \times k}, \mtx{T}]$
            \State $J = J[\lslice{k}]$
            \State \textbf{return} $\mtx{X}, J$
            \Else
            \State $\mtx{X}, I = \code{osid{\_}qrcp}(\mtx{Y}^{\trans}, k, \text{axis}=1)$
            \State $\mtx{Z} = \mtx{X}^{\trans}$
            \State \textbf{return} $\mtx{Z}, I$
            \EndIf
        \EndIndent 
    \end{algorithmic}
\end{algorithm}

The one-sided ID interface is
\[
    \mtx{M}, P = \code{OneSidedID}(\mtx{A}, k, s, \text{axis}).
\]
The output value $\mtx{M}$ is the interpolation matrix and $P$ is the length-$k$ vector of skeleton indices.
When $\text{axis} = 1$ we are considering a row ID and so obtain the approximation $\Aa = \mtx{M}\Ao[P, :]$ to $\Ao$.
When $\text{axis} = 2$, we are considering the low-rank column ID $\Aa = \Ao[:,P]\mtx{M}$.
Implementations of this interface perform internal calculations with sketches of rank $k + s$.

\begin{algorithm}[H]
    \setstretch{1.0}
    \caption{\code{OSID1} : implements \code{OneSidedID} by re-purposing an ID of a sketch. Besides the original source \cite[\S 5.1]{VM:2016:CUR}, more information on this algorithm can be found in \cite[\S 10.4]{Martinsson:2018_ish} and \cite[\S 13.4]{MT:2020}.}\label{alg:osid1}
    \begin{algorithmic}[1]
        \State \textbf{function} $\code{OSID1}(\mtx{A}, k, \text{axis})$\vspace{0.5pt} 
        \Indent
            \Statex \quad Inputs:
            \Statex \begin{quote}
                $\mtx{A}$ is an $m \times n$ matrix and $k \ll \min\{m, n\}$ is a positive integer. \\
                $\text{axis}$ is an integer, equal to $1$ for row ID or $2$ for column ID.
            \end{quote}\vspace{4pt}
            \Statex \quad Output:
            \Statex \begin{quote}
                A matrix $\mtx{Z}$ and vector $I$ satisfying $\mtx{Y}[I, :] = (\mtx{ZY}[I, :])[I, :]$ \\
                \textit{or} \\
                a matrix $\mtx{X}$ and vector $J$ satisfying $\mtx{Y}[:, J] = (\mtx{Y}[:, J] \mtx{X})[:, J]$.
            \end{quote}\vspace{4pt}
            \Statex \quad Abstract subroutines:
            \Statex \begin{quote}
                \code{TallSketchOpGen} and \code{osid{\_}qrcp}
            \end{quote}\vspace{4pt}
            \Statex \quad Tuning parameters:
            \Statex \begin{quote}
                $s$ is a nonnegative integer. The algorithm internally works with a sketch of rank $k + s$.
            \end{quote}\vspace{4pt}
            \setstretch{1.1}
            \If{axis == 1} \codecomment{row ID}
            \State $\mtx{S} = \code{TallSketchOpGen}(\mtx{A}, k + s)$
            \State $\mtx{Y} = \mtx{A}\mtx{S}$
            \State $\mtx{Z}, I = \code{osid{\_}qrcp}(\mtx{Y}, k, \text{axis}=0)$
            \State \textbf{return} $\mtx{Z}, I$
            \Else
            \State $\mtx{S} = \code{TallSketchOpGen}(\mtx{A}^{\trans}, k + s)^{\trans}$
            \State $\mtx{Y} = \mtx{S}\mtx{A}$
            \State $\mtx{X}, J = \code{osid{\_}qrcp}(\mtx{Y}, k, \text{axis}=1)$
            \State \textbf{return} $\mtx{X}, J$
            \EndIf
        \EndIndent 
    \end{algorithmic}
\end{algorithm}

Consider the following interface for (randomized) row and column subset selection algorithms
\[
    P = \code{RowOrColSelection}(\mtx{A},k,s,\text{axis}).
\]
The index vector $P$ and oversampling parameter is understood in the same way as the \code{OneSidedID} interface.
That is, $P$ is a partial permutation of the row index set $\idxs{m}$ (when $\text{axis} = 1$) or the column index set $\idxs{n}$ (when $\text{axis} = 2$).
Implementations are supposed to perform internal calculations with sketches of rank $k+s$.

\begin{algorithm}[H]
    \setstretch{1.0}
    \caption{\code{ROCS1} : implements \code{RowOrColSelection} by QRCP on a sketch}\label{alg:rocs1}
    \begin{algorithmic}[1]
        \State \textbf{function} $\code{ROCS1}(\mtx{A}, k, s, \text{axis})$\vspace{0.5pt} 
        \Indent
            \Statex \quad Inputs:
            \Statex \begin{quote}
                $\mtx{A}$ is an $m \times n$ matrix and $k \ll \min\{m, n\}$ is a positive integer. \\
                 $\text{axis}$ is an integer, equal to $1$ for row selection or $2$ for column selection.
            \end{quote}\vspace{4pt}
            \Statex \quad Output:
            \Statex \begin{quote}
                $I$: a row selection vector of length $k$ \\
                or \\
                $J$: a column selection vector of length $k$.
            \end{quote}\vspace{4pt}
            \Statex \quad Abstract subroutines:
            \Statex \begin{quote}
                \code{TallSketchOpGen}
            \end{quote}\vspace{4pt}
            \Statex \quad Tuning parameters:
            \Statex \begin{quote}
                $s$ is a nonnegative integer. The algorithm internally works with a sketch of rank $k + s$.
            \end{quote}\vspace{4pt}
            \setstretch{1.1}
            \If{axis == 1}
            \State $\mtx{S} = \code{TallSketchOpGen}(\mtx{A}, k + s)$
            \State $\mtx{Y} = \mtx{A}\mtx{S}$
            \State $\mtx{Q},\mtx{R},I = \code{qrcp}(\mtx{Y}^{\trans})$
            \State \textbf{return} $I[:k]$
            \Else
            \State $\mtx{S} = \code{TallSketchOpGen}(\mtx{A}^{\trans}, k + s)$
            \State $\mtx{Y} = \mtx{S}\mtx{A}$
            \State $\mtx{Q}, \mtx{R}, J = \code{qrcp}(\mtx{Y})$
            \State \textbf{return} $J[:k]$
            \EndIf
        \EndIndent 
    \end{algorithmic}
\end{algorithm}

\chapter[Correctness of Preconditioned Cholesky QRCP]{Correctness of Preconditioned \\ Cholesky QRCP}\label{app:cholqrcp}

In this appendix we prove \cref{prop:chol_exactness}.
Since this would involve a fair amount of bookkeeping if we used the notation of \cref{new_notation_QRRQR}, we begin with a more detailed statement of the algorithm.

Let $\mtx{A}$ be $m \times n$ and $\mtx{S}$ be $d \times m$ with $n \leq d \ll m$.
\begin{enumerate}
    \item Compute the sketch $\mtx{A}^{\mathrm{sk}} = \mtx{S}\mtx{A}$
    \item Decompose $[\mtx{Q}^{\mathrm{sk}},\mtx{R}^{\mathrm{sk}}, J] = \code{qrcp}(\mtx{A}^{\mathrm{sk}})$
    \begin{enumerate}
        \item $J$ is a permutation vector for the index set $\idxs{n}$.
        \item Abbreviating $\mtx{A}^{\mathrm{sk}}_J = \mtx{A}^{\mathrm{sk}}[\fslice,J]$, we have $\mtx{A}^{\mathrm{sk}}_J = \mtx{Q}^{\mathrm{sk}}\mtx{R}^{\mathrm{sk}}$. \label{a_sk_j_ref}
        \item Let $k = \rank(\mtx{A}^{\mathrm{sk}})$.
        \item $\mtx{Q}^{\mathrm{sk}}$ is $m \times k$ and column-orthonormal.
        \item $\mtx{R}^{\mathrm{sk}} = [\mtx{R}^{\mathrm{sk}}_1,~\mtx{R}^{\mathrm{sk}}_2]$ is $k \times n$ upper-triangular. \label{r_sk_def}
        \item $\mtx{R}^{\mathrm{sk}}_1$ is $k \times k$ and nonsingular.
    \end{enumerate}
    \item Abbreviate $\mtx{A}_J = \mtx{A}[\fslice{},J]$ and explicitly from $\mtx{A}^{\mathrm{pre}} = \mtx{A}_J[\fslice{},\lslice{k}] (\mtx{R}^{\mathrm{sk}}_{1})^{-1}$. \label{a_pre_def}
    \item Compute an unpivoted QR decomposition $\mtx{A}^{\mathrm{pre}} = \mtx{Q}\mtx{R}^{\mathrm{pre}}$. \label{a_pre_decomp}
    \begin{enumerate}
        \item If $\rank(\mtx{A}) = k$ then $\mtx{Q}$ is an orthonormal basis for the range of $\mtx{A}$.
        \item For the purposes of this appendix, it does not matter what algorithm we use to compute this decomposition. We assume the decomposition is~exact.
    \end{enumerate}
    \item Explicitly form $\mtx{R} = \mtx{R}^{\mathrm{pre}}\mtx{R}^{\mathrm{sk}}$\label{r_def}
\end{enumerate}

The goal of this proof is to show that the equality $\mtx{A}[:, J] = \mtx{QR}$ holds under the assumption that $\rank(\mtx{S}\mtx{A}) = \rank(\mtx{A})$.
Let us first establish some useful identities. 
By steps \ref{a_pre_def} and \ref{a_pre_decomp} of the algorithm description above, we know that
\[
    \mtx{R}^{\mathrm{pre}} = \mtx{Q}^{ \trans}\mtx{A}_{J}[:,  \fslice{k}](\mtx{R}^{\mathrm{sk}}_{1})^{-1}.
\]
Combining this with the characterization of $\mtx{R}$ from Steps \ref{r_sk_def} and \ref{r_def}, we have
\begin{equation*}
\mtx{R} = \mtx{Q}^{ \trans}\mtx{A}_{J}[:,  \fslice{k}](\mtx{R}^{\mathrm{sk}}_{1})^{-1}[\mtx{R}^{\mathrm{sk}}_{1},~ \mtx{R}^{\mathrm{sk}}_{2}].
\end{equation*}
We may further expand this expression as such:
\begin{equation*}
\mtx{R} = \mtx{Q}^{ \trans}\mtx{A}_{J}[:,  \fslice{k}][\mtx{I}_{k \times k},~ (\mtx{R}^{\mathrm{sk}}_{1})^{-1}\mtx{R}^{\mathrm{sk}}_{2}].    
\end{equation*}
Since $\mtx{Q}$ is an orthonormal basis for the range of $\mtx{A}$ and, consequently, $\mtx{A}_{J}$, we have~that
\begin{equation}\label{eq:key_for_cholqrcp_proof}
\mtx{Q}\mtx{R} = \mtx{A}_{J}[:,  \fslice{k}][\mtx{I}_{k \times k},~ (\mtx{R}^{\mathrm{sk}}_{1})^{-1}\mtx{R}^{\mathrm{sk}}_{2}]. 
\end{equation}
We use \eqref{eq:key_for_cholqrcp_proof} to establish the claim by a columnwise argument.
That is, we show that $\mtx{QR}[:, \ell] = \mtx{A}_{J}[:, \ell]$ for all $1 \leq \ell \leq n$. 

First, consider the case when $\ell \leq k$. 
Let $\vct{\delta}^{n}_{\ell}$ be the $\ell^{\text{th}}$ standard basis vector in $\R^{n}$. 
Then, consider the following series of identities:
\begin{align*}
\mtx{QR}[:, \ell] & = \mtx{QR}\vct{\delta}^{n}_{\ell} \\
                  & = \mtx{A}_{J}[:,  \fslice{k}][\mtx{I}_{k \times k},~ (\mtx{R}^{\mathrm{sk}}_{1})^{-1}\mtx{R}^{\mathrm{sk}}_{2}]\vct{\delta}^{n}_{\ell} \\
                  & =  \mtx{A}_{J}[:,  \fslice{k}]\vct{\delta}^{k}_{\ell}  = \mtx{A}_{J}[:, \ell], 
\end{align*}
hence the desired statement holds for $\ell \leq k$.

It remains to show that $\mtx{QR}[:, \ell] = \mtx{A}_{J}[:, \ell]$ for $\ell > k$.
Note that 
\begin{align*}
    \mtx{Q}\mtx{R}[:, \ell]  & = \mtx{A}_{J}[:,  \fslice{k}][\mtx{I}_{k \times k},~ (\mtx{R}^{\mathrm{sk}}_{1})^{-1}\mtx{R}^{\mathrm{sk}}_{2}]\vct{\delta}^{n}_{\ell} \\
                       & = \mtx{A}_{J}[:,  \fslice{k}]((\mtx{R}^{\mathrm{sk}}_{1})^{-1}\mtx{R}^{\mathrm{sk}}_{2})[:, \ell - k].
\end{align*}
Let $\vct{\gamma} = ((\mtx{R}^{\mathrm{sk}}_{1})^{-1}\mtx{R}^{\mathrm{sk}}_{2})[:, \ell - k]$.
Therefore, in order to obtain the desired identity for $\ell > k$, we will need to show that 
\begin{equation*}
\mtx{A}_{J}[:, k]\vct{\gamma} = \mtx{A}_{J}[:, \ell].
\end{equation*}

\begin{proposition} \label{prop: similarity}
    If $\mtx{A}^{\mathrm{sk}}_{J}[:, \ell] = \mtx{A}^{\mathrm{sk}}_{J}[:,  \fslice{k}]\vct{u}$ for some $\vct{u} \in \R^{k}$, then \\ $\mtx{A}_{J}[:, \ell] = \mtx{A}_{J}[:,  \fslice{k}]\vct{u}$.
\end{proposition}
\begin{proof}
    To simplify notation, define the $m \times k$ matrix $\mtx{X} = \mtx{A}_{J}[:,  \fslice{k}]$ and the $m$-vector $\vct{y} = \mtx{A}_{J}[:, \ell]$.
    
    Suppose to the contrary that $\vct{y} \neq \mtx{X}\vct{u}$ and $\mtx{S}\vct{y} = \mtx{S}\mtx{X}\vct{u}$.
    Then,  $\mtx{S}\mtx{X}\vct{u} - \mtx{S}\vct{y} = 0$.
    Define $U = \mathrm{ker}(\mtx{S}[\mtx{X},~ \vct{y}])$ and $V = \mathrm{ker}([\mtx{X},~ \vct{y}])$. Clearly, $U$ contains $V$. Additionally, if $U$ contains a nonzero vector that is not in $V$, then $\mathrm{dim}(U) > \mathrm{dim}(V)$. This would further imply that $\rank(\mtx{S}[\mtx{X},~ \vct{y}]) < \rank([\mtx{X},~ \vct{y}])$.
    
    If $\mtx{S}\mtx{X}\vct{u} - \mtx{S}\vct{y} = 0$, then $(\vct{u}, -1)$ is a nonzero vector in $U$ that is not in $V$. However, by our assumption, the sketch does not drop rank. Consequently, no such vector $(\vct{u}, -1)$ can exist, and we must have $\vct{y} = \mtx{X}\vct{u}$.
\end{proof}

We now prove that $\mtx{A}^{\mathrm{sk}}_{J}[:,  \fslice{k}]\vct{\gamma} = \mtx{A}^{\mathrm{sk}}_{J}[:, \ell]$.
To do this, start by noting that $\mtx{A}^{\mathrm{sk}}_{J}[:,  \fslice{k}] = \mtx{Q}^{\mathrm{sk}}\mtx{R}^{\mathrm{sk}}_{1}$.
Plugging in the definition of $\vct{\gamma}$, we have
\begin{equation*}
\mtx{A}^{\mathrm{sk}}_{J}[:,  \fslice{k}]\vct{\gamma} =  \mtx{Q}^{\mathrm{sk}}\mtx{R}^{\mathrm{sk}}_{1}(\mtx{R}^{\mathrm{sk}}_{1})^{-1} (\mtx{R}^{\mathrm{sk}}_{2})[:, \ell - k] = \mtx{Q}^{\mathrm{sk}}(\mtx{R}^{\mathrm{sk}}_{2})[:, \ell - k] .
\end{equation*}
The next step is to use the simple observation that $\mtx{R}^{\mathrm{sk}}_{2}[:, \ell - k] = \mtx{R}^{\mathrm{sk}}[:,\ell]$ to find 
\begin{equation*}
\mtx{A}^{\mathrm{sk}}_{J}[:,  \fslice{k}]\vct{\gamma} = (\mtx{Q}^{\mathrm{sk}}\mtx{R}^{\mathrm{sk}})[:, \ell] = \mtx{A}^{\mathrm{sk}}_{J}[:, \ell].
\end{equation*}
Combining the above results and \cref{prop: similarity} proves \cref{prop:chol_exactness}.

\chapter[Bootstrap Methods for Error Estimation]{Bootstrap Methods
\\ for Error Estimation}
\label{app:bootstrap}

\minitoc
\bigskip


Whenever a randomized algorithm produces a solution, a question immediately arises: Is the solution sufficiently accurate?
In many situations, it is possible to estimate numerically the error of the solution using the available problem data --- a process that is often referred to as \emph{(a posteriori) error estimation}.%
\footnote{This should be contrasted with \emph{(a priori) error bounds} often used in theoretical development of \RandNLA{} algorithms, in which one bounds rather than estimates the error, and does so in a worst-case way that does not depend on the problem data.}
In addition to resolving uncertainty about the quality of a solution, another key benefit of error estimation is that it enables computations to be done more adaptively. 
For instance, error estimates can be used to determine if additional iterations should be performed, or if tuning parameters should be modified.
In this way, error estimates can help to incrementally refine a rough initial solution so that ``just enough'' work is done to reach a desired level of accuracy.
\\

%
%

In this appendix, we provide a brief overview of \textit{bootstrap methods} for error estimation in \RandNLA{}.
Up to now, these tools (which are common in statistics and statistical data analysis) have been designed for a handful of sketch-and-solve type algorithms, and the development of bootstrap methods for a wider range of randomized algorithms is an open direction of research.
Our main purpose in writing this appendix is to record the consideration we have given to bootstrap methods.
Our secondary purpose is to provide a starting point for non-experts to survey this literature as it evolves.

\section{Bootstrap methods in a nutshell}

Bootstrap methods have been studied extensively in the statistics literature for more than four decades, and they comprise a very general framework for quantifying uncertainty \cites{Efron:1994:Bootstrap_book,Shao:2012:jackknife_bootstrap}.
One of the most common uses of these methods in statistics is to assess the accuracy of parameter estimates.
This use-case provides the connection between bootstrap methods and error estimation in \RandNLA{}.
Indeed, an exact solution to a linear algebra problem can be viewed as an ``unknown parameter,'' and a randomized algorithm can be viewed as providing an ``estimate'' of that parameter.
Taking the analogy a step further, a random sketch of a matrix can also be viewed as a ``dataset'' from which the estimate of the ``population'' quantity is computed. Likewise, when bootstrap methods are applied in \RandNLA{}, the rows or columns of a sketched matrix often play the role of ``data vectors''.
 
We now formulate the task of error estimation in a way that is convenient for discussion of bootstrap methods.
First, suppose the existence of some fixed but unknown ``true parameter'' $\theta \in \R$.
Suppose we estimate this parameter by a value $\hat\theta$ depending on random samples from some probability distribution.
The error of $\hat\theta$ is defined as $\hat{\epsilon} = |\hat\theta-\theta|$, which we emphasize is both random and unknown.
From this standpoint, it is natural to seek the tightest upper bound on $\hat{\epsilon}$ that holds with a specified probability, say $1-\alpha$.
This ideal bound is known as the $(1-\alpha)$-quantile of $\hat{\epsilon}$, and is defined more formally as
\[
    q_{1-\alpha}=\inf\{t\in[0,\infty) \, | \, \mathbb{P}(\hat{\epsilon}\leq t)\geq 1-\alpha\}.
\]
An error estimation problem is considered solved if it is possible to construct a quantile estimate $\hat q_{1-\alpha}$ such that the inequality $\hat{\epsilon}\leq \hat q_{1-\alpha}$ holds with probability that is close to $1-\alpha$. 

The bootstrap approach to estimating $q_{1-\alpha}$ is based on imagining a scenario where it is possible to generate many independent samples $\doublehat{\epsilon}_1,\dots,\doublehat{\epsilon}_N$ of the random variable $\hat{\epsilon}$.
Of course, this is not possible in practice, but if it were, then an estimate of $q_{1-\alpha}$ could be easily obtained using the empirical $(1-\alpha)$-quantile of the samples $\doublehat{\epsilon}_1,\dots,\doublehat{\epsilon}_N$. 
The key idea that bootstrap methods use to circumvent the difficulty is to generate ``approximate samples'' of $\hat{\epsilon}$, which \emph{can} be done in practice. 

To illustrate how approximate samples of $\hat{\epsilon}$ can be constructed, consider a generic situation where the estimate $\hat\theta$ is computed as a function of a dataset $X_1,\dots,X_n$.
That is, suppose $\hat\theta=f(X_1,\dots,X_n)$ for some function $f$.
Then, a \textit{bootstrap sample} of $\hat{\epsilon}$, denoted $\doublehat{\epsilon}$, is computed as follows: 
\begin{itemize}
\item Sample $n$ points $\{\doublehat{X}_i\}_{i=1}^n$ with replacement from the original dataset $\{X_i\}_{i=1}^n$.
\item Compute $\doublehat{\theta} := f(\doublehat{X}_1,\dots,\doublehat{X}_n)$
\item Compute $\doublehat{\epsilon} :=|\doublehat{\theta} - \hat\theta|$.
\end{itemize}
By performing $N$ independent iterations of this process, a collection of bootstrap samples $\doublehat{\epsilon}_1,\dots,\doublehat{\epsilon}_N$ can be generated.
Then, the desired quantile estimate $\hat q_{1-\alpha}$ can be computed as the smallest number $t\geq 0$ for which the inequality
\[
\frac{1}{N}\sum_{i=1}^N \mathbb{I}\{\doublehat{\epsilon}_i \leq t\}\geq 1-\alpha
\]
is satisfied, where $\mathbb{I}\{\cdot\}$ refers to the $\{0,1\}$-valued indicator function.
This quantity is also known as the empirical $(1-\alpha)$-quantile of $\doublehat{\epsilon}_1,\dots,\doublehat{\epsilon}_n$.
We will sometimes denote it by $\text{quantile}[\doublehat{\epsilon}_1,\dots,\doublehat{\epsilon}_n; 1-\alpha]$.

To provide some intuition for the bootstrap, the random variable $\doublehat\theta$ can be viewed as a ``perturbed version'' of $\hat\theta$, where the perturbing mechanism is designed so that the deviations of $\doublehat\theta$ around $\hat\theta$ are statistically similar to the deviations of $\hat\theta$ around $\theta$~\cite{Efron:1994:Bootstrap_book}.
Equivalently, this means that the histogram of $\doublehat{\epsilon}_1,\dots,\doublehat{\epsilon}_N$ will serve as a good approximation to the distribution of the actual random error variable $\hat{\epsilon}$.
Furthermore, it turns out that this approximation is asymptotically valid (i.e., $n\rightarrow \infty$) and supported by quantitative guarantees in a broad range of situations~\cite{Shao:2012:jackknife_bootstrap}.

\section{Sketch-and-solve least squares}\label{subapp:bootstrap:least_squares}

There is a direct analogy between the discussion above and the setting of sketch-and-solve algorithms for least squares. First, the ``true parameter'' $\theta$ is the exact solution $\vct{x}_{\star} = \textup{argmin}_{\vct{x}\in{\R}^n}\|\mtx{A}\vct{x}-\vct{b}\|_2^2$.
Second, the dataset $X_1,\dots,X_n$ corresponds to the sketches $[\hat{\mtx{A}},\hat{ \vct{b}}]=\mtx{S}[\mtx{A},\vct{b}]$.
Third, the estimate $\hat{\theta}$ corresponds to the sketch-and-solve solution  $\hat{\vct{x}} = \textup{argmin}_{\vct{x}\in{\R}^n}\|\hat{\mtx{A}} \vct{x}-\hat{\vct{b}}\|_2^2$.
Fourth, the error variable can be defined as $\hat{\epsilon} = \rho(\hat{\vct{x}}, \vct{x}_{\star})$, for a preferred metric $\rho$, such as that induced by the $\ell_2$ or $\ell_{\infty}$ norms.

Once these correspondences are recognized, the previous bootstrap sampling scheme can be applied.
For further background, as well as extensions to error estimation for iterative randomized algorithms for least squares, we refer to~\cite{LWM:2018:lstsq_err_est}.\\

%

\begin{method}[!ht]
{
	\caption{(Bootstrap error estimation for sketch-and-solve least squares).}\label{alg:bootstrap:ols}
	\normalfont
	\vspace{0.1cm}
	{\bf Input:} A positive integer $B$, the sketches $\hat{\mtx{A}}\in\R^{d\times n}$, $\hat{\vct{b}} \in \R^d$, and $\hat{\vct{x}} \in \R^n$.\\[0.2cm]
	{\bf For } $\ell \in \idxs{B}$\; {\bf do in parallel} 
	\begin{enumerate}
			\item Draw a vector $I:=(i_1,\dots,i_d)$ by sampling $d$ numbers with replacement from $\idxs{d}$.
			\item Form the matrix $\doublehat{\mtx{A}}:=\hat{\mtx{A}}[I,:]$, and vector $\doublehat{\vct{b}}:=\hat{\vct{b}}[I]$. 
			\item Compute the following vector and scalar,
			\begin{equation}\label{eqn:tildexalg}
			\doublehat{\vct{x}} := \argmin_{\vct{x}\in\R^n}\|\doublehat{\mtx{A}} \vct{x}-\doublehat{\vct{b}}\|_2 \text{ \ \  \ \ \ and \ \ \ \ \  } \doublehat{\epsilon}_{\ell} :=\|\doublehat{\vct{x}}-\hat{\vct{x}}\|.
			\end{equation}
	\end{enumerate}
	{\bf Return:} The estimate $\text{quantile}[\doublehat{\epsilon}_1,\dots,\doublehat{\epsilon}_B; 1-\alpha]$ for the $(1-\alpha)$-quantile of $\|\hat{\vct{x}}-\vct{x}_{\star}\|$.
}
\end{method}
To briefly comment on some of the computational characteristics of this method, it should be emphasized that the for loop can be implemented in an embarrassingly parallel manner, which is typical of most bootstrap methods. Second, the method only relies on access to sketched quantities, and hence does not require any access to the full matrix $\mtx{A}$. Likewise, the computational cost of the method is independent of the number of rows of $\mtx{A}$.

\section{Sketch-and-solve one-sided SVD}\label{subapp:bootstrap:svd}

We call the problem of computing the singular values and right singular vectors of a matrix a ``one-sided SVD.''
We further use the term ``sketch-and-solve one-sided SVD'' for an algorithm that approximates the top $k$ singular values and singular vectors of $\mtx{A}$ by those of a sketch $\hat{\mtx{A}} = \mtx{S}\mtx{A}$.
Here we consider estimating the error incurred by such an algorithm.
As matters of notation, we let $\{(\sigma_j,\vct{v}_j)\}_{j=1}^k$ denote the top $k$ singular values and right singular vectors of $\mtx{A}$ and $\{(\hat{\sigma}_j,\hat{\vct{v}}_j)\}_{j=1}^k$ the corresponding quantities for $\hat{\mtx{A}}$.
We suppose that error is measured uniformly over $j\in\idxs{k}$, which leads us to consider error variables of the form 
\[
\e_{_{\Sigma}}:=\max_{j\in\idxs{k}}|\hat{\sigma}_j-\sigma_j|\quad\text{and}\quad\e_{_{V}}:=\max_{j\in\idxs{k}}\rho(\hat{\vct{v}}_j,\vct{v}_j).
\]
The following bootstrap method, developed in~\cite{LEM:2020:svd_err_est}, provides estimates for the $(1-\alpha)$-quantiles of $\e_{_{\Sigma}}$ and $\e_{_{V}}$.

\begin{method}[!ht]
{
	\caption{(Bootstrap error estimation for sketch-and-solve SVD).}\label{alg:bootstrap:svn}
	\normalfont
	\vspace{0.1cm}
	\noindent {\bf Input}: The sketch $\hat{\mtx{A}}\in\R^{d\times n}$ and its top $k$ singular values and right singular vectors $(\hat{\sigma}_1,{\hat{\vct{v}}}_1),\dots,(\hat{\sigma}_k,{\hat{\vct{v}}}_k)$, a number of samples $B$, a parameter $\alpha\in(0,1)$. \\[-0.1cm]
	
	$\bullet $ {\bf For } $\ell \in \idxs{B}$\; {\bf do \ in \ parallel}
	\begin{enumerate}
		\item Form $\doublehat{\mtx{A}} \in {\R}^{d\times n}$ by sampling $d$ rows from $\hat{\mtx{A}}$ with replacement. 
		\item Compute the top $k$ singular values and right singular vectors of $\doublehat{\mtx{A}}$, denoted as $ \doublehat{\sigma}_1,\dots,\doublehat{\sigma}_k$ and ${\doublehat{\vct{v}}}_1,\dots,{\doublehat{\vct{v}}}_k$. Then, compute the bootstrap samples
		%
		\begin{align}
		& \doublehat\e_{_{\mtx{\Sigma},\ell}}:=\max_{j\in\idxs{k}} |\doublehat{\sigma}_j - \hat{\sigma}_j|\label{eqn:Sigmasamples}\\[0.2cm]
		& \doublehat\e_{_{\mtx{V},\ell}}:=\max_{j\in\idxs{k}} \rho( {\doublehat{\vct{v}}}_j,{\hat{\vct{v}}}_j).\label{eqn:Vsamples}
		\end{align}
	\end{enumerate}
	{\bf Return:} The estimates $\text{quantile}[\doublehat\e_{_{\mtx{\Sigma},1}},\dots,\doublehat \e_{_{\mtx{\Sigma},B}};1-\alpha]$ and $\text{quantile}[\doublehat\e_{_{\mtx{V},1}},\dots,\doublehat \e_{_{\mtx{V},B}};1-\alpha]$ for the $(1-\alpha)$-quantiles of $\e_{_{\Sigma}}$ and  $\e_{_{V}}$.
}
\end{method}
Although this method is only presented with regard to singular values and right singular vectors, it is also possible to apply a variant of it to estimate the errors of approximate left singular vectors. However, a few extra technical details are involved, which may be found in~\cite{LEM:2020:svd_err_est}.

Another technique to estimate error in the setting of sketch-and-solve one-sided SVD is through the spectral norm $\|\hat{\mtx{A}}^{\trans}\hat{\mtx{A}}-\mtx{A}^{\trans}\mtx{A}\|_2$.
Due to the Weyl and Davis-Kahan inequalities, an upper bound on $\|\hat{\mtx{A}}^{\trans}\hat{\mtx{A}}-\mtx{A}^{\trans}\mtx{A}\|_2$ directly implies upper bounds on the errors of all the sketched singular values $\hat{\sigma}_1,\dots,\hat{\sigma}_n$ and sketched right singular vectors $\hat{\vct{v}}_1,\dots,\hat{\vct{v}}_n$.
Furthermore, the quantiles of the error variable $\|\hat{\mtx{A}}^{\trans}\hat{\mtx{A}}-\mtx{A}^{\trans}\mtx{A}\|_2$ can be estimated via the bootstrap, as shown in~\cite{LEM:2019:opnorm_err_est}.

\clearpage
\fancyhead{}
\fancyhead[LE, RO]{\slshape Bibliography}

\ifUseBIBLATEX
    \printbibliography[]
\else
    \small
    \bibliographystyle{alpha}
    \bibliography{references/refs}
\fi
\addcontentsline{toc}{chapter}{Bibliography}
\end{document}